\documentclass[12pt]{article}
\usepackage{amsmath}
\usepackage{amssymb}
\usepackage{amsthm}
\usepackage{natbib}
\usepackage{ragged2e}
\usepackage{geometry}
\usepackage{hyperref}
\usepackage{algorithm}
\usepackage{graphicx} 
\usepackage{algpseudocode}
\usepackage{authblk}
\usepackage[capitalise,nameinlink
    ,noabbrev]{cleveref}
  \hypersetup{
    colorlinks=true,
    linkcolor=blue,
    citecolor=blue
}
\usepackage{autonum}
\usepackage{xcolor}
\usepackage{multirow}  
\usepackage{setspace}
\usepackage{pifont}

\newcommand{\R}{\mathbb{R}}

\newcommand{\mf}{\mathbf}
\newcommand{\bs}{\boldsymbol}

\newcommand{\bb}{\overline{b}}

\newcommand{\h}{\mathbf{h}}
\newcommand{\kk}{\mathbf{k}}
\newcommand{\G}{\mathcal G}
\newcommand{\tr}{\mathrm{tr}}
\newcommand{\s}{\mathbf{s}}
\renewcommand{\u}{\mathbf{u}}

\newcommand{\X}{\widetilde{X}}
\newcommand{\bS}{\mathbf{S}}
\newcommand{\w}{\mathbf{w}}

\newtheorem{theorem}{Theorem}[section]
\newtheorem{assumption}{Assumption}[section]

\newtheorem{proposition}{Proposition}[section]

\newtheorem{remark}{Remark}[section]
\newtheorem{lemma}{Lemma}[section]
\newtheorem{corollary}{Corollary}[section]
\definecolor{darkgreen}{rgb}{0.0, 0.4, 0.1} 

\newcommand{\x}{\mathbf{x}}

\title{Validating spatial-temporal separability for stationary  processes}
\author[1]{Lujia Bai}
\author[1]{Holger Dette}
\author[1]{Zihao Yuan}
\affil[1]{Department of Mathematics, Ruhr University Bochum}
\begin{document}

\maketitle

\begin{abstract}
A crucial assumption to reduce computational complexity  in spatial-temporal data  analysis is separability, which factors the covariance structure  into a purely spatial and a purely temporal component. In this paper, we develop statistical inference tools for validating this assumption for a second-order stationary process  under  both domain-expanding-infill  asymptotics
and domain-expanding  asymptotics. In contrast to previous work on this subject, the methodology neither requires the assumption of normally distributed data, nor uses spectral methods. Our approach is based on nonparametric estimates of measures for the  deviation between  the covariance matrix  and separable approximations, which vanish if and only 
if the  assumption of separability is satisfied. We derive the asymptotic distributions of appropriate estimators for these measures with non-standard limiting distributions and use these results to develop inference tools for validating the assumption of separability. More specifically, we derive confidence intervals for the deviation measures, tests  for the  hypothesis of exact separability, and  for the hypothesis that the deviation from separability is smaller than a prespecified threshold.
\end{abstract}

{\it Keywords and phrases:} best rank-one approximation,  central limit theorem for dependent data, nonparametrics, spatial-temporal data,  $U$-statistics

{\it AMS Classification:} 62G10, 62G20, 62M30

\section{Introduction}
\label{sec introduction}
   \def\theequation{1.\arabic{equation}}	
   \setcounter{equation}{0}

Space--time data arise ubiquitously in environmental science, climatology, epidemiology, and econometrics, where accurate modeling of dependence structures is essential for inference and prediction. A central object in such analyses is the space-time covariance function, which governs both spatial dependence and temporal evolution.  Nowadays,  technological advances often  yield  ultra-high-dimensional   space-time data  and
 a completely nonparametric estimation of the covariance  operator is computationally very expensive (if not impossible). For example, common  prediction methods require precise estimation and inversion  of the 
covariance matrix. If the temporal dimension is  $500$ and the spatial dimension is $20$, one has to determine the inverse of a $10000 \times 10000$ matrix. A common approach to deal with such difficulties is  the assumption of a separable covariance structure  where the covariance factors into purely spatial and purely temporal components \citep[see]{KyrJour-review-1999,GneiGentGut2007,MarcGenton2007,Sher2011,CresWik-book-2011,Cres-book-2015}. This yields, in the above-mentioned example, an inversion of a $500 \times 500$ and a $20 \times 20 $ matrix.

Despite its importance, separability is rarely directly observed, and falsely imposing such an  assumption may cause a  substantial bias in estimation, and prediction and as a consequence misleading  conclusions. Therefore, the  assumption of separability must be  carefully  justified in applications, and several    statistical tests
for this hypothesis  have been developed in the literature.  For example,
\cite{MatYaj2004,ScaMar2005,Fuentes2006,CruFer-CasGon-Man2010} proposed tests based on spectral methods for stationarity, axial symmetry, and separability. In    \citet{MitGentGump2005,MitGentGump2006},  likelihood ratio tests under the assumption of a normal distribution were  investigated.
  Finite-dimensional nonparametric tests for separability were considered in \cite{LiGentSher2007}, and \cite{BevilacquaMateuPorcuZhangZini2010} considered the problem of testing the assumption of separability in parametric models.
  More recently, a considerable amount of literature considers the problem    in a functional and nonparametric framework, where  independent data $X_1(s,t), \ldots , X_n(s,t) $  are observed and one is interested in testing the  separability of the covariance structure    Cov($X_1(s,t),X_1(s',t')$) \citep[see][]{ConstKokReim2017,AstPigTav17,constantinou2018testing,pramita2020,DetDieKut21,DetteDierickxSkutta2025}.

In this paper, we develop novel nonparametric inference for validating the assumption of   separability in the  space-time covariance structure 
 \begin{align}
\label{d2x}
    C(\h, v) = \mathrm{Cov}\{X_{(\mf s, t)}, X_{(\mf s + \mf h   , t+v)}\}
\end{align}
of a zero-mean, second-order stationary spatial-temporal process  $\{X_{(\mf s, t)}: \mf s \in \R^2, t \in \mathbb{N} \}$ from  observations sampled at   irregular spatial and regular temporal locations. Our approach is based on two measures for the deviation between the  covariance matrix of the process and a separable approximation. The measures  which will vanish if and only if the covariance structure is  separable.  We construct consistent estimators of these  deviation measures with   tractable limiting distributions under weak dependence and irregular sampling designs.
We then use these results to  construct several statistical inference tools. First, we develop  tests for exact separability, which are consistent against a broad class of non-separable alternatives. In contrast to the existing methods on testing this assumption, our tests work in the spatial-temporal domain and do not require the assumption of a Gaussian process.
Second, the assumption  of exact separability is rarely  made because one believes in it, but because of the enormous computational savings if it is at least approximately satisfied. To address this issue, we develop asymptotically valid confidence intervals for the measures of deviation from separability as an alternative to testing for exact separability.
Third, as a further alternative,  we construct novel testing procedures for the hypothesis  that {\it the deviation  from separability is small.} More recently, hypotheses of this type have attracted considerable interest within the framework of {\it tolerant testing}, where one assesses whether the data are consistent with any distribution lying within a prescribed neighborhood of the candidate \citep[see, for example,][and the references therein]{canonne22a, kania2026testingimprecisehypotheses}. 

Section~\ref{sec2} introduces the modeling framework, formalizes the  separability considered in this paper, and defines two measures of deviation from separability as well as a nonparametric  estimator for the covariance function.  In Section~\ref{sec3}, we investigate  the  asymptotic properties of this estimator under 
 both \textit{domain-expanding-infill (DEI) asymptotics} and \textit{domain-expanding (DE) asymptotics} \citep[see][]{lu2014nonparametric}. These results are used in Section \ref{sec4} for the development of the statistical methodology for validating the separability  assumption, whose properties are examined by means of a simulation study  in Section \ref{sec5}. Finally, all proofs of the theoretical results are deferred to the Appendix.

\section{Two measures for deviations from separability}
\label{sec2}
 \def\theequation{2.\arabic{equation}}	
   \setcounter{equation}{0}

In this section, we introduce two measures for deviations from separability and a nonparametric estimator for the covariance matrix. The first measure is related to the concept of partial trace \citep[see][among others]{bhatia2003partial, constantinou2018testing, pramita2020, masak2023separable}  and the second is based on the concept of best rank-one approximation. Both measures are functions of the covariance matrix $\mf C$ of the observations at predefined  spatial-temporal locations (see equation  \eqref{d1} below for a precise definition), and  vanish if and only if the hypothesis of separability is satisfied. Estimates for these measures can then be obtained by replacing the matrix $\mf C$ by a nonparametric estimator introduced later in this section

We consider a second-order stationary spatial-temporal process $\{X_{(\s,t)}:(\s,t)\in \mathbb{R}^2\times \mathbb{N}\}$ and assume that  
at $n$ spatial locations $\{\s_{1n}, \ldots , \s_{nn}\}$,  we observe the time series $\{X_{(\s_{in},t)}:t=1,...,T\}$. 
Thus, the full data set   is given by 
\begin{align}
    \label{eq definition of data set} \textbf{Data}_n:= 
  \big \{ X_{(\s_{in},t)}   \big |~ i=1, \ldots , n; ~t=1, \ldots , T \big \}. 
\end{align}

Suppose that there exists a smooth function 
$\mathcal{C}: (\mathbb{R}^2\times \mathbb{N})^2 \to \mathbb{R}$  such that 
\begin{align}
\label{covariance}
    \text{Cov}(X_{(\s,t)},X_{(\s',t')}) =\mathcal{C}\big (\s,\s',t,t'\big)
\end{align}
holds for any given $n$ and $(\s,t)$, $ (\s',t')\in \mathbb{R}^2\times \mathbb{N}$. We are interested in checking   the assumption of separability at  a group of spatial and temporal locations,  say $\{(\s^*_i,t^*_j)\}_{1\leq i\leq M';1\leq j\leq N'}$, for some positive integers $M'$ and $N'$. Then, we have a well-defined $M'N'\times M'N'$-dimensional covariance matrix of these evaluation locations.  

Furthermore, we assume second-order stationarity of the  spatial-temporal process, that is 
\begin{align}
  C(\h, v):=  \mathcal{C}\big (\s,\s + \h ,t,t+ v \big)= \mathrm{Cov}(  X_{(\mf s, t)}, X_{(\mf s + \mf h   , t+v)}) ,
\end{align}
where the function $C$ is defined on  $\mathbb{R}^2\times \mathbb{R}_0^+$, and satisfies
\begin{align}
    \label{dx3}
\lim_{(\h,v)\to \infty}|\text{Cov}(X_{(\mf \s ,t )}, X_{(\s +\h,t+ v)})| = \lim_{(\h,v)\to \infty}| C (\h , v )| =0.
\end{align}
Thus, instead of the aforementioned $M'N'\times M'N'$ covariance matrix, we only need to focus on the following matrix,
\begin{align}\label{d1} 
    \mathbf C = \begin{pmatrix}
        C(\h_1, v_1),& \cdots,& C(\h_1, v_N)\\
        \cdots, &\cdots, &\cdots\\
         C(\h_M, v_1), &\cdots,& C(\h_M, v_N)
    \end{pmatrix}, 
\end{align}
where $\h_i$'s and $v_j$'s are \textit{evaluation distances} generated by the differences $\s^*_{k}-\s^*_{k'}$ and $t_l^*-t_{l'}^*$, i.e.,
\begin{align}
\begin{split}
\label{eq evaluation distance}
\{ \h_1 , \ldots \h_M\} &= \{ \s_k^* - \s_{k'}^* ~|~1 \leq k \leq k' \leq M' \} ~,\\
\{ v_1 , \ldots v_N\} & = \{ t_l^* - t_{l'}^* ~|~1 \leq l \leq l' \leq N' \}
.
\end{split}
\end{align}
 Note that we do not reflect the  dependence on $M$ and $N$ in the notation of the matrix $\mf C$, for example in \eqref{d1}, whenever this is clear from the  context.  If the (stationary) process is separable, the covariance satisfies $C(\h , v) =C_S(\h) C_T(v)$, where the functions 
 $C_S: \mathbb{R}^2\to \mathbb  R$ 
and $C_T:  \mathbb{R}
\to \mathbb  R$ 
depend only on  spatial and temporal lags, respectively. In this case, the matrix $\mathbf C$ can be rewritten as 
   \begin{align}
    \mathbf C = (C_S(\h_1), \ldots, C_S(\h_M))^{\top} (C_T(v_1),\ldots, C_T(v_N)),\label{eq:separable}
\end{align}
and  we call  the  matrix $\mathbf  C$ \textit{(grid) separable}.

\subsection{Measuring the deviation by a  partial trace approach}
\label{sec21}
We define a measure for the deviations of the matrix $\mf C$ in \eqref{d1} from separability in two steps, which is related to the concept of the  partial trace operator as considered in 
\cite{AstPigTav17,pramita2020} and \cite{DetDieKut21}
for the analysis of covariance operators of random surfaces. 
First, we are interested in measuring this deviation  for a fixed vector $\mf c_1^* \in \mathbb{R}^M$, that is 
\begin{align}
\label{d3}
    D^*(\mf c_1^*) 
    = \min_{\mf c_2 \in \mathbb R^N} \| \mathbf C -  {\mf c}_1^*   {\mf c}_2^{\top}\|_F^2 = \|\mf C\|_F^2 - \frac{\| (\mf c_1^*)^{\top} \mf C\|_F^2}{\|\mf c_1^*\|_F^2},
\end{align}
where $\| \mathbf{A} \|_F$ denotes the Frobenius norm of a matrix $ \mathbf{A}$ (for a vector, it is the Euclidean norm), and 
the second equality follows by the Cauchy-Schwarz inequality, which is achieved for $$(\mf c_2^*)^\top  = {(\mf c_1^*)^{\top} \mf C \over \|\mf c_1^*  \|_F^2 } .$$
In order to minimize \eqref{d3},  we note that in the case where $\mf C$  is separable, that is  $\mf C= \mf  a_1 \mf a_2^\top$,   an application of the Cauchy-Schwarz inequality gives 
$$
D^*(\mf c_1^*) = \|\mf C\|_F^2 - \frac{\| (\mf c_1^*)^{\top} \mf a_1 \mf a^\top_2 \|_F^2}{\|\mf c_1^*\|_F^2} \geq 
0 ~,
$$
where there is equality for any vector $\mf c_1^* = \lambda \mf a_1 $ with  $\lambda \ne 0$.
In the following, we  choose $\lambda = \mf a^\top_2 \bs \psi $  for some vector $ \bs \psi $ (such that $\bs a_2^\top \bs \psi \ne 0$), which gives 
$\mf c_1^* = \mf a_1  \mf a^\top_2   \bs \psi = \mf C   \bs  \psi $. Using the choice $\mf c_1^* = \mf C   \bs  \psi$ in the distance \eqref{d3} also in the case when $\mf C$ is not necessarily separable, we  obtain 
\begin{align}
\label{d7}
      D_{\bs  \psi} (\mf C) := 
        D^* ( \mf C  \bs  \psi ) =
       \|\mf C\|_F^2 - \frac{\| \bs  \psi^{\top}\mf C^{\top} \mf C\|_F^2}{\|\mf C \bs  \psi\|_F^2} \geq 0 . 
\end{align}
Note that this distance depends on the covariance matrix $\mf C$, which can be estimated (see Section \ref{sec3}), and the vector $\bs \psi$, which can be specified. Moreover, by the discussion in the previous paragraph $\mf C$ is (grid) separable if and only if  $ D_{ \bs \psi} = 0$ (independently of the vector $\bs \psi $).


\subsection{Measuring the deviation by the best rank-one approximation}

\label{sec22}

Another measure for the deviation from separability is obtained by  computing the norm of the difference between $\mf C$ and its best rank-one approximation, that is  
\begin{align}
    D(\mf C)  :=\min_{\bs \eta_1\in\mathbb{R}^M, \bs \eta_2\in \mathbb{R}^N}\|\mf C-\bs \eta_1\bs \eta_2^{\top}\|_F^2 ,
    \label{d8}
\end{align}
 The Eckart–Young–Mirsky theorem \citep[see, for example,][]{golub2013matrix} for the Frobenius norm gives 
 \begin{align}
    D(\mf C) = ||\mf C||_{F}^2-\sigma_{sv,1}^2(\mf C),
 \end{align}
where $\sigma_{sv,1}$ is the largest singular value of the matrix $\mf C$. 
Moreover, by definition, we have
\begin{align}
    \label{d4}
    D_{\bs \psi} (\mf C)   \geq D(\mf C)  
\end{align}
for any vector $\bs \psi \in \mathbb{R}^N$.

\begin{remark}
{\rm 
    If the matrix $\mf C =\bs a_1 \bs a_2^\top $  is of rank one and separable, we obtain from the discussion in Section \ref{sec22}  the representation 
    $$
    \mf C = \bs c_1^*(\bs c_2^*)^\top
    = {\mf C \bs \psi \over \| \mf C  \bs \psi \|_F} {(\mf C^\top \mf C \bs \psi )^\top \over \| \mf C^\top \mf C \bs \psi \|_F} {\| \mf C^\top \mf C \bs \psi \|_F \over \| \mf C  \bs \psi \|_F} . 
    $$
 As the SVD  of $\mf C = \sigma_{sv,1} \bs u \bs v^\top $ is unique, we have  
 $$\sigma_{sv,1} ={\| \mf C^\top \mf C \bs \psi \|_F \over \| \mf C  \bs \psi \|_F}~,~~ \bs u={\mf C \bs \psi \over \| \mf C  \bs \psi \|_F}~\text{ and } ~~ \bs v={\mf C^\top \mf C \bs \psi \over \| \mf C^\top \mf C \bs \psi \|_F}
 $$
 for any vector $\bs \psi $ with $\mf C^\top  \mf C \bs \psi \ne 0 $.
Note that this representation does not depend on the choice of the vector $\bs \psi $ (up to a multiplication of the vectors $\bs u $ and $\bs v$ with $-1$). Moreover,  we can choose $\bs \psi $ as the $j$th unit vector and obtain 
\begin{align}
    \label{det12a}
\bs c_1^* = \big  ( C(\h_1, v_j) ,  \ldots , C(\h_M, v_j) \big  )^\top ~.
\end{align}
If additionally the $i$th element of this vector does not vanish, this yields 
\begin{align}
    \label{det12b}
\bs c_2^*= {1 \over C(\mf h_i,v_j)}  \big  ( C(\h_i, v_1) ,  \ldots , C(\h_i, v_N) \big  )^\top ~.
\end{align}
}
\end{remark}

\subsection{Estimation of Covariance} \label{sec23} 
We introduce the following statistic  based on a rescaling procedure   \citep[see][for a similar approach for the estimation of the  mean function]{kurisu2022nonparametric}, 
\begin{align}
    &\hat{C}(\mathbf{h}_0,v_0)= \frac{\sum_{1\leq i\neq i'\leq n}\sum_{1\leq t\neq t'\leq T} K_{0b}(\s_{in}-\s_{i'n},t-t')X_{(\s_{in},t)}X_{(\s_{i'n},t')}}{\sum_{1\leq i\neq i'\leq n} \sum_{1\leq t\neq t'\leq T}K_{0b}(\s_{in}-\s_{i'n},t-t')}, \label{eq covariance estimator}
    \end{align} 
   as an estimator of the covariance function \eqref{covariance} at the point  $(\h_0,v_0)=\big ( (h_{0,1},h_{0,2})^{\top}, v_0 \big )  \in \mathbb{R}^2 \times \mathbb{R}_0^+$ from the data \eqref{eq definition of data set}, where we use the notation 
  
    \begin{align}
  K_{0b}(\s_{in}-\s_{i'n},t-t') & := K_{0ii'}K_{0tt'}, \label{eq K0b}
\end{align} 
with 
\begin{align}
\label{eq K0b1}
 K_{0ii'} & := K \Big (\frac{\lambda_n^{-1}(\s_{in,1}-\s_{i'n,1}-h_{0,1})}{b} \Big  )K \Big  (\frac{(\lambda_n\bb_n)^{-1}(\s_{in,2}-\s_{i'n,2}-h_{0,2})}{b} \Big  ) ,~~~~~~~ \\
K_{0tt'}  &:=   
  K\Big (\frac{|t-t'|-v_0}{ Tb}\Big ), \label{eq K0b2}
\end{align} 
 $b$ is a bandwidth, $\lambda_n$ is an increasing sequence converging to $\infty$ and $\bb_n $ is a non-increasing sequence converging to a non-negative  constant. Here $K: \mathbb{R} \to [0 , \infty)  $ is a non-negative, symmetric kernel with compact support, say   $[-1,1]$,  which  is non-increasing on the interval  $[0,1]$. 
 
 Furthermore, for the grid $\{(\h_{i},v_{j})\}_{1\leq i\leq M,1 \leq  j\leq N}$ introduced in Section \ref{sec2}, we define
\begin{align}
    \widehat{\mf C}=(\hat{C}(\h_{i},v_j))_{1 \leq i\leq M,1 \leq j\leq N}\in \mathbb{R}^{M\times N}, \label{eq covariance matrix estimator}
\end{align}
as the estimator of the 
 covariance matrix  \eqref{d1}.  In Section \ref{sec3}, we  study the asymptotic properties of the estimator $\widehat{\mf C}$ in  \eqref{eq covariance matrix estimator}. These results   will be used in Section \ref{sec4} to determine 
 the asymptotic distribution 
  of   the estimates $D_{\bs \psi } (\widehat{\mf C} )$ and  $D (\widehat{\mf C} )$
   for the measures $D_{\bs \psi } (\mf C) $ and $D(\mf C) $ of the deviation from separability      defined in \eqref{d7} and \eqref{d8}, respectively, and  to  develop statistical inference for validating the assumption of  separability.

  \begin{remark}
      {\rm  Note that we do not adapt a triangular-array setting in our model as it is often done in modeling the mean function $m(r) = E (X_{  Tr ,T} ) $ of a sequence of stationary stochastic processes $ \{\{X_{Tr,T}\}_{r\in [0,1]}:T \in \mathbb{N} \}$ such that
      $\{ \{ X_{t,  T} \} _{t=1, \ldots , T} : T \in \mathbb{N} \} $ is a triangular array with mean 
      $m({t \over T} )  = E [X_{t,  T}] $ 
      \citep[see, for example,][]{ROBINSON20124}.  In fact, this modeling approach is hard to extend to covariance functions and is not reasonable in the present context. More specifically, assume that we would define in the same spirit  a spatial-temporal   sequence of stochastic processes  
 $\{\{X_{(\Lambda_n \kk, Tr)}\}_{(\kk,r)\in \mf D\times [0,1]}:n,T \in \mathbb{N} \}$, where $\mf D \subset  [0,1]^2$ is a compact spatial domain,
 such that for $\h \in 
 [-\Lambda_n, \Lambda_n] \subset \mathbb{R}^2$, $v \in \{ 1,2 \ldots, T-1 \}$
  \begin{align}
 \label{eq Cr=C}    C_{\text{r}}\Big (\Lambda_n^{-1}\h,\frac{v}{T}\Big )= \text{Cov} \big ( X_{(\h, v)}, X_{(0, 0)}\big )
 \end{align}
 for some  covariance function $C_{\text{r}}:\mf D\times [0,1]\to\mathbb{R}$, which does not depend on $(n,T)$.   This  representation \eqref{eq Cr=C} implies
for   any point $(\kk_0,r_0)  \in \mf D\times [0,1]$,   that
\begin{align}
C_{\text{r}}  (\kk_0,r_0 )= \text{Cov} \big ( X_{(\Lambda_n\kk_0, Tr_0)}, X_{(0, 0)}\big )
  \end{align}
  holds for all  $n,T \in \mathbb{N}$. Together with \eqref{dx3}, we therefore obtain 
\begin{align}    C_{\text{r}} (\kk_0,r_0  )=\lim_{n,T\to\infty} \text{Cov} \big ( X_{(\Lambda_n\kk_0, Tr_0)}, X_{(0, 0)}\big )=0, 
  \end{align}
  which  yields  $ C_{\text{r}}\equiv0$ on $\mf D\times [0,1]$ and  makes the problem of testing for separability trivial.
  }
  \end{remark}

\section{Asymptotic properties of covariance matrix estimator}
\label{sec3}

 \def\theequation{3.\arabic{equation}}	
   \setcounter{equation}{0}
   
\subsection{Technical assumptions}
\label{sec 3.1}

 To state our results rigorously, we first introduce several assumptions.
Recall the definition of the set $\Gamma_{n}:= \{\s_{1n},\ldots ,\s_{nn}\} $ of $n$ spatial locations and that, at each location, the spatial-temporal process  is observed at time points $t= 1, \ldots , T$. We assume that $\s_{in}=\Lambda_n\x_{in}$, where $\{\x_{in}:1\leq i\leq n\}$ are independent identically distributed random variables with a density  $f$. We further assume that the    support of $f$, say $\mathbf{D}\subset [-1,1]^2$, is compact with positive Lebesgue measure, and define
$$\Lambda_n=\text{diag}(\lambda_n,\lambda_n\bb_n),$$  where $\lambda_n$ is a diverging sequence and $\bb_n$ is a non-increasing sequences converging to some non-negative constant or $ 0$.  We  also define  $\mathbf{R}_n=\lambda_n \mathbf{D}$  as the {\it sampling region}. Furthermore, we impose the following assumptions about the density  $f$  of spatial locations and its support $\mathbf{D}$.

\begin{assumption}
    \label{as spatial locations 1}
   The support of $f$ is given by   $\mathbf{D}=\big  \{(x,y)\in[-1,1]^2: y\in [l(x)-\bb,l(x)+\bb] \big \}$, where  $l:[-1,1]\rightarrow[-1,1]$ is a continuous function and  $\bb>0$ a  constant. Moreover, we also assume that there  exist   positive constants $\underline{f}, \overline{f} $ such that $0<\underline{f} \leq f \leq   \overline{f} $.
\end{assumption}

Assumption \ref{as spatial locations 1}  is widely used in spatial data modeling with random sampling \citep[see, for example][]{lahiri2003central,kurisu2022nonparametric}.  
By  this assumption, the set $\mathbf{D}$ is actually assumed to be a “band”  centered at the function $l$ with width $\bb$, which contains \cite{lahiri2003central}'s setting as a special case (take $l(x)=0$ and  $\bb=1$). This is more appropriate for  many examples, where  the pattern of the  feasible sampling region 
$\mathbf{R}_n$ is complex, like rivers or seashores. It can be easily extended to the setting where $\mathbf{D}$ is the union of multiple bands sharing a common width.

Next, we specify the dependence structure for the sequence of the random field
 $\{X_{(\s,t)}:(\s,t)\in \mathbb{R}^2\times \mathbb{N}\}$. For this purpose, let  $|A|$ denote the cardinality of the set $A$, and for two subsets $A,B$ of the Euclidean space (its dimension will always be clear from the context),  we denote by 
 $\rho(A,B)=\inf_{a\in A,\ b\in B}||a-b||_{F}$ the {\it distance between $A$ and $B$}.
\begin{assumption}
\label{dependence}
There exist a non-increasing function $\beta:\mathbb{R}^+\rightarrow\mathbb{R}^+$ and a coordinate-wise non-decreasing function $\Psi: \mathbb{N} \times \mathbb{N} \to \mathbb{R}^+$, such that 
    \begin{align}
    \label{d1a}
        E\Big |\sup_{B\in\sigma(V)}(\mathbb{P}(B|\sigma (U)) -\mathbb{P}(B))\Big |\leq \Psi(|U|,|V|)\beta(\rho (U,V)),
    \end{align}
for any  disjoint and finite sets $U,\ V\subset \mathbb{R}^2\times \mathbb{N}$, where, for a set  $ A \subset \mathbb{R}^2  \times \mathbb{N}$,  
 $\sigma(A)$ denotes the sigma field  generated by $\{ X_{(\s,t)} :  (\s,t) \in A \} $.  
\end{assumption}
 The assumption of finite sets $U$ or $V$ \eqref{d1a} is crucial for $\beta$-mixing-type processes. Otherwise, as pointed out by Theorem 1 in  \cite{BRADLEY1989489}, \cref{dependence} would imply  $m$-dependence, which would be a very restrictive assumption for asymptotic analysis with an expanding spatial domain.

\begin{assumption}
    \label{moment conditions}
  ~~
  
  \begin{itemize}
        \item [(M1)] 
       $\{X_{(\s,t)}: (\s,t)\in \mathbb{R}^2\times \mathbb{N}\}$ is a marginally stationary stochastic process. If  $X_0$ denotes a random variable with this distribution, we further assume that  $E[X_0]=0$ and  that $E|X_0|^{4(1+\delta)}<\infty$  for some $\delta>0$.
        \item [(M2)] For the constant $\delta$  from condition  (M1) and  $k= 4(1+\delta) $, we have 
    $$
    M_{k}=\max_{(\s,t)\neq (\s',t')}\{E|X_{(\s,t)}X_{(\s',t')}|^k\lor E|X_{(\s,t)}|^kE|X_{(\s',t')}|^k \} < \infty .$$

        \item[(M3)] For $k= 4(1+\delta)$,
          we have   \begin{align}
                \widetilde{M}_{k}& =\max_{(\s,t)\neq (\s',t'), (\mf u,v),(\mf u',v')}\{E|X_{(\s,t)}X_{(\s',t')}X_{(\u,v)}X_{(\u',v')}|^k\} < \infty ,    \\   \widehat{M}_{k}& =\max_{(\s,t)\neq (\s',t')\neq (\u,v)}\{E|X_{(\s,t)}X_{(\s',t')}X_{(\u,v)}|^k\}<\infty.
             \end{align}
    \end{itemize}
\end{assumption}

To investigate the asymptotic properties of the estimator $\widehat{\mf  C}$, we consider both \textit{domain-expanding-infill (DEI) asymptotics} and \textit{domain-expanding (DE) asymptotics}, i.e., $n/\lambda_n^2 \bb_n\to\infty$ and $n/\lambda_n^2 \bb_n\to c$, for some constant $c>0$. As pointed out  by \cite{lu2014nonparametric}, DEI asymptotics is sometimes more appropriate  than the less complex DE asymptotics for modeling data in 
real life, such as  real-state price data and  data used in environmental science.

\begin{assumption}
\label{technical assumptions}
The bandwidth satisfies  $b= c n^{-\eta}$ for constants $\eta>0$, $c> 0$, such that (with the constant
$\delta > 0$  from  Assumption \ref{moment conditions}) 
    \begin{itemize}
        \item [(T1)]  For some constant $0<\xi\leq \frac{1}{3}$,
        $$
        \frac{n}{\lambda_n^2\bb_n}b^{\xi}=o(1)~,~~
        nTb^{2+2\xi}\nearrow\infty .
        $$ 
        \item [(T2)] (i) The mixing coefficients  $\beta(x)$ are  non-increasing  (as a function of $x$) such that  $\beta(0)=1$, and for any $ x,y\geq 0$:  $\beta(x+y)\leq C_{\beta}\beta^{\frac{1}{2}}(x)\beta^{\frac{1}{2}}(y)$, for some  constant  $C_{\beta}>0$. \\ (ii) For any $a,r,s>0$, there exists some constant $C'_{\beta}$ 
        such that $\beta^s(ar)\leq C'_{\beta}\beta^s(a)\beta^{s}(r)$.
         \item [(T3)] $\sum_{l=1}^{\infty}\beta^{\frac{\delta}{2(1+\delta)}}(l)<\infty$, $\sum_{l=1}^{\infty}l^2\beta^{\frac{\delta}{(1+\delta)}}(l)<\infty$, where $\delta $ is the constant from Assumption \ref{moment conditions}.
          \item [(T4)] There exists a diverging  sequence $(q_n)_{n \in \mathbb{N}}$ in $\mathbb{N}$  such that 
        \begin{itemize}
            \item [(1)] $nT\beta^{\frac{\delta}{1+\delta}}(q_n)=o(1)$,  ~~~~~~~~ (2) $(\lambda_n\bb_n)^{-1}q_n=o(n^{-\theta})$ for some $\theta>0$, 
            \item[(3)]  $b^{0.5}q_n^8(\log n)^2=o(n^3)$ and $(q_n\log n)^6b=o(1)$, \item[(4)] $\frac{q_n^4\log T}{n^3b^{1.5}}\lor\frac{q_n^3\log T}{n^2b^{0.5}}=o(1)$, ~~~  (5) 
           $q_n=o(\sqrt{n}\land b^{-\frac{1}{2}})$,
           \item[(6)] $\left( nT(q_n\log n)^3\right)^{4(1+\delta)}b^{3-2(1+\delta)}\beta^{\frac{\delta}{1+\delta}}(q_n)=o((nT)^4b^6)$.
          \end{itemize}
    \end{itemize}
\end{assumption}

Assumption \ref{technical assumptions} can be generally understood as the “spatial version” of the technical assumptions used in \cite{fan1999central}.    Similar assumptions have been made in  previous research on the analysis of spatial data with randomly sampled locations  \citep[see, for example][and the references therein]{kurisu2022nonparametric,hu2025self}. More specifically, condition (T1) is used to describe the degree of “infill asymptotics” and excludes the case of pure infill asymptotics, which often leads to inconsistency \citep[see][]{lahiri1996inconsistency}.
Condition (T3) is a standard assumption on the  $\beta$-mixing coefficients. Condition (T2) can be easily satisfied when $\beta$-mixing coefficients decay at a  polynomial rate.  Condition (T4) is a technical condition regarding a sequence $(q_n)_{n\in \mathbb{N}}$ which is introduced in the proofs of the main results. We emphasize that this sequence does not have to be chosen in practice and does not appear in the definition of the statistic \eqref{eq covariance estimator}. Moreover, Condition (T4) simplifies substantially if the decay rate of the mixing coefficients has been specified more concretely. For example, if the mixing coefficients decay at a  geometric rate, i.e. $\beta (l) \leq c' \theta^{-l} $  for constants $ c'>0, \theta >1$ and $T/n \to c $ for a constant $c>0$,   we can choose $q_n = C_q\log n $, where $C_q>0$ is independent of sample size. If $C_q$ is chosen sufficiently large, condition (T4) is {always} satisfied.
\medskip

We conclude the technical assumptions by introducing  regularity conditions on the  covariance function and the selected evaluation distances defined in \eqref{eq evaluation distance}.

\begin{assumption}
    \label{smoothness}
  The covariance function $C$   defined in \eqref{covariance} is twice  continuously differentiable  with the
  first (and second) partial derivatives $\partial_k C$ (and $\partial_{k,\ell} C$) with respect to its $k$-th (and  $k$ and $l$-th) coordinate(s).
\end{assumption}

Our final assumption refers to the chosen evaluation distances  $\{\h_i\}_{1 \leq i\leq M}$ and $\{v_j\}_{1 \leq j\leq N}$ in \eqref{d1}. As we consider a sampling  region $\mf R_n $ expanding as the sample size grows, we  make a similar assumption for the evaluation locations as well.

\begin{assumption}
    \label{as regularity}
There  exists a diverging sequence  $(e_n)_{n\in \mathbb{N}}$ such that
\begin{align}
    \label{eq en}
    ce_n\leq \min_{1\leq i\leq M\atop 1\leq  j\leq N}\{||\h_i||,v_j\}\leq \max_{1\leq i\leq M \atop
    1 \leq j\leq N}\{||\h_i||,v_j\}\leq Ce_n,
\end{align}
for some constants $c,C>0$. Moreover, there exist 
 non-negative functions $\widetilde{\gamma}_0$, $\widetilde{\gamma}_1$ and $\widetilde{\gamma}_2$  (which might depend on $\lambda_n$) such that
    \begin{itemize}
         \item [(E1)] $\min_{1\leq i\leq M , 1\leq j\leq N} |C(\h_i,v_j) | \geq \widetilde{\gamma}_0(e_n)$ and  $\liminf_{n\to \infty}  nTb^{1.5}\widetilde{\gamma}_0(e_n)>0$.
        \item [(E2)] For any  
        $(\h, v)$ with  $\min\{||\h||, v\}\geq e_n$, we have  
        $|\partial_k C(\h,v)|\lesssim \widetilde{\gamma}_1(e_n)$ and $|\partial_{k,l} C(\h,v)|\lesssim \widetilde{\gamma}_2(e_n)$ for $k,l=1,2,3$.  
        \item [(E3)] For any $m=1,2$, we have  $(\lambda_n \lor T)^m\widetilde{\gamma}_m(e_n)=O(1)$.
    \end{itemize}
\end{assumption}
Condition (E1) controls the properties of the covariance for large lags, because in  regions where the covariance is very close to $0$,  deviations from separability are  very hard to detect.
Condition (E2)  is an assumption on the growth of  the partial derivatives of the covariance function for large lags, and condition (E3) specifies their rates. These assumptions are  required for our inference framework and  are satisfied by many of the common covariance models. A prominent example is the covariance  kernel
\begin{align}
    C(\h;v)=\frac{\sigma^2}{(|v|^{2\alpha}+1)}\exp\left(-\frac{||\h||^{2\eta}}{(|v|^{2\alpha}+1)^{\eta}}\right),~~\sigma^2 >0,\ \ \eta, \alpha>0 , \label{eq:Gneiting}
\end{align}
which was introduced in equation (4) of \cite{Gneiting01062002}. For example, if  $0.5<\alpha<1$ and $\eta=\frac{1}{2(1-\alpha)}$,  \cref{smoothness} is satisfied. Furthermore, by choosing $e_n=\log n$, \cref{as regularity}  holds as well with $\widetilde{\gamma}_m(e_n)=\widetilde{\gamma}_m(e_n, \lambda_n)=(\lambda_n)^m\exp(-e_n)$, $m=0,1,2$.

\subsection{Weak convergence of the covariance estimator} \label{sec32}

To present a mathematically precise  statement about  the weak convergence of the estimator \eqref{eq covariance matrix estimator}, let  
$$\mathbf{X}_{(\s_{in},T)}:=(X_{(\s_{in},1)},...,X_{(\s_{in},T)})^\top \in\mathbb{R}^T,$$ 
 denote the vector of observations at the (random) spatial location  $\s_{in}$ ($i=1, \ldots , n$), and define 
$\mathbf{X}_{\mf S_n,T}:= (\mathbf{X}_{(\s_{1n},T)}^\top ,...,\mathbf{X}_{(\s_{nn},T)}^\top)^\top  $
as the $nT$-dimensional vector of all observations (the  vectorized version of the data set $\textbf{Data}_n$ in  \eqref{eq definition of data set}).

Note that, due to the  Gluing lemma in optimal transport theory  \citep[see][]{berti2015gluing},  the joint distribution of the vector $\mathbf{X}_{\mf S_n,T}$ and the vector $\mathbf S_n=(\s_{1n},...,\s_{nn})^{\top}$ of spatial locations exists, such that the  random vector 
\begin{align}
    (\mf X^{\top}_{\mf S_n,T},\mf S_n^{\top})=(\mf X^{\top}_{(\s_{1n},T)},...\mf X_{(\s_{nn},T)}^{\top},\s_{1n},...,\s_{nn})
\end{align}
is well defined. 
By the  disintegration theorem, the  regular conditional distribution $\mathbb{P}^{\mf X_{(\mf S_{n},T)}|\mf S_n=S_n}$ of $\mf X_{\mf S_n,T}$ given  $\mf S_n=S_n\in \mathbf{R}_n^n \subset \mathbb{R}^{2n}$ is also well defined. We define  the set of conditional distributions 
\begin{align}
    \mathcal{P}_{n,T}=\{\mathbb{P}^{\mathbf{X}_{(\mf S_{n},T)} | \mf S_n=S_n }:  S_n\in\mathbf{R}_n^n\}, \quad n,T \in\mathbb{N}. 
    \label{det2}
\end{align}

\par Suppose, for each $n$ and $T$, $H_{n,T}:\mathbb{R}^{nT}\to \mathbb{R}^d$ is a vector-valued measurable mapping (for some positive integer $d$).  The conditional distribution of the random variable $H_{n,T}(\mf X_{(\mf S_n,T)})$ given  $\mf S_n=S_n$ is denoted by
\begin{align}
 \mathbb{P}
 ^{H_{n,T}|\mf S_n=S_n}((\infty ,\mf t])=   \mathbb{P}
 (H_{n,T}(\mf X_{\mf S_n,T})\in (\infty ,\mf t] | \mf S_n=S_n),\ \forall\ \mf t\in \mathbb{R}^d.
\end{align}

Furthermore, the  stochastic process,
$ 
\mathbb{S}:=    (\mf S_{n})_{n\geq 1}
$
 is also well defined,  and we denote its distribution by  $\mathbb{Q} = \mathbb{P}^{\mathbb{S}}$. Based on the aforementioned notation, we can discuss the weak convergence of the statistic $H_{n,T}(\mathbf{X}_{\mf S,T})$ in an almost sure sense, that is:  there exists a distribution function $G$ defined on $\mathbb{R}^d$, such that  for any continuity point   $\mf t$  of  $G$, 
\begin{align}
   \lim_{n\to \infty, T \to \infty }\mathbb{P}^{H_{n,T}|\mf S_n=S_n}((-\infty,\mf t])= \lim_{n\to\infty, T \to \infty}\mathbb{P}(H_{n,T}(\mathbf{X}_{\mf S_n,T})\in (\infty, \mf t] |\mf S_n=S_n)=G(\mf t)
\end{align}
holds almost surely with respect to $\mathbb{Q}$. We write for this statement $H_{n,T} \xrightarrow[]{d} G $  almost surely with respect to $\mathbb{Q}$.
Similar concepts of weak convergence have been discussed in \cite{lahiri2003central} and \cite{kurisu2022nonparametric}.

With these preparations, we are able to establish the weak convergence of the estimator $\hat{C}(\h_0,v_0)$ of the covariance at any point  $(\h_0,v_0)\in \{(\h_i,v_j) \}_{1 \leq i\leq M,1 \leq j\leq N}$.

\begin{theorem}
    \label{Theorem CLT Covariance Estimator}
    Suppose \cref{as spatial locations 1} - \ref{as regularity} hold,  $T/n\rightarrow c \in (0, \infty) $ and $nTb^{3.5}\rightarrow 0$, and  
    let   $(\h_0,v_0)$ be a  point from $\{(\h_i,v_j)\}_{1 \leq i\leq M;1 \leq j\leq N}$. Then, we have
    \begin{align}   &{nTb^{1.5}} \Big (\hat{C}(\h_0,v_0)-C(\h_0,v_0)-\frac{\textbf{Bias}(\h_0,v_0)}{2 \mf  A_1}\Big  )\xrightarrow[]{d} N\Big (0, \frac{\sigma^2}{4\mf A_1^2}\Big ) 
            \end{align} 
            almost surely with respect to $\mathbb{Q}$, where
\begin{align}
      \label{d7b}
      \sigma^2 & = (E[X_{(0,0)}^2])^2\mathbf{A}_2\mathbf{B}_2,
\\
         \label{d7c}
         \mathbf{A}_k & =
\int_{\mathbf{D}}f(\x)^2d\x ~\Big ( \int_{[-1,1]}K^k(u)du\Big )^2,  
           \\
             \label{d7d}
        \mathbf{B}_k & =2\int_{-1}^{1}K^k(u)du,  \\
        \label{nhd1}
        \textbf{Bias}(\h_0,v_0)
        & = \Big( \sum_{k=1}^2\partial_{k}C(\h_0,v_0)   V_T^{(0)} W_n^{(1) } (k)   \Big ) 
+\partial_{3}C(\h_0,v_0) {W_n^{(0)} V_T^{(1)}},
    \end{align}
     and for $\ell=0,1$, 
 \begin{align}
  V_T^{(\ell)}&= \frac{1}{T^2b}\sum_{1\leq t\neq t'\leq T}K\Big (\frac{|t-t'|-v_0}{ Tb}\Big ) \Big(|t-t'|-v_0\Big)^{\ell }, \\
 W_n^{(\ell ) } (k)   &= \frac{1}{n^2b^2}\sum_{1\leq i\neq i'\leq n} K \Big  (\frac{\s_{in,1}-\s_{i'n,1}-h_{0,1}}{\lambda_nb} \Big  )K \Big  (\frac{\s_{in'2}-\s_{i'n,2}-h_{0,2}}{\lambda_n\bb_nb} \Big  ) \big (\s_{in,k}-\s_{i'n,k}-h_{0k}\big )^\ell.
 \end{align}
\end{theorem}

Note  that \cref{Theorem CLT Covariance Estimator} cannot be directly used  for statistical inference. On the one hand, the term $\mf A_1$ contains the density $f$, which might be unknown in specific applications. On the other hand, the bias in \eqref{nhd1} is difficult to estimate.
In our next theorem, we present a result based on  undersmoothing, which can be used for statistical inference.

\begin{theorem}
\label{Theorem Feasible Inference}
Let the assumptions of  Theorem \ref{Theorem CLT Covariance Estimator} be  satisfied and assume that $nTb^{2.5}\to 0$. For    any  point $(\h_0,v_0)$ from $\{(\h_i,v_j)\}_{1 \leq i\leq M, 1 \leq j\leq N}$, we have 
\begin{align}
\label{nhd2}
    &nTb^{1.5} \big (  \hat{C}(\h_0,v_0)-C(\h_0,v_0)\big)\xrightarrow[]{d} 
    N (0, \tau^2)
\end{align}
 almost surely with respect to $\mathbb{Q}$, where 
 \begin{align}
       \label{limvar}
     \tau^2 :=  {\sigma^2 \over 4 \mf A_1^2 } = {(E [X_{(0,0)}^2])^2 \mf B_2^3 \over  16~ \mf I },
   \end{align}
   $\sigma^2$ and $\mf A_1$
 are defined in \eqref{d7b} and \eqref{nhd1}, respectively. 
 Moreover, the asymptotic variance 
  can be consistently estimated by 
  \begin{align}
 \hat \tau^2  :=  { \widehat{(E[X_{(0,0)}^2])^2}\mf B_2^3  \over  16~ \hat{ \mf I}} ,
    \label{nhd4}
    \end{align} 
 where     
\begin{align}    
    \label{det10}
    \widehat{E[X_{(0,0)}^2]}& :=\frac{1}{nT}\sum_{i=1}^{n}\sum_{t=1}^{T}X^2_{(\s_{in},t)},\\
    \label{det10a}
    \hat{\mf I}&:=\frac{1}{n^2b^2}\sum_{i\neq i'}K_{0ii'},
\end{align}
and  $K_{0ii'}$ is defined in \eqref{eq K0b1}.
\end{theorem}

\begin{remark}\label{rm:sigma}
   {\rm 
  Consistency of $\widehat{E[X_{(0,0)}^2]}$ can be easily obtained by repeating the Step 1 of the proof of \cref{Theorem CLT Covariance Estimator}. 
   Furthermore, following  \cite{lahiri2003central} and  \cite{kurisu2022nonparametric}, we may assume that $f$ and $\lambda_n$ ($\lambda_n\bb_n$) can be obtained from prior information.  Since $\mf I $ and $\mf B_2$ are known constants   if the density $f$ is known,  the asymptotic variance 
    in \eqref{nhd2} can in this case be consistently estimated by 
    \begin{align}
\hat \tau^2=   { \widehat{(E[X_{(0,0)}^2])^2}\mf B_2^3  \over  16~ { \mf I}} .
    \label{nhd5}
    \end{align}}
\end{remark}

Recalling the definition of the estimator of the covariance matrix in \eqref{eq covariance matrix estimator} and denoting  by  
 $\text{vec}(\widehat{\mf C})$  its  $MN$-dimensional vectorized version, which is obtained by  stacking the columns of $\widehat{\mf C}$ in one vector. We now give a central limit theorem for the  vector
\begin{align}
    \text{vec}(\widehat{\mf C})-\text{vec}(\mf C). 
\end{align}

\begin{theorem}
    \label{Theorem MCLT Covariance Estimator} 
    If the assumptions of Theorem \ref{Theorem Feasible Inference} are satisfied, 
   we have 
 \begin{align}
      nTb^{1.5}(\text{vec}(\widehat{\mf C})-\text{vec}(\mathbf{C})) \overset{d}{\to} N \big ( \mf  0, \tau^2 I_{MN} \big ) 
      \label{nhd3}
 \end{align}  
almost surely with respect to $\mathbb{Q}$, where $\tau^2$ is defined in \eqref{limvar}. Moreover, the asymptotic variance in \eqref{nhd3} can be consistently estimated by \eqref{nhd4} (or by  \eqref{nhd5} if the density $f$ is known).
\end{theorem}

\section{Statistical applications}
\label{sec4}
\def\theequation{4.\arabic{equation}}
\setcounter{equation}{0}

In this section, the asymptotic results derived in Section \ref{sec3} are used to develop several tools for statistical inference on the separability of the matrix $\mf C$ defined in \eqref{d1}.

\subsection{Exact hypotheses for spatial-temporal separability}
\label{sec41}

By the discussion in Sections \ref{sec21} and \ref{sec22},   exact (grid) separability 
 is equivalent  to 
 \begin{align}
     \label{d6}
  H_0:~   D_{\bs \psi} = D = 0,
 \end{align}
 where the measures $ D_{\bs \psi} $ and $  D$ are defined in \eqref{d7} and \eqref{d8}, respectively. Note that a non-separable covariance matrix $\mf C$   yields values  $D_{\bs \psi} > 0 $ and $D>0$. In this section, we construct test statistics based on estimators of $D_{\bs \psi} $ and $D$
 and investigate their asymptotic properties under the null hypothesis of exact (grid) separability and the alternative hypothesis of non-separability. 

 \subsubsection{Partial trace approach}
 \label{sec411}

 The estimator of the  partial-trace  distance  in \eqref{d7} is defined as  
\begin{align}
\label{d9}
\widehat{D}_{\psi} = \|\widehat {\mf C}\|_F^2 - \frac{\|\bs  \psi^{\top}\widehat {\mf C}^{\top} \widehat {\mf C}\|_F^2}{\|\widehat {\mf C} \bs  \psi\|_F^2} ,  \end{align}
 where  
\begin{align}
  \label{d12}  
\widehat{\mathbf C} =( \hat{C} (\h_i, v_j) )_{i=1, \ldots, M}^{j=1, \ldots N},
\end{align}
is the (undersmoothed)  estimator defined in \eqref{eq covariance matrix estimator}. 
    The following result establishes asymptotic properties of $\widehat{D}_{\psi}$ under the null hypothesis \eqref{d6}.

 \begin{theorem}\label{thm:partialexact}
      If the assumptions of Theorem \ref{Theorem Feasible Inference}  and the  null hypothesis  \eqref{d6} are satisfied, we have 
       \begin{align}
       \label{det11}
           & (nT)^2b^{3}\widehat{D}_{\bs \psi}\overset{d}{\to} 
           {\cal L} 
           \end{align}
       almost surely with respect to $\mathbb{Q}$,   where
        \begin{align}
        {\cal L}&:= \left\|\tau \G- \frac{\tau  \G \bs  \psi \bs   \psi^{\top} \mf C^{\top} \mf C}{\|\mf C \bs \psi \|_F^2 } \right\|_F^2 - \left\| \frac{\tau \G^{\top} \mf C \bs \psi}{\|\mf C \bs \psi\|_F}-\frac{\tau \mf C^{\top} \G \bs  \psi}{\|\mf C \bs \psi\|_F}\right\|_F^2
       \end{align}
       and 
 $\mathcal{G}$ is an $M\times N$  random matrix with independent standard normal distributed entries,  $\tau^2$ is defined  in
    \eqref{limvar} and can be consistently estimated by  \eqref{nhd4}.
   \end{theorem}

\begin{remark}
    {\rm 
    Let  $q_{1-\alpha}$ denote the $(1-\alpha)$ quantile of the distribution of the random variable ${\cal L}$ in \eqref{det11}.
We can estimate the quantile of the distribution of  $\mathcal L$ 
replacing 
    the matrix $\mf C$ by its estimate  $\widehat{\mf C}$ and  simulating the distribution of  $\mathcal G $.
    For this purpose we generate an $M \times N$ random matrix  $\mathcal{G}$ with independent standard normal distributed entries,
calculate the variance estimator  $ \hat{\tau }^2$ in \eqref{nhd4}  and define
  $$
\hat{\mathcal L} :=\left\| \hat{\tau }\mathcal{G}- \frac{ \hat{\tau }\mathcal{G} \bs  \psi \bs   \psi^{\top} \widehat{\mf C}^{\top} \widehat {\mf C}}{\|\widehat {\mf C} \bs \psi \|_F^2 } \right\|_F^2 - \left\| \frac{\hat{\tau }\mathcal{G}^{\top} \widehat {\mf C} \bs \psi}{\|\widehat {\mf C} \bs \psi\|}-\frac{\hat{\tau }\widehat {\mf C}^{\top}  \mathcal{G}  \bs  \psi}{\|\widehat {\mf C} \bs \psi\|}\right\|_F^2.
$$
Then, it is easy to see that this  estimator satisfies that  
$
\hat{\mathcal L} 
\xrightarrow[]{d} 
{\cal L} 
$
holds almost surely with respect to $\mathbb Q$. Therefore, we use the $(1-\alpha)$-quantile $\hat q_{1-\alpha}
$ of the distribution of $\hat{\mathcal L} $ as approximation for the corresponding quantile $q_{1-\alpha}$ of the distribution of ${\mathcal L} $. The quantile $\hat q_{1-\alpha}$   can be estimated with arbitrary precision by the empirical $(1-\alpha)$ quantile of $B$  simulations from  $\hat{\mathcal L}$, if $B$ is chosen sufficiently large.
Finally, we propose to reject the hypothesis of  exact spatial-temporal separability in \eqref{d6}  at the significance level $\alpha$, whenever
 \begin{align}
 \label{eq:reject} 
          \widehat D_{\bs \psi} > {\hat q_{1-\alpha} \over (nT)^2b^3}.
\end{align}
}
\end{remark}

Our next result establishes the asymptotic properties of the estimate $\widehat D_{\bs \psi}$  in the case $ D_{\bs \psi} \ne 0$ (or equivalently $ D \ne 0$). On the one hand, it yields the consistency of the test \eqref{eq:reject} for the hypotheses \eqref{d6}. On the other hand, it can be used for the construction of confidence intervals for the measure $D_{\bs \psi}$ of deviation from separability.  
Moreover, it is a basic tool for constructing a test for relevant hypotheses, which will be discussed in Section \ref{sec412} below.

  \begin{theorem} \label{thm42} If the assumptions of Theorem \ref{Theorem Feasible Inference} are satisfied  and  
  $\mf C$ is not (grid) separable (that is $D_{\bs \psi} > 0$),  we have 
       \begin{align}
         nTb^{1.5}  (\widehat  D_{\bs \psi} - D_{\bs \psi})\overset{d}{\to} N \big (0, 4 \tau ^2 \| \mf W\|^2_F \big ) 
       \end{align}
       almost surely with respect to $\mathbb{Q}$, where  $\tau^2$ is defined in \eqref{limvar} and 
       \begin{align}
        \label{d11}
           \mf W = 
           \mf C
           -  \mf C \bs \psi \bs \psi^{\top}\mf  C^{\top} \mf C  +  \|\mf C\|_F^2  \mf C \bs \psi \bs \psi^{\top}  -  \mf C \mf C^{\top} \mf C \bs \psi \bs \psi^{\top}.
\end{align}\label{thm:alterexact}
   \end{theorem}
 It follows from \cref{thm:partialexact} and \ref{thm:alterexact} that the decision rule \eqref{eq:reject}  defines a consistent and asymptotic level $\alpha$ test for the hypotheses in  \eqref{d6}. 

\begin{remark} \label{rem42}
    {\rm We can use Theorem \ref{thm42} to construct an asymptotic  confidence interval for the measure $D_{\bs \psi }$ of deviation from separability. More precisely, if $\widehat{\mf W}$ denotes the estimator of the matrix $\mf W$ in \eqref{d11}, which is obtained by replacing $\mf C$ by the estimator $\widehat {\mf C}$ defined in \eqref{eq covariance matrix estimator},  we propose to use
    \begin{align}
        \label{d13}
        \hat \vartheta^2_{\bs \psi } := 4 \hat \tau ^2 
   \| \widehat{\mf W} \|^2_F 
   \end{align}
   as an estimator for the asymptotic variance in Theorem \ref{thm42},  where the  $\hat{\tau}^2$ is defined in \eqref{nhd4}. It is then easy to see that the interval
   \begin{align}
   \label{conf1}
   \Big [\widehat  D_{\bs \psi} - {\hat \vartheta_{\bs \psi } \over  nTb^{1.5}} u_{1-\alpha /2 } ~,~
   \widehat  D_{\bs \psi} + {\hat \vartheta_{\bs \psi } \over  nTb^{1.5}} u_{1-\alpha /2} \Big ]
   \end{align}
   defines an asymptotic $(1-\alpha)$-confidence interval  for the measure $D_{\bs \psi}$, where $u_{1-\alpha}$ denotes the $(1-\alpha)$ quantile of the standard normal distribution.
    }
\end{remark}
\subsubsection{SVD approach}
\label{sec4 2}
For estimating the minimum distance  $D$ in \eqref{d8} for the best rank-one approximation,  we consider the statistic
\begin{align}
\label{d10}
    \widehat{D}=||\widehat {\mf C}||_F^2-\sigma_{sv,1}^2(\widehat {\mf C}),
\end{align}
where $\widehat {\mf C}$ is the undersmoothed estimator of matrix $\mf C$ defined in \eqref{d12}. 
Our first result provides the asymptotic distribution of the statistic \eqref{d10} under the null hypothesis \eqref{d6} of exact (grid)-separability. 

\begin{theorem}
\label{th exact test SVD}
Let $\mf C \not = \mf 0 $. If the assumptions of  Theorem \ref{Theorem Feasible Inference} and the  null hypothesis  \eqref{d6} are satisfied,
we have
\begin{align}
\label{det13}
(nT)^2b^3 (\widehat{D}-D)\overset{d}{\to}  {\cal L}_{sv}:=\tau^2 {\mathbf G}^{\top}  {\mf P}  {\mathbf G}   
\end{align}
almost surely with respect to $\mathbb{Q}$,
where $ {\mathbf G} $  is an $MN$-dimensional vector of independent standard normal  distributed random variables, $\tau^2$ is defined in \eqref{limvar}  and 
$$
{\mf P}  = \Big (\mf I_{M}-\frac{\mf c_1^* (\mf c_1^*)^\top }{||{\mf c}_1^*||^2_F}
\Big) 
\otimes
 \Big ( \mf I_{N}-\frac{\mf c_2^* (\mf c_2^*)^\top}{||{\mf c}_2^*||_F^2}  \Big ),
$$
and $\mf c_1^*$
 and $\mf c_2^*$ are defined in \eqref{det12a} and \eqref{det12b}, respectively.
\end{theorem}
 
\begin{remark} {\rm 
    In practice, the quantiles of the distribution of the random variable on the right-hand side of \eqref{det13} have to be estimated. For this purpose, we recall the definition of  the estimate $\hat{\tau}^2$  in \eqref{nhd4}, introduce the estimators 
\begin{align}
\hat{ \bs c}_1^*  & = \big  ( \hat C(\h_1, v_j) ,  \ldots , \hat C(\h_M, v_j) \big  )^\top ~, \\ 
\hat{ \bs c}_2^* & = {1 \over \hat C(\h_i,v_j)}  \big  ( \hat C(\h_i, v_1) ,  \ldots , \hat C(\h_i, v_N) \big  )^\top ~,
\end{align}
and define 
$$
\hat{\mf P}  = \Big (\mf I_{M}-\frac{ \hat{ \mf c}_1^* (\hat{ \mf c}_1^*)^\top }{||{\hat{ \mf c}}_1^*||^2_F}
\Big) 
\otimes
 \Big ( \mf I_{N}-\frac{\hat{ \mf c} _2^* (\hat{ \mf c}_2^*)^\top}{||{\hat{  \mf c}}_2^*||_F^2}  \Big ).
$$
Then it is easy to see that 
$$
\hat {\cal L}_{sv} :=
\hat{ \tau}^2 \mf G^{\top} \widehat{\mf P} \mf G
\xrightarrow[]{d} 
{\cal L}_{sv} ~
$$
holds almost surely with respect to $\mathbb{Q}$,
where  the random variable ${\cal L}_{sv} $ is defined  in \eqref{det13}. Consequently, we can use the quantiles of the distribution of  $\hat {\cal L}_{sv}$
to approximate the quantiles of the distribution of  ${\cal L}_{sv} $. If $\hat v_{1-\alpha}$ denotes the  $(1-\alpha)$-quantile of the distribution  of  $\hat {\cal L}_{sv}$, we propose to reject the null hypothesis in \eqref{d6}, whenever 
\begin{align}
    \label{testsv}
     \widehat{D} > {\hat v_{1-\alpha} \over (nT)^2b^3}.
\end{align}
By Theorem \ref{th exact test SVD} this test has asymptotic level $\alpha$.
}
\end{remark}

Our next result specifies the asymptotic distribution of $\widehat{D}$ under the alternative. In particular, it shows that the test \eqref{testsv} for the hypotheses \eqref{d6} is consistent. 

\begin{theorem} 
\label{thm44}
    \label{th relevant test SVD}
    If the assumptions of Theorem \ref{th exact test SVD} are satisfied and $D>0$, then we have  
    \begin{align}
        nTb^{1.5}(\widehat{D}-D)\overset{d}{\to}N({0},\vartheta_{sv}^2)
    \end{align}
    almost surely with respect to $\mathbb{Q}$, 
  where 
 \begin{align}
     \label{nhd6}
  \vartheta_{sv}^2= 4 \tau^2 \|\mf Q\|^2_F,
 \end{align}
$\tau^2$ is defined in \eqref{limvar}  and  $\mf Q=\mf C-\mf c_1^*(\mf c_2^*)^{\top}$. 
\end{theorem}
Similarly, as in Remark \ref{rem42} we can derive an asymptotic confidence interval for the measure $D$ of the deviation from separability based on the best rank-one approximation, that is

   \begin{align}
   \label{conf2}
   \Big [ \widehat{D} - {\hat \vartheta_{sv} \over  nTb^{1.5}} u_{1-\alpha /2 } ~,~
   \widehat{D} - {\hat \vartheta_{sv} \over  nTb^{1.5}} u_{1-\alpha /2} \Big ], 
   \end{align}
   where $\widehat{D} $ is defined in \eqref{d10} and (recall that the statistic $\hat \tau ^2$ in \eqref{nhd4} is an estimator of $\tau^2$)
   \begin{align}
   \label{d14}
    \hat \vartheta_{sv}^2 :=  4\hat \tau^2  \| \widehat{ \mf Q}\|_F^2,\ 
   \end{align}
  and  $\widehat{\mf Q}=\widehat {\mf C}-\hat{\mf c}_1^*(\hat{\mf c}_2^*)^\top$ 
   is an estimator of the matrix $\mf Q $ in \cref{th relevant test SVD}.

\begin{remark}
{\rm 
If a rank-one approximation is not reasonable to reduce the computational complexity of spatial-temporal data analysis, 
one could try to approximate the matrix $\mf C$ by a matrix of rank $k\geq 2$,  and the methodology developed in this section can be extended to this situation. More specifically, we consider the deviation 
   \begin{align}
       D_k(\mf C):=\min_{\text{rank}(\mf B)\leq k} ||\mf C-\mf B||_F^2=||\mf C||_F^2-\sum_{j=1}^k\sigma^2_{sv,j}(\mf C)
   \end{align}
   between the covariance matrix $\mf C $ and its best rank-$k$ approximation, where 
   the equality is again a consequence of  the Eckart–Young–Mirsky theorem \citep[see][]{golub2013matrix} and  $\sigma_{sv,1}(\mf C) \geq  \sigma_{sv,2}(\mf C) \geq  \ldots  \geq\sigma_{sv,k}(\mf C)$  are the ordered singular values of the matrix $\mf C $. 
   Then, similar statements   as 
  given in \cref{th exact test SVD} and \ref{th relevant test SVD}  can be derived for the estimator
  $D_k(\widehat{\mf C})$, where $\widehat{\mf C}$ is the undersmoothed estimator of the matrix ${\mf C}$ in \eqref{eq covariance matrix estimator}.
  For example, if $\text{rank}(\mf C)=k$, that is   $ D_k(\mf C)=0$, one can show that, under the assumptions of Theorem \ref{th exact test SVD}, we have 
    \begin{align}
    \label{eq MCLT rank k approximation}
   (nT)^2b^3 (D_k(\widehat{\mf C})-D_k(\mf C))\overset{d}{\to} 
   \tau^2 \mathbf{G}^{\top}\mf P^*_k \mathbf{G}  
    \end{align}
almost surely with respect to $\mathbb{Q},$
    where $ {\mathbf G} $  is an $MN$-dimensional vector of independent standard normal  distributed random variables, $\tau^2$ is defined in \eqref{limvar}, 
    $$
    \mf P_k^*= (I_M-\mf U_1\mf U_1^{\top})\otimes(I_{N}-\mf U_2\mf U_2^{\top})
    $$
    and $\mf U_1 \in\mathbb{R}^{M\times k}$, $\mf U_2 \in \mathbb{R}^{N\times k}$ 
    are the orthogonal matrices in the singular value decomposition of the matrix  $\mf C=\mf U_1\mf \Sigma \mf U_2$ and  $\mf \Sigma=\text{diag}(\sigma_{sv,j}(\mf C))_{j=1}^k$.
    Quantiles of the limiting distribution can be obtained in a similar way as described before. More precisely, estimates 
 of  $\mf U_1$ and $\mf U_2$ are directly obtained from the  SVD of the matrix   $\widehat {\mf C}$ and the estimate of the factor  $\tau^2 $    is defined in \eqref{nhd4}. 
 
If  $D_k(\mf C)>0$,  that is  $\text{rank}(\mf C)>k$, one can show by similar arguments as given in the proof  of \cref{th relevant test SVD} that 
 \begin{align}
     nTb^{1.5}\big (D_k(\widehat{\mf C})-D_k(\mf C) \big )\xrightarrow[]{d} N\Big (\mf 0,4 \tau ^2||\mf C-\mf C_k||^2_F \Big ) 
 \end{align}
     almost surely with respect to $\mathbb{Q}$, 
 where $\mf C_k$ denotes the best rank-$k$ approximation of the matrix $\mf C$ with respect to the Frobenius norm.
 }
\end{remark}

\subsection{Relevant  hypotheses for spatial-temporal separability}
\label{sec412} 

As mentioned in \cref{sec introduction}, the assumption of separability is typically imposed not because exact separability is believed to hold, but to simplify the computational complexity of spatial-temporal data analysis \citep{AstPigTav17}, with the expectation that valid inference with affordable computational
costs remains feasible even when the assumption is only approximately satisfied. Under this paradigm, testing the hypothesis of exact separability is not reasonable, and the confidence intervals defined by \eqref{conf1} and \eqref{conf2} can be used to investigate if the deviation from separability is not too large. Alternatively, we can also use Theorem \ref{thm42} and \ref{thm44} to develop a consistent asymptotic level $\alpha$ test for the hypothesis that the deviation from separability (measured by $D_{\bs \psi}$ or $D$) does not exceed a given threshold, say $\Delta$. 

For the sake of brevity, we discuss such a test for the measure $D$ based on the best rank-one approximation 
(Tests for the measure $D_{\bs  \psi}$ based on the partial trace approach can be obtained in exactly the same way). More specifically, we consider two types of hypothesis pairs, that is
\begin{align}
    \label{det111}
    H_0^{\rm rel}:  D  \leq \Delta ~~~~{ \rm versus }  ~~~~ H_1^{\rm rel}:  D  > \Delta ~, \\
      \label{det112}
    H_0^{\rm eq}: D   \geq   \Delta ~~~~{ \rm versus }  ~~~~ H_1^{\rm eq}:  D  <  \Delta ~, 
\end{align}
where $\Delta>0$ is a prespecified threshold (not necessarily the same in \eqref{det111} and \eqref{det112}). Hypotheses of this form are called {\it relevant or precise hypotheses} and reflect the fact  that, as pointed out by \cite{berger1987}, in many instances it is {\it ``rare, and perhaps impossible, to have a null hypothesis that can be exactly modeled by a parameter being precisely $0$'' (in our case $D=D_{\bs \psi}=0$)}.
Note that \eqref{det111} contains the hypothesis of exact separability as a special case $\Delta=0$.  The hypotheses \eqref{det112} seem to be more useful in the present context, because by rejecting $H_0$ it is possible to decide at a controlled type I error that the deviation from separability (measured by the SVD-based approach) is smaller than the threshold $\Delta$. Such hypotheses are quite popular in biostatistics, where one is interested in bioequivalence testing \citep{wellek2010testing}. For this reason, we use the index ``eq'' to distinguish them from \eqref{det11}.  Note also 
that more recently hypotheses of the form \eqref{det111}  have attracted considerable interest within the framework of 
{\it  tolerant testing}, where one assesses whether the data is consistent with any distribution that lies within a given neighborhood of the candidate \citep[see, for example,][and the references therein]{canonne22a, kania2026testingimprecisehypotheses}.

We propose to  reject the null hypothesis in \eqref{det111}  if 
\begin{align}
    \label{test1}
 \widehat{D} > \Delta + u_{1-\alpha} {\hat \vartheta_{sv} \over nTb^{1.5} } ,
\end{align}
where the variance estimator $\hat \vartheta_{sv}^2$ is defined in \eqref{d14}.
Similarly, we propose to reject the null hypothesis in \eqref{det112}, whenever
\begin{align}
    \label{test2}
  \widehat{D} <  \Delta +  u_{\alpha} {\hat \vartheta_{sv} \over nTb^{1.5} } .
\end{align}
Our next result shows that both decision rules define consistent and asymptotic level $\alpha$ tests for the hypotheses \eqref{det111} and \eqref{det112}.

\begin{corollary}
    \label{cor2}
     ~~
    Let the conditions  of Theorem \ref{thm44} be satisfied.
    \begin{itemize}
        \item [(a)] For the  test \eqref{test1}, we have 
\begin{align}
\lim_{n\rightarrow\infty}
\mathbb{P }\Big(\widehat{D} >  \Delta + u_{1- \alpha} {\hat \vartheta_{sv} \over nTb^{1.5} }  \Big) =     
\begin{cases}
1,      & \text{if }  {D}  >  \Delta,\\ 
\alpha, & \text{if }  {D}  = \Delta  ,\\
0,      & \text{if }  {D}  < \Delta.
\end{cases}
\end{align}
        \item[(b)] For the test \eqref{test2}, we have 
\begin{align}
\lim_{n\rightarrow\infty}
\mathbb{P }\Big(\widehat{D} <  \Delta + u_{\alpha} {\hat \vartheta_{sv} \over nTb^{1.5}  }  \Big) =     
\begin{cases}
1,      & \text{if }   {D} <  \Delta,\\ 
\alpha, & \text{if }   {D} = \Delta  ,\\
0,      & \text{if }   {D} > \Delta.
\end{cases}
\end{align}
           \end{itemize}
\end{corollary}

\begin{remark}\label{choiceofdelta} 
{\rm An important ingredient in the hypotheses \eqref{det111} or \eqref{det112}  is the threshold $\Delta$, which has to be carefully discussed in each application. We emphasize that   
$\Delta$ can also be chosen in a data-driven way. To be precise, we consider the hypotheses in \eqref{det112}, which are nested:  for different values of $\Delta$, rejection of the null hypothesis by the test \eqref{test2} at $\Delta=\Delta_0$ also implies rejection for all $\Delta \geq \Delta_0$. By the sequential rejection principle, the hypotheses in \eqref{det112} can thus be tested simultaneously  for different values of $\Delta$ to determine the minimal threshold  say
\begin{align}
    \widehat \Delta_\alpha:=  \min   
    \Big \{
    \big \{ 0 \big \} \cup \Big \{\Delta \ge 0 \,\Big| \,  \widehat{D} <  \Delta +  u_{\alpha} {\hat \vartheta_{sv} \over nTb^{1.5} } \Big \} \Big\},  
\end{align}
such that the null hypothesis in \eqref{det112} is rejected. Note that the null hypothesis is not rejected for all thresholds $\Delta \leq  \widehat\Delta_\alpha$ and rejected for $\Delta >  \widehat \Delta_\alpha$. Therefore, the quantity $\widehat \Delta_\alpha$ may be interpreted as a measure of evidence against the null hypothesis, with smaller values indicating stronger support for the assumption of separability.
}
\end{remark}

  \section{Finite sample properties}
  \label{sec5}
  \def\theequation{5.\arabic{equation}}
\setcounter{equation}{0}

In this section, we investigate the finite sample properties of the proposed methodology by means of a simulation study.  
All results presented here  are based on $1000$ simulation runs. 
For the sake of brevity, we concentrate on the hypothesis of exact separability in \eqref{d6} and investigate the  performance of the test \eqref{eq:reject} based on the partial trace approach  and the test \eqref{testsv} based on the best rank-one approximation.

We consider a centered, second-order and stationary real-valued Gaussian process
$$
 \big \{X(\mathbf{s},t) : \mathbf{s}\in [0,\sqrt{n}]\times  [0, \sqrt{n}],\, t\in \{1,\ldots T\} \big \}.
 $$
 and denote  by \( \mathbf{h}=\mathbf{s}-\mathbf{s}' \) the spatial lag and \( \tau=t-t' \) the temporal lag. 
 For the spatial locations, we choose $n$ 
 uniformly distributed  points $\mf s_1, \ldots, \mf s_{n}$ from the square $\big [0, \lfloor \sqrt{n} \rfloor \big ]\times  \big [0, \lfloor  \sqrt{n}\rfloor \big ]$.  For the evaluation distances in the matrix \eqref{d1}, we consider two scenarios. In the first one, we choose $M=3$ spatial lags and $N=3$ temporal lags  
 \begin{align}
\label{scenario1}
 \begin{pmatrix}  \mf h_1 &  \mf h_2 &  \mf h_3 
 \end {pmatrix}= \begin{pmatrix}
         1 & 2 & 3 \\
         1.5 &  1.75& 1 
     \end{pmatrix} 
     \times
   {\log(n) \over  \log(50)}~,~~ \bs  \tau = (2,3,0.5)^{\top} \times  {\log(n) \over  \log(50)}.
 \end{align}
 In  the second one, we choose $M=5$ spatial lags and $N=5$ temporal lags 
 \begin{align}
 \label{scenario2} 
   \begin{pmatrix}  \mf h_1 &  \mf h_2 &  \mf h_3 
 \end {pmatrix}= \begin{pmatrix}
     1 & 2 & 3& 3.5& 2.7\\ 
     1.5 &  1.75 & 1 & 0.5 & 1.3
     \end{pmatrix} 
    \times  {\log(n) \over  \log(50)}, ~~\bs  \tau  = (2,3,0.5, 1, 3.5)^{\top} \times  {\log(n) \over  \log(50)}.
 \end{align}
 The choice of the bandwidth $b$ in the procedures is a very  difficult problem, as, due to undersmoothing, it is not possible to balance bias and variance in the estimates.
Based on an intensive numerical study, we  propose the following rule,  which worked well in our simulation studies.
Noting that the bias of $\hat{\mf C}$ depends on the partial derivatives of  $ {\mf C}$, we introduce the quantity  
 \begin{align}
f_{MN} = \|C_{MN}\|_F /(\|\partial_1C_{MN}\|_F + \|\partial_2C_{MN}\|_F+\|\partial_3 C_{MN}\|_F)
 \end{align} 
 (in practice, this quantity has to be estimated), and use   
 \begin{align}
     b_{\rm PT} &=(E[X_{(0,0)}^2])^{3}T^{-0.75+0.01N} n^{-0.01-0.01M}(MN)^{0.2} f_{MN}^{0.2}  ,\label{eq:rule1} \\
   b_{\rm SVD}  &=  (E[X_{(0,0)}^2])^{2}  T^{-0.4+0.01N} n^{-0.15-0.01M}(MN)^{0.3}f_{MN}^{0.15}. \label{eq:rule2}
 \end{align}
 for the partial trace 
 and  the SVD approach, respectively.  The partial trace approach additionally requires the choice of a vector $\bs \psi $, and we used  $\bs \psi = (1,0,0)^{\top}$ in the case $M=N=3$ and $\bs \psi = (1,0,0,0,0)^{\top}$ in the case $M=N=5$ .

\smallskip

 \begin{table}[h]
     \centering
     \begin{tabular}{c|c|c c||c c|}
              & & \multicolumn{2}{c}{$M= N =3$}& \multicolumn{2}{c}{$M= N =5$}\\
         $n$  & $T$ & \eqref{eq:reject} & \eqref{testsv} & \eqref{eq:reject} & \eqref{testsv} \\
        \hline 
  75 & 100 &  4.20 (0.20) &  4.40 (0.21) &  5.60 (0.23) &  3.40 (0.18) \\ 
  100 & 100 &  4.60 (0.21) &  4.50 (0.21) &  4.20 (0.20) &  3.90 (0.19) \\ 
  150 & 100 &  5.10 (0.22) &  3.60 (0.19) &  2.40 (0.15) &  3.70 (0.19) \\ 
  100 & 200 &  3.40 (0.18) &  2.30 (0.15) &  2.90 (0.17) &  4.20 (0.20) \\ 
  150 & 200 &  3.90 (0.19) &  3.80 (0.19) &  5.10 (0.22) &  5.20 (0.22) \\ 
  200 & 200 &  6.30 (0.24) &  4.50 (0.21) &  4.50 (0.21) &  5.70 (0.23) \\ 
  200 & 250 &  4.10 (0.20) &  3.30 (0.18) &  4.90 (0.22) &  5.00 (0.22) \\ 
  250 & 250 &  5.10 (0.22) &  4.20 (0.20) &  5.00 (0.22) &  5.90 (0.24) \\ 
     \end{tabular}
     \caption{\it  Simulated rejection probabilities of
     the tests \eqref{eq:reject}  and the test \eqref{testsv}   under the null hypothesis of a separable  model  defined in   \eqref{dsim1}. The nominal level is $5\%$, and the 
     numbers in brackets denote the  standard deviation.  
       }
     \label{tab:null1}
 \end{table}
 
\medskip

To study the approximation of the nominal level under the null hypothesis, we consider the separable covariance structure  
\begin{align}
    \label{dsim1}
C_0(\mathbf{h},\tau)=2\exp\!\Big(-\|\mathbf{h}\|-|\tau|\Big).
\end{align}
The empirical rejected probabilities of both tests are displayed  in Table \ref{tab:null1}, where the nominal level is chosen as $5\%$. We observe a reasonable approximation of the nominal level in most cases.

To investigate the power, we consider  
  three alternatives. The first one is a model from the extended Gneiting class \citep[see][]{allard2020simulating}
\begin{align} \label{model1}
C_1(\h,\tau)=\frac{2}{|v|^{3/2}+1}\exp\left(\frac{-0.01||\h||^{2}}{|v|^{3/2}+1}\right).
\end{align}
Since  the generation of data with the  non-separable 
covariance structure \eqref{model1} is very memory consuming using the Cholesky decomposition, we adopt the spectral method proposed in  Algorithm 4 of \cite{allard2020simulating}.  The second model in our study is the product-sum model considered in  \cite{DEIACO20132002}, that is 
\begin{align}
\label{model2}
    C_2(\h,\tau) = 0.5 \times \{C_0(\h,\tau) +  C_0(\h/2,\tau/5) + C_s(\h) + C_t(\tau)\},
\end{align}
where $C_0$ is defined in \eqref{dsim1} and 
\[
C_s(\mathbf{h})= \exp\!\big(-\|\mathbf{h}\|\big),\qquad
C_t(\tau)=\exp\!\big(-|\tau|\big).
\]
The third model is given by  
\begin{align}
\label{model3}
C_3(\h,\tau)=\frac{\sigma^2}{(a|v|^{2\alpha}+1)^{\beta/2}}\exp\left(-\frac{r||\h||}{(a|v|^{2\alpha}+1)^{\beta/2}}\right),\label{model3}
\end{align}
 
\citep[see equation (4)  in][]{Gneiting01062002},  
where $\sigma^2 = 1.8$, $2\alpha = 1.544$, $a = 0.901$, $\beta = 0.61$, $r = 0.00134$, similar to the setting for the modeling of the Irish wind data in \cite{Gneiting01062002}. For its implementation, we implement Algorithm 1 and 3 in \cite{allard2020simulating}.

The rejection probabilities of both tests are displayed in Table \ref{tab:alt41} - \ref{tab:alt43}. We observe that both have reasonable power, and there is no clear winner between the two methods. Note that the power in Table \ref{tab:alt2} is substantially lower compared to Table \ref{tab:alt41} and  Table \ref{tab:alt43}.


 \begin{table}[!h]
     \centering
     \begin{tabular}{c|c|c c||c c|}
     & & \multicolumn{2}{c}{$M= N =3$}& \multicolumn{2}{c}{$M= N =5$}\\
        $n$  & $T$ & \eqref{eq:reject} & \eqref{testsv}& \eqref{eq:reject} & \eqref{testsv} \\
        \hline 
      75 & 100 &  47.10 (0.50) &  58.90 (0.49) &  45.80 (0.50) &  67.40 (0.47) \\ 
  100 & 100 &  77.80 (0.42) &  70.40 (0.46) &  58.70 (0.49) &  62.20 (0.48) \\ 
  150 & 100 &  77.70 (0.42) &  66.10 (0.47) &  77.60 (0.42) &  77.00 (0.42) \\ 
  100 & 200 &  81.80 (0.39) &  68.80 (0.46) &  69.30 (0.46) &  70.90 (0.45) \\ 
  150 & 200 &  85.30 (0.35) &  88.20 (0.32) &  75.70 (0.43) &  80.90 (0.39) \\ 
  200 & 200 &  99.20 (0.09) &  77.40 (0.42) &  92.40 (0.26) &  84.80 (0.36) \\ 
  200 & 250 &  92.70 (0.26) &  85.10 (0.36) &  88.90 (0.31) &  84.60 (0.36) \\ 
  250 & 250 &  80.40 (0.40) &  88.00 (0.32) &  99.60 (0.06) &  92.20 (0.27) \\ 
     \end{tabular}
     \caption{\it Simulated rejection probabilities of
     the tests \eqref{eq:reject}  and  \eqref{testsv}  under the alternative  \eqref{model1} of a  non-separable spatio-temporal model. The nominal level is $5\%$, and the numbers in brackets denote the  standard deviation. }
     \label{tab:alt41}
 \end{table}

  \begin{table}[!h]
     \centering
     \begin{tabular}{c|c|c c||c c|}
        & & \multicolumn{2}{c}{$M= N =3$}& \multicolumn{2}{c}{$M= N =5$}\\
        $n$  & $T$ & \eqref{eq:reject} & \eqref{testsv}& \eqref{eq:reject} & \eqref{testsv} \\
        \hline 
   75 & 100 &  62.40 (0.48) &  41.50 (0.49) &  43.60 (0.50) &  53.40 (0.50) \\ 
  100 & 100 &  53.20 (0.50) &  48.60 (0.50) &  22.40 (0.42) &  48.30 (0.50) \\ 
  150 & 100 &  39.50 (0.49) &  37.60 (0.48) &  25.70 (0.44) &  43.90 (0.50) \\ 
  100 & 200 &  50.60 (0.50) &  55.30 (0.50) &  58.20 (0.49) &  40.30 (0.49) \\ 
  150 & 200 &  52.80 (0.50) &  65.90 (0.47) &  44.40 (0.50) &  63.70 (0.48) \\ 
  200 & 200 &  50.40 (0.50) &  66.50 (0.47) &  46.70 (0.50) &  66.30 (0.47) \\ 
  200 & 250 &  48.00 (0.50) &  51.80 (0.50) &  59.90 (0.49) &  72.10 (0.45) \\ 
  250 & 250 &  51.60 (0.50) &  44.30 (0.50) &  54.50 (0.50) &  47.40 (0.50) \\ 
     \end{tabular}
     \caption{\it Simulated rejection probabilities of
     the tests \eqref{eq:reject}  and  \eqref{testsv}  under the alternative  \eqref{model2} of a  non-separable spatio-temporal model. The nominal level is $5\%$, and the numbers in brackets denote the  standard deviation. }
     \label{tab:alt2}
 \end{table}
 \begin{table}[!h]
     \centering
     \begin{tabular}{c|c|c c||c c|}
        & & \multicolumn{2}{c}{$M= N =3$}& \multicolumn{2}{c}{$M= N =5$}\\
        $n$  & $T$ & \eqref{eq:reject} & \eqref{testsv}& \eqref{eq:reject} & \eqref{testsv} \\
        \hline 
   75 & 100 &  89.90 (0.30) &  75.10 (0.43) &  100.00 (0.00) &  7.60 (0.26) \\ 
  100 & 100 &  85.10 (0.36) &  87.40 (0.33) &  100.00 (0.00) &  76.60 (0.42) \\ 
  150 & 100 &  99.40 (0.08) &  89.40 (0.31) &  96.50 (0.18) &  85.10 (0.36) \\ 
  100 & 200 &  96.10 (0.19) &  87.60 (0.33) &  100.00 (0.00) &  92.00 (0.27) \\ 
  150 & 200 &  100.00 (0.00) &  100.00 (0.00) &  96.00 (0.20) &  96.00 (0.20) \\ 
  200 & 200 &  100.00 (0.00) &  96.00 (0.20) &  100.00 (0.00) &  94.80 (0.22) \\ 
  200 & 250 &  100.00 (0.00) &  96.10 (0.19) &  100.00 (0.00) &  92.00 (0.27) \\ 
  250 & 250 &  100.00 (0.00) &  94.30 (0.23) &  93.60 (0.24) &  99.30 (0.08) \\ 
     \end{tabular}
     \caption{\it  Simulated rejection probabilities of
     the tests \eqref{eq:reject}  and  \eqref{testsv}  under the alternative  \eqref{model3} of a  non-separable spatio-temporal model. The nominal level is $5\%$ and the numbers in brackets denote the  standard deviation. }
     \label{tab:alt43}
 \end{table}

\newpage

 {\bf Acknowledgments. } 
This work  has been supported
by the Deutsche Forschungsgemeinschaft (DFG)  TRR
391 {\it Spatio-temporal Statistics for the Transition of Energy and Transport}, project number
520388526 (DFG).

\bibliographystyle{apalike}
\bibliography{main}

\newpage
\section*{Appendix}
In Appendix \ref{seca}, we provide the detailed proofs of all the main results in the main paper. In Appendix \ref{secb}, we prove almost surely central limit theorems for triangular arrays, which are useful for the proofs in Appendix \ref{seca}. The technical results required for the proofs of the main results are documented in Appendix \ref{secc}.

\begin{appendix}

\section{Proofs of main results}  \label{seca}

\def\theequation{A.\arabic{equation}}
\setcounter{equation}{0}


\subsection{\texorpdfstring{Proof of Theorem \ref{Theorem CLT Covariance Estimator}}{Proof of Theorem 3.1}} \label{seca11}

 The proof is decomposed into two steps. In the first step, we prove the weak convergence of the numerator of the statistic  \eqref{eq covariance estimator} (see Section \ref{seca13}). The numerator is a $U$-statistic with kernel depending on $n$ for dependent data in the time space domain, the proof techniques presented in this section might  be of its own interest. Next, in the second step, we use this result to establish the weak convergence of the estimator of the covariance (see Section \ref{seca12}). Meanwhile, we introduce the following notation which is important in Section \ref{seca12} 
 \begin{align}
 \label{eq abbreviation}
     \h_{0n}=(h_{0,1n},h_{0,2n}):=\Lambda_n^{-1}\h_0=\text{diag}(\lambda_n^{-1}h_{01}, (\lambda_n\bb_n)^{-1}h_{02}),\ v_{0T}=v_0/T.
 \end{align}

\subsubsection{Weak convergence of the numerator of the statistic  \eqref{eq covariance estimator}} \label{seca13}

Before we start the proof, we introduce  the following  notation. Given $(\h_0,v_0)\in\{(\h_i,v_j)\}_{1 \leq i \leq M, 1 \leq j \leq N}$, for any given $\bS_n= S_n$, we define $X_{it}=X_{(s_{in},t)}$ and 
\begin{align}
    \label{eq definition of K0tt'}
  K_{0tt'} & =K\Big (\frac{|t-t'|-v_0}{ Tb}\Big )=K\Big(\frac{|t-t'|/T-v_{0T}}{b}\Big),\\
  K_{0ii'} &=K \Big (\frac{\lambda_n^{-1}(s_{in,1}-s_{i'n,1}-h_{0,1})}{b} \Big  )K \Big  (\frac{(\lambda_n\bb_n)^{-1}(s_{in,2}-s_{i'n,2}-h_{0,2})}{b} \Big  )\\
      \label{eq definition of K0ii'}
  &\ \ \ \ \ \ \ \ =K \Big  (\frac{x_{in,1}-x_{i'n,1}-h_{0,1n}}{b} \Big  )K \Big  (\frac{x_{in'2}-x_{i'n,2}-h_{0,2n}}{b} \Big  ).
  \end{align}
With this notation, the following identity holds for any given $\bS_n=S_n$, 
\begin{align}
    \label{d18}
K_{0b}(s_{in}-s_{i'n},t-t') = K_{0tt'}  K_{0ii'}, 
\end{align}
and we will use  the notation on the left and right hand side of \eqref{d18} simultaneously in the proof. For example, on condition of $\bS_n=S_n$, 
the  covariance estimator in  \eqref{eq covariance estimator}  can be rewritten as follow, 
\begin{align}
      \hat{C}(\mathbf{h}_0,v_0)=\frac{\sum_{1\leq t\neq t'\leq T}\sum_{1\leq i\neq i'\leq n} K_{0ii'}K_{0tt'}X_{it}X_{i't'}}{\sum_{1\leq t\neq t'\leq T}\sum_{1\leq i\neq i'\leq n} K_{0ii'}K_{0tt'}},
\end{align}
where $b$ is the chosen bandwidth.  Additionally, we introduce the following notation
\begin{align} 
\label{d15a} 
U_n & =\sum_{1\leq t\neq t'\leq T}\sum_{1\leq i\neq i'\leq n}K_{0ii'}K_{0tt'}(X_{it}X_{i't'}-E_{|S_n}[X_{it}X_{i't'}]); \\
\label{d15b} 
Z_{it}& =(s_{in},t/T, X_{(s_{in},t)})=(s_{in},t/T, X_{it});  \\
\label{d15c} 
g_{n}(Z_{it},Z_{i't'}) & = K_{0b}(s_{in}-s_{i'n},t-t')X_{it}X_{i't'};  \\
\label{d15d} 
H_{n}(Z_{it},Z_{i't'})& =g_{n}(Z_{it},Z_{i't'})-E_{|S_n}[g(Z_{it},Z_{i't'})];
\end{align}
where  $E_{|S_n}[\cdot] = E[\cdot {|\bS_n }=S_n]$ is the conditional expectation of given value, $\bS_n=S_n$. Note that \eqref{d15a}  is a centered version of the numerator in \eqref{eq covariance estimator}.
With the notation    $m(x,y)=xy-E_{|S_n}[xy]$, we  have  
\begin{align}
    \label{d15e}
m(X_{it},X_{i't'})=X_{it}X_{i't'}-E_{|S_n}[X_{it}X_{i't'}],
\end{align}
and  finally obtain for the the random variable  $H_n$ in \eqref{d15d} the representation
$$
H_{n}(Z_{it},Z_{i't'})=K_{0b}(s_{in}-s_{i'n},t-t')m(X_{it},X_{i't'})=K_{0ii'}K_{0tt'}m(X_{it},X_{i't'}).
$$

\par

\begin{theorem}
    \label{Theorem CLT U-Statistic}
Under Assumption \ref{as spatial locations 1}-\ref{technical assumptions}, for any given $(\h_0,v_0)\in\{(\h_i,v_j)\}_{1\leq i\leq M;1\leq j\leq N}$, 
\begin{align}
    &U_n/\sigma_n \xrightarrow[]{d} N(0,1)\ \ \ a.s.-[\mathbb{Q}],
    \end{align}
    where
    \begin{align}\sigma^2_n&=(nT)^2b^3(E[X^2_0])^2\mathbf{A}_2\mathbf{B}_2,
\end{align} 
and $\mathbf{A}_2$ and $\mathbf{B}_2$ are defined in \eqref{d7c} and \eqref{d7d}, respectively.
\end{theorem}


\begin{proof}
We first introduce the  spatial-temporal blocks and derive a decomposition of the statistic $U_n$, which turns out to be  crucial for proving the central limit theorem.  The construction of the spatial-temporal is valid for sufficiently large $n$ and $T$  and illustrated in Figure \ref{fig:block}. 
\begin{figure}
    \centering
    \includegraphics[width=0.6\linewidth]{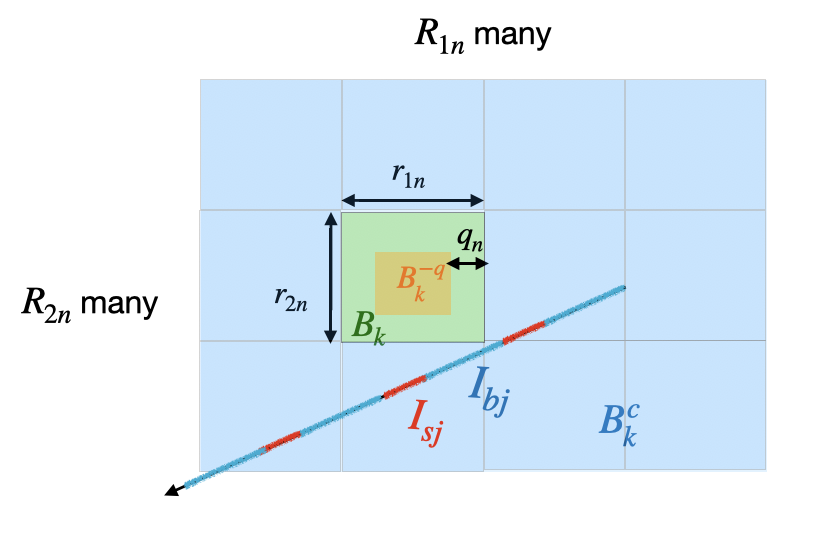}
    \caption{Illustration of blocking in the spatial and temporal domain.}
    \label{fig:block}
\end{figure}
\medskip

\par We start defining \textbf{spatial blocks} partitioning  $\mathbf{R}_n=\lambda_n\mathbf{D}\subset [-\lambda_n,\lambda_n]^2$. For a given $0<r_{1n}\leq \lambda_n/2$, parallel to y-axis, we first decompose $[-\lambda_n,\lambda_n]^2$ into $R_{1n}=[2\lambda_n/r_{1n}]+1$ smaller rectangles whose longer edge is of length $2\lambda_n$ and shorter edge is of length $r_{1n}$. 
This also gives a vertical partition of the set $\mathbf{R}_n$, which is now refined in horizontal direction.  For this purpose we chose  $0<r_{2n}\leq \frac{\lambda_n\bb_n}{2}$, define $R_{2n}=[2\lambda_n\bb_n/r_{2n}]+1$, we can partition $\mathbf{R}_n$ into $R_{2n}$ sets of the form  
\begin{align}
    \{(x,y)\in \mathbf{R}_n: \lambda_nl(x)-2\lambda_n\bb_n+(k-1)r_{2n}\leq y\leq \lambda_nl(x)-2\lambda_n\bb_n+kr_{2n} \},
\end{align}
($k=1,2,...,R_{2n}$). By combining these two partitions together, we obtain a  partition of $\mathbf{R}_n$ into $R_n=R_{1n}R_{2n}$ smaller disjoint segments
$B_1 , \ldots , B_{R_n}$
and the area of each segment is at most  $r_n:= r_{1n}r_{2n}$. For a constant  $q_n>0$, we define $B^{-q_n}_{k}=B_{k}\backslash (B^{c}_{k})^{+q_n}$, where, for  $A\subset \mathbb{R}^2$   the  set   $A^{+q_n}=\{x\in\mathbb{R}^2:\inf_{y\in A}||x-y||<q_n\}$ denotes  its $q_n$-blow-up. We also let 
\begin{align}
   r_{kn}=q_n\log n, \ k=1,2, \label{eq:r1r2}
\end{align}
which is valid according to point (2) in (T4) of Assumption \ref{technical assumptions}.
\medskip 

 We now construct   corresponding \textbf{time blocks} setting  
$q_T=q_n$, where $q_n$ is introduced in spatial blocking and choosing $p_{T}$ such that   $0<r_{3T}=p_{T}+q_{T}\leq T/2$, $q_T/p_T=\frac{1}{\log T}$, $p_{T}/T=o(1)$. Then, we can partition $\mathcal{T}:=\{1,\ldots  ,T\}$ into $J_n=[T/r_{3T}]$   blocks. More specifically, we decompose $\mathcal{T}$  into consecutive  ``small'' blocks   $I_{sj}=\{b_{j},...,b_{j}+q_{T}-1\}$ of size $q_T$ and ``big'' blocks $I_{bj}=\{a_{j-1},...,a_{j-1}+p_T-1\}$ of size $p_T$ and a remaining block $I_{r}=\{1,2,...,T\}\backslash\bigcup_{j=1}^{J_n}(I_{bj}\cup I_{sj})$, where we start with a big block, that is $a_0=1$.

\par Based on these  blocks we define,  for any $k=1,2,..,R_n$, $j=1,2,...,J_n$, the sets 
\begin{align}
    &\tilde{B}^-_{k,j}=B^{-q_n}_{k}\times I_{bj},\\
\label{d21}    &\tilde{B}_{k,j}=B_{k}\times (I_{bj}\cup I_{sj}\cup I_{s(j+1)})=:B_{k}\times I_j,\\
 &I_{s(J_n+1)}=I_r.
\end{align}
(see also Figure \ref{fig:block}) and  obtain the following decomposition  of the statistic $U_n$ in \eqref{d15a}. 
\begin{align}
\label{d16} 
    U_n&=U_{1n}+U_{2n}+U_{3n}+U_{4n}+U_{5n},\\
\end{align} 
where 
\begin{align}
\label{d16a}
U_{1n}&= \sum_{k=1}^{R_n}\sum_{j=1}^{J_n}\sum_{(i,t)\in\tilde{B}^-_{k,j}}\sum_{(i',t')\in \tilde{B}^c_{k,j}}H_n(Z_{it},Z_{i't'}), \\
\label{d16b}
    U_{2n}&=\sum_{k=1}^{R_n}\sum_{j=1}^{J_n}\sum_{(i,t)\in\tilde{B}^-_{k,j}}\sum_{(i',t')\in \tilde{B}_{k,j}}H_n(Z_{it},Z_{i't'}),\\
    \label{d16c}
    U_{3n}&= \sum_{k=1}^{R_n}\sum_{j=1}^{J_n}\sum_{(i,t)\in \tilde{B}_{k,j}\backslash\tilde{B}^-_{k,j}}\sum_{(i',t')\in \tilde{B}_{k,j}}H_n(Z_{it},Z_{i't'}),\\
    \label{d16d}
    U_{4n}&=\sum_{k=1}^{R_n}\sum_{j=1}^{J_n}\sum_{(i,t)\in \tilde{B}_{k,j}\backslash\tilde{B}^-_{k,j}}\sum_{(i',t')\in \tilde{B}_{k,j}^c}H_n(Z_{it},Z_{i't'}),\\
    \label{d16e}
    U_{5n}&= U_{n}-(U_{1n}+U_{2n}+U_{3n}+U_{4n}).
\end{align}

The main strategy of this proof is to prove that the asymptotic normality of $(nT)b^3U_{1n}$ and $(nT)b^3U_{kn}=o_{\mathbb{P}_{|S_n}}(1)$, $2\leq k\leq 5$, hold almost surely with respect to $\mathbb{Q}$.

\bigskip

 \textbf{Step 1} (\textbf{Shadow Sample}) We introduce a  shadow sample of random variables  $\{\X_{it}:1\leq i\leq n, 1\leq t\leq T\}$ such that 



\begin{itemize}
    \item [(P1)] $\{\X_{it}:1\leq i\leq n, 1\leq t\leq T\}$ is a group of mutually independent random variables.
    \item [(P2)] For every $i$ and $t$, $\X_{it}$ is identically distributed as $X_{it}$.
    \item [(P3)] For each $n$ and $T$, the sigma algebras generated by $\{\X_{it}:1\leq i\leq n, 1\leq t\leq T\}$ and $\{X_{n}(\s_{in},t),\s_{in}:1\leq i\leq n, 1\leq t\leq T\}$ are independent from each other.
\end{itemize}
In the following discussion, we focus on presenting  technical details mainly for the  terms $U_{1n}$ and $U_{2n}$ since the other terms can be treated in the same way. More specifically, it will be shown that $U_{1n}$ is the leading term.\\
\par \textbf{Step 2 (calculation of the  conditional variance  $\text{Var}_{|S_n}(U_{1n})$ given  $(s_{in}:1\leq i\leq n)$).}
 Recall the definition of $U_{1n}$ in \eqref{d16a}, define  $Y_{k,j}(Z_{it}):=\sum_{(i',t')\in \tilde{B}^c_{k,j}}H_n(Z_{it},Z_{i't'})$  and  $T_{k,j} := \sum_{(i,t)\in\tilde{B}^-_{k,j}}Y_{k,j}(Z_{it})$, $k=1,\ldots, R_n$, $j = 1, \ldots, J_n$, then we obtain the representation
\begin{align}
     U_{1n}=:\sum_{(k,j)}\sum_{(i,t)\in\tilde{B}^-_{k,j}}Y_{k,j}(Z_{it})=\sum_{(k,j)}T_{k,j}.
\end{align}
Since $E_{|S_n}[H_n(Z_{it},Z_{i't'})]=0$ holds for any $1\leq i,i'\leq n$ and $1\leq t,t'\leq T$, we have
\begin{align}
    \text{Var}_{|S_n}[U_{1n}]=V_1+V_2, \label{eq:V1V2}
\end{align}
where
\begin{align}
    V_1& =E_{|S_n}[U^2_{1n}]= \sum_{(k,j)}E_{|S_n}[T^2_{k,j}],
    \\
    V_2 & =\sum_{(k,j)\neq (k',j')}E_{|S_n}[T_{k,j}T_{k',j'}].
\end{align}
\textbf{Step 2.1  (calculation of $V_1$ of \eqref{eq:V1V2}).} For each term  $E_{|S_n}[T^2_{k,j}]$ we have 
\begin{align}
\label{d17}
    E_{|S_n}[T^2_{k,j}]=\sum_{(i,t)\in \tilde{B}_{k,j}^-}E_{|S_n}[Y_{k,j}^2(Z_{it})]+\sum_{(i,t)\neq (i',t')}E_{|S_n}[Y_{k,j}(Z_{it})Y_{k,j}(Z_{i't'})]=:I_{k,j}+II_{k,j}.
\end{align}
Therefore, we can investigate $V_1$ by studying  $I_{k,j}$, $II_{k,j}$.

\textbf{Step 2.1.1(calculation of $I_{k,j}$ in \eqref{d17}).}
Simple algebra gives
\begin{align}
    I_{k,j}&=\sum_{(i,t)\in\tilde{B}^-_{k,j}}\sum_{(i',t')\in \tilde{B}^c_{k,j}}E_{|S_n}[H_{n}^2(Z_{it},Z_{i't'})] \\
    &+\sum_{(i,t)\in\tilde{B}^-_{k,j}}\sum_{
    (i',t'),(i'',t'')\in\tilde{B}^c_{k,j}}E_{|S_n}[H_{n}(Z_{it},Z_{i't'})H_{n}(Z_{it},Z_{i''t''})]=:I_{k,j,1}+I_{k,j,2} \label{eq:IKj}
\end{align}
%
%
\textbf{Step 2.1.1.a (calculation of $I_{k,j,1}$ in  \eqref{eq:IKj}).}
Recall that $ m(X_{it},X_{i't'})=X_{it}X_{i't'}-E_{|S_n}[X_{it}X_{i't'}]$ and $H_{n}(Z_{it},Z_{i't'})=K_{0b}(s_{in}-s_{i'n},t-t')m(X_{it},X_{i't'})$.   Based on the shadow sample $\X_{it}$ introduced in Step 1, together with the fact that $||(i,t)-(i',t')||_{F}>q_n$, we have 
\begin{align}
    E_{|S_n}[m^2(X_{it},X_{i't'})]=(E_{|S_n}[m^{2}(X_{it},X_{i't'})]-E_{|S_n}[m^{2}(\X_{it},\X_{i't'})])+E_{|S_n}[m^{2}(\X_{it},\X_{i't'})].
\end{align}
Lemma \ref{lemma 3} and Assumption \ref{dependence} yield
\begin{align}
    &|E_{|S_n}[m^{2}(X_{it},X_{i't'})]-E_{|S_n}[m^{2}(\X_{it},\X_{i't'})]|\\
    =&|E_{|S_{n}}(X_{it}X_{i't'}-E_{|S_n}[X_{it}X_{i't'}])^2- E_{|S_{n}}(\X_{it}\X_{i't'}-E_{|S_n}[\X_{it}\X_{i't'}])^2|                                                             \\
    \leq &|E_{|S_{n}}(X_{it}X_{i't'})^2- E_{|S_{n}}(\X_{it}\X_{i't'})^2|+(E_{|S_n}[X_{it}X_{i't'}])^2 \\
    \leq& 4\Psi(1,1)M_{2(1+\delta)}\beta^{\frac{\delta}{1+\delta}}(q_n)+16 \Psi^2(1,1)M_{1+\delta}^{2}\beta^{\frac{2\delta}{1+\delta}}(q_n)
\end{align}
where $M_{1}$ and $M_2$ are defined in condition (M2) of Assumption \ref{moment conditions}. Together with condition (M1) of Assumption \ref{moment conditions} and the fact that $E_{|S_n}[m^2(\X_{it},\X_{i't'})]=E[m^2(\X_{it},\X_{i't'})]=E[\X^2_{it}]E[\X^2_{i't'}]=(E[X^2_0])^2$, we obtain
\begin{align}
    I_{k,j,1}=& \sum_{(i,t)\in\tilde{B}^-_{k,j}}\sum_{(i',t')\in \tilde{B}^c_{k,j}}K^2_{0b}(s_{in}-s_{i'n},t-t')(E_{|S_n}[m^{2}(X_{it},X_{i't'})]-E_{|S_n}[m^{2}(\X_{it},\X_{i't'})])\\
    &+ \sum_{(i,t)\in\tilde{B}^-_{k,j}}\sum_{(i',t')\in \tilde{B}^c_{k,j}}K^2_{0b}(s_{in}-s_{i'n},t-t')E_{|S_n}[m^{2}(\X_{it},\X_{i't'})]=:H_1(k,j)+H_2(k,j),
\end{align}
where $H_1$ and $H_2$ satisfy
\begin{align}    
   | H_1(k,j)|\leq& 4\Psi(1,1)(E|X_{0}|^{^{2(1+\delta)}})^{\frac{1}{1+\delta}}\beta^{\frac{\delta}{1+\delta}}(q_n) \sum_{(i,t)\in\tilde{B}^-_{k,j}}\sum_{(i',t')\in \tilde{B}^c_{k,j}}K^2_{0b}(s_{in}-s_{i'n},t-t'),\\
    H_2(k,j)=&\sum_{(i,t)\in\tilde{B}^-_{k,j}}\sum_{(i',t')\in \tilde{B}^c_{k,j}}K^2_{0b}(s_{in}-s_{i'n},t-t')(E[X^2_{0}])^2.
\end{align}
\par Then, for each $(k,j)$,
\begin{align}
\label{proof 2}
    I_{k,j,1}=H_{2}(k,j)(1+o(1))
\end{align}
holds almost surely with respect to $\mathbb{Q}$.
%
%

\textbf{Step 2.1.1.b  (calculation of  $I_{k,j,2}$  in \eqref{eq:IKj}).} 
Recalling the notation  $g_{n}(Z_{it},Z_{i't'})=K_{0b}(s_{in}-s_{i'n},t-t')X_{it}X_{i't'}$ and $H_{n}(Z_{it},Z_{i't'})=g_{n}(Z_{it},Z_{i't'})-E_{|S_n}[g(Z_{it},Z_{i't'})]$ we have 
\begin{align}
   I_{k,j,2}=&\sum_{(i,t)\in\tilde{B}^-_{k,j}}\sum_{(i',t'),(i'',t'')\in \tilde{B}^c_{k,j} }\text{Cov}_{|S_n}(g_n(Z_{it},Z_{i't'}),g_n(Z_{it}Z_{i''t''}))\\
   =&\sum_{(i,t)\in\tilde{B}^-_{k,j}}\sum_{(i',t'),(i'',t'')\in \tilde{B}^c_{k,j} }K_{0b}(s_{in}-s_{i'n}, t-t')K_{0b}(s_{in}-s_{i''n},t-t'')\text{Cov}_{|S_n}(X_{it}X_{i't'},X_{it}X_{i''t''}), \label{eq:Ikj2}
\end{align}  
where $\text{Cov}_{|S_n}$ is the conditional covariance  given  $S_n$.   The Cauchy Schwartz inequality yields
\begin{align}
    |\text{Cov}_{|S_n}(X_{it}X_{i't'},X_{it}X_{i''t''})|\leq |E_{|S_n}[(X_{it}X_{i't'})(X_{it}X_{i''t''})]|+|E_{|S_n}[X_{it}X_{i't'}]||E_{|S_n}[X_{it}X_{i''t''}]|, \label{eq:cov12}
\end{align}
and we derive an  upper bound for  $I_{k,j,2}$  by investigating the first and second terms of \eqref{eq:cov12} separately.
 
\textbf{The first term in  \eqref{eq:cov12}.}
 Observing the definition of  
$\{\X_{it}\}$ in Step 1 we have 
\begin{align}
    &|E_{|S_n}[(X_{it}X_{i't'})(X_{it}X_{i''t''})]|\leq \sum_{k=1}^4E_{k}, \label{eq:E14}        
    \end{align}  
{where}
  \begin{align}
& E_{1}=|E_{|S_n}[(X_{it}X_{i't'}-\X_{it}\X_{i't'})(X_{it}X_{i''t''}-\X_{it}\X_{i''t''})]|  \\
&E_2=|E_{|S_n}[(X_{it}X_{i't'}-\X_{it}\X_{i't'})\X_{it}\X_{i''t''}]| \\
&E_3=|E_{|S_n}[\X_{it}\X_{i't'}(X_{it}X_{i''t''}-\X_{it}\X_{i''t''})]| \\
&E_4=|E_{|S_n}[\X^2_{it}\X_{i't'}\X_{i''t''}]|. 
\end{align}
Since $\{\X_{it}:1\leq i\leq n, 1\leq t\leq T\}$ is mutually independent and $\X_{it}$ is identically distributed as $X_{it}$, we immediately obtain 
\begin{align}
    E_4 = 0. \label{eq:E4}
\end{align}
Furthermore, due to the independence between $\{\X_{it}\}$ and $\{X_{it}\}$, we also have
\begin{align}
    E_2&=|E_{|S_n}[X_{it}X_{i't'}]E_{|S_n}[\X_{it}]E_{|S_n}[\X_{i''t''}]-E_{|S_n}[\X^2_{it}]E_{|S_n}[\X_{i't'}]E_{|S_n}[\X_{i''t''}]|=0 \label{eq:E2}\\
    E_3&=|E_{|S_n}[X_{it}X_{i't'}]E_{|S_n}[\X_{it}]E_{|S_n}[\X_{i't'}]-E_{|S_n}[\X^2_{it}]E_{|S_n}[\X_{i't'}]E_{|S_n}[\X_{i''t''}]|=0 \label{eq:E3}.
\end{align}
Thus, we only need to focus on $E_1$. It is obvious that 
\begin{align}
    E_1=&\Big|E_{|S_n}[X^2_{it}X_{i't'}X_{i''t''}]-E_{|S_n}[\X_{it}]E_{|S_n}[\X_{i't'}]E_{|S_n}[X_{it}X_{i''t''}]\\
    &-E_{|S_n}[X_{it}X_{i't'}]E_{|S_n}[\X_{it}]E_{|S_n}[\X_{i''t''}]+E_{|S_n}[\X^2_{it}\X_{i't'}\X_{i''t''}]\Big|\\
    =&|E_{|S_n}[X^2_{it}X_{i't'}X_{i''t''}]-E_{|S_n}[\X^2_{it}\X_{i't'}\X_{i''t''}]|.
\end{align}
Here the second equality follows from  the following three facts: (1)$ \X_{it}$'s are mutually independent and $\{\X_{it}\}$ is independent of $\{X_{it}\}$ and $\{\s_{in}\}$; (2) $E_{|S_n}[\X_{it}]=E_{|S_n}[X_{it}]=0$; (3) $a+E_{|S_n}[\X^2_{it}\X_{i't'}\X_{i''t''}]=a-E_{|S_n}[\X^2_{it}\X_{i't'}\X_{i''t''}]$ holds for any $a\in\mathbb{R}$ since $E_{|S_n}[\X^2_{it}\X_{i't'}\X_{i''t''}]=0$. Then, the following result holds for every $(i,t)$, $(i',t')$ and $(i'',t'')$,
\begin{align}
   E_1 &\leq |E_{|S_n}[X^2_{it}X_{i't'}X_{i''t''}]-E_{|S_n}[\X^2_{it}X_{i't'}X_{i''t''}]|+|E_{|S_n}[\X^2_{it}X_{i't'}X_{i''t''}]-E_{|S_n}[\X^2_{it}\X_{i't'}\X_{i''t''}]|\\
   &=\left|\int_{\mathbb{R}^3}x^2(yz)d\mathbb{P}_{X_{it},(X_{i't'},X_{i''t''})|S_n}-\int_{\mathbb{R}^3}x^2(yz)d\mathbb{P}_{X_{it}|S_n}\otimes\mathbb{P}_{(X_{i't'},X_{i''t''})|S_n}\right|\\
   &+E[\X^2_{it}]\left|\int_{\mathbb{R}^2}xyd\mathbb{P}_{(X_{i't'},X_{i''t''})|S_n}-\int_{\mathbb{R}^2}xyd\mathbb{P}_{X_{i't'}|S_n}\otimes\mathbb{P}_{X_{i''t''}|S_n}\right|,
\end{align}
where $\mathbb{P}_{W|S_n}\otimes \mathbb{P}_{V|S_n}$ denotes the product-type joint distribution of random vector $(W,V)$ on condition of $S_n$. Then, by applying Lemma \ref{lemma 3}, we have
\begin{align}
    E_1\leq& 4\Psi(1,2)(E|X_{0}|^{2(1+\delta)})^{\frac{1}{1+\delta}}||X_{i't'}X_{i''t''}||_{1+\delta}|\beta^{\frac{\delta}{1+\delta}}(q_n)  \\
    &+4\Psi(1,1)E[X_0^2]||X_0||_{1+\delta}^2\beta^{\frac{\delta}{1+\delta}}(||(i',t')-(i'',t'')||_{F}) \label{eq:E1}.
\end{align}
Combining \eqref{eq:E14}, \eqref{eq:E1}, \eqref{eq:E2}, \eqref{eq:E3} and \eqref{eq:E4}, we conclude that
\begin{align}
\label{proof 3}
    |E_{|S_n}[X_{it}X_{i't'}(X_{it}X_{i''t''})]|\leq& 4\beta^{\frac{\delta}{1+\delta}}(q_n)\Psi(1,2)(E|X_{0}|^{2(1+\delta)})^{\frac{2}{1+\delta}} \\&+4\Psi(1,1)E[X_0^2]||X_0||_{1+\delta}^2\beta^{\frac{\delta}{1+\delta}}(||(i',t')-(i'',t'')||_{F}).
\end{align}
 \textbf{The second term of \eqref{eq:cov12}.}
By similar arguments for the first term we have   for any $(i,t)\neq (j,s)$  
\begin{align}
    |E_{|S_n}[X_{it}X_{js}]|=|\text{Cov}_{|S_n}(X_{it},X_{js})|\leq 4||X_0||^2_{1+\delta}\Psi(1,1)\beta^{\frac{\delta}{1+\delta}}(\min\{||s_{in}-s_{jn}||_{F}, |t-s|\}),
\end{align}
which asserts
\begin{align}
    \label{proof 4}
   | E[X_{it}X_{i't'}]||E[X_{it}X_{i''t''}]|\leq 16||X_0||^4_{1+\delta}\Psi^2(1,1)\beta^{\frac{2\delta}{1+\delta}}(q_n).
\end{align}
Hence, according to \eqref{eq:Ikj2},  \eqref{eq:cov12}, \eqref{proof 3} and \eqref{proof 4}, there exist  positive constants $C_{X_0}$ and $C'_{X_0}$ associated with the moment conditions of $X_0$ (see Assumption \ref{moment conditions}) but independent of $n$ such that
\begin{align}
\label{proof 5}
    |I_{k,j,2}|\leq &C_{X_0}\sum_{(i,t)\in\tilde{B}^-_{k,j}}\sum_{(i',t'),(i'',t'')\in \tilde{B}^c_{k,j} \atop ||(i',t')-(i'',t'')||_{F}\leq q_n}K_{0b}(s_{in}-s_{i'n},t-t')K_{0b}(s_{in}-s_{i''n},t-t'') \\
    & ~~~~~~~~~~~~~~~~~~~~~~~~ \times  \beta^{\frac{\delta}{1+\delta}}(||(i',t')-(i'',t'')||_{F}) \\
    &+C'_{X_0}\sum_{(i,t)\in\tilde{B}^-_{k,j}}\sum_{(i',t'),(i'',t'')\in \tilde{B}^c_{k,j} \atop ||(i',t')-(i'',t'')||_{F}> q_n}K_{0b}(s_{in}-s_{i'n},t-t')K_{0b}(s_{in}-s_{i''n},t-t'')\beta^{\frac{\delta}{1+\delta}}(q_n) \\
    =& I_{k,j,21}+I_{k,j,22}
\end{align}
holds almost surely with respect to $\mathbb{Q}$. Proposition \ref{prop 6}   and \ref{prop 8}  yield
\begin{align}
    \lim_{n\to\infty}\frac{|I_{k,j,21}|}{n\textbf{Card}(\widetilde{B}^-_{k,j}\cap \Gamma_n)Tp_Tb^4}=0, \\
      \lim_{n\to\infty}\frac{|I_{k,j,22}|}{n\textbf{Card}(\widetilde{B}^-_{k,j}\cap \Gamma_n)Tp_Tb^3}=0
\end{align}
almost surely with respect to $\mathbb{Q}$, respectively, which gives,  for every $k,j$,
\begin{align}
    \label{proof 6}
   \lim_{n\to\infty}\frac{|I_{k,j,2}|}{n\textbf{Card}(\widetilde{B}^-_{k,j}\cap \Gamma_n)Tp_Tb^3}=0
\end{align}
holds almost surely with respect to $\mathbb{Q}$, uniformly for $k=1,\ldots, R_n$, $j = 1,\ldots, J_n$.

Summarizing these calculations, we obtain for the first term in \eqref{d17} that 
\begin{align}
\label{proof 7}
    I_{k,j}= H_{2}(k,j)(1+o(1))+o(n\textbf{Card}(\widetilde{B}^-_{k,j}\cap \Gamma_n)Tp_Tb^3)
\end{align}
holds almost surely with respect to $\mathbb{Q}$ for each $k,j$, where the small $o$'s do not depend on $k$ and $j$.

\textbf{Step 2.1.2 (calculation of $II_{k,j}$ in \eqref{d17}).}   For each $\bS_n=S_n$, simple algebra yields
\begin{align}
    II_{k,j}=&\sum_{(i,t)\neq (i',t')\in \widetilde{B}^-_{k,j}}\sum_{(r,s)\in \widetilde{B}^c_{k,j}}\sum_{(r',s')\in \widetilde{B}^c_{k,j}}E_{|S_{n}}[H_n(Z_{it},Z_{rs})H_{n}(Z_{i't'},Z_{r's'})]\\
    =&\sum_{(i,t)\neq (i',t')\in \widetilde{B}^-_{k,j}}\sum_{(r,s)\in \widetilde{B}^c_{k,j}}\sum_{(r',s')\in \widetilde{B}^c_{k,j}}K_{0b}(s_{in}-s_{rn},t-s)K_{0b}(s_{i'n}-s_{r'n},t'-s') \\
    & ~~~~~~~~~~~~~~~~~\times  E_{|S_{n}}[m(X_{it},X_{rs})m(X_{i't'},X_{r's'})]\\
    =&\sum_{A_{k,j} }K_{0b}(s_{in}-s_{rn},t-s)K_{0b}(s_{i'n}-s_{r'n},t'-s')E_{|S_{n}}[m(X_{it},X_{rs})m(X_{i't'},X_{r's'})]\\
    &+\sum_{A_{k,j}^c }K_{0b}(s_{in}-s_{rn},t-s)K_{0b}(s_{i'n}-s_{r'n},t'-s')E_{|S_{n}}[m(X_{it},X_{rs})m(X_{i't'},X_{r's'})]\\
    =&:II_{k,j,1}+II_{k,j,2},
\end{align}
where 
\begin{align}
    A_{k,j}&=\{(i,t)\neq (i',t')\in \widetilde{B}^{-}_{k,j}; (r,s),(r',s')\in\widetilde{B}^{c}_{k,j}:\max\{ ||(i,t)-(i',t')||_{F},||(r,s)-(r',s')||_{F}\}>q_n\}\},\\
    A^c_{k,j}&=\{(i,t)\neq (i',t')\in \widetilde{B}^{-}_{k,j}; (r,s),(r',s')\in\widetilde{B}^{c}_{k,j}:\max\{||(i,t)-(i',t')||_{F},||(r,s)-(r',s')||_{F}\}\leq q_n \}\}.
\end{align}
Note that 
\begin{align}
\label{proof 8}
    &|E_{|S_n}[m(X_{it},X_{rs})m(X_{i't'},X_{r's'})]| \\
    =&|E_{|S_n}[X_{it}X_{i't'}X_{rs}X_{r's'}]-E_{|S_n}[X_{it}X_{rs}]E_{|S_n}[X_{i't'}X_{r's'}]| \\
    \leq& |E_{|S_n}[X_{it}X_{i't'}]E_{|S_n}[X_{rs}X_{r's'}]|+|E_{|S_n}[X_{it}X_{rs}]E_{|S_n}[X_{i't'}X_{r's'}]|+|Cov_{|S_n}(X_{it}X_{i't'},X_{rs}X_{r's'})| \\
    \leq &16M_{1+\delta}^2\beta^{\frac{\delta}{1+\delta}}(||(i,t)-(i',t')||_{F})\beta^{\frac{\delta}{1+\delta}}(||(r,s)-(r',s')||_{F})+16M_{1+\delta}^2\beta^{\frac{2\delta}{1+\delta}}(q_n)+4(\widetilde{M}_{1+\delta}\lor M_{1+\delta})\beta^{\frac{\delta}{1+\delta}}(q_n).
\end{align}

Thus,  with condition (T2)(i) of Assumption \ref{technical assumptions}, \eqref{proof 8} we obtain 
\begin{align}
    |II_{k,j,1}|&\leq C_M\sum_{A_{k,j} }K_{0b}(s_{in}-s_{rn},t-s)K_{0b}(s_{i'n}-s_{r'n},t'-s') \beta^{\frac{\delta}{1+\delta}}(q_n),\\
    |II_{k,j,2}|&\leq C_M \sum_{A_{k,j}^c }K_{0b}(s_{in}-s_{rn},t-s)K_{0b}(s_{i'n}-s_{r'n},t'-s') \beta^{\frac{\delta}{1+\delta}}(||(i,t)-(i',t')||_{F})\beta^{\frac{\delta}{1+\delta}}(||(j,s)-(j',s')||_{F}),
\end{align}
where $C_M>0$ is a constant only determined by $M_{1+\delta}$ and $\widetilde{M}_{1+\delta}$. Proposition \ref{prop 7} implies that
\begin{align}
    \label{proof 9}
    \lim_{n \to \infty } \frac{|II_{k,j,2}|}{n\textbf{Card}(\widetilde{B}^-_{k,j}\cap \Gamma_n)Tp_Tb^4}=0,\ \ a.s.-[\mathbb{Q}].
\end{align}
holds almost surely with respect to $\mathbb{Q}$ (uniformly with respect to ($k,j)$). As for $II_{k,j,1}$, (E2) of Proposition \ref{prop 7} yields that
(uniformly with respect to ($k,j)$) 
\begin{align}
   \lim_{n\to\infty}\frac{|II_{k,j,1}|}{n\textbf{Card}(\widetilde{B}^-_{k,j}\cap \Gamma_n)Tp_Tb^3}=0\ \ a.s.-[\mathbb{Q}].
\end{align}

\par Combining these estimates  with \eqref{proof 5} gives a corresponding bound for  $II_{k,j}$,  and we obtain for term $V_1$ in of \eqref{proof 2}. 
\begin{align}   
    V_1=\sum_{k,j} I_{k,j}+\sum_{k,j}II_{k,j}= \sum_{k,j}H_{2}(k,j)(1+o(1))+o(n^2T^2b^3 )
\end{align}
 almost surely with respect to $\mathbb{Q}$.
 
Observing the definition of $H_{2}(k,j)$ we obtain  with Proposition \ref{prop 9} that
\begin{align}
\label{proof 10}
    \lim_{n\to\infty}\frac{V_1}{\sigma_n^2}=1
\end{align}
holds almost sure with respect to $\mathbb{Q}$.

\noindent\textbf{Step 2.2 (calculation of $V_2$ of \eqref{eq:V1V2}).} Note that 
\begin{align}
    &V_2=\sum_{(k,j)\neq (k',j')}\sum_{(i,t)\in\widetilde{B}_{k,j}^{-}}\sum_{(i',t')\in\widetilde{B}_{k',j'}^{-}}E_{|S_n}[Y_{k,j}(Z_{it})Y_{k',j'}(Z_{i't'})]\\
    =&\sum_{(k,j)\neq (k',j')}\sum_{(i,t)\in\widetilde{B}_{k,j}^{-}}\sum_{(i',t')\in\widetilde{B}_{k',j'}^{-}}\sum_{(r,u)\in \widetilde{B}_{k,j}^{c}}\sum_{(r',u')\in \widetilde{B}_{k',j'}^{c}}E_{|S_n}[H_{n}(Z_{it},Z_{ru})H_{n}(Z_{i't'},Z_{r'u'})]\\
    =&\sum_{(k,j)\neq (k',j')}V_{2}(k,j,k',j')
\end{align}
and 
\begin{align}
    E_{|S_n}[H_{n}(Z_{it},Z_{ru})H_{n}(Z_{i't'},Z_{r'u'})]&=K_{0b}(s_{in}-s_{rn},t-u)K_{0b}(s_{i'n}-s_{r'n},t'-u')\text{Cov}(X_{it}X_{ru},X_{i't'}X_{r'u'})\\
    &=:K_{irtu}K_{i'r't'u'}\text{Cov}(X_{it}X_{ru},X_{i't'}X_{r'u'}),
\end{align}
the  simple algebra yields that
\begin{align}
    V_{2}(k,j,k',j')=&\sum_{(i,t)\in\widetilde{B}^{-}_{k,j}}\sum_{(i',t')\in\widetilde{B}^-_{k',j'}}\sum_{(r,u),(r',u')\in C^c_{(i,t)(i',t')}}K_{irtu}K_{i'r't'u'}\text{Cov}(X_{it}X_{ru},X_{i't'}X_{r'u'})\\
  &+\sum_{(i,t)\in\widetilde{B}^{-}_{k,j}}\sum_{(i',t')\in\widetilde{B}^-_{k',j'}}\sum_{(r,u), (r',u')\in C_{(i,t)(i',t')} }K_{irtu}K_{i'r't'u'}\text{Cov}(X_{it}X_{ru},X_{i't'}X_{r'u'})\\
  =&: V_{21}(k,j,k',j')+V_{22}(k,j,k',j'),\label{eq:V2}
    \end{align}
where
\begin{align}
     & C_{(i,t)(i',t')}=\{(r,u)\in \widetilde{B}_{k,j}^c, (r',u')\in\widetilde{B}_{k',j'}^c:||(r,u)-(i',t')||_{F}\leq q_n, ||(r',u')-(i,t)||_{F}\leq q_n\},\\
  &C^c_{(i,t)(i',t')}=\{(r,u)\in \widetilde{B}_{k,j}^c, (r',u')\in\widetilde{B}_{k',j'}^c: \max\{||(r,u)-(i',t')||_{F}, ||(r',u')-(i,t)||_{F}\}> q_n\}.
\end{align}
According to Proposition \ref{prop 10}, we obtain 
\begin{align}
    \lim_{n\to\infty}\frac{V_2}{\sigma_n^2}=0 \ \ a.s.-[\mathbb{Q}].
\end{align}

\par Combining the results from  Step 2.1 and 2.2  we are able to calculate the conditional variance of $U_n$, that is 
 \begin{align}
    \label{d19}
    \lim_{n}\frac{\text{Var}_{|S_n}(U_{1n})}{\sigma_n^2}=1
\end{align}
holds almost surely with respect to $\mathbb{Q}$.\\\\
\noindent \textbf{Step 3 (calculation of ($\text{Var}_{|S_n}(U_{kn})$, $k\geq 2$).} The main goal of this step is to show 
\begin{align}
    \label{d20}
\lim_{n\to\infty}\sigma^{-2}_{n}\text{Var}_{|S_n}(U_{kn})=0
\end{align}
 almost surely with respect to $\mathbb{Q}$ which implies $\sigma_{n}^{-1}U_{kn}=o_{\mathbb{P}_{|S_n}}(1)$ holds almost surely with respect to $\mathbb{Q}$. 

\par First, we focus on $U_{2n}$ defined in \eqref{d16b}.  Similar to the decomposition of $U_{1n}$, we have
\begin{align}
    U_{2n}=\sum_{(k,j)}\sum_{(i,t)\in \widetilde{B}^-_{k,j}}Y^*_{k,j}(Z_{it})=\sum_{(k,j)}T^*_{k,j},
\end{align}
where
\begin{align}
  Y^*_{k,j}(Z_{it})  &=  \sum_{(i',t')\in\widetilde{B}_{k,j}}H_n(Z_{it},Z_{i't'}), 
 \\
T^*_{k,j}  &=  \sum_{(i,t)\in\widetilde{B}^-_{k,j}}Y_{k,j}(Z_{it}). 
\end{align}

This gives 
\begin{align}
&\text{Var}_{|S_n}(U_{2n})=V^*_{1}+V^*_{2}+V^*_{3}+V_{4}^*,
\end{align}
where
   \begin{align}
    V^*_{1}&= \sum_{k,j}\sum_{(i,t)\in\widetilde{B}^-_{k,j}}\sum_{(i',t')\in\widetilde{B}_{k,j}}E_{|S_n}[H_{n}^2(Z_{it},Z_{i't'})],\\
    V^*_{2}&=\sum_{k,j}\sum_{(i,t)\in\widetilde{B}^-_{k,j}}\sum_{(i',t')\neq (i'',t'')\in\widetilde{B}_{k,j}}E_{|S_n}[H_{n}(Z_{it},Z_{i't'})H_{n}(Z_{it},Z_{i''t''})],\\
     V^*_{3}&=\sum_{k,j}\sum_{(i,t)\neq (i',t')\in\widetilde{B}^-_{k,j}}\sum_{(r,s)\neq (r',s')\in\widetilde{B}_{k,j}}E_{|S_n}[H_{n}(Z_{it},Z_{rs})H_{n}(Z_{i't'},Z_{r's'})],\\
     V^*_{4}&=\sum_{(k,j)\neq (k',j')}\sum_{(i,t)\in\widetilde{B}_{k,j}^{-}}\sum_{(i',t')\in\widetilde{B}_{k',j'}^{-}}\sum_{(r,u)\in \widetilde{B}_{k,j}}\sum_{(r',u')\in \widetilde{B}_{k',j'}}E_{|S_n}[H_{n}(Z_{it},Z_{ru})H_{n}(Z_{i't'},Z_{r'u'})].
\end{align}
Recalling the notation of  
$g_n$, $ H_{n}$ and $m$ in \eqref{d15c}, \eqref{d15d} and \eqref{d15e}, respectively, 
and observing  $0<\lim_{n\to\infty}\frac{\sigma_n^2}{(nT)^2b^3}<\infty$, Propositions \ref{prop 11} -  \ref{prop 13}  we obtain for $k=1,2,3,4$
\begin{align}
    \lim_{n\to\infty}\frac{V^*_k}{\sigma_n^2}=0,\ \ \ a.s.-[\mathbb{Q}], 
\end{align}
which gives 
 $\sigma_{n}^{-1}U_{2n}=o_{\mathbb{P}_{|S_n}}(1)$ holds almost surely with respect to $\mathbb{Q}$.\\

\par Note that $U_{3n}$ and $U_{4n}$ have nearly the same structure as $U_{2n}$ and $U_{1n}$ respectively. By similar arguments as given  in the  previous steps, we can therefore  show $\sigma_{n}^{-1}U_{kn}=o_{\mathbb{P}_{|S_n}}(1)$ 
 almost surely with respect to $\mathbb{Q}$ ($k=3,4$).   As for the term $U_{5n}$, notice that the only difference between $U_{5n}$ and $\sum_{k=1}^4U_{kn}$ is that the number of time locations involved in $U_{5n}$ is strictly smaller than $p_T+q_T$. By repeating the method demonstrated above, together with the fact that $\frac{p_T+q_T}{T}=o(1)$, we can show  $\sigma_{n}^{-1}U_{5n}=o_{\mathbb{P}_{|S_n}}(1)$ holds almost surely with respect to $\mathbb{Q}$.
\par Summarizing the discussion, we manage to show that the statistic $U_n$ in \eqref{d16} satisfies 
\begin{align}
\label{proof 12}
    \sigma^{-1}_{n}U_{n}=\sigma_{n}^{-1}U_{1n}+o_{\mathbb{P}_{|S_n}}(1)
\end{align}
almost surely with respect to $\mathbb{Q}$, where $U_{1n}$ is defined in \eqref{d16a}. This means that we only need to focus on the leading term $\sigma_n^{-1}U_{1n}$.\\\\
\noindent \textbf{Step 4 (asymptotic properties of the leading term of $U_{1n}$).} In this step, we aim to simplify $U_{1n}$  such that we can then employ Theorem \ref{theorem CLT 2} to prove the central limit theorem of our U-statistic. A key observation here is that, for each set $\widetilde{B}_{k,j}$   introduced in \eqref{d21}, 
$$
\widetilde{B}^c_{L,k,j}:= \bigcup_{(k',j')\neq (k,j)}\widetilde{B}^-_{k',j'} \subset \widetilde{B}^c_{k,j},
$$
which gives the decomposition  
\begin{align}
    U_{1n}&=\sum_{k,j}\sum_{(i,t)\in\widetilde{B}^-_{k,j}}\sum_{(i',t')\in\widetilde{B}^c_{L,k,j}}H_{n}(Z_{it},Z_{i't'})+\sum_{k,j}\sum_{(i,t)\in\widetilde{B}^-_{k,j}}\sum_{(i',t')\in\widetilde{B}^c_{k,j}\backslash\widetilde{B}^c_{L,k,j}}H_{n}(Z_{it},Z_{i't'})\\
    &=: U_{1n,L}+U_{1n,R}.
\end{align}
We are going to show that, for every $\epsilon>0$, 
\begin{align}
   \lim_{n\to\infty}\mathbb{P}_{|S_n}( |\sigma_n^{-1}U_{1n,R}|>\epsilon)=0 \ \ \ a.s.-[\mathbb{Q}].
\end{align}
It suffice to prove $\sigma_n^{-2}\text{Var}_{|S_n}(U_{1n,R})=0$ holds almost surely with respect to $\mathbb{Q}$. For this purpose, we use the decomposition 
\begin{align}
    &\text{Var}_{|S_n}(U_{1n,R})= \widetilde{V}_1+\widetilde{V}_2+\widetilde{V}_3+\widetilde{V}_4,
    \end{align}
    where
    \begin{align}
    \widetilde{V}_1&=\sum_{k,j}\sum_{(i,t)\in\widetilde{B}_{k,j}}\sum_{(i',t')\in\widetilde{B}^c_{k,j}\backslash\widetilde{B}^c_{L,k,j}}E[H^2_{n}(Z_{it},Z_{i't'})]\\
    \widetilde{V}_{2}&=\sum_{k,j}\sum_{(i,t)\in\widetilde{B}^-_{k,j}}\sum_{(i',t')\neq (i'',t'')\in\widetilde{B}^c_{k,j}\backslash\widetilde{B}^c_{L,k,j}}E_{|S_n}[H_{n}(Z_{it},Z_{i't'})H_{n}(Z_{it},Z_{i''t''})],\\
     \widetilde{V}_{3}&=\sum_{k,j}\sum_{(i,t)\neq (i',t')\in\widetilde{B}^-_{k,j}}\sum_{(r,s), (r',s')\in\widetilde{B}^c_{k,j}\backslash\widetilde{B}^c_{L,k,j}}E_{|S_n}[H_{n}(Z_{it},Z_{rs})H_{n}(Z_{i't'},Z_{r's'})],\\
     \widetilde{V}_{4}&=\sum_{(k,j)\neq (k',j')}\sum_{(i,t)\in\widetilde{B}_{k,j}^{-}}\sum_{(i',t')\in\widetilde{B}_{k',j'}^{-}}\sum_{(r,s)\in \widetilde{B}^c_{k,j}\backslash\widetilde{B}^c_{L,k,j}}\sum_{(r',s')\in \widetilde{B}^c_{k',j'}\backslash\widetilde{B}^c_{L,k',j'}}E_{|S_n}[H_{n}(Z_{it},Z_{ru})H_{n}(Z_{i't'},Z_{r'u'})],
\end{align}
and   prove that $\lim_{n\to\infty}\frac{\widetilde{V}_k}{\sigma_n^2}=0$  almost surely with respect to $\mathbb{Q}$ ($1\leq k\leq 4$). Compared with Step 2, a crucial difference here is that
\begin{align}
    \lim_{n\to\infty}\max_{k,j}\left|\frac{\textbf{Leb}(\Lambda_{n}^{-1}(\widetilde{B}_{k,j,S}^c\backslash\widetilde{B}^c_{L,k,j,S}))}{\textbf{Leb}(\Lambda_n^{-1}\widetilde{B}^c_{k,j,S})}\right|=0,
\end{align}
where $\widetilde{B}^c_{L,k,j,S}$ is the projection of set $\widetilde{B}^c_{L,k,j}$ on spatial plane. Together with the design of spatial blocks, 
Proposition \ref{prop 5} implies for the event 
\begin{align}
    \mathcal{G}=\bigcap_{k=1}^{R_n}\bigcap_{j=1}^{J_n}\Big\{\lim_{n\to\infty}n^{-1}\textbf{Card}\Big((\widetilde{B}^c_{k,j,S}\backslash\widetilde{B}^c_{L,k,j,S})\cap\Gamma_n\Big)=0\Big\},
\end{align}
 that 
  \begin{align}
  \label{proof 14}
      \mathbb{Q}(\overline{\lim}_{n\to\infty}\mathcal{G}^c_n)=1,
  \end{align}

  where $\mathcal{G}^c$ is the complement of set $\mathcal{G}$. Then, based on \eqref{proof 14}, by repeating the method used in Step 2, we can easily obtain that 
  \begin{align}
      \label{proof 15}
      \lim_{n\to\infty}\frac{\text{Var}_{|S_n}(U_{1n,R})}{\sigma_n^2}=0\    \  \text{and}
      \lim_{n\to\infty}\frac{\text{Var}_{|S_n}(U_{1n,L})}{\sigma_n^2}<\infty
  \end{align}
hold almost surely with respect to $\mathbb{Q}$. Therefore, we obtain from the representation 
$$\text{Var}_{|S_n}(U_{1n})=\text{Var}_{|S_n}(U_{1n,L})+\text{Var}_{|S_n}(U_{1n,R})+2\text{Cov}_{|S_n}(U_{1n,L},U_{1n,R}),$$
\eqref{proof 15} and the  Cauchy-Schwartz inequality  that 
\begin{align}
\label{proof 16}
    \lim_{n\to\infty}\frac{\text{Var}_{|S_n}(U_{1n,L})}{\sigma_n^2}=1
\end{align}
holds almost surely with respect to $\mathbb{Q}$.

\par Consequently, \eqref{proof 16} implies that we only need to focus on proving the asymptotic normality of $\frac{U_{1n,L}}{\sigma_n}$ (almost surely with respect to $\mathbb{Q}$). An important point here  is that $\{\widetilde{B}^-_{k,j}\}$ is an array of disjoint sets, which means that there is a line connecting all $\widetilde{B}^-_{k,j}$'s in linear order.  By naming $R_nJ_n=M_n$, we can rename $\{\widetilde{B}^-_{k,j}: 1\leq k\leq R_n, 1\leq j\leq J_n\}$ as $\{\widetilde{B}^-_{m}:m=1,2,\dots,M_n\}$, see Figure \ref{fig:chess}.

\begin{figure}
    \centering
    \includegraphics[width=0.5\linewidth]{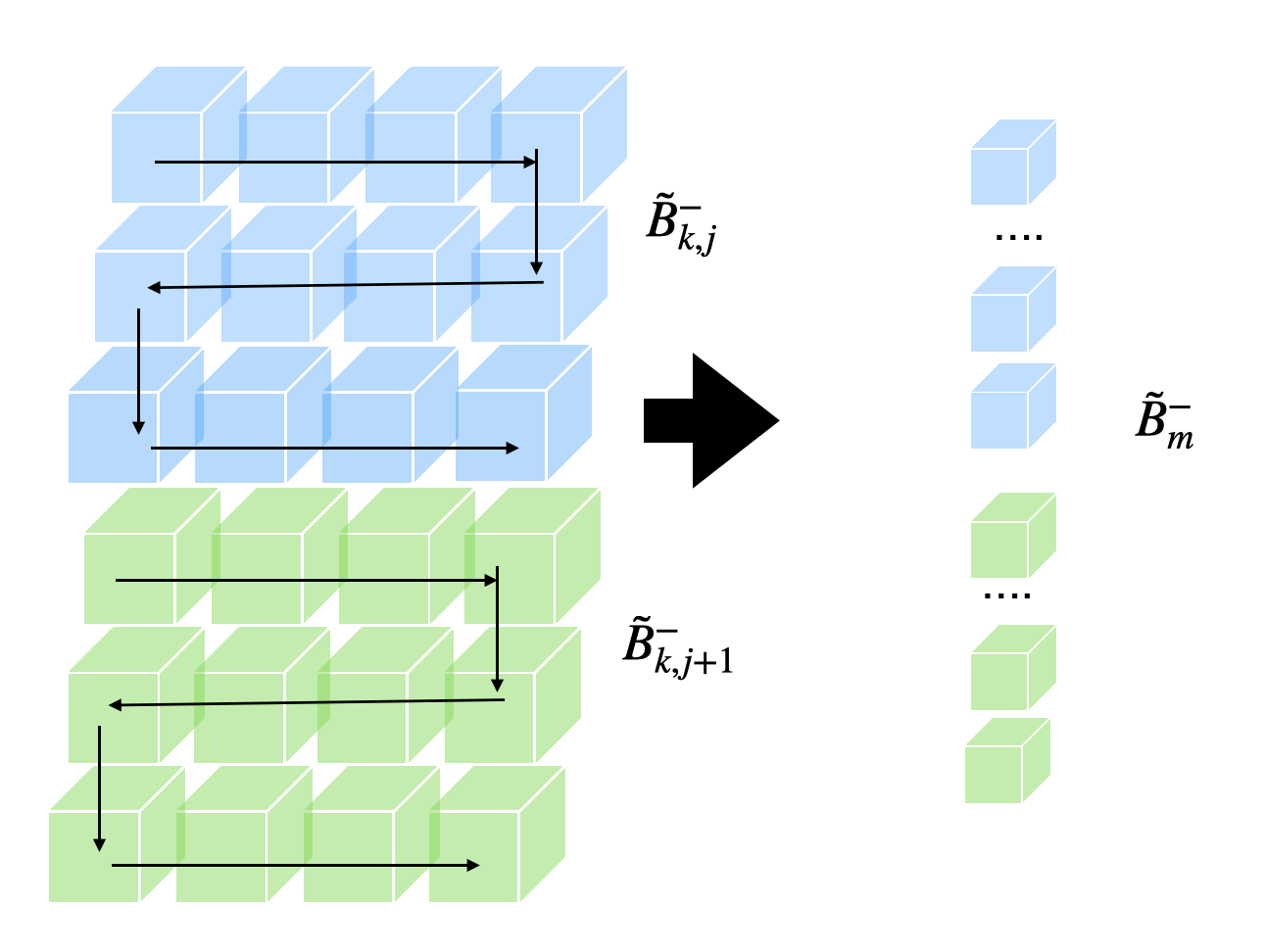}
    \caption{Linear order of the disjoint blocks.}
    \label{fig:chess}
\end{figure}
Furthermore, we have $d(\widetilde{B}^-_{m},\widetilde{B}^-_{m'})\geq q_n$, where $d(A,B)=\inf_{a\in A,b\in B}||a-b||_{F}$, for any $A,B\subset \mathbb{R}^3$. With this notation  rewrite $U_{1n,L}$ as follows,
\begin{align}
    \sigma_n^{-1}U_{1n,L}&=\sigma_n^{-1}\sum_{m=1}^{M_n}\sum_{(i,t)\in\widetilde{B}^-_{m}}\sum_{m'\neq m\atop 1\leq m'\leq M_n}\sum_{(i',t')\in\widetilde{B}^-_{m'}}H_n(Z_{it},Z_{i't'}) \\
    &=2\sigma_n^{-1}\sum_{m=2}^{M_n}\sum_{(i,t)\in\widetilde{B}^-_{m}}\sum_{ 1\leq m'<m}\sum_{(i',t')\in\widetilde{B}^-_{m'}}H_n(Z_{it},Z_{i't'}) \\
    &=: \sum_{m=1}^{M_n}\widetilde{T}_{m,n}, \label{eq:tildeT}
\end{align}
where 
\begin{align}
    \label{d22}
    \widetilde{T}_{m,n} ={2 \over \sigma_n}\sum_{(i,t)\in\widetilde{B}^-_{m}}\sum_{ 1\leq m'<m}\sum_{(i',t')\in\widetilde{B}^-_{m'}}H_n(Z_{it},Z_{i't'})
\end{align}
Hence, we only need to show the asymptotic normality of $\sum_{m=1}^{M_n}\widetilde{T}_{m,n}$ holds almost surely with respect to $\mathbb{Q}$. Note that we can regard $\bS_n$ and $\mathbb{Q}$ above as the $O_n$ and $Q$ mentioned in Theorem \ref{theorem CLT 2}. Then, the asymptotic normality follows from Theorem \ref{theorem CLT 2} if we are able to show the  conditions (D1)-(D3) introduced there.\\\\
\noindent\textbf{Step 5} This step is divided into three parts, in which we will verify  conditions (D1) to (D3) of Theorem \ref{theorem CLT 2}. Define $\mathcal{F}_{k}$ is the sigma field generated by $\{\widetilde{T}_{m,n}:1\leq m\leq k\}$. 
\par \textbf{Step 5.1} ($\sum_{m=2}^{M_n}E_{|S_n}[\widetilde{T}_{m,n}|\mathcal{F}_{m-1}] = o_{\mathbb{P}_{| S_n}} (1) \ \ a.s.-[\mathbb{Q}]$) First, note that $E_{|S_n}[\widetilde{T}_{m,n}|\mathcal{F}_{m-1}]$ is well defined since we can understand it as a real-valued random variable measurable with respect to the sigma field $\mathcal{F}_{m-1}$ and its distribution is determined by the realization $\bS_n=S_n$. Since $\sigma_n=O(nTb^{1.5})$, it is sufficient  to  prove $\sum_{m=2}^{M_n}E_{|S_n}[T_{m,n}|\mathcal{F}_{m-1}]=o_{\mathbb{P}_{|S_n}} (nTb^{1.5})\ \ a.s.-[\mathbb{Q}]$, where  $ T_{m,n} = {\sigma_n \over 2}  \widetilde T_{m,n}$. Simple algebra shows that for every realization $\bS_n=S_n$, we have 
\begin{align}
    &\Big|E_{|S_n}[T_{m,n}|\mathcal{F}_{m-1}]\Big|=\Big|\sum_{(i,t)\in\widetilde{B}^-_{m}}\sum_{(i',t')\in\widetilde{B}^-_{m'}\atop 1\leq m'<m}K_{0b}(s_{in}-s_{i'n},t-t')E_{|S_n}[X_{it}X_{i't'}-E_{|S_n}[X_{it}X_{i't'}]|\mathcal{F}_{m-1}]\Big|\\
    =&\Big|\sum_{(i,t)\in\widetilde{B}^-_{m}}\sum_{(i',t')\in\widetilde{B}^-_{m'}\atop 1\leq m'<m}K_{0b}(s_{in}-s_{i'n},t-t')(E_{|S_n}[X_{it}X_{i't'}|\mathcal{F}_{m-1}]-E_{|S_n}[X_{it}X_{i't'}])\Big|\\
    \leq & \sum_{(i,t)\in\widetilde{B}^-_{m}}\sum_{(i',t')\in\widetilde{B}^-_{m'}\atop 1\leq m'<m}K_{0b}(s_{in}-s_{i'n},t-t')\Big|E_{|S_n}[X_{it}X_{i't'}|\mathcal{F}_{m-1}]\Big|\\
    &\ \ \ \ \ \ \ \ \ \ \ \ \ \ +\sum_{(i,t)\in\widetilde{B}^-_{m}}\sum_{(i',t')\in\widetilde{B}^-_{m'}\atop 1\leq m'<m}K_{0b}(s_{in}-s_{i'n},t-t')|\text{Cov}_{|S_n}(X_{it},X_{i't'})|\\
    =&:\mathbf{E}_{1m}+\mathbf{E}_{2m},
\end{align}
where the second equality is due to the fact that $E_{|S_n}[E_{|S_n}[X_{it}X_{i't'}]|\mathcal{F}_{m-1}]=E_{|S_n}[X_{it}X_{i't'}]$ holds almost surely with respect to $\mathbb{P}_{|S_n}$. Since condition (M1) in  Assumption \ref{moment conditions} implies that $E[X_{i't'}X_{it}]=\int_{\mathbb{R}}X_{i't'}xd\mathbb{P}_{X_{it}}=0$ holds for every realization $X_{i't'}=x'$ and $\bS_n=S_n$,  Lemma \ref{lemma 2} yields 
\begin{align}
   E_{|S_n} | \mathbf{E}_{1m}|& \leq  \sum_{(i,t)\in\widetilde{B}^-_{m}}\sum_{(i',t')\in\widetilde{B}^-_{m'}\atop 1\leq m'<m}K_{0b}(s_{in}-s_{i'n},t-t')\Big|E_{|S_n}[X_{it}X_{i't'}|\mathcal{F}_{m-1}]-g_{S_n}(X_{i't'})\Big|\\
    &\leq 3(M_{1+\delta})^{\frac{1}{1+\delta}}\beta^{\frac{\delta}{1+\delta}}(q_n)\sum_{(i,t)\in\widetilde{B}^-_{m}}\sum_{(i',t')\in\widetilde{B}^-_{m'}\atop 1\leq m'<m}K_{0b}(s_{in}-s_{i'n},t-t')
\end{align}
  and $2\leq m\leq M_n$, where $M_n$ is introduced in Step 4. Moreover, Lemma \ref{lemma 3} also implies that, for every $1\leq m\leq M_n$,
\begin{align}
    E_{|S_n} | \mathbf{E}_{2m}| \leq 4(M_{1+\delta})^{\frac{1}{1+\delta}}\beta^{\frac{\delta}{1+\delta}}(q_n)\sum_{(i,t)\in\widetilde{B}^-_{m}}\sum_{(i',t')\in\widetilde{B}^-_{m'}\atop 1\leq m'<m}K_{0b}(s_{in}-s_{i'n},t-t').
\end{align}
Above all, we obtain that 
\begin{align}
    &E\Big|\sum_{m=2}^{M_n}E_{|S_n}[T_{m,n}|\mathcal{F}_{m-1}]\Big|\leq \sum_{m=2}^{M_n}(E|\mathbf{E}_{1m}|+E|\mathbf{E}_{2m}|)\\
    \leq& 7(M_{1+\delta})^{\frac{1}{1+\delta}}\beta^{\frac{\delta}{1+\delta}}(q_n)\sum_{m=2}^{M_n}\sum_{(i,t)\in\widetilde{B}^-_{m}}\sum_{(i',t')\in\widetilde{B}^-_{m'}\atop 1\leq m'<m}K_{0b}(s_{in}-s_{i'n},t-t'),\ \ \ \forall\ \bS_n=S_n.
\end{align}
Together with condition  (T4)(1) in Assumption \ref{technical assumptions} and the settings of array $\{\s_{in}\}$, some simple algebra shows 
\begin{align}
    \sum_{m=2}^{M_n}E_{|S_n}[T_{m,n}|\mathcal{F}_{m-1}]=o_{\mathbb{P}_{| S_n}}(nTb^{3})\ \ a.s.-[\mathbb{Q}],
\end{align}
which gives
\begin{align}
 {\sum_{m=2}^{M_n}E_{|S_n}[\widetilde{T}_{m,n}|\mathcal{F}_{m-1}]}{b^{1.5}}  = o_{\mathbb{P}_{|S_n}} (1) \ \  \ \ a.s.-[\mathbb{Q}].
\end{align}
Therefore,   $\{\widetilde{T}_{m,n}\}$ satisfies condition (D1) of Theorem \ref{theorem CLT 2}.
\par \textbf{Step 5.2} ($\sum_{m=2}^{M_n}E_{|S_n}[\widetilde{T}^2_{m,n}|\mathcal{F}_{m-1}]\xrightarrow[n\to\infty]{\mathbb{P}_{|S_n}} 1\ \ a.s.-[\mathbb{Q}]$) By  \eqref{proof 16} and the definition of $\widetilde{T}_{m,n}$  it is sufficient to prove
\begin{align}
    \Big(\text{Var}_{|S_n}(\sum_{m=2}^{M_n}T_{m,n})\Big)^{-1}\sum_{m=2}^{M_n}E_{|S_n}[T^2_{m,n}|\mathcal{F}_{m-1}]\xrightarrow[n\to\infty]{\mathbb{P}_{|S_n}} 1\ \ a.s.-[\mathbb{Q}],
\end{align}
where  $ T_{m,n} = {\sigma_n \over 2}  \widetilde T_{m,n}$ and  $\widetilde T_{m,n}$ is defined in \eqref{d22}.
Since 
\begin{align}
    \text{Var}_{|S_n} \Big (\sum_{m=2}^{M_n}T_{m,n} \Big )=\sum_{m=2}^{M_n}E_{|S_n}[T^2_{m,n}]+2\sum_{m=2}^{M_n}\sum_{1\leq m'<m}E_{|S_n}[T_{m,n}T_{m',n}]=:V_{T1}+V_{T2},
\end{align}
we only need to prove  that 
\begin{align}
\label{proof 17}
    &\frac{\sum_{m=2}^{M_n}E_{|S_n}[T^2_{m,n}|\mathcal{F}_{m-1}]}{V_{T1}}\xrightarrow[n\to\infty]{\mathbb{P}_{|S_n}}1,\\
    \label{proof 18}
    &\text{and}\ \ \ \ \ \ \ \ \ \                  \lim_{n\to\infty}\frac{V_{T2}}{V_{T1}}=0
\end{align}
almost surely with respect to $\mathbb{Q}$. First, note that, by repeating Step 2.2, we can obtain $\lim_{n\to\infty}\sigma^{-2}_{n}V_{T2}=0$ almost surely with respect to $\mathbb{Q}$. Moreover, \eqref{proof 16} implies that 
\begin{align}
    \lim_{n\to\infty}\sigma_n^{-2}\text{Var}_{|S_n}(\sum_{m=2}^{M_n}T_{m,n})=\frac{1}{4},
\end{align}
which, together with the Cauchy-Schwartz inequality gives  $\lim_{n\to\infty}\sigma^{-2}_{n}V_{T1}=\frac{1}{4}$ and $\lim_{n\to\infty}\frac{V_{T2}}{V_{T1}}=0$   almost surely with respect to $\mathbb{Q}$, which proves \eqref{proof 18}. 

\par To prove \eqref{proof 17},  we use  the decomposition 
\begin{align}
   \sum_{m=2}^{M_n} E_{|S_n}[T^2_{m,n}|\mathcal{F}_{m-1}]&=\sum_{m=2}^{M_n}\sum_{(i,t)\in\widetilde{B}^-_{m}}E_{|S_n}[Y^2_{m}(Z_{it})|\mathcal{F}_{m-1}]+\sum_{m=2}^{M_n}\sum_{(i,t)\neq (r,s)\in\widetilde{B}^-_{m}}E_{|S_n}[Y_{m}(Z_{it})Y_{m}(Z_{rs})|\mathcal{F}_{m-1}]  \\
   =&: \mathbf{F}_1+\mathbf{F}_2,  \label{eq:F1F2}  
   \end{align}
where
\begin{align}
    Y_{m}(Z_{it})&=\sum_{(i',t')\in\widetilde{B}^-_{m'}\atop 1\leq m'<m}H_n(Z_{it},Z_{i't'}).\label{d33} 
\end{align}

\noindent \textbf{Step 5.2.1} ($V_{T1}^{-1}\mathbf{F}_2 =o_{{\mathbb{P}_{|S_n}}}(1),\ a.s.-[\mathbb{Q}]$) Since $\lim_{n\to\infty}\sigma^{-2}_{n}V_{T1}=\frac{1}{4}$ holds almost surely with respect to $\mathbb{Q}$, it suffices to show 
\begin{align}
\label{proof 19}
    \frac{\mathbf{F}_2}{(nT)^2b^3}\xrightarrow[n\to\infty]{\mathbb{P}_{|S_n}} 0\ \ \ a.s.-[\mathbb{Q}].
\end{align}
Using the notation
\begin{align}
    &K_{0ir}=K(\frac{\lambda_n^{-1}(s_{in,1}-s_{rn,1}-h_{0,1})}{b})K(\frac{(\lambda_n\bb_n)^{-1}(s_{in,2}-s_{rn,2}-h_{0,2})}{b}),\\
 & K_{0ts}=K(\frac{|t-s|-v_0}{Tb}).
\end{align}
simple algebra yields  
\begin{align}
    \mathbf{F}_2= &\sum_{m=2}^{M_n}\sum_{(i,t)\neq (r,s)\in\widetilde{B}^-_m}\sum_{(i',t')\in \widetilde{B}^-_{m'}\atop 1\leq m'<m}\sum_{(r',s')\in \widetilde{B}^-_{m''}\atop 1\leq m''<m}K_{0ir}K_{0i'r'}K_{0ts}K_{0t's'}E_{|S_n}[X_{it}X_{i't'}X_{rs}X_{r's'}|\mathcal{F}_{m-1}]\\
    &- \sum_{m=2}^{M_n}\sum_{(i,t)\neq (r,s)\in\widetilde{B}^-_m}\sum_{(i',t')\in \widetilde{B}^-_{m'}\atop 1\leq m'<m}\sum_{(r',s')\in \widetilde{B}^-_{m''}\atop 1\leq m''<m}K_{0ir}K_{0i'r'}K_{0ts}K_{0t's'}E_{|S_n}[X_{it}X_{i't'}]E_{|S_n}[X_{rs}X_{r's'}]\\
    =&:\mathbf{F}_{21}+\mathbf{F}_{22}\label{F21F22}
\end{align}
 With the notation  $g_{S_n}(X_{i't'}X_{r's'})=X_{i't'}X_{r's'}\int_{\mathbb{R}}xyd\mathbb{P}_{X_{it},X_{rs}|S_n}$ we have  
\begin{align}
    E_{|S_n}[X_{it}X_{i't'}X_{rs}X_{r's'}|\mathcal{F}_{m-1}]
    =&E_{|S_n}[X_{it}X_{i't'}X_{rs}X_{r's'}|\mathcal{F}_{m-1}]-g_{S_n}(X_{i't'}X_{r's'})+g_{S_n}(X_{i't'}X_{r's'})\\
    =&:\triangle_{S_n}(X_{i't'}X_{r's'})+g_{S_n}(X_{i't'}X_{r's'}) \label{DeltaSngSn}
\end{align}
holds for every $\bS_n=S_n$. Based on the definition of $\widetilde{B}^-_{m}$'s,  Lemma \ref{lemma 2} gives  
$$E_{|S_n}|\triangle_{S_n}(X_{i't'}X_{r's'})|=E_{|S_n}[X_{it}X_{i't'}X_{rs}X_{r's'}|\mathcal{F}_{m-1}]-g_{S_n}(X_{i't'}X_{r's'})\leq 3C_{M,\delta}\beta^{\frac{\delta}{1+\delta}}(q_n),$$ where the constant  $C_{M,\delta}>0$ depends solely on $\delta$ and the  moment conditions in Assumption \ref{moment conditions}. Then, we can rewrite $F_{21}$ in \eqref{F21F22} using \eqref{DeltaSngSn} as 
\begin{align}
\label{d25a}
    &\mathbf{F}_{21}=\mathbf{F}_{21,1}+\mathbf{F}_{21,2} \\
    &:=2\sum_{m=2}^{M_n}\sum_{(i,t)\neq (r,s)\in\widetilde{B}^-_m}\sum_{(i',t')\in \widetilde{B}^-_{m'}\atop 1\leq m'<m}\sum_{(r',s')\in \widetilde{B}^-_{m''}\atop 1\leq m''\leq m'}K_{0ir}K_{0i'r'}K_{0ts}K_{0t's'}\triangle_{S_n}(X_{i't'}X_{r's'})\\
    &+2\sum_{m=2}^{M_n}\sum_{(i,t)\neq (r,s)\in\widetilde{B}^-_m}\sum_{(i',t')\in \widetilde{B}^-_{m'}\atop 1\leq m'<m}\sum_{(r',s')\in \widetilde{B}^-_{m''}\atop 1\leq m''\leq m'}K_{0ir}K_{0i'r'}K_{0ts}K_{0t's'}g_{S_n}(X_{i't'}X_{r's'}).
\end{align}
Thus, the non-negativity of kernel function yields
\begin{align}
    E_{|S_n}|\mathbf{F}_{21,1}|&\leq 6C_{\delta,M}\sum_{m=2}^{M_n}\sum_{(i,t)\neq (r,s)\in\widetilde{B}^-_m}K_{0ir}K_{0ts} \sum_{(i',t')\in \widetilde{B}^-_{m'}\atop 1\leq m'<m}\sum_{(r',s')\in \widetilde{B}^-_{m''}\atop 1\leq m''\leq m'}K_{0i'r'}K_{0t's'}\beta^{\frac{\delta}{1+\delta}}(q_n)
    \\
    &\leq 6C_{\delta,M}\sum_{m=2}^{M_n}\sum_{(i,t)\neq (r,s)\in\widetilde{B}^-_m}K_{0ir}K_{0ts} \sum_{(i',t')\in \widetilde{B}^c_{m}}\sum_{(r',s')\in \widetilde{B}^c_{m}}K_{0i'r'}K_{0t's'}\beta^{\frac{\delta}{1+\delta}}(q_n).
\end{align}
Note  that
\begin{align}
    \sum_{(i',t')\in \widetilde{B}^c_{m}}\sum_{(r',s')\in \widetilde{B}^c_{m}}K_{0i'r'}K_{0t's'}\beta^{\frac{\delta}{1+\delta}}(q_n) 
   & \leq \sum_{(i',t')}\sum_{(r',s')}K_{0i'r'}K_{0t's'}\beta^{\frac{\delta}{1+\delta}}(q_n) \\ 
   & \lesssim \Big (\sum_{s_{i'n},s_{r'n}\in\Gamma_n}K_{0i'r'}\Big )\Big (\sum_{t'\neq s'\in \{1,...,T\}}K_{0t's'} \Big )\beta^{\frac{\delta}{1+\delta}}(q_n)\\
   & \lesssim  \overline{E} (nT)^2b^3\beta^{\frac{\delta}{1+\delta}}(q_n)
\end{align}
almost surely with respect to $\mathbb{Q}$,
where 
$\overline{E}=\max_{n}\max_{i,r}||b^{-2}E_{|S_{n}}[K_{0ir}]||_{\infty}$.
and the last inequality follows  from the strong law of large numbers.
Therefore, we obtain (observing condition (T4)(i)) 
\begin{align}
\label{proof 37}
    \lim_{n\to\infty}\frac{\max_{m}\sum_{(i',t')\in \widetilde{B}^c_{m}}\sum_{(r',s')\in \widetilde{B}^c_{m}}K_{0i'r'}K_{0t's'}\beta^{\frac{\delta}{1+\delta}}(q_n)}{(nT)b^3}=0.
\end{align}
Similar to the proof of Proposition \ref{prop 10}, the definition of $\widetilde{B}^-_{m}$ and Proposition \ref{prop 5} imply that 
\begin{align}
    \lim_{n\to\infty}\frac{\sum_{m=2}^{M_n}\textbf{Card}(\widetilde{B}^-_{m,S}\cap\Gamma_n)p_T}{nT}<\infty\ \ a.s.-[\mathbb{Q}],
\end{align}
and it follows
 \begin{align}
\label{proof 38}
    &\lim_{n\to\infty}\frac{\sum_{m=2}^{M_n}\sum_{(i,t)\neq (r,s)\in\widetilde{B}^-_m}K_{0ir}K_{0ts}}{(\frac{nr_{1n}r_{2n}}{\lambda_n^2\bb_n})p_Tb^3(\sum_{m=2}^{M_n}\textbf{Card}(\widetilde{B}^-_{m,S}\cap\Gamma_n)p_T)}
    =\lim_{n\to\infty}\frac{\sum_{m=2}^{M_n}\sum_{(i,t)\neq (r,s)\in\widetilde{B}^-_m}K_{0ir}K_{0ts}}{(\frac{nr_{1n}r_{2n}}{\lambda_n^2\bb_n})p_TnTb^3}<\infty 
\end{align}
$a.s.-[\mathbb{Q}]$.
Now, using condition  (T4)(3) and (T1) in  Assumption \ref{technical assumptions}, \eqref{eq:r1r2},  Markov's inequality, \eqref{proof 37} and \eqref{proof 38} shows
\begin{align}
    \label{proof 39}
    \frac{\mathbf{F}_{21,1}}{(nT)^2b^6(\frac{nr_{1n}r_{2n}}{\lambda_n^2\bb_n})p_T}\xrightarrow[n\to\infty]{\mathbb{P}_{|S_n}}0\ \ \ a.s.-[\mathbb{Q}], \ \ \ \text{where}\ \frac{(nT)^2b^6(\frac{nr_{1n}r_{2n}}{\lambda_n^2\bb_n})p_T}{(nT)^2b^3}=o(1).
\end{align}
This yields 
\begin{align}
\label{d24}
    \frac{\mathbf{F}_{21,1}}{(nT)^2b^3} = o_{\mathbb{P}_{|S_n}} (1) 
    \end{align}
    almost surely with respect to $\mathbb{Q}$. To prove a corresponding statement for $\frac{\mathbf{F}_{21,2}}{(nT)^2b^3}$,
we first introduce the  decomposition
\begin{align}
    \mathbf{F}_{21,2}=&2\sum_{m=2}^{M_n}\sum_{(i,t)\neq (r,s)\in\widetilde{B}^-_m\atop ||(i,t)-(r,s)||_{F}>q_n}\sum_{(i',t')\in \widetilde{B}^-_{m'}\atop 1\leq m'<m}\sum_{(r',s')\in \widetilde{B}^-_{m''}\atop 1\leq m''\leq m'}K_{0ir}K_{0i'r'}K_{0ts}K_{0t's'}g_{S_n}(X_{i't'}X_{r's'})\\
    &+2\sum_{m=2}^{M_n}\sum_{(i,t)\neq (r,s)\in\widetilde{B}^-_m\atop ||(i,t)-(r,s)||_{F}\leq q_n}\sum_{(i',t')\in \widetilde{B}^-_{m'}\atop 1\leq m'<m}\sum_{(r',s')\in \widetilde{B}^-_{m''}\atop 1\leq m''\leq m'}K_{0ir}K_{0i'r'}K_{0ts}K_{0t's'}g_{S_n}(X_{i't'}X_{r's'})\\
    =&:\widetilde{F}_{21,21}+\widetilde{F}_{21,22
    },
    \label{d23}
\end{align}
and estimate the terms  $\widetilde{F}_{21,21}$ and $\widetilde{F}_{21,22}$ separately. Since $||(i,t)-(r,s)||_{F}>q_n$ and 
\begin{align}
    g_{S_n}(X_{i't'}X_{r's'})=X_{i't'}X_{r's'}E_{|S_n}[X_{it}X_{rs}]=X_{i't'}X_{r's'}\text{Cov}_{|S_n}(X_{it},X_{rs}),
\end{align}
we obtain for the conditional expectation of $| \widetilde{F}_{21,21}| $  for each $\bS_n=S_n$,
\begin{align}
    E_{|S_n}|\widetilde{F}_{21,21}|\leq& 2\sum_{m=2}^{M_n}\sum_{(i,t)\neq (r,s)\in\widetilde{B}^-_m\atop ||(i,t)-(r,s)||_{F}>q_n}\sum_{(i',t')\in \widetilde{B}^-_{m'}\atop 1\leq m'<m}\sum_{(r',s')\in \widetilde{B}^-_{m''}\atop 1\leq m''\leq m'}K_{0ir}K_{0i'r'}K_{0ts}K_{0t's'}E_{|S_n}|X_{i't'}X_{r's'}|\text{Cov}_{|S_n}(X_{it},X_{rs})\\
    \leq& 2M_{1}(M_{1+\delta})^{\frac{1}{1+\delta}}\beta^{\frac{\delta}{1+\delta}}(q_n)\sum_{m=2}^{M_n}\sum_{(i,t)\neq (r,s)\in\widetilde{B}^-_m\atop ||(i,t)-(r,s)||_{F}>q_n}\sum_{(i',t')\in \widetilde{B}^-_{m'}\atop 1\leq m'<m}\sum_{(r',s')\in \widetilde{B}^-_{m''}\atop 1\leq m''\leq m'}K_{0ir}K_{0i'r'}K_{0ts}K_{0t's'}, 
\end{align}
where $M_1$ and $M_{1+\delta}$ are the constants from Assumption \ref{moment conditions}. Repeating the calculations to derive a bound for the   term $E_{|S_n}|\mathbf{F}_{21,1}|$  yields
\begin{align}
\label{proof C1}
    \frac{\widetilde{F}_{21,21}}{(nT)^2b^3}\xrightarrow[n\to\infty]{\mathbb{P}_{|S_n}}0\ \ a.s.-[\mathbb{Q}].
\end{align}

\par For the term  $\widetilde{F}_{21,22}$ we use the decomposition 
\begin{align}
    \widetilde{F}_{21,22}=&2\sum_{m=2}^{M_n}\sum_{(i,t)\neq (r,s)\in\widetilde{B}^-_m\atop ||(i,t)-(r,s)||_{F}\leq q_n}\sum_{(i',t')\in \widetilde{B}^-_{m'}\atop 1\leq m'<m}\sum_{(r',s')\in \widetilde{B}^-_{m''}\atop  m''= m'}K_{0ir}K_{0i'r'}K_{0ts}K_{0t's'}g_{S_n}(X_{i't'}X_{r's'})\\
    &+2\sum_{m=2}^{M_n}\sum_{(i,t)\neq (r,s)\in\widetilde{B}^-_m\atop ||(i,t)-(r,s)||_{F}\leq q_n}\sum_{(i',t')\in \widetilde{B}^-_{m'}\atop 1\leq m'<m}\sum_{(r',s')\in \widetilde{B}^-_{m''}\atop 1\leq m''< m'}K_{0ir}K_{0i'r'}K_{0ts}K_{0t's'}g_{S_n}(X_{i't'}X_{r's'})\\
    =&:\widetilde{F}_{21,22,1}+\widetilde{F}_{21,22,2}.
\end{align}

\par Since
\begin{align}
    &2\widetilde{F}_{21,22,2}=\sum_{m=2}^{M_n}\sum_{(i,t)\neq (r,s)\in\widetilde{B}^-_m\atop ||(i,t)-(r,s)||_{F}\leq q_n}\sum_{(i',t')\in \widetilde{B}^-_{m'}\atop 1\leq m'\neq m\leq M_n}\sum_{(r',s')\in \widetilde{B}^-_{m''}\atop 1\leq m''\neq m\leq M_n}K_{0ir}K_{0i'r'}K_{0ts}K_{0t's'}g_{S_n}(X_{i't'}X_{r's'})-2\widetilde{F}_{21,22,1}\\
  =&  \sum_{m=2}^{M_n}\sum_{(i,t)\neq (r,s)\in\widetilde{B}^-_m\atop ||(i,t)-(r,s)||_{F}\leq q_n}\sum_{(i',t')\in \widetilde{B}^-_{m'}\atop 1\leq m'\neq m\leq M_n}\sum_{(r',s')\in \widetilde{B}^-_{m''}\atop 1\leq m''\neq m\leq M_n}K_{0ir}K_{0i'r'}K_{0ts}K_{0t's'}(X_{i't'}X_{r's'}-E_{|S_n}[X_{i't'}X_{r's'}])E_{|S_n}[X_{it}X_{rs}]\\
   &+\sum_{m=2}^{M_n}\sum_{(i,t)\neq (r,s)\in\widetilde{B}^-_m\atop ||(i,t)-(r,s)||_{F}\leq q_n}\sum_{(i',t')\in \widetilde{B}^-_{m'}\atop 1\leq m'\neq m\leq M_n}\sum_{(r',s')\in \widetilde{B}^-_{m''}\atop 1\leq m''\neq m\leq M_n}K_{0ir}K_{0i'r'}K_{0ts}K_{0t's'}E_{|S_n}[X_{i't'}X_{r's'}]E_{|S_n}[X_{it}X_{rs}]-2\widetilde{F}_{21,22,1}\\
   =&:\mathbf{A}+\mathbf{B}-2\widetilde{F}_{21,22,1}, \label{d30}
\end{align}
which asserts
\begin{align}
  \widetilde{F}_{21,22,1}+  \widetilde{F}_{21,22,2}=\mf A+\mf B.
\end{align}
Thus, the incoming two steps focus on proving 
$\frac{\mathbf{A}}{(nT)^2b^3}= o_{{\mathbb{P}_{|S_n}}}(1) $ and $\frac{\mathbf{B}}{(nT)^2b^3}= o_{{\mathbb{P}_{|S_n}}}(1)$ hold almost surely with respect to $\mathbb{Q}$ and we name these steps as \textbf{Step 5.2.1.(a)}, \textbf{Step 5.2.1.(b)}.
\smallskip

\par \textbf{Step 5.2.1(a)} ($\frac{\mathbf{A}}{(nT)^2b^3}= o_{{\mathbb{P}_{|S_n}}}(1) \ a.s.-[\mathbb{Q}]$) We   use decomposition 
\begin{align}
    \mathbf{A}=&\sum_{m=2}^{M_n}\sum_{(i,t)\neq (r,s)\in\widetilde{B}^-_m\atop ||(i,t)-(r,s)||_{F}\leq q_n}K_{0ir}K_{0ts}E_{|S_n}[X_{it}X_{rs}] \sum_{(i',t')\in \widetilde{B}^-_{m'}\atop 1\leq m'\neq m\leq M_n}\sum_{(r',s')\in \widetilde{B}^-_{m''}\atop 1\leq m''\neq m\leq M_n}K_{0i'r'}K_{0t's'}\big (X_{i't'}X_{r's'}-E_{|S_n}[X_{i't'}X_{r's'}] \big )\\
    =&\sum_{m=2}^{M_n}\sum_{(i,t)\neq (r,s)\in\widetilde{B}^-_m\atop ||(i,t)-(r,s)||_{F}\leq q_n}K_{0ir}K_{0ts}E_{|S_n}[X_{it}X_{rs}]\sum_{(i',t')\in \widetilde{B}^-_{m'}\atop 1\leq m'\leq M_n}\sum_{(r',s')\in \widetilde{B}^-_{m''}\atop 1\leq m''\leq M_n}K_{0i'r'}K_{0t's'} \big (X_{i't'}X_{r's'}-E_{|S_n}[X_{i't'}X_{r's'}] \big  )  \\
    &-\sum_{m=2}^{M_n}\sum_{(i,t)\neq (r,s)\in\widetilde{B}^-_m\atop ||(i,t)-(r,s)||_{F}\leq q_n}K_{0ir}K_{0ts}E_{|S_n}[X_{it}X_{rs}] \sum_{(i',t')\in \widetilde{B}^-_{m'}\atop 1\leq m'\neq m\leq M_n}\sum_{(r',s')\in \widetilde{B}^-_{m}}K_{0i'r'}K_{0t's'}\big  (X_{i't'}X_{r's'}-E_{|S_n}[X_{i't'}X_{r's'}] \big ) \\
    &-\sum_{m=2}^{M_n}\sum_{(i,t)\neq (r,s)\in\widetilde{B}^-_m\atop ||(i,t)-(r,s)||_{F}\leq q_n}K_{0ir}K_{0ts}E_{|S_n}[X_{it}X_{rs}] \sum_{(i',t')\in \widetilde{B}^-_{m}}\sum_{(r',s')\in \widetilde{B}^-_{m''}\atop 1\leq m''\neq m\leq M_n}K_{0i'r'}K_{0t's'}\big (X_{i't'}X_{r's'}-E_{|S_n}[X_{i't'}X_{r's'}] \big ) \\
    &-\sum_{m=2}^{M_n}\sum_{(i,t)\neq (r,s)\in\widetilde{B}^-_m\atop ||(i,t)-(r,s)||_{F}\leq q_n}K_{0ir}K_{0ts}E_{|S_n}[X_{it}X_{rs}]\sum_{(i',t')\in \widetilde{B}^-_{m}}\sum_{(r',s')\in \widetilde{B}^-_{m}}K_{0i'r'}K_{0t's'}\big (X_{i't'}X_{r's'}-E_{|S_n}[X_{i't'}X_{r's'}] \big ) \\
    =&:\mathbf{A}_{L}- \sum_{k=1}^3\mathbf{A}_{R,k}.
\end{align}

\par The terms $\mathbf{A}_{R,1}, \mathbf{A}_{R,2}$ and $\mathbf{A}_{R,3}$ can be bounded in the same way and we concentrate on $\mathbf{A}_{R,1}$ for the sake of brevity. A key observation is that Assumption \ref{moment conditions} implies (note that the kernels are non-negative)
\begin{align}
    E_{|S_n}|\mathbf{A}_{R,1}|\leq 2M^2_1\sum_{m=2}^{M_n}\sum_{(i,t)\neq (r,s)\in\widetilde{B}^-_m\atop ||(i,t)-(r,s)||_{F}\leq q_n}K_{0ir}K_{0ts}\sum_{(i',t')\in \widetilde{B}^-_{m'}\atop 1\leq m'\neq m\leq M_n}\sum_{(r',s')\in \widetilde{B}^-_{m}}K_{0i'r'}K_{0t's'},\ \forall\ \bS_n=S_n.
\end{align}
Using the same arguments as given  in \textbf{Step 5.2.1(a)}, we obtain 
\begin{align}
    \frac{\mathbf{A}_{R,1}}{(nT)^{2}b^6(p_T(\frac{nr_{1n}r_{2n}}{\lambda_n^2\bb_n}))^2}=O_{\mathbb{P}_{|S_n}}(1)\ \  \ a.s.-[\mathbb{Q}], 
\end{align}
and condition   (T4)(3) in Assumption \ref{technical assumptions} yields 
\begin{align}
    \label{proof C3} 
    \frac{\mathbf{A}_{R,1}}{(nT)^2b^3}\xrightarrow[n\to\infty]{\mathbb{P}_{|S_n}}0\ \ a.s.-[\mathbb{Q}].
\end{align}
Corresponding estimates for $\mathbf{A}_{R,2}$  and $\mathbf{A}_{R,3}$ are shown in exactly the same way, that is 
\begin{align}
\label{proof C4}
    \frac{\mathbf{A}_{R,\ell }}{(nT)^2b^3}\xrightarrow[n\to\infty]{\mathbb{P}_{|S_n}}0\ \ a.s.-[\mathbb{Q}]   ~~~(\ell =2,3).
\end{align}

\par As the inner two sums of the  term $\mathbf{A}_{L}$ have a similar structure  as the statistic $U_n$ in \eqref{d15a}, we can use the same arguments as given  in \textbf{Step 2} to show that
\begin{align}
\label{proof C5}
    \lim_{n\to\infty}\frac{1}{(nT)^2b^3}\text{Var}_{|S_n}\Big(\sum_{(i',t')\in \widetilde{B}^-_{m'}\atop 1\leq m'\leq M_n}\sum_{(r',s')\in \widetilde{B}^-_{m''}\atop 1\leq m''\leq M_n}K_{0i'r'}K_{0t's'}(X_{i't'}X_{r's'}-E_{|S_n}[X_{i't'}X_{r's'}])\Big)<\infty \ \ a.s.-[\mathbb{Q}].
\end{align}
Therefore,  Chebyschev's inequality implies, for any $\xi_n\nearrow\infty$,
\begin{align}
    \frac{1}{(nT)^2b^3}\sum_{(i',t')\in \widetilde{B}^-_{m'}\atop 1\leq m'\leq M_n}\sum_{(r',s')\in \widetilde{B}^-_{m''}\atop 1\leq m''\leq M_n}K_{0i'r'}K_{0t's'}(X_{i't'}X_{r's'}-E_{|S_n}[X_{i't'}X_{r's'}])=o_{\mathbb{P}_{|S_n}} \Big (\frac{\xi_n}{nTb^{1.5}} \Big )\ \ a.s.-[\mathbb{Q}],
\end{align}
and the same arguments  as given before show that for any sequence $(\xi_n)_{n\in \mathbb{N}}$ converging to infinity
\begin{align}
    \frac{\mathbf{A}_{L}}{(nT)^2b^{4.5}(\frac{nq_n^3}{\lambda_n^2\bb_n})\xi_n}\xrightarrow[n\to\infty]{\mathbb{P}_{|S_n}}0 \ \ a.s.-[\mathbb{Q}].
\end{align}
By condition (T4)(3) in  Assumption \ref{technical assumptions} we have   $q_n^3=o(b)$, and by  (T1) of Assumption \ref{technical assumptions}, it follows that 
\begin{align}
    \label{proof C6}
    \frac{\mathbf{A}_{L}}{(nT)^2b^3}\xrightarrow[n\to\infty]{\mathbb{P}_{|S_n}}0\ \ a.s.-[\mathbb{Q}].
\end{align}
Thus, we finish the proof of \textbf{Step 5.2.1(a)}.\\
\smallskip

\par \textbf{Step 5.2.1(b)}($\frac{\mathbf{B}}{(nT)^2b^3}= o_{{\mathbb{P}_{|S_n}}}(1) \ \  a.s.-[\mathbb{Q}]$) By the triangle inequality and condition (M2) of Assumption \ref{moment conditions} we have 
\begin{align}
    |\mathbf{B}|&\leq 2M_1\sum_{m=2}^{M_n}\sum_{(i,t)\neq (r,s)\in\widetilde{B}^-_m\atop ||(i,t)-(r,s)||_{F}\leq q_n}\sum_{(i',t')\in \widetilde{B}^-_{m'}\atop 1\leq m'< m}\sum_{(r',s')\in \widetilde{B}^-_{m''}\atop 1\leq m''<m}K_{0ir}K_{0i'r'}K_{0ts}K_{0t's'}|\text{Cov}_{|S_n}(X_{i't'},X_{r's'})|\\
    &=2M_1(M_{1+\delta})^{\frac{1}{1+\delta}}\mathbf{B}' ,  \label{d26a}
\end{align}
where
\begin{align}
\mathbf{B}' &= \sum_{m=2}^{M_n}\sum_{(i,t)\neq (r,s)\in\widetilde{B}^-_m\atop ||(i,t)-(r,s)||_{F}\leq q_n}\sum_{(i',t')\in \widetilde{B}^-_{m'}\atop 1\leq m'< m}\sum_{(r',s')\in \widetilde{B}^-_{m''}\atop 1\leq m''<m}K_{0ir}K_{0i'r'}K_{0ts}K_{0t's'}\beta^{\frac{\delta}{1+\delta}}(d(\widetilde{B}_{m'},\widetilde{B}_{m''})).   
\end{align}
 We use $A_{S}$ to denote the projection of set $A\subset \mathbf{R}_n\times \{1,2,...,T\}$ onto the spatial component $\mathbf{R}_n$.  Then, for every $\bS_n=S_n$,
\begin{align}
\label{d26b}
    &\mathbf{B}'\leq \mathbf{B}'_{S}\mathbf{B}'_T,
        \end{align}    
where
\begin{align}&\mathbf{B}'_{S}=\sum_{m=2}^{M_{n,S}}\sum_{s_{in}\neq s_{rn}\in\widetilde{B}^-_{m,S}}K_{0ir}\Big(\sum_{s_{i'n}\in\widetilde{B}^-_{m',S}\atop 1\leq m'<m} \sum_{s_{r'n}\in\widetilde{B}^-_{m'',S}\atop 1\leq m''<m} K_{0i'r'}\beta^{\frac{\delta}{(1+\delta)}}(d(\widetilde{B}^-_{m',S},\widetilde{B}^-_{m'',S}))\Big),\\
    &\mathbf{B}'_{T}=\sum_{m=2}^{M_{n,T}}\sum_{t\neq s\in\widetilde{B}^-_{m,T}}K_{0ts}\Big(\sum_{t'\in\widetilde{B}^-_{m',T}\atop 1\leq m'<m}\sum_{s'\in\widetilde{B}^-_{m'',T}\atop 1\leq m''<m}K_{0t's'}\Big),\\ 
    & M_n=M_{n,S}M_{n,T}.
\end{align}
Note that  we can regard $\{\widetilde{B}^-_{m,S}\}$ as a matrix of two-dimensional cubes, where  the distance between any two cubes is $q_n$. Furthermore, by regarding this matrix of cubes as a matrix of the lattice, we know, for any $m'\neq m''$, $d(\widetilde{B}^-_{m',S},\widetilde{B}^-_{m'',S}))\geq k_{m',m''}q_n$, where $k_{m',m''}$ is some positive integer representing the “lattice-distance”. Then,
\begin{align}
    \mathbf{B}'_{S}&\leq\sum_{m=2}^{M_{n,S}}\sum_{s_{in}\neq s_{rn}\in\widetilde{B}^-_{m,S}}K_{0ir}\Big(\sum_{s_{i'n}\in\widetilde{B}^-_{m',S}\atop 1\leq m'<m}\sum_{k_{m',m''}=0}^{\infty} \sum_{s_{r'n}\in \mathcal{B}(k_{m',m''})} K_{0i'r'}\beta^{\frac{\delta}{(1+\delta)}}(k_{m',m''}q_n)\Big) \leq C\mathbf{B}''_{S}\beta^{\frac{\delta}{1+\delta}}(q_n), 
\end{align}
where   the last inequality is due to  condition (T2)(ii) in Assumption \ref{technical assumptions}, and
$$
\mathbf{B}''_{S}:=\sum_{m=2}^{M_{n,S}}\sum_{s_{in}\neq s_{rn}\in\widetilde{B}^-_{m,S}}K_{0ir}\sum_{s_{i'n}\in\widetilde{B}^-_{m',S}\atop 1\leq m'<m}\sum_{k_{m',m''}=0}^{\infty} \sum_{s_{r'n}\in \mathcal{B}(k_{m',m''})} K_{0i'r'}\beta^{\frac{\delta}{(1+\delta)}}(k_{m',m''})
$$
and the 
constant $C>0$ is independent of the sample size. As $\mathcal{B}(k_{m',m''})$ is the set of all $\widetilde{B}^-_{m''}$ whose aforementioned “lattice-distance” to $\widetilde{B}^-_{m'}$ is equal to $k_{m',m''}\in\mathbb{N}$, it follows that 
\begin{align}
    \label{d26b+}
\mathbf{B}'\leq C\mathbf{B}''_{S}\mathbf{B}'_{T}\beta^{\frac{\delta}{1+\delta}}(q_n). 
\end{align}
Simple algebra, Lemmas \ref{lemma 5}, \ref{lemma 6} and condition  (T4)(1) in Assumption \ref{technical assumptions} imply
\begin{align}
  \mathbf{B}'_{T}\beta^{\frac{\delta}{1+\delta}}(q_n)\lesssim (T^2b)(q_T\log T\lor Tb)q_T(\log T)^2\beta^{\frac{\delta}{1+\delta}}(q_n)=o(b(q_T\log T\lor Tb)q_T(\log T)^2). \label{d26c}
\end{align}

\par To derive a  bound for  $\mathbf{B}''_{S}$, we  use the strong law of large numbers (note that the $\{\s_{in}\}$ are row-wise iid) and obtain 
\begin{align}
    \mathbf{B}''_{S}\leq& C b^4\overline{E}\Big(\frac{nr_{1n}r_{2n}}{\lambda_n^2\bb_n}\Big)\Big(\sum_{m=2}^{M_{n,S}}|\widetilde{B}^-_{m,S}|^2\Big)\Big(\sum_{m'=2}^{M_{n,S}}|\widetilde{B}^-_{m',S}|\Big)\Big(\sum_{l=0}^{\infty}l\beta^{\frac{\delta}{1+\delta}}(l)\Big)\\
    =&C\overline{E}\Big(\sum_{l=0}^{\infty}l\beta^{\frac{\delta}{1+\delta}}(l)\Big)\Big(\frac{nr_{1n}r_{2n}}{\lambda_n^2\bb_n}\Big)^2n^2b^4=o(n^2b^3)  \label{d26d}
\end{align}
 almost surely with respect to $\mathbb{Q}$.  Summarizing \eqref{d26a}, \eqref{d26b+}, \eqref{d26c}   and \eqref{d26d}, we obtain 
\begin{align}
    \lim_{n\to\infty}\frac{\mathbf{B}}{n^2b^4(q_T\log T\lor Tb)q_T(\log T)^2}=0\ \ \ a.s.-[\mathbb{Q}],
\end{align}
which gives 
\begin{align}
     \label{proof 20}
      \lim_{n\to\infty}\frac{\mathbf{B}}{(nT)^2b^3}=0\ \ \ a.s.-[\mathbb{Q}].
\end{align}

\par Combining  the results from \textbf{Step 5.2.1(a)} and \textbf{Step 5.2.1(b)} therefore yields  
$$
\frac{\widetilde{F}_{21,22}}{(nT)^2b^3}\xrightarrow[n\to\infty]{\mathbb{P}_{|S_n}}0\ a.s.-[\mathbb{Q}], 
$$
and combining this result with \eqref{d23} and \eqref{proof C1}
yields 
$$
\frac{\mathbf{F}_{21,2}}{(nT)^2b^3}\xrightarrow[n\to\infty]{\mathbb{P}_{|S_n}}0\ a.s.-[\mathbb{Q}]. 
$$
This gives (observing \eqref{d24}  and \eqref{d25a}) 
$$
\frac{\mathbf{F}_{21}}{V_{T1}}\xrightarrow[n\to\infty]{\mathbb{P}_{|S_n}}0\ a.s.-[\mathbb{Q}].
$$
To derive a corresponding statement for  $\mathbf{F}_{22}$, we note that 
\begin{align}
    |E_{|S_n}[X_{it}X_{i't'}]E_{|S_n}[X_{rs}X_{r's'}]|
   &  \leq 16(M_{1+\delta})^{\frac{1}{1+\delta}}\beta^{\frac{\delta}{1+\delta}}(d(\widetilde{B}^-_{m},\widetilde{B}^-_{m'}))\beta^{\frac{\delta}{1+\delta}}(d(\widetilde{B}^-_{m},\widetilde{B}^-_{m''}))\\
  &  \leq  16 (M_{1+\delta})^{\frac{1}{1+\delta}}\beta^{\frac{\delta}{1+\delta}}(d(\widetilde{B}^-_{m,S},\widetilde{B}^-_{m',S}))\beta^{\frac{\delta}{1+\delta}}(d(\widetilde{B}^-_{m,S},\widetilde{B}^-_{m'',S})), 
\end{align}
and similar arguments as used  in \textbf{Step 5.2.1(b)} show 
\begin{align}
    \label{proof 21}
     \frac{\mathbf{F}_{22}}{V_{T1}}\xrightarrow[n\to\infty]{\mathbb{P}_{|S_n}}0 \ \ \ a.s.-[\mathbb{Q}].
\end{align}
Therefore, Markov inequality and \eqref{F21F22} yield 
\begin{align}
    \label{proof 22}
    \frac{\mathbf{|F}_{2}|}{V_{T1}}\xrightarrow[n\to\infty]{\mathbb{P}_{|S_n}} 0\ \ \ a.s.-[\mathbb{Q}],
\end{align}
which completes the proof of  \textbf{Step 5.2.1}.

\noindent \textbf{Step 5.2.2} ($\frac{\mathbf{F}_{1}}{V_{T1}}\xrightarrow[n\to\infty]{\mathbb{P}_{|S_n}} 1\ \ a.s.-[\mathbb{Q}]$) 
Recalling  the definition of $\mathbf F_1$ in \eqref{eq:F1F2}   reveals that 
\begin{align}
    \mathbf{F}_{1}=&\sum_{m=2}^{M_n}\sum_{(i,t)\in\widetilde{B}^-_{m}}\sum_{(i',t')\in\widetilde{B}^{-}_{m'}\atop 1\leq m'<m}E_{|S_n}[H^2_{n}(Z_{it},Z_{i't'})|\mathcal{F}_{m-1}]\\
    &+\sum_{m=2}^{M_n}\sum_{(i,t)\in\widetilde{B}^-_{m}}\sum_{(i',t')\neq (i'',t'')\in\widetilde{B}^{c}_{m,S}}E_{|S_n}[H_{n}(Z_{it},Z_{i't'})H_{n}(Z_{it},Z_{i''t''})|\mathcal{F}_{m-1}]\\
    =&:\mathbf{F}_{11}+\mathbf{F}_{12}
\end{align}
where $\widetilde{B}^c_{m}=\bigcup_{1\leq m'<m}\widetilde{B}^-_{m'}$. By  the definition of $\tilde T_{m^{\prime},n}$ in  \eqref{eq:tildeT}, for each $\bS_n=S_n$, $\mathcal{F}_{m-1}=\sigma(\widetilde{T}_{m',n}: 1\leq m'\leq m-1)$ and $\widetilde{T}_{m',n}$ is measurable with respect to $\sigma(X_{it},X_{i't'}:(i,t)\in\widetilde{B}^-_{m}, (i',t')\in \bigcup_{1\leq m'<m}\widetilde{B}^-_{m'})=:\mathcal{G}_{m-1}$. Hence, by the tower property of conditional expectation,   it is sufficient to focus on the following two terms
\begin{align}
    \widetilde{\mathbf{F}}_{11}:&=\sum_{m=2}^{M_n}\sum_{(i,t)\in\widetilde{B}^-_{m}}\sum_{(i',t')\in\widetilde{B}^{-}_{m'}\atop 1\leq m'<m}E_{|S_n}[H^2_{n}(Z_{it},Z_{i't'})|\mathcal{G}_{m-1}],\\
    \widetilde{\mathbf{F}}_{12}:&=\sum_{m=2}^{M_n}\sum_{(i,t)\in\widetilde{B}^-_{m}}\sum_{(i',t')\neq (i'',t'')\in\widetilde{B}^{c}_{m}}E_{|S_n}[H_{n}(Z_{it},Z_{i't'})H_{n}(Z_{it},Z_{i''t''})|\mathcal{G}_{m-1}].
\end{align}
Note that
\begin{align}
   \widetilde{\mathbf{F}}_{12}\leq &\sum_{m=2}^{M_n}\sum_{(i,t)\in\widetilde{B}^-_{m}}\sum_{(i',t')\neq (i'',t'')\in\widetilde{B}^{c}_{m}}K_{0b}(s_{in}-s_{i'n},t-t')K_{0b}(s_{in}-s_{i''n},t-t'')|E_{|S_n}[X^2_{it}X_{i't'}X_{i''t''}|\mathcal{G}_{m-1}]|\\
   +& \sum_{m=2}^{M_n}\sum_{(i,t)\in\widetilde{B}^-_{m}}\sum_{(i',t')\neq (i'',t'')\in\widetilde{B}^{c}_{m}}K_{0b}(s_{in}-s_{i'n},t-t')K_{0b}(s_{in}-s_{i''n},t-t'')|E_{|S_n}[X_{it}X_{i't'}]E_{|S_n}[X_{it}X_{i''t''}]|\\
   =&:\widetilde{F}_{12,1}+\widetilde{F}_{12,2}. \label{d29}
\end{align}
The same arguments used in deriving a bound for  the term “$I_{k,j,2}$” in “Step 2.1.1.b” show that 
\begin{align}
    \label{d28}
\frac{\widetilde{F}_{12,2}}{V_{T1}}\xrightarrow[n\to\infty]{\mathbb{P}_{|S_n}}0
\end{align}
almost surely with respect to $\mathbb{Q}$. For the term  $\widetilde{F}_{12,1}$ note that  $g(X_{it},X_{i't'})=\int_{\mathbb{R}}X_{it}^2X_{i't'}xd\mathbb{P}_{X_0} =0$ (for every realization of $(X_{it},X_{i't'})$). Thus, Lemma \ref{lemma 2} implies  
\begin{align}
    E_{|S_n}|\widetilde{F}_{12,1}|& \leq 3(\widetilde{M}_{1+\delta})^{\frac{1}{1+\delta}}\sum_{m=2}^{M_n}\sum_{(i,t)\in\widetilde{B}^-_{m}}\sum_{(i',t')\neq (i'',t'')\in\widetilde{B}^{c}_{m}}K_{0ii'}K_{0ii''}K_{0tt'}K_{0tt''}\beta^{\frac{\delta}{1+\delta}}(d(\{(i,t),(i',t')\},\{(i'',t'')\}))\\
    &\leq 3(\widetilde{M}_{1+\delta})^{\frac{1}{1+\delta}}\sum_{m=2}^{M_n}\sum_{(i,t)\in\widetilde{B}^-_{m}}\sum_{(i',t')\neq (i'',t'')\in\widetilde{B}^{c}_{m}}K_{0ii'}K_{0ii''}K_{0tt'}K_{0tt''}\beta^{\frac{\delta}{1+\delta}}(d(\{s_{in},s_{i'n}\},\{s_{i''n}\})),
\end{align}
where $d(A,B)=\inf_{a\in A,b\in B}||a-b||_{F}$, for any $A,B\subset \R_n$. By similar arguments as   used in Step 5.2.1, we have
\begin{align}
    &\sum_{m=2}^{M_n}\sum_{(i,t)\in\widetilde{B}^-_{m}}\sum_{(i',t')\neq (i'',t'')\in\widetilde{B}^{c}_{m}}K_{0ii'}K_{0ii''}K_{0tt'}K_{0tt''}\beta^{\frac{\delta}{1+\delta}}(d(\{s_{in},s_{i'n}\},\{s_{i''n}\})) 
    \leq \triangle_{S}\triangle_{T} ,
\end{align}
where
\begin{align}
 \triangle_{S} &:= \sum_{m=2}^{M_{n,S}}\sum_{s_{in}\in\widetilde{B}^-_{m,S}}\sum_{s_{i'n}\neq s_{i''n}\in \widetilde{B}^c_{m, S}}K_{0ii'}K_{0ii''}\beta^{\frac{\delta}{1+\delta}}(d(\{s_{in},s_{i'n}\},\{s_{i''n}\})) , \\
 \triangle_{T} &:= \sum_{m=2}^{M_{n,T}}\sum_{t\in\widetilde{B}^-_{m,T}}\sum_{t'\neq t''\in\widetilde{B}^c_{m,T}}K_{0tt'}K_{0tt''} . 
\end{align}

Similar to arguments as given  in Step 5.2.1 (b), we have
\begin{align}
    \triangle_{S}\leq C_{\beta}\sum_{m=2}^{M_{n,S}}\sum_{s_{in}\in\widetilde{B}^-_{m}}\sum_{s_{i'n}\in\widetilde{B}^c_{m}}K_{0ii'}\sum_{k_{m',m''}=0}^{\infty}\sum_{s_{i''n}\in \mathcal{B}(k_{m',m''})}K_{0ii''}\beta^{\frac{\delta}{1+\delta}}(k_{m',m''})\beta^{\frac{\delta}{1+\delta}}(q_n)=:\widetilde{\triangle}_S\beta^{\frac{\delta}{1+\delta}}(q_n),
\end{align}
which asserts 
\begin{align}
 E_{|S_n}|\widetilde{F}_{12,1}|\leq 3(\widetilde{M}_{1+\delta})^{\frac{1}{1+\delta}}   \triangle_{S}\triangle_{T}\leq 3(\widetilde{M}_{1+\delta})^{\frac{1}{1+\delta}}C_{\beta}\widetilde{\triangle}_S(\triangle_T\beta^{\frac{\delta}{1+\delta}}(q_n)).
\end{align}
Hence, similar to the arguments used in Step 5.2.1,  Markov's inequality, Lemmas \ref{lemma 5} and \ref{lemma 6} and  condition (T4)(1)  in Assumption \ref{technical assumptions} yield $ \frac{\widetilde{F}_{12,1}}{V_{T1}}\xrightarrow[n\to\infty]{\mathbb{P}_{|S_n}}0$ (almost surely with respect to $\mathbb{Q}$), which asserts  observing \eqref{d28} $ \frac{\widetilde{\mathbf{F}}_{12}}{V_{T1}}\xrightarrow[n\to\infty]{\mathbb{P}_{|S_n}}0.$ By \eqref{d29} and  \eqref{d28} we therefore obtain
\begin{align}
    \label{proof 23}
\frac{\mathbf{F}_{12}}{V_{T1}}\xrightarrow[n\to\infty]{\mathbb{P}_{|S_n}}0  \ \ a.s.-[\mathbb{Q}].
\end{align}

\par Now we focus on showing 
\begin{align}
    \label{d32}\frac{\widetilde{\mathbf{F}}_{11}}{V_{T1}}\xrightarrow[n\to\infty]{\mathbb{P}_{|S_n}}1\ \ a.s.-[\mathbb{Q}].
    \end{align}
    First, note that similar arguments as used in  
\begin{align}
\label{proof 24}&\lim_{n\to\infty}\frac{\Delta_{n}}{V_{T1}}=1,
\end{align}
almost surely with respect to $\mathbb{Q}$, where
\begin{align}
    \Delta_n=\sum_{m=2}^{M_n}\sum_{(i,t)\in\widetilde{B}^-_{m}}\sum_{(i',t')\in\widetilde{B}^-_{m'} \atop 1\leq m'<m}K_{0b}^2(s_{in}-s_{i'n},t-t')(E[X^2_0])^2. 
    \label{proof 24a}
\end{align}

Recalling the notation of  
 $ H_{n}$ and $m$ in  \eqref{d15d} and \eqref{d15e}, respectively,  we have 

\begin{align}
    \widetilde{\mathbf{F}}_{11}&= \sum_{m=2}^{M_n}\sum_{(i,t)\in\widetilde{B}^-_{m}}\sum_{(i',t')\in\widetilde{B}^{-}_{m'}\atop 1\leq m'<m}E_{|S_n}[H^2_{n}(Z_{it},Z_{i't'})|\mathcal{G}_{m-1}]\\
    &= \sum_{m=2}^{M_n}\sum_{(i,t)\in\widetilde{B}^-_{m}}\sum_{(i',t')\in\widetilde{B}^{-}_{m'}\atop 1\leq m'<m}K^2_{0b}(s_{in}-s_{i'n},t-t')E_{|S_n}[m^2(X_{it},X_{i't'})|\mathcal{G}_{m-1}].
\end{align}

For the conditional expectation  we obtain 
\begin{align}
    E_{|S_n}[m^2(X_{it},X_{i't'})|\mathcal{G}_{m-1}]& 
    =(E_{|S_n}[X_{it}^2X_{i't'}^2|\mathcal{G}_{m-1}]-E[X^2_{i't'}\widetilde{X}^2_{it}])+E[X^2_{i't'}\widetilde{X}^2_{it}]\\
    &-2(E_{|S_n}[X_{it}X_{i't'}|\mathcal{G}_{m-1}]-E[X_{i't'}\widetilde{X}_{it}])\text{Cov}_{|S_n}(X_{it},X_{i't'})\\
    &+(\text{Cov}_{|S_n}(X_{it},X_{i't'}))^2 \\
    &
    =:\mathbf{E}_{1}+E[X^2_{i't'}\widetilde{X}^2_{it}]-\mathbf{E}_{2}+\mathbf{E}_{3},
\end{align}
where $\{\X_{it}\}$ is the shadow sequeneces introduced in Step 1. The decomposition above yields, on condition of $\bS_n=S_n$, 
\begin{align}
    \widetilde{\mathbf{F}}_{11}= &\sum_{m=2}^{M_n}\sum_{(i,t)\in\widetilde{B}^-_{m}}\sum_{(i',t')\in\widetilde{B}^{-}_{m'}\atop 1\leq m'<m}K^2_{0b}(s_{in}-s_{i'n},t-t')(\mathbf{E}_{1}-\mathbf{E}_2+\mathbf{E}_3)\\
    &+ \sum_{m=2}^{M_n}\sum_{(i,t)\in\widetilde{B}^-_{m}}\sum_{(i',t')\in\widetilde{B}^{-}_{m'}\atop 1\leq m'<m}K^2_{0b}(s_{in}-s_{i'n},t-t')E[X^2_{i't'}\widetilde{X}^2_{it}]=:\widetilde{F}_{11,1}+\widetilde{F}_{11,2}.
\end{align}
Similar arguments as given  in Step 2,  Lemmas \ref{lemma 2} and \ref{lemma 3} yield for the first term $\frac{\widetilde{F}_{11,1}}{V_{T1}}\xrightarrow[n\to\infty]{\mathbb{P}_{|S_n}}0$. 
On the  other hand, we have 
\begin{align}
 \widetilde{F}_{11,2}
 =&\widetilde{F}_{11,21}+\Delta_n,
\end{align}
where 
\begin{align}
   \label{holger} 
\widetilde{F}_{11,21}:= 
\sum_{m=2}^{M_n}\sum_{(i,t)\in\widetilde{B}^-_{m}}\sum_{(i',t')\in\widetilde{B}^{-}_{m'}\atop 1\leq m'<m}K^2_{0b}(s_{in}-s_{i'n},t-t')(X^2_{i't'}-E[X_0^2])E[X^2_0]
\end{align}
and $\Delta_n$ is introduced in \eqref{proof 24a} and satisfies \eqref{proof 24}. Hence, to finish Step 5.2.2, it suffices to show that $\frac{\widetilde{F}_{11,21}}{(nT)^2b^3}\xrightarrow[n\to\infty]{\mathbb{P}_{|S_n}}0$  almost surely with respect to $\mathbb{Q}$. 

\par For a measurable function $h:\mathbf{R}_n\to\mathbb{R}$, we  define 
\begin{align}
    E_{|s_{in}}[h(s_{in},s_{jn})]=\int_{\mf D} h(s_{in},\Lambda_nx) f(x)dx,
\end{align}
where $f$ is the density introduced in Assumption \ref{as spatial locations 1}. Then, for each $\bS_n=S_n$,
\begin{align}
    2\widetilde{F}_{11,21}=&E[X^2_0]\sum_{m=2}^{M_n}\sum_{(i,t)\in\widetilde{B}^-_{m}}\sum_{(i',t')\in\widetilde{B}^{-}_{m'}\atop 1\leq m'< m}K^2_{0ii'}K^2_{0tt'}(X^2_{i't'}-E[X_0^2])
      =\mathbf{A}+\mathbf{B},\label{eq:AB}
\end{align}
where
\begin{align}
    \mathbf{A}& :=  E[X^2_0]b^2\sum_{m=2}^{M_n}\sum_{(i,t)\in\widetilde{B}^-_{m}}\sum_{(i',t')\in\widetilde{B}^{c}_{m}}[b^{-2}(K^2_{0ii'}-E_{|s_{i'n}}[K^2_{0ii'}])]K^2_{0tt'}(X^2_{i't'}-E[X_0^2]) \\
    \mathbf{B} & := E[X^2_0]b^2\sum_{m=2}^{M_n}\sum_{(i,t)\in\widetilde{B}^-_{m}}\sum_{(i',t')\in\widetilde{B}^{c}_{m}}(b^{-2}E_{|s_{i'n}}[K^2_{0ii'}])K^2_{0tt'}(X^2_{i't'}-E[X_0^2])
\end{align}
and 
$\widetilde{B}^{c}_{m}=\bigcup_{1\leq m'< m}\widetilde{B}^-_{m'}$. Now we focus on showing 
\begin{align}
    \label{d31a}
& \frac{\mathbf{A}}{(nT)^2b^3}\xrightarrow[n\to\infty]{\mathbb{P}_{|S_n}}0, \\  
&\frac{\mathbf{B}}{(nT)^2b^3}\xrightarrow[n\to\infty]{\mathbb{P}_{|S_n}}0
 \label{d31b}
\end{align}
 almost surely with respect to $\mathbb{Q}$.  

We  start with \eqref{d31b} and denote by   $\overline{E}=\max_{n}\max_{1\leq i,i'\leq n}||b^{-2}E[K_{0ii'}]||_{\infty}$, which  is a finite constant independent of sample size. We also denote $\textbf{Card}(A\cap\Gamma_n)=|A|$, $\forall\ A\subset \mf R_n$ and note that  $K_{0tt'}\neq 0$ if and only if $v_0-Tb\leq |t'-t|\leq v_0+Tb$. Hence, together with triangle inequality, by denoting $\mathcal{T}_{t,v_0}=\{t'\in\{1,2,..,T\}:v_0-Tb\leq |t'-t|\leq v_0+Tb\}$, some simple algebra shows that,
\begin{align}
|\mathbf{B}|\leq &
\label{proof 26}
E[X_0^2]b^2\sum_{m=2}^{M_n}\sum_{(i,t)\in\widetilde{B}^-_{m}}\left|\sum_{(i',t')\in\widetilde{B}^{c}_{m}}(b^{-2}E_{\s_{i'n}}[K^2_{0ii'}])K^2_{0tt'}1[t'\in\mathcal{T}_{t,v_0}](X^2_{i't'}-E[X_0^2])\right|.
\end{align}
For the given $v_0$, we name the set of $t$'s such that $\mathcal{T}_{t,v_0}\cap\{1,2,...,T\}$ is asymptotically non-empty as $\mathbf{T}_{v_0}$. More specifically, we decompose $\mathbf{T}_{v_0}$ as ``boundary'' points and ``interior'' points,
\begin{align}
   &\mathbf{T}_{v_0}=\mathbf{T}_{v_0,1}\bigcup\mathbf{T}_{v_0,2},
   \end{align}
where
\begin{align}
& \mathbf{T}_{v_0,1}=\Big\{t\in \{1,...,T\}: \lim_{n\to\infty}\frac{\textbf{Card}(\mathcal{T}_{t,v_0}\cap\{1,...,T\})}{Tb}=0\Big\},\\
   & \mathbf{T}_{v_0,2}=\Big\{t\in \{1,...,T\}: \lim_{n\to\infty}\frac{\textbf{Card}(\mathcal{T}_{t,v_0}\cap\{1,...,T\})}{Tb}\in (0,4]\Big\}.
\end{align}
Together with \eqref{proof 26}, this asserts 
\begin{align}
    |\mathbf{B}|\leq& E[X_0^2]b^2\sum_{m=2}^{M_n}\sum_{\s_{in}\in\widetilde{B}^-_{m,S}}\sum_{t\in\widetilde{B}^-_{m,T}\cap\mathbf{T}_{v_0,1}}\left|\sum_{(i',t')\in\widetilde{B}^{c}_{m}}(b^{-2}E_{\s_{i'n}}[K^2_{0ii'}])K^2_{0tt'}1[t'\in\mathcal{T}_{t,v_0}](X^2_{i't'}-E[X_0^2])\right|\\
    &+E[X_0^2]b^2\sum_{m=2}^{M_n}\sum_{\s_{in}\in\widetilde{B}^-_{m,S}}\sum_{t\in\widetilde{B}^-_{m,T}\cap\mathbf{T}_{v_0,2}}\left|\sum_{(i',t')\in\widetilde{B}^{c}_{m}}
    (b^{-2}
    E_{\s_{i'n}}[K^2_{0ii'}])K^2_{0tt'}1[t'\in\mathcal{T}_{t,v_0}]
    (X^2_{i't'}-E[X_0^2])\right|\\
    =&:\mathbf{B}_1+\mathbf{B}_2.
\end{align}
Note that  $\{(b^{-2}E_{\s_{i'n}}[K^2_{0ii'}])K^2_{0tt'}\}$ as an array of non-negative constants with upper bound independent of $(i,t,i',t')$ and the sample size. This array will be denoted by  $\{c_{it,i't'}\}$ in the following discussion. Moreover, with the notation 
\begin{align}
  O_{m,T,v_0,k}:=   \overline{B}^c_{m,T}\bigcap \mathbf{T}_{v_0,k},\  \ \mathbf{O}_{m,v_0,k}:= O_{m,T,v_0,k}\times \widetilde{B}^c_{m,S},\ \ \ k=1,2,
\end{align}
we obtain $\textbf{Card}(O_{m,T,v_0,1})=o(Tb)$ and $\textbf{Card}(O_{m,T,v_0,2})\leq 4Tb$, see the illustration in 
 Figure \ref{eq:Oterms}.
\begin{figure}
    \centering    \includegraphics[width=0.5\linewidth]{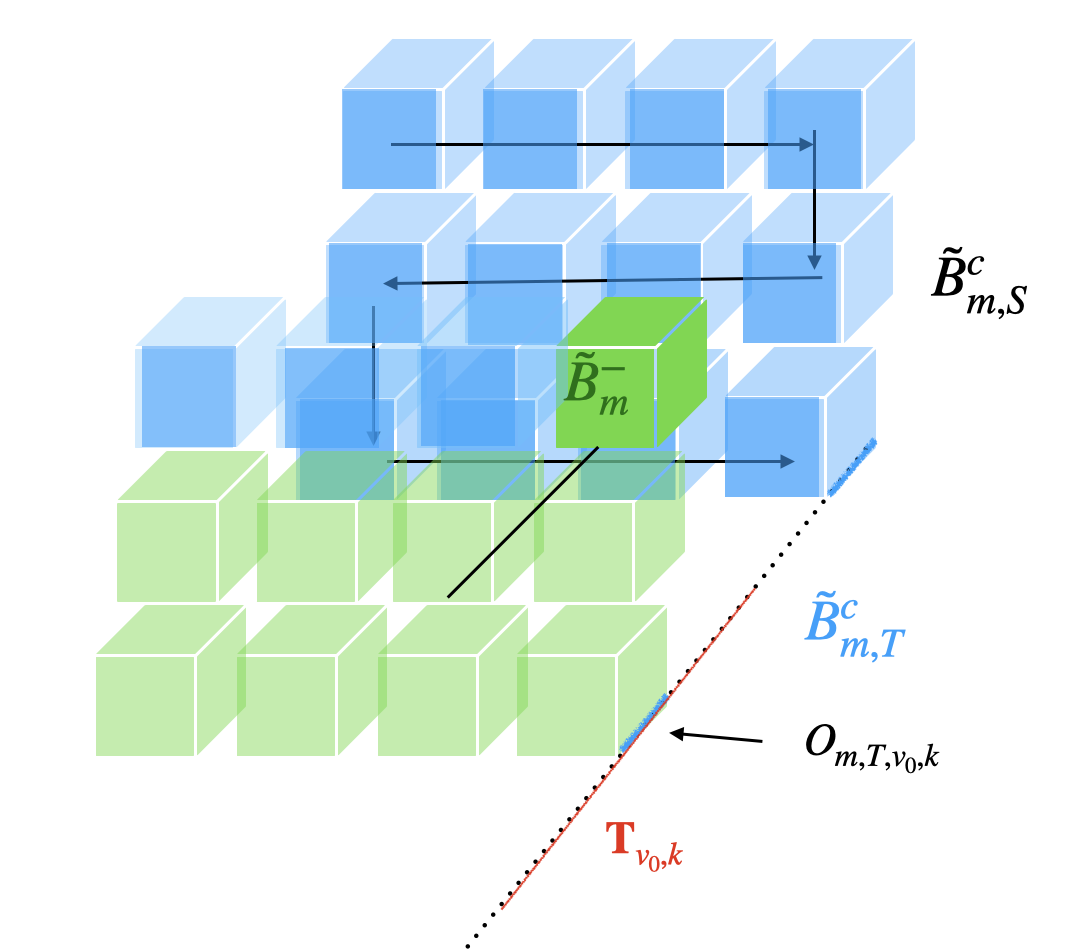}
    \caption{An illustration of $\widetilde B_{m}^{-}$, $\widetilde B_{m, S}^{c}$ and $O_{m, T, v_0, k}$ based on the ordering in Figure \ref{fig:chess}. The light blue cubes denote $\widetilde B_{m}^{c}$, while the deeper blue faces denote $\widetilde B_{m, S}^{c}$ and the deeper blue lines denote $\widetilde B_{m, T}^{c}$. The red line denote $\mf T_{v_0, k}$, the set of non-zero term in the time dimension. The intersection of the red line and the blue lines is $O_{m, T, v_0, k}$. }
    \label{eq:Oterms}
\end{figure}
Similar to the proof of Proposition \ref{prop 10}, by introducing proper sets and event, Proposition \ref{prop 5} yields that, $\textbf{Card}(\widetilde{B}^c_{m,S})/n \to 1$.
Therefore, for every $1\leq m\leq M_n$ 
\begin{align}
    &\lim_{n\to\infty}\frac{\textbf{Card}(\mathbf{O}_{m,v_0,1})}{nTb}=0,\ \lim_{n\to\infty}\frac{\textbf{Card}(\mathbf{O}_{m,v_0,1})}{n}>0\ \text{and}\\ 
    &\lim_{n\to\infty}\frac{\textbf{Card}(\mathbf{O}_{m,v_0,2})}{nTb}<\infty, \lim_{n\to\infty}\frac{\log n\textbf{Card}(\mathbf{O}_{m,v_0,2})}{nTb}>0
\end{align}
hold almost surely with respect to $\mathbb{Q}$. We further introduce events 
\begin{align}
    &\mathcal{C}_{k}:= \Big\{\lim_{n\to\infty}\mathbb{P}\Big(\Big|\textbf{Card}(\mathbf{O}_{m,v_0,k})^{-1}\sum_{(i',t')\in\mathbf{O}_{m,v_0,k}}c_{it,i't'}(X^2_{i't'}-E[X_0^2])\Big|>\epsilon_{k,n}\Big)=0\Big\},\ \  k=1,2,
    \end{align}
where
 \begin{align}
    \epsilon_{1,n}= \sqrt{\frac{\xi_n}{n}},\  \epsilon_{2,n}=\sqrt{\frac{\xi_n\log n}{nTb}}
\end{align}
and $(\xi_n)_{n \in \mathbb{N}} \subset \mathbb{R}$ is a diverging non-negative sequence. The reason why $\epsilon_{k,n}$ is invariant with respect to $m$ is because the growth condition of $\textbf{Card}(\mathbf{O}_{m,v_0,k})$ is independent of $m$.
\par Together with Lemma \ref{lemma 3}, Chebyschev's inequality and typical Bernstein-blocking techniques, we can show 
\begin{align}
\label{proof 27}
   \frac{\mathbf{B}_{1}\lor \mathbf{B}_2}{(nT)^2 b^3}\xrightarrow[n\to\infty]{\mathbb{P}_{|S_n}}0 \ \ \ a.s.-[\mathbb{Q}].
\end{align}
Similarly, by repeating the  concentration arguments for the term  $\mathbf{A}$ in \eqref{d31a}, we can show $\frac{\mathbf{A}}{\mathbf{B}}\xrightarrow[n\to\infty]{\mathbb{P}_{|S_n}}0 \ \ a.s.-[\mathbb{Q}]$, which asserts 
\begin{align}
\label{proof 28}
    \frac{\mathbf{A}}{(nT)^2b^3}\xrightarrow[n\to\infty]{\mathbb{P}_{|S_n}}0 \ \ a.s.-[\mathbb{Q}]
\end{align}
Combining \eqref{eq:AB}, \eqref{proof 27} and \eqref{proof 28} yield $\frac{\widetilde{F}_{11,21}}{(nT)^2b^3}\xrightarrow[n\to\infty]{\mathbb{P}_{|S_n}}0 \ \ a.s.-[\mathbb{Q}]$, which implies  \eqref{d32}. 
Then, together with tower property of conditional expectation, we obtain 
$$
\frac{\mathbf{F}_{1}}{V_{T1}}\xrightarrow[n\to\infty]{\mathbb{P}_{|S_n}}1 \ \ a.s.-[\mathbb{Q}]
$$
and thus have completed  the proof of {Step 5.2.2} and the whole {Step 5.2}.
\smallskip

\par \textbf{Step 5.3} (For any $\epsilon>0$, $\sum_{m=2}^{M_n}E_{|S_n}[\widetilde{T}^2_{m,n}1[|\widetilde{T}_{m,n}|\geq \epsilon]|\mathcal{F}_{m-1}]\xrightarrow[n\to\infty]{\mathbb{P}_{|S_n}}0,\ a.s.-[\mathbb{Q}]$)
\par By  Markov's inequality and $\lim_{n\to\infty}\frac{\sigma^4_n}{(nT)^4b^6}\in (0,\infty)$, it suffices to prove that 
\begin{align}
\label{proof 29}
    \lim_{n\to\infty}\frac{\sum_{m=2}^{M_n}E_{|S_n}[T^4_{m,n}]}{(nT)^4b^6}=0
\end{align}
almost surely with respect to $\mathbb{Q}$. Recalling the notation \eqref{d33}, \eqref{d15d} and \eqref{d15e}, 
we have  
\begin{align}
    T_{m,n}&=\sum_{(i,t)\in\widetilde{B}^-_{m}}Y_m(Z_{it})=\sum_{(i,t)\in\widetilde{B}^-_{m}}\sum_{(i',t')\in\widetilde{B}^-_{m'}\atop 1\leq m'<m}H_{n}(Z_{it},Z_{i't'})=\sum_{(i,t)\in\widetilde{B}^-_{m}}\sum_{(i',t')\in\widetilde{B}^c_{m}}H_{n}(Z_{it},Z_{i't'})\\
    &=\sum_{(i,t)\in\widetilde{B}^-_{m}}\sum_{(i',t')\in\widetilde{B}^c_{m}}K_{0b}(s_{in}-s_{i'n},t-t')m(X_{it},X_{i't'}),
\end{align}
where $ \widetilde{B}^c_{m}=\bigcup_{1\leq m'<m}\widetilde{B}^-_{m'}.$
For any given $\bS_n=S_n$, let $\mathbb{P}_{S_n, m}$ be the distribution of random vector 
$(X_{it}, (i,t)\in\widetilde{B}^-_{m})$. 
In the following, we will apply Lemma \ref{lemma 2}, to approximate   $E_{|S_n}[T^4_{m,n}]$ by $E_{|S_{n}}[\int (\sum_{(i,t)\in\widetilde{B}^-_{m}}Y(z_{it}))^4 d\mathbb{P}_{S_n,m}]$. 

Jensen inequality and condition (M2) of Assumption \ref{moment conditions} yield for any $1\leq m\leq M_n$  
\begin{align}   
    &\max\Big \{E_{|S_n}|T^4_{m,n}|^{1+\delta},E_{|S_n} \Big [\int |T^4_{m,n}|^{1+\delta}d\mathbb{P}_{S_n,m} \Big ] \Big \} \\
 &   \ \  \ \  \ \  \ \   \leq M_{4(1+\delta)}\left( nT\Big(\frac{nr_{1n}r_{2n}}{\lambda^2\bb_n}\Big)p_T\right)^{4(1+\delta)}\Big(\frac{1}{|\widetilde{B}^-_{m}||\widetilde{B}^c_{m}|}\sum_{(i,t)\in\widetilde{B}^-_{m}}\sum_{(i',t')\in\widetilde{B}^c_{m}}K^{4(1+\delta)}_{0b}(s_{in}-s_{i'n},t-t')\Big) \\
 &  \ \  \ \  \ \  \ \   \label{proof 30}
 \leq C\left( nT\Big(\frac{nr_{1n}r_{2n}}{\lambda^2\bb_n}\Big)p_T\right)^{4(1+\delta)}b^3
  \end{align}
almost surely with respect to $\mathbb{Q}$,    where we used the Bernstein  inequality for the last inequaltity.
Now Lemma \ref{lemma 2}, condition (T1)(6) in Assumption \ref{technical assumptions} and the design of the blocks give (up to some constant factor independent of $m$ and the sample size) 
\begin{align}
    \label{proof 31}
    E_{|S_n} [T^4_{m,n}]&\leq E_{|S_{n}}\Big [\int \Big (\sum_{(i,t)\in\widetilde{B}^-_{m}}Y_{{m}}(Z_{it})\Big )^4 d\mathbb{P}_{S_n,m}\Big ]+\Big ( nT\Big(\frac{nr_{1n}r_{2n}}{\lambda^2\bb_n}\Big)p_T\Big  )^{4(1+\delta)}b^3\beta^{\frac{\delta}{1+\delta}}(q_n) \\
    &\leq E_{|S_{n}}\Big [\int \Big (\sum_{(i,t)\in\widetilde{B}^-_{m}}Y_{{m}}(Z_{it})\Big )^4 d\mathbb{P}_{S_n,m}\Big ]+\big( nT(q_n\log n)^3\big )^{4(1+\delta)}b^{3-2(1+\delta)}\beta^{\frac{\delta}{1+\delta}}(q_n) \\
    &\leq E_{|S_n }\Big [\int \Big (\sum_{(i,t)\in\widetilde{B}^-_{m}}Y_{m}(Z_{it}) \Big )^4 d\mathbb{P}_{S_n,m} \Big ]+o((nT)^4b^6)=:L_{m}+o((nT)^4b^6)
\end{align}
 almost surely with respect to $\mathbb{Q}$. Now we focus on $L_m$. For any given $\bS_n=S_n$, 
  by denoting $E_{\mathbb{G}_{m}}$ as the expectation taken with respect to $\{X_{it}:(i,t)\in\widetilde{B}^c_{m}\}$,  Fubini theorem asserts
\begin{align}
    L_m=&\int E_{\mathbb{G}_{m}}[(\sum_{(i,t)\in\widetilde{B}_{m}}Y_{m}(z_{it}))^4]d\mathbb{P}_{S_n,m} = \sum_{k=1}^{5}L_{mk}, 
   \end{align}
where
    \begin{align}
    L_{m1}& =\int\sum_{(i,t)\in\widetilde{B}_{m}}E_{\mathbb{G}_m}[Y_{m}^4(Z_{it})]d\mathbb{P}_{S_n,m}  \\
    L_{m2} & = \int\sum_{(i,t)\neq(i',t')\in\widetilde{B}_{m}}E_{\mathbb{G}_m}[Y_{m}^2(Z_{it})Y_{m}^2(Z_{i't'})]d\mathbb{P}_{S_n,m}\\
    L_{m3} &=\int\sum_{(i,t)\neq(i',t')\in\widetilde{B}_{m}}E_{\mathbb{G}_m}[Y_{m}^3(Z_{it})Y_{m}(Z_{i't'})]d\mathbb{P}_{S_n,m}\\
   L_{m4} &= \int\sum_{(i,t)\neq(i',t')\neq (i'',t'')\in\widetilde{B}_{m}}E_{\mathbb{G}_m}[Y_{m}^2(Z_{it})Y_{m}(Z_{i't'})Y_{m}(Z_{i''t''})]d\mathbb{P}_{S_n,m}\\
   L_{m4} &= \int\sum_{(i,t)\neq(i',t')\neq (i'',t'')\neq (i'''t''')\in\widetilde{B}_{m}}E_{\mathbb{G}_m}[Y_{m}(Z_{it})Y_{m}(Z_{i't'})Y_{m}(Z_{i''t''})Y_{m}(Z_{i'''t'''})]d\mathbb{P}_{S_n,m}
\end{align}
Therefore, to finish the proof of \textbf{Step 5.3}, it suffices to prove that,  for every $k=1,2,...,5$, 
\begin{align}
    \lim_{n\to\infty}\frac{\sum_{m=2}^{M_n}L_{mk}}{(nT)^4b^6}=0\ \ a.s.-[\mathbb{Q}]
\end{align}


\noindent \textbf{Step 5.3.1} ($\lim_{n\to\infty}\frac{L_{m1}}{(nT)^4b^3}=0$) For any given $m$ and $(i,t)\in\widetilde{B}^-_{m}$ we have 
\begin{align}
Y^4_m(Z_{it})=&\sum_{(r,s)\in \widetilde{B}^c_{m}}H_{n}^4(Z_{it},Z_{rs})+\sum_{(r,s)\neq (r',s')\in\widetilde{B}^c_{m}}H_{n}^2(Z_{it},Z_{rs})H_{n}^2(Z_{it},Z_{r's'})\\
    &+\sum_{(r,s)\neq (r',s')\in\widetilde{B}^c_{m}}H_{n}^3(Z_{it},Z_{rs})H_{n}(Z_{it},Z_{r's'})\\
    &+\sum_{(r,s)\neq (r',s')\neq (r'',s'')\in\widetilde{B}^c_{m}}H_{n}^2(Z_{it},Z_{rs})H_{n}(Z_{it},Z_{r's'})H_{n}(Z_{it},Z_{r''s''})\\
    &+\sum_{(r,s)\neq (r',s')\neq (r'',s'')\neq (r''',s''')\in\widetilde{B}^c_{m}}H_{n}(Z_{it},Z_{rs})H_{n}(Z_{it},Z_{r's'})H_{n}(Z_{it},Z_{r''s''})H_{n}(Z_{it},Z_{r'''s'''})\\
    =&: \mathbf{H}_{1}+\mathbf{H}_{2} +...+\mathbf{H}_{5},
\end{align}
which gives  
\begin{align}
    \label{d34}
L_{m1}=\sum_{k=1}^5 \int E_{\mathbb{G}_m}[\mathbf{H}_{k}]d\mathbb{P}_{S_{n},m}=:\sum_{k=1}^5L_{m1,k}.
\end{align}
\par By repeating the calculation for the  term $I_{k,j,1}$ in {\bf Step 2}, it can be shown that 
\begin{align}
\label{proof 32}
    \lim_{n\to\infty}\frac{|L_{m1,1}|}{\textbf{Card}(\widetilde{B}_{m,S}\cap\Gamma_n)\textbf{Card}(\widetilde{B}_{m,T})(nT)b^3}<\infty,\ \ a.s.-[\mathbb{Q}],
\end{align}
where $A_S$ is the projection of set $A\subset \mathbf{R}_n\times \{1,2,..,T\}$ on the spatial plane $\mathbb{R}^2$. 

\par  For the term $L_{m1,2}$ we have  
\begin{align}
    L_{m1,2}=L_{m1,21}+L_{m1,22}
        \end{align}
        where 
    \begin{align}
         L_{m1,21}& = \sum_{(i,t)\in\widetilde{B}^-_{m}}\sum_{(r,s)\neq (r',s')\in\widetilde{B}^c_{m}\atop ||(r,s)-(r',s')||_{F}\leq q_n}K^2_{0b}(s_{in}-s_{rn},t-s)K^2_{0b}(s_{in}-s_{r'n},t-s') \\
         & ~~~~~~~~~~~~~~~~~~~~~~~~~~~~~~~~~~~~~~~~~~~~~~~~~~~ \times E_{X_{it}}[E_{\mathbb{G}_{m}}[m^2(X_{it},X_{rs})m^2(X_{it},X_{r's'})]]\\   L_{m1,22}& = \sum_{(i,t)\in\widetilde{B}^-_{m}}\sum_{(r,s)\neq (r',s')\in\widetilde{B}^c_{m}\atop ||(r,s)-(r',s')||_{F}> q_n}K^2_{0b}(s_{in}-s_{rn},t-s)K^2_{0b}(s_{in}-s_{r'n},t-s') \\
         & ~~~~~~~~~~~~~~~~~~~~~~~~~~~~~~~~~~~~~~~~~~~~~~~~~~~ \times E_{X_{it}}[E_{\mathbb{G}_{m}}[m^2(X_{it},X_{rs})m^2(X_{it},X_{r's'})]].
\end{align}
The estimate  \eqref{lemma 4 eq-1} in  Lemma \ref{lemma 4}, condition (M1) and  Assumption \ref{moment conditions} give 
\begin{align}
\label{proof 33}
    \lim_{n\to\infty}\frac{L_{m1,21}}{(\textbf{Card}(\widetilde{B}^-_{m,S}\cap\Gamma_n)\frac{n^2q_n^{2}}{\lambda_n^2\bb_n})(\textbf{Card}(\widetilde{B}_{m,T})Tq_T)b^5}<\infty \ \ a.s.-[\mathbb{Q}].
\end{align}

To bound $L_{m1,22}$, we use condition  (M1) in  Assumption \ref{moment conditions} and Lemma \ref{lemma 2} and obtain 
\begin{align}
    L_{m1,22}\leq C_M\sum_{(i,t)\in\widetilde{B}^-_{m}}\sum_{(r,s)\neq (r',s')\in\widetilde{B}^c_{m}\atop ||(r,s)-(r',s')||_{F}> q_n}K^2_{0b}(s_{in}-s_{rn},t-s)K^2_{0b}(s_{in}-s_{r'n},t-s')\beta^{\frac{\delta}{1+\delta}}(q_n), 
\end{align}
 where $C_M$ is a positive constant depending only on moment conditions introduced in Assumption \ref{moment conditions}. Then,  Lemma \ref{lemma 4} shows  that 
\begin{align}
\label{proof 34}
    \lim_{n\to\infty}\frac{L_{m1,22}}{(\textbf{Card}(\widetilde{B}_{m,S}\cap\Gamma_n)\textbf{Card}(\widetilde{B}_{m,T}))(nT)^2b^5}<\infty, \ \ a.s.-[\mathbb{Q}].
\end{align}
Combining \eqref{proof 33} and \eqref{proof 34} asserts that
\begin{align}
    \label{proof 35}
     \lim_{n\to\infty}\frac{L_{m1,2}}{(\textbf{Card}(\widetilde{B}_{m,S}\cap\Gamma_n)\textbf{Card}(\widetilde{B}_{m,T}))(nT)^2b^5}<\infty \ \ a.s.-[\mathbb{Q}]
\end{align}
holds for all $1\leq m\leq M_n$ simultaneously.For the term $L_{m1,3}$ we obtain 
\begin{align}
    &L_{m1,3}\leq  L_{m1,31}+L_{m1,32}
        \end{align}
        where
    \begin{align}
    L_{m1,31}&:=\sum_{(i,t)\in\widetilde{B}^-_{m}}\sum_{(r,s)\neq (r',s')\in\widetilde{B}^c_{m}\atop ||(r,s)-(r',s')||_{F}\leq q_n}K^2_{0b}(s_{in}-s_{rn},t-s)K^2_{0b}(s_{in}-s_{r'n},t-s')\\
   & ~~~~~~~~~~~~~~~~~
   ~~~~~~~~~~~~~~~~~
   ~~~~~~~~~~~~~~~~~
\times    E_{X_{it}}|E_{\mathbb{G}_m}[m^3(X_{it},X_{rs})m(X_{it},X_{r's'})]|\\
   L_{m1,32} &:=\sum_{(i,t)\in\widetilde{B}^-_{m}}\sum_{(r,s)\neq (r',s')\in\widetilde{B}^c_{m}\atop ||(r,s)-(r',s')||_{F}> q_n}K^2_{0b}(s_{in}-s_{rn},t-s)K^2_{0b}(s_{in}-s_{r'n},t-s') \\
   & ~~~~~~~~~~~~~~~~~
   ~~~~~~~~~~~~~~~~~
   ~~~~~~~~~~~~~~~~~
\times    
   E_{X_{it}}|E_{\mathbb{G}_m}[m^3(X_{it},X_{rs})m(X_{it},X_{r's'})]|.
\end{align}
Since $m(X_{it},X_{rs})=X_{it}X_{rs}-E_{\mathbb{P}_{|S_n}}[X_{it}X_{rs}]$, $E_{\mathbb{G}_m}[m(X_{it},X_{rs})]=0$ holds almost surely with respect to $\mathbb{P}_{|S_n}$, which implies 
$$
E_{X_{it}}|E_{\mathbb{G}_m}[m^3(X_{it},X_{rs})m(X_{it},X_{r's'})]|=E_{\mathbb{P}_{|S_n}}|\text{Cov}_{\mathbb{G}_m}(m^3(X_{it},X_{rs}),m(X_{it},X_{r's'}))|,
$$
where the calculation of $\text{Cov}_{\mathbb{G}_m}(m^3(X_{it},X_{rs}),m(X_{it},X_{r's'}))$ should treat $X_{it}$ as a constant. Thus, Lemma \ref{lemma 2} gives 
\begin{align}
    \lim_{n\to\infty}\frac{L_{m1,32}}{L_{m1,22}}=0\ \ a.s.-[\mathbb{Q}],
\end{align}
and, similar to $L_{m1,21}$, we obtain
\begin{align}
    \lim_{n\to\infty}\frac{L_{m1,31}}{(\textbf{Card}(\widetilde{B}^-_{m,S}\cap\Gamma_n)\frac{n^2q_n^{2}}{\lambda_n^2\bb_n})(\textbf{Card}(\widetilde{B}_{m,T})Tq_T)b^5}<\infty\ \ a.s.-[\mathbb{Q}].
\end{align}
Above all, 
\begin{align}
\label{proof 36}
     \lim_{n\to\infty}\frac{L_{m1,3}}{(\textbf{Card}(\widetilde{B}_{m,S}\cap\Gamma_n)\textbf{Card}(\widetilde{B}_{m,T}))(nT)^2b^5}<\infty, \ \ a.s.-[\mathbb{Q}].
\end{align}

For the term  $L_{m1,4}$ in \eqref{d34} recall that  $(r,s)$, $(r',s')$ and $(r'',s'')$ are elements of the  set $\widetilde{B}^c_{m}=\cup_{1\leq m'<m}\widetilde{B}^-_{m'}$ and that $d(\widetilde{B}^-_{m'},\widetilde{B}^-_{m''})\geq 2q_n$ holds for any $1\leq m'\neq m''<m$ which gives the  decomposition 
 \begin{align}
     L_{m1,4}=&\sum_{(i,t)\in\widetilde{B}^-_{m}}\sum_{(r,s)\neq (r',s')\neq (r'',s'')\in\widetilde{B}^c_{m}}K^2_{0ir}K_{0ir'}K_{0ir''}K^2_{0ts}K_{0ts'}K_{0ts''}E_{X_{it}}
     \\
     & ~~~~~~~~~~~~~~~~~~~~~~~~~~~~~~~~~~~~~~~~~~  ~~~~~~~~~~~~~~ \times [E_{\mathbb{G}_m}[m^2(X_{it},X_{rs})m(X_{it},X_{r's'})m(X_{it},X_{r''s''})]]\\
     =&\sum_{(i,t)\in\widetilde{B}^-_{m}}(\sum_{(1)}+\sum_{(2)}+\sum_{(3)})K^2_{0ir}K_{0ir'}K_{0ir''}K^2_{0ts}K_{0ts'}K_{0ts''}      \\
     & ~~~~~~~~~~~~~~~~~~~~~~~~~~~~~~~~~~~~~~~~~~  ~~~~~~~~~~~~~~ \times E_{X_{it}}[E_{\mathbb{G}_m}[m^2(X_{it},X_{rs})m(X_{it},X_{r's'})m(X_{it},X_{r''s''})]]\\
     =&:L_{m1,41}+L_{m1,42}+L_{m1,43}, \label{d35}
 \end{align}
where  $\sum_{(k)}$ denotes the summation of elements in set  \begin{align}
\label{d36}
    \mathcal{D}_{m,k}:= &\{(r,s)\neq (r',s')\neq (r'',s'')\in \widetilde{B}^c_{m}:\\
    &\text{$k$ out of these three locations belong to the same $\widetilde{B}^-_{m'}$, $1\leq m'<m$}\},\ \ 1\leq k\leq 3.
\end{align}
Note that 
\begin{align}
\textbf{Card}(\mathcal{D}_{m,1}) & \leq m(m-1)(m-2)(\frac{nr_{1n}r_{2n}}{\lambda^2_{n}\bb_n}), ~\textbf{Card}(\mathcal{D}_{m,2})\leq m(m-1)(\frac{nr_{1n}r_{2n}}{\lambda^2_{n}\bb_n}),
\\
\textbf{Card}(\mathcal{D}_{m,3})& \leq m(\frac{nr_{1n}r_{2n}}{\lambda^2_{n}\bb_n}).
\end{align}
To derive an estimate for $L_{m1,41}$, we denote by  $A_{S}$ and $A_T$ the projection of $A\subset \mathbb{R}_n\times \{1,2,...,T\}$ onto the spatial plane and the time axis respectively, and   regard $\{\widetilde{B}^-_{m,S}\}$ as a lattice, in which we consider $\widetilde{B}^-_{m,S}$ as a lattice point. Then, the lattice distance, denoted as $d_{\text{lattice}}$, between any two $\widetilde{B}^-_{m,S}$'s is always positive integer. Then, we have the following representation of the set $\mathcal{D}_{m,1,S}$, 
\begin{align}
\label{proof 40}
    &\mathcal{D}_{m,1,S}=\bigcup_{\atop 1\leq m'<m}\bigcup_{\atop 1\leq m''\neq m'<m}\bigcup_{l=1}^{\infty}\bigcup_{D_{m',m''}(l)}\{(\s_{rn},\s_{r'n},\s_{jn}):\s_{rn}\in\widetilde{B}^-_{m',S},\s_{r'n}\in\widetilde{B}^-_{m'',S},\s_{jn}\in D_{m',m''}(l)\},
        \end{align}
        where 
    \begin{align}
 D_{m',m''}(l)=\{\widetilde{B}^-_{j,S}:j\neq m',\ j\neq m'', d_{\text{lattice}}(\widetilde{B}_{j,S}, \{\widetilde{B}_{m,S},\widetilde{B}_{m',S}\} )=l\}. 
\end{align}
Note that $\max_{m',m''}\textbf{Card}(D_{m',m''}(l))\leq 2((2l+1)^2-(2l-1))^2=16l$. Additionally, since $(r,s)$, $(r',s')$ and $(r'',s'')$ located in different $\widetilde{B}^-_{j}$'s whose $j$'s are smaller than $m$, together with the property that $E_{\mathbb{G}_m}[m(X_{it},X_{r''s''})]=0$ holds for every given value of $X_{it}$. Then, for every $X_{it}=x$ and $\bS_n=S_n$, provided that $(r,s)\in\widetilde{B}^-_{m',S}$, $(r',s')\in\widetilde{B}^-_{m'',S}$ and $(r'',s'')\in\widetilde{B}^-_{j,S}$ and $m'\neq m''\neq j$,
\begin{align}
    &|E_{\mathbb{G}_m}[m^2(x,X_{rs})m(x,X_{r's'})m(x,X_{r''s''})]|=|\text{Cov}_{\mathbb{G}_m}(m^2(x,X_{rs})m(x,X_{r's'}),m(x,X_{r''s''}))|\\
    \leq &4 \widetilde{C}(x) \beta^{\frac{\delta}{1+\delta}}(d(\widetilde{B}_{j},\widetilde{B}^-_{m'}\cup \widetilde{B}^-_{m''}))\leq 4 \widetilde{C}(x) \beta^{\frac{\delta}{1+\delta}}(d(\widetilde{B}_{j,S},\widetilde{B}^-_{m',S}\cup \widetilde{B}^-_{m'',S}))\\
    \leq& 4 \widetilde{C}(x) \beta^{\frac{\delta}{1+\delta}}(d_{\text{lattice}}(\widetilde{B}_{j,S},\{\widetilde{B}^-_{m',S}, \widetilde{B}^-_{m'',S}\})q_n),
\end{align}
where
\begin{align}
    \widetilde{C}^{1+\delta}(x)=\max\{&E_{\mathbb{G}_{m}}|m^2(x,X_{rs})m(x,X_{r's'})m(x,X_{r''s''})|^{1+\delta},\\
    &E_{\mathbb{G}_{m}}|m^2(x,X_{rs})m(x,X_{r's'})|^{1+\delta}E_{\mathbb{G}_m}|m(x,X_{r''s''})|^{1+\delta}\}.
\end{align}
By  Assumption \ref{moment conditions} we have  $\widetilde{C}:= \int_{\mathbb{R}}\widetilde{C}(x)d\mathbb{P}_{X_{it}}<\infty$, which gives
\begin{align}
    E_{X_{it}}|E_{\mathbb{G}_m}[m^2(x,X_{rs})m(x,X_{r's'})m(x,X_{r''s''})]|\leq 4 \widetilde{C} \beta^{\frac{\delta}{1+\delta}}(d_{\text{lattice}}(\widetilde{B}_{j,S},\{\widetilde{B}^-_{m',S}, \widetilde{B}^-_{m'',S}\})q_n) .
\end{align}
Therefore we obtain from the   representation \eqref{proof 40}
\begin{align}
    |4L_{m1,41}|= & \Big  |\sum_{s_{in}\in\widetilde{B}^-_{m,S}}\sum_{s_{rn}\in\widetilde{B}^-_{m',S}\atop 1\leq m'\neq m\leq M_n}\sum_{s_{r'n}\in\widetilde{B}^-_{m'',S}\atop 1\leq m''\neq m'\neq m\leq M_n}\sum_{l=1}^{\infty}\sum_{s_{r''n}\in D_{m',m''}(l)}K^2_{0ir}K_{0ir'}K_{0ir''}K^2_{0ts}K_{0ts'}K_{0ts''}
    \\ 
    & ~~~~~~~~~~~~~~~~~~
    ~~~~~~~~~~~~~~~~~~
    \times  E_{\mathbb{G}_m}[m^2(X_{it},X_{rs})m(X_{it},X_{r's'})m(X_{it},X_{r''s''})]  \Big | \\
    \leq &\widetilde{C}C'_{\beta}\Big (\sum_{s_{in}\in\widetilde{B}^-_{m,S}}\sum_{s_{rn}\in\widetilde{B}^-_{m',S}\atop 1\leq m'\neq m\leq M_n}\sum_{s_{r'n}\in\widetilde{B}^-_{m'',S}\atop 1\leq m''\neq m'\neq m\leq M_n}\sum_{l=1}^{\infty}\sum_{s_{r''n}\in D_{m',m''}(l)}K^2_{0ir}K_{0ir'}K_{0ir''}\beta^{\frac{\delta}{1+\delta}}(l)\\
   &~~~~~~~~~~~~~~~~~~ \times \sum_{t
    \in\widetilde{B}^-_{m,T}}\sum_{s\neq s'\neq s'' \in\mathcal{D}_{m,1,T}}K^2_{0ts}K_{0ts'}K_{0ts''}\beta^{\frac{\delta}{1+\delta}}(q_n) \Big ) =:\widetilde{C}C'_{\beta}\mathbf{A}\mathbf{B},
\end{align}
where the first equality is due to the symmetry of $
E_{\mathbb{G}_m}[m^2(X_{it},X_{rs})m(X_{it},X_{r's'})m(X_{it},X_{r''s''})] 
$ and $C_{\beta}'$ is the constant from condition  (T2)(ii) in Assumption \ref{technical assumptions}. Proposition \ref{prop 5} and strong law of large number could show that the following inequality holds almost surely with respect to $\mathbb{Q}$ up to a constant factor independent of $m$ and sample size,
\begin{align}
    \mathbf{A}\leq \overline{E}\textbf{Card}(\widetilde{B}^-_{m,S}\cap\Gamma_n)n^2b^6\Big(\frac{nr_{1n}r_{2n}}{\lambda_n^2\bb_n}\Big)(\sum_{l=1}^{\infty}l\beta^{\frac{\delta}{1+\delta}}(l)),
\end{align}
where $\max_{i,n}||b^{-2}E_{\s_{rn}}[K_{0ir}]||_{\infty}\leq \overline{E}<\infty$. To bound $\mathbf{B}$, a noteworthy fact is that, for any given $v_0$ and $t\in\widetilde{B}^-_{m}$, a necessary condition that $K_{ts}\neq 0$ is $s\in \{v_0-Tb\leq |t-s|\leq v_0+ Tb\}$. Hence, together with point (1) of (T4) of Assumption \ref{technical assumptions}, we obtain
\begin{align}
    \mathbf{B}=o( \textbf{Card}(\widetilde{B}_{m,T})T^3b^3 (nT)^{-1})=O(\textbf{Card}(\widetilde{B}_{m,T})Tb^3),
\end{align}
where the second inequality is due to the assumption without loss of generality, $n=T$. Summarizing,    condition (T4)(3) of Assumption \ref{technical assumptions} gives 
\begin{align}
\label{proof 41}
    \lim_{n\to\infty}\frac{L_{m1,41}}{(\textbf{Card}(\widetilde{B}_{m,S}\cap\Gamma_n)\textbf{Card}(\widetilde{B}_{m,T}))n^2Tb^{7.5}}=0 \ \ a.s.-[\mathbb{Q}].
\end{align}
By similar arguments we can also obtain   
\begin{align}
    \lim_{n\to\infty}\frac{L_{m1,4k}}{(\textbf{Card}(\widetilde{B}_{m,S}\cap\Gamma_n)\textbf{Card}(\widetilde{B}_{m,T}))n^2Tb^{7.5}}=0\ \ \ a.s.-[\mathbb{Q}]
\end{align}
for $k=2,3$, which implies 
\begin{align}
\label{proof 42}
    \lim_{n\to\infty}\frac{L_{m1,4}}{(\textbf{Card}(\widetilde{B}_{m,S}\cap\Gamma_n)\textbf{Card}(\widetilde{B}_{m,T}))n^2Tb^{7.5}}=0 \ \ a.s.-[\mathbb{Q}].
\end{align}
To derive a bound for $L_{m1,5}$ we use the same sets  $\mathcal{D}_{m,k}$'s introduced in \eqref{d36} and the decomposition 
\begin{align}
    L_{m1,5}&=\sum_{(i,t)\in\widetilde{B}^-_{m}} \Big (\sum_{(1)}+\sum_{(2)}+\sum_{(3)}+\sum_{(4)} \Big )K_{0ir}K_{0ir'}K_{0ir''}K_{0ir'''}K_{0ts}K_{0ts'}K_{0ts''}K_{0ts'''} \\
    & ~~~~~~~~~~~~~~
    ~~~~~~~~~~~~~~
    \times    E_{X_{it}}[E_{\mathbb{G}_m}[m(X_{it},X_{rs})m(X_{it},X_{r's'})m(X_{it},X_{r''s''})m(X_{it},X_{r'''s'''})]].
\end{align}
Now similar techniques used to deliver the upper bound of for   $L_{m1,4}$ give  
\begin{align}
    \label{proof 43}
    &\lim_{n\to\infty}\frac{|L_{m1,5}|}{(\textbf{Card}(\widetilde{B}^-_{m,S}\cap\Gamma_n)\textbf{Card}(\widetilde{B}^-_{m,T}))n^3T^2b^{10}(\frac{nr_{1n}r_{2n}}{\lambda_n^2\bb_n})} \\
   &  ~~~~~~~~~~~~~~ ~~~~~~~~~~~~~~  = 
   \lim_{n\to\infty}\frac{|L_{m1,5}|}{(\textbf{Card}(\widetilde{B}^-_{m,S}\cap\Gamma_n)\textbf{Card}(\widetilde{B}^-_{m,T}))n^3T^2b^{8.5}}=0\ \ \ \  a.s.-[\mathbb{Q}],
\end{align}
where the  equality holds because  of  the design of $r_{1n}$ and $r_{2n}$ and condition (T1) of Assumption \ref{technical assumptions}.
 Since $\lim_{n \to \infty }\frac{\sigma_n^2}{(nT)^4b^6}\in (0,\infty)$ holds almost surely with respect to $\mathbb{Q}$, we obtain \eqref{proof 32}, \eqref{proof 35}, \eqref{proof 36}, \eqref{proof 42} and \eqref{proof 43}  
\begin{align}
    \label{proof 44}
    \lim_{n\to\infty}\frac{L_{m1}}{\textbf{Card}(\widetilde{B}^-_{m,S}\cap\Gamma_n)\textbf{Card}(\widetilde{B}^-_{m,T})\eta_n}=0\ \ a.s.-[\mathbb{Q}], 
        \end{align}
        where
$    
    \eta_n=\max\big \{(nT)b^3, (nT)^2b^5, n^3T^2b^{8.5} \big \}, 
$
which asserts
\begin{align}
    \label{Proof F1}
    \lim_{n\to\infty}\frac{\sum_{m=2}^{M_n}L_{m1}}{(nT)^4b^6}=0 \ \ a.s.-[\mathbb{Q}].
\end{align}

\noindent \textbf{Step 5.3.2} ($\lim_{n\to\infty}\frac{\sum_{m=2}^{M_n}L_{m2=k}}{(nT)^4b^6}=0\ \ a.s.-\mathbb{Q}$, $k=2,3$) We first focus on the case $k=2$ using Cauchy-Schwartz inequality, which gives 
\begin{align}
    L_{m2}=&\sum_{(i,t)\neq (i',t')}\int E[Y^2(z_{it})Y^2(z_{i't'})]d\mathbb{P}_{S_n,m}\\
    \leq & \sum_{(i,t)\neq (i',t')}\left(\int E[Y^4(z_{it})]d\mathbb{P}_{S_n,m}\right)^\frac{1}{2}\left(\int E[Y^4(z_{i't'})]d\mathbb{P}_{S_n,m}\right)^{\frac{1}{2}}
\end{align}
By similar arguments as  used for  proving \eqref{proof 32}, \eqref{proof 35}, \eqref{proof 36}, \eqref{proof 42} and \eqref{proof 43}, we obtain, for any given $2\leq m\leq M_n$ and $(i,t)\in\widetilde{B}^-_{m}$, 
\begin{align}
    \lim_{n\to\infty}\frac{\int E[Y^4(z_{it})]d\mathbb{P}_{S_n,m}}{\eta_n}=0 \ \ \ \ \ a.s.-[\mathbb{Q}],
\end{align}
where $\eta_n$ is defined in \eqref{proof 44}. This implies
\begin{align}
    \lim_{n\to\infty}\frac{\sum_{m=2}^{M_n}L_{m2}}{nT\eta_n(\frac{nr_{1n}r_{2n}}{\lambda_n^2\bb_n})p_T}= \lim_{n\to\infty}\frac{\sum_{m=2}^{M_n}L_{m2}}{nT\eta_nb^{-1.5}}=0 \ \ \ a.s.-[\mathbb{Q}],
\end{align}
where the first equality is based on the design of spatial blocks and conditions (T1) and  (T4)(3) in Assumption \ref{technical assumptions}. Together with   Assumption \ref{technical assumptions} we obtain
\begin{align}
    \label{proof F2}
     \lim_{n\to\infty}\frac{\sum_{m=2}^{M_n}L_{m2}}{(nT)^4b^6}=0 \ \ \ a.s.-[\mathbb{Q}]. 
\end{align}
The proof of the corresponding statement for $L_{m3}$ is the nearly the same and  omitted  for the sake of  brevity.
\smallskip

\noindent \textbf{Step 5.3.3} ($ \lim_{n\to\infty}\frac{\sum_{m=2}^{M_n}L_{m4}}{(nT)^4b^6}=0 \ \ \ a.s.-[\mathbb{Q}]$) We use the decomposition 
\begin{align}
    L_{m4}
    =&L_{m4,1}+L_{m4,2} 
       \end{align}
       with 
   \begin{align}
     L_{m4,1} &= 
    \sum_{(i,t)\neq (i',t')\neq (i'',t'')\atop d(\{(i,t),(i',t')\},(i'',t''))> q_n}\int E_{\mathbb{G}_m}[ Y^2_{m}(Z_{it})Y_{m}(Z_{i't'})Y(Z_{i''t''})    ]d\mathbb{P}_{S_n,m}\\
    L_{m4,2}  & = \sum_{(i,t)\neq (i',t')\neq (i'',t'')\atop d(\{(i,t),(i',t')\},(i'',t''))\leq q_n}\int E_{\mathbb{G}_m}[ Y_m^2(Z_{it})Y_m(Z_{i't'})Y_m(Z_{i''t''})    ]d\mathbb{P}_{S_n,m}
\end{align}
For the term $L_{m4,1}$ we note that, for any given $\bS_n=S_n$,   Fubini's theorem gives 
\begin{align}
   | L_{m4,1}|&\leq \sum_{(i,t)\neq (i',t')\neq (i'',t'')\atop d(\{(i,t),(i',t')\},(i'',t''))> q_n}E_{\mathbb{G}_m}|\int Y_m^2(Z_{it})Y_m(Z_{i't'})Y_m(Z_{i''t''})    d\mathbb{P}_{it,i't',i''t''}|\\
    & =\sum_{(i,t)\neq (i',t')\neq (i'',t'')\atop d(\{(i,t),(i',t')\},(i'',t''))> q_n}E_{\mathbb{G}_m}\Big|\int Y_m^2(Z_{it})Y_m(Z_{i't'})Y_m(Z_{i''t''})    d\mathbb{P}_{it,i't',i''t''}\\
    & \ \ \ \ \ \ \ \ \ \ \ \ \ \ \ \ \ \ \ \ \ \ \ \ \ 
    -\int Y_m^2(Z_{it})Y_m(Z_{i't'})    d\mathbb{P}_{it,i't'}\int Y_m(Z_{i''t''})    d\mathbb{P}_{i''t''}\Big|,
\end{align}
where $\mathbb{P}_{(i_1t_1,...,i_kt_k)}$ denotes  the distribution of the  vector $(X_{i_1t_1},\dots,X_{i_kt_k})\in\mathbb{R}^k$ and the second equality follows from the fact that $\int Y(Z_{i''t''})d\mathbb{P}_{i''t''}=0$  (almost surely with respect to $\mathbb{P}_{|S_n}$). Fubini's theorem  shows 
\begin{align}
&E_{\mathbb{G}_m} \Big [\int |Y_m^2(Z_{it})Y_m(Z_{i't'})Y_m(Z_{i''t''})|^{1+\delta}    d\mathbb{P}_{it,i't',i''t''} \Big ]\\
=&E_{\mathbb{G}_m}\Big[\int  \Big|\Big(\sum_{(r,s)\in\widetilde{B}^c_m}K_{0ir}K_{0ts}m_{it,rs} \Big)^2 \sum_{(r',s')\in\widetilde{B}^c_m}K_{0i'r'}K_{0t's'}m_{i't',r's'}\\
&\ \ \ \ \ \ \ \ \ \ \ \ \ \ \ \ \ \ \ \times \sum_{(r'',s'')\in\widetilde{B}^c_m}K_{0i''r''}K_{0t''s''}m_{i''t'',r''s''} \Big |^{1+\delta}    d\mathbb{P}_{it,i't',i''t''}\Big]\\
=&E_{\mathbb{G}_m}\Big[\int \Big |\sum_{(r_1,s_1), (r_2,s_2),(r',s')(r'',s'')\in\widetilde{B}^c_m}K_{0ir_1}K_{0ts_1}K_{0ir_2}K_{0ts_2}K_{0i'r'}K_{0t's'}K_{0i''r''}K_{0t''s''} \\
& ~~~~~~~~~~~~~~~
\times m_{it,r_1s_1}m_{it,r_2s_2}m_{i't',r's'}m_{i''t'',r''s''}\Big |^{1+\delta}    d\mathbb{P}_{it,i't',i''t''}\Big]
\end{align}
where 
$m_{it,rs}=m(X_{it},X_{rs})$.
We also denote $\widetilde{B}^c_{m,U}$ as the set of spatial temporal location $(r,s)$ such that $(r,s)\in\widetilde{B}^c_{m}$ and $(r,s)$ is not ruled out by the compactness of the support of kernel functions. Thus, Proposition \ref{prop 5} implies that 
$$
\lim_{n\to\infty}\frac{\textbf{Card}(\widetilde{B}^c_{m,U})}{nTb}<\infty \ \ \ a.s.-[\mathbb{Q}]
$$ for all $m$, and  
 using the basic $C_r$-inequality  yields
\begin{align}
\label{proof 45}
    &(\textbf{Card}(\widetilde{B}^c_{m,U}))^{-4(1+\delta)} E_{\mathbb{G}_m} \Big [\int |Y_m^2(Z_{it})Y_m(Z_{i't'})Y_m(Z_{i''t''})|^{1+\delta}    d\mathbb{P}_{it,i't',i''t''} \Big ] \\
    & ~~~  \leq \frac{C_M}{(\textbf{Card}(\widetilde{B}^c_{m,U}))^4}\sum_{(r_1,s_1), (r_2,s_2),(r',s')(r'',s'')\in\widetilde{B}^c_{m,U}}|K_{0ir_1}K_{0ts_1}K_{0ir_2}K_{0ts_2}K_{0i'r'}K_{0t's'}K_{0i''r''}K_{0t''s''}|^{1+\delta}  \\
    & = O\Big (
    {\left(\begin{array}{c}
          4 \\
          \mu  \end{array}\right)^{-1}}\sum_{(r_1,s_1)\neq (r_2,s_2)\neq (r',s')\neq (r'',s'')\in\widetilde{B}^c_{m,U}}|K_{0ir_1}K_{0ts_1}K_{0ir_2}K_{0ts_2}K_{0i'r'}K_{0t's'}K_{0i''r''}K_{0t''s''}|^{1+\delta}\big )  
\end{align}
almost surely with respect to $\mathbb{Q}$, 
where  $\mu =\textbf{Card}(\widetilde{B}^c_{m,U})$,  $C_M$ is a  constant determined by the moment conditions in Assumption \ref{moment conditions}, 
and the last equality follows by   considering the sum   as a $V$-statistic of order $4$ and using Proposition \ref{prop 5} and the  non-negativity of the  kernel functions. 
Since $\{\s_{in}\}$ is iid and kernel functions are uniformly bounded, Proposition \ref{prop 5} and Lemma \ref{lemma 1} imply
\begin{align}
\label{proof 46}
    \lim_{n\to\infty}\frac{(\textbf{Card}(\widetilde{B}^c_{m,U}))^{-4(1+\delta)} }{b^8}  E_{\mathbb{G}_m} \Big [\int |Y_m^2(Z_{it})Y_m(Z_{i't'})Y_m(Z_{i''t''})|^{1+\delta}    d\mathbb{P}_{it,i't',i''t''} \Big ]  
    <\infty\ \ \ a.s.-[\mathbb{Q}],
\end{align}
for any given $(i,t)$, $(i',t')$ and $(i'',t'')$. We thus obtain 
\begin{align}
   \lim_{n\to\infty}\frac{ (E_{\mathbb{G}_m}[\int |Y_m^2(Z_{it})Y_m(Z_{i't'})Y_m(Z_{i''t''})|^{1+\delta}    d\mathbb{P}_{it,i't',i''t''}])^{\frac{1}{1+\delta}}}{(nT)^4b^{4+\frac{8}{1+\delta}}}<\infty\ \ a.s.-[\mathbb{Q}].
\end{align}
Similarly, we obtain 
\begin{align}
    \lim_{n\to\infty}\frac{(E_{\mathbb{G}_m}[\int |Y_m^2(Z_{it})Y_m(Z_{i't'})Y_m(Z_{i''t''})|^{1+\delta}    d\mathbb{P}_{it,i't'}d\mathbb{P}_{i''t''}])^{\frac{1}{1+\delta}}}{(nT)^4b^{4+\frac{8}{1+\delta}}}<\infty \ \ \ a.s.-[\mathbb{Q}].
\end{align}
Then, Lemma \ref{lemma 3} yields 
\begin{align}
   \lim_{n\to\infty} \frac{ |L_{m4,1}| }{\textbf{Card}(\widetilde{B}^-_{m,S})\textbf{Card}(\widetilde{B}^-_{m,T})(nT)^4b^{4+\frac{8}{1+\delta}}(\frac{nr_{1n}r_{2n}}{\lambda^2_n\bb_n}p_T)^2\beta^{\frac{\delta}{1+\delta}}(q_n)}<\infty\ \ a.s.-[\mathbb{Q}],
\end{align}
and the design of spatial-temporal blocks, and conditions (T1) and (T4)(3) in Assumption \ref{technical assumptions},  yield
\begin{align}
    \label{proof 47}
\lim_{n\to\infty} \frac{ |\sum_{m=2}^{M_n}L_{m4,1}| }{(nT)^4b^{10}}<\infty\ \ a.s.-[\mathbb{Q}].
\end{align}

To derive a bound for the term $L_{m4,2}$) we define  $G_{m,it}=\int E_{\mathbb{G}_m}[Y^4(z_{it})]d\mathbb{P}_{S_n,m}=E_{\mathbb{G}_m}[\int Y^4(z_{it})d\mathbb{P}_{it}]$, and a repetitive application of Cauchy-Schwartz inequality asserts
\begin{align}
    \sum_{m=2}^{M_n}|L_{m4,2}|\leq \sum_{m=2}^{M_n}\sum_{(i,t)\neq (i',t')\neq (i'',t'')\atop d(\{(i,t),(i',t')\},(i'',t''))\leq q_n}G_{m,it}^\frac{1}{2}G^{\frac{1}{4}}_{m,i't'}G^{\frac{1}{4}}_{m,i''t''}. 
\end{align}
For any given $m$ and $(i,t)\in\widetilde{B}^-_{m}$,
\begin{align}
    \lim_{n\to\infty}\frac{G_{m,it}}{\eta_n}<\infty \ \ a.s.-[\mathbb{Q}],
\end{align}
where $\eta_n=\max\{(nT)b^3, (nT)^2b^5, n^3T^2b^{8.5}\}$ has been introduced in \eqref{proof 44}. Now similar arguments as given in the proof of Proposition \ref{prop 10}, conditions  (T4)(3) and (T1) in Assumption \ref{technical assumptions}, and  Proposition \ref{prop 5} yield
\begin{align}
    \label{proof F3}
\lim_{n\to\infty}\frac{\sum_{m=2}^{M_n}|L_{m4,2}|}{(nT)^4b^6}=0\ \ \ a.s.-[\mathbb{Q}].
\end{align}
\noindent \textbf{Step 5.3.4} ($ \lim_{n\to\infty}\frac{\sum_{m=2}^{M_n}L_{m5}}{(nT)^4b^6}=0 \ \ \ a.s.-[\mathbb{Q}]$) As the proof of \textbf{Step 5.3.3} is very similar to \textbf{Step 5.3.3},  we only sketch it  and highlight the key steps. We have 
\begin{align}
  &|L_{m5}|\leq {L}_{m5,1}+{L}_{m5,2}
  \end{align}
  where
    \begin{align}
{L}_{m5,1} &=\sum_{(i,t)\neq(i',t')\neq (i'',t'')\neq (i'''t''')\in\widetilde{B}_{m}\atop d(\{(i,t),(i',t'),(i'',t'')\}, (i''',t'''))> q_n}E_{\mathbb{G}_m}\left|\int[Y_m(Z_{it})Y_m(Z_{i't'})Y_m(Z_{i''t''})Y_m(Z_{i'''t'''})]d\mathbb{P}_{it,i't',i''t'',i'''t'''}\right|\\
 {L}_{m5,2} &= \sum_{(i,t)\neq(i',t')\neq (i'',t'')\neq (i'''t''')\in\widetilde{B}_{m}\atop d(\{(i,t),(i',t'),(i'',t'')\}, (i''',t'''))\leq  q_n}E_{\mathbb{G}_m}\left|\int[Y_m(Z_{it})Y_m(Z_{i't'})Y_m(Z_{i''t''})Y_m(Z_{i'''t'''})]d\mathbb{P}_{it,i't',i''t'',i'''t'''}\right|_{m5,2}
\end{align}
By  the same techniques used in  \textbf{Step 5.3.3}, we obtain
\begin{align}
    &\lim_{n\to\infty}\frac{ (E_{\mathbb{G}_m}[\int |Y_m(Z_{it})Y_m(Z_{i't'})Y_m(Z_{i''t''})Y_m(Z_{i'''t'''})|^{1+\delta}    d\mathbb{P}_{it,i't',i''t'',i'''t'''}])^{\frac{1}{1+\delta}}}{(nT)^4b^{12}}<\infty\ \ a.s.-[\mathbb{Q}] , \\
    & \lim_{n\to\infty}\frac{ (E_{\mathbb{G}_m}[\int |Y_m(Z_{it})Y_m(Z_{i't'})Y_m(Z_{i''t''})Y_m(Z_{i'''t'''})|^{1+\delta}    d\mathbb{P}_{it,i't',i''t''}d\mathbb{P}_{i'''t'''}])^{\frac{1}{1+\delta}}}{(nT)^4b^{12}}<\infty\ \ a.s.-[\mathbb{Q}] ~, 
\end{align}
and  with Lemma \ref{lemma 2}, we can show
\begin{align}
\label{proof 48}
    \lim_{n\to\infty}\frac{\sum_{m=2}^{M_n}{L}_{m5,2}}{(nT)^4b^{9}}=0\ \ a.s.-[\mathbb{Q}].
\end{align}
Similar as given in  \textbf{Step 5.3.3}, condition  (T4)(3) and (T1) in  Assumption \ref{technical assumptions} yield
\begin{align}
    \label{proof 49}
     \lim_{n\to\infty}\frac{\sum_{m=2}^{M_n}{L}_{m5,1}}{(nT)^4b^{6}}=0\ \ a.s.-[\mathbb{Q}], 
\end{align}
and a combination of \eqref{proof 48} and \eqref{proof 49} shows 
\begin{align}
    \label{proof C7}
     \lim_{n\to\infty}\frac{\sum_{m=2}^{M_n}\widetilde{L}_{m5}}{(nT)^4b^{6}}=0\ \ a.s.-[\mathbb{Q}].
\end{align}

\par Now, by collecting the results obtained in \textbf{Step 5.3.1}-\textbf{Step 5.3.4}, we have finished the proof of \textbf{Step 5.3}. Therefore the assertion of Theorem  \ref{Theorem CLT U-Statistic}  follows by an application of  Theorem \ref{theorem CLT 2}.
\end{proof}



\subsubsection{\texorpdfstring{Proof of Theorem \ref{Theorem CLT Covariance Estimator}}{Proof of Theorem 3.1}} 
\label{seca12}
The proof is a direct application of Theorem \ref{Theorem CLT U-Statistic}
but the treatment of  the bias term is a bit different from the techniques used for usual kernel smoother since it still contains randomness (note that we observe the process at random locations   $\{\s_{in}\}$). Recall the notation of $U_n$ in \eqref{d15a}, define $\widetilde{U}_n=\frac{1}{(nT)^2b^3}U_n$, then Theorem \ref{Theorem CLT U-Statistic} yields
\begin{align}
\label{Theorem CLT Covariance Estimator eq 1}
   nTb^{1.5}\widetilde{U}_{n} \xrightarrow{d} N(0,\sigma^2)\ \ a.s.-[\mathbb{Q}],
\end{align}
where  $
\sigma^2 $ is defined in \eqref{d7b}. 
Therefore we obtain 
\begin{align}
\label{Theorem CLT Covariance Estimator eq 2}
&nTb^{1.5}\Big (\hat C(\h_0,v_0)-C(\h_0,v_0)-\frac{\textbf{Bias}(\h_0,v_0)}{2\mf A_1}\Big )\\
=&\frac{nTb^{1.5}\widetilde{U}_n}{2\mf A_1(\widetilde{K}_n(\h_0,v_0)/2\mf A_1)}+\frac{nTb^{1.5}\mf B_n}{2\mf A_1(\widetilde{K}_n(\h_0,v_0)/2\mf A_1)}-\frac{nTb^{1.5}\textbf{Bias}(\h_0,v_0)}{2\mf A_1}
      \end{align} 
    where 
      \begin{align}   
      \widetilde{K}_n (\h_0,v_0) &  = \frac{1}{(nT)^2b^3}\sum_{1\leq i\neq i'\leq n}\sum_{1\leq t\neq t'\leq T}K_{0b}(\s_{in}-\s_{i'n},t-t') \\
      \mf B_n& = \frac{1}{(nT)^2b^3}\sum_{1\leq t\neq t'\leq T}\sum_{1\leq i\neq i'\leq n}K_{0b}(\s_{in}-\s_{i'n},t-t')(E_{|S_n}[X_{it}X_{i't'}]-C(\h_0,v_0)). 
\end{align}
The strong law of large number implies that $\widetilde{K}_{n}(\h_0,v_0)\rightarrow 2\mf A_1\ \ a.s.-[\mathbb{Q}]$, which together with \eqref{Theorem CLT Covariance Estimator eq 2} asserts
\begin{align}
\label{Theorem CLT Covariance Estimator eq 2+}
    &nTb^{1.5}\Big (\hat C(\h_0,v_0)-C(\h_0,v_0)-\frac{\textbf{Bias}(\h_0,v_0)}{2\mf A_1}\Big )
    =\frac{nTb^{1.5}\widetilde{U}_n}{2\mf A_1}+\frac{nTb^{1.5}(\mf B_n-\textbf{Bias}(\h_0,v_0))}{2\mf A_1} + o_(1) ~~\text{a.s.}
\end{align}
By \eqref{Theorem CLT Covariance Estimator eq 1} we have 
\begin{align}
    \frac{nTb^{1.5}\widetilde{U}_{n}}{2\mf A_1}\xrightarrow{d} N(0,\frac{\sigma^2}{4\mf A_1^2})\ \ \  \ a.s.-[\mathbb{Q}],
\end{align}
and we only need to focus on the second term on the right hand side of \eqref{Theorem CLT Covariance Estimator eq 2+}. Recall that $\Lambda^{-1}_n\s_{in}=\x_{in}$, where $\x_{in}\in\mathbf{D}$ and $\Lambda_n=\lambda_n\text{diag}(1,\bb_n)$. Observing  condition (M1) in Assumption \ref{moment conditions} and \eqref{covariance}, we have
\begin{align}
    \mf B_n=& \frac{1}{(nT)^2b^3}\sum_{1\leq t\neq t'\leq T}\sum_{1\leq i\neq i'\leq n}K_{0b}(\s_{in}-\s_{i'n},t-t')(\text{Cov}_{|S_n}(X_{it},X_{i't'})-C(\h_0,v_0))\\
    =&\frac{1}{(nT)^2b^3}\sum_{1\leq t\neq t'\leq T}\sum_{1\leq i\neq i'\leq n}K_{0ii'}K_{0tt'}\Big \{ C \big (\s_{in}-\s_{i'n},|t-t'| \Big )-C(\h_0,v_0) \Big \} \\
    =&\frac{1}{n^2b^2}\sum_{1\leq i\neq i'\leq n}K_{0ii'}\Big (  \frac{1}{T^2b}\sum_{1\leq t\neq t'\leq T}K_{0tt'}\Big  \{ C \Big  (\s_{in}-\s_{i'n},|t-t'| \Big )-C(\s_{in}-\s_{i'n},v_0) \Big \} \Big )  \\
    &+\frac{1}{T^2b}\sum_{1\leq t\neq t'\leq T}K_{0tt'}\Big(\frac{1}{n^2b^2}\sum_{1\leq i\neq i'\leq n}K_{0ii'} \Big \{ C \Big  (\s_{in}-\s_{i'n},v_0 \Big )-C(\h_0,v_0) \Big \} \Big)\\
    =&:\mf B_{1n}+\mf B_{2n}
\end{align}
where $K_{0tt'}$
 and $K_{0ii'}$ are defined in \eqref{eq definition of K0tt'} and \eqref{eq definition of K0ii'}, respectively.
We now consider the two  terms 
$\mf B_{1n}$ and $\mf B_{2n}$ separately.
\smallskip

 For the term $\mf B_{2n}$,  \cref{smoothness},the notation introduced in \eqref{eq abbreviation} and a second-order Taylor expansion yields
\begin{align}
    &C(\s_{in}-\s_{i'n},v_0)-C(\h_0,v_0)=C(\Lambda_n(\x_{in}-\x_{i'n}),v_0)-C(\Lambda_n \h_{0n},v_0)\\
    =&(\lambda_n)^l\sum_{l=1}^2\partial_{1}^l C(\h_0,v_0)((\x_{in,1}-\x_{i'n,1})-h_{0,1n})^l+(\lambda_n\bb_n)^l\sum_{l=1}^2\partial_{2}^l C(\h_0,v_0)((\x_{in,2}-\x_{i'n,2})-h_{0,2n})^l\\
    &\ \ \ \ \ \ \ +\lambda_n^2\bb_n\partial_{12}C (\theta_{0ii'})\prod_{k=1}^2((\x_{in,k}-\x_{i'n,k})-h_{0,nk}),
\end{align}
where $\partial^l_{k}C$ denotes the $l$-th partial derivative of the covariance function $C$ with respect to the $k$-th coordinate, $\partial_{k,j}C$  denotes the  second-order partial derivative of $C$  with respect to the $k$-th and $j$-th coordinate and
$\theta_{0ii'}=\Lambda_n(\lambda(\x_{in}-\x_{i'n})+(1-\lambda)\h_{0,n})$, $\lambda\in (0,1)$. Furthermore, it follows from \eqref{eq en} and \cref{as regularity} that 
\begin{align}
   &C(\s_{in}-\s_{i'n},v_0)-C(\h_0,v_0)\\
   =&\lambda_n \partial_1C(\h_0,v_0)(\x_{in,1}-\x_{i'n,1}-h_{0,1n})+ \lambda_n\bb_n \partial_2C(\h_0,v_0)(\x_{in,2}-\x_{i'n,2}-h_{0,2n})+O(b^2)\\
   =&\sum_{k=1}^2\partial_kC(\h_0,v_0)(\s_{in,k}-\s_{i'n,k}-h_{0k})+O(b^2)
\end{align}
where the remainder  does not depend on  $i$, $i'$ and $(\h_0,v_0)$.
Therefore,
\begin{align}
   \mf B_{2n} & =\frac{1}{T^2b}\sum_{1\leq t\neq t'\leq T}K_{0tt'}\Big(\sum_{k=1}^2\frac{1}{n^2b^2}\sum_{1\leq i\neq i'\leq n}K_{0ii'}\partial_{k}C(\h_0,v_0)((\s_{in,k}-\s_{i'n,k})-h_{0k})\Big)+O(b^2)\\
    & =\sum_{k=1}^2\partial_{k}C(\h_0,v_0)\Big\{\frac{1}{T^2b}\sum_{1\leq t\neq t'\leq T}K_{0tt'}\Big(\frac{1}{n^2b^2}\sum_{1\leq i\neq i'\leq n}K_{0ii}
    ((\s_{in,k}-\s_{i'n,k})-h_{0k})\Big)\Big\}+O(b^2).\label{Theorem CLT Covariance Estimator eq 7}
\end{align}

To derive a corresponding expansion for the term $\mf B_{1n}$, first note that 
\begin{align}
    &C(\s_{in}-\s_{i'n},|t-t'|)-C(\s_{in}-\s_{i'n},v_0)=C \Big (\s_{in}-\s_{i'n},T\Big  (\frac{|t-t'|}{T}\Big )\Big )-C(\s_{in}-\s_{i'n},Tv_{0T})\\
    &=T\partial_3 C(\s_{in}-\s_{i'n},v_0)\Big (\frac{|t-t'|}{T}-v_{0T}\Big  )+T^2\partial_3^2(\s_{in}-\s_{i'n},\theta_{0tt'})\Big (\frac{|t-t'|}{T}-v_{0T}\Big )^2,
\end{align}
where, similar to the previous $\theta_{0ii'}$, $\theta_{0,tt'}$ is a the intermediate point between $u_{0}$ and $|t-t'|$. Furthermore, \cref{as regularity} yields 
\begin{align}
\label{Theorem CLT Covariance Estimator eq 3}
    C(\s_{in}-\s_{i'n},|t-t'|)-C(\s_{in}-\s_{i'n},v_0)=T\partial_3 C(\s_{in}-\s_{i'n},v_0)(|t-t'|/T-v_{0T})+O(b^2),
\end{align}
where the absolute value of remainder has bound independent of $t,t'$ and $(\h_0,v_0)$. Additionally, mean-value theorem  implies
\begin{align}
\label{Theorem CLT Covariance Estimator eq 4}
    &\partial_3C(\s_{in}-\s_{i'n},v_0)-\partial_3C(\h_0,v_0)\\
    =&\lambda_n\partial_{31} C(\xi,v_0)(\x_{in,1}-\x_{i'n,1}-h_{0,1n})+\lambda_n\bb_n\partial_{32}C(\xi,v_0)(\x_{in,2}-\x_{i'n,2}-h_{0,2n}).
\end{align}
Together with Assumption \ref{as regularity}, substituting \eqref{Theorem CLT Covariance Estimator eq 4} to \eqref{Theorem CLT Covariance Estimator eq 3} yields
\begin{align}
    \label{Theorem CLT Covariance Estimator eq 5}
    &C(\s_{in}-\s_{i'n},|t-t'|)-C(\s_{in}-\s_{i'n},v_0)\\
    =&T\{\lambda_n\partial_{31} C(\xi,v_0)(\x_{in,1}-\x_{i'n,1}-h_{0,1n})+\lambda_n\bb_n\partial_{32}C(\xi,v_0)(\x_{in,2}-\x_{i'n,2}-h_{0,2n})\}(|t-t'|/T-v_{0T})\\
    &+T\partial_3(\h_0,v_0)(|t-t'|/T-v_{0T})+O(b^2)\\
    =&\partial_3C(\h_0,v_0)(|t-t'|-v_0)+O(b^2),
\end{align}
where the absolute value of remainder has bound independent of $t,t'$ and $(\h_0,v_0)$. Above all, we obtain 
\begin{align}
    \label{Theorem CLT Covariance Estimator eq 6}
    \mf B_{1n}=\frac{1}{n^2b^2}\sum_{1\leq i\neq i'\leq n}K_{0ii'}\Big(\frac{1}{T^2b}\sum_{1\leq t\neq t'\leq T}K_{0tt'}\partial_3C(\h_0,v_0)(|t-t'|-v_0)\Big)+O(b^2).
\end{align}

Combining \eqref{Theorem CLT Covariance Estimator eq 6} and \eqref{Theorem CLT Covariance Estimator eq 7} yields
\begin{align}
\label{Theorem CLT Covariance Estimator eq 8}
\mf B_n=& \mf B_{1n}+\mf B_{2n}\\
=&\frac{1}{n^2b^2}\sum_{1\leq i\neq i'\leq n}K_{0ii'}\Big(\frac{1}{T^2b}\sum_{1\leq t\neq t'\leq T}K_{0tt'}\partial_3C(\h_0,v_0)(|t-t'|-v_0)\Big)\\
&+\sum_{k=1}^2\partial_{k}C(\h_0,v_0)\Big\{\frac{1}{T^2b}\sum_{1\leq t\neq t'\leq T}K_{0tt'}\Big(\frac{1}{n^2b^2}\sum_{1\leq i\neq i'\leq n}K_{0ii}
    ((\s_{in,k}-\s_{i'n,k})-h_{0k})\Big)\Big\}+O(b^2)\\
    =&\textbf{Bias}(\h_0,v_0)+O(b^2)
\end{align}
where the remainder is  independent of $i,i',t,t'$ and $(\h_0,v_0)$. Together with condition $nTb^{3.5}\to 0$,   the proof is completed.

\subsection{\texorpdfstring{Proof of Theorem \ref{Theorem Feasible Inference}}{Proof of Theorem 3.2}} 
According to \eqref{Theorem CLT Covariance Estimator eq 2},  \eqref{Theorem CLT Covariance Estimator eq 8} and Slutsky's Theorem, we only need to prove that, when $nTb^{2.5}\to 0$ holds, we have
\begin{align}
    nTb^{1.5}\mf B_{kn}\to 0 \ \ a.s-[\mathbb{Q}],\ \ k=1,2,
\end{align}
where $\mf B_{kn}$'s are introduced in \eqref{Theorem CLT Covariance Estimator eq 8}. 

Note that 
\begin{align}
    nTb^{1.5}\mf B_{1n}=nTb^{1.5}\frac{1}{n^2b^2}\sum_{1\leq i\neq i'\leq n}K_{0ii'}\Big(\frac{1}{T^2b}\sum_{1\leq t\neq t'\leq T}K_{0tt'}T\partial_3C(\h_0,v_0)\Big (\frac{|t-t'|-v_0  )}{T}\Big )\Big).
\end{align}
(E3) of \cref{as regularity} indicates that $T\partial_3C(\h_0,v_j)=O(1)$. Together with the fact that $\frac{|t-t'|-v_0)}{T}]=O(b)$. Some simple algebra yields
\begin{align}
    \label{Theorem Feasible Inference eq 1}
     nTb^{1.5}\mf B_{1n}=nTb^{2.5}\to 0\ \ a.s-[\mathbb{Q}].
\end{align}

Similarly, we have 
\begin{align}
    nTb^{1.5}\mf B_{1n}&= nTb^{1.5}\sum_{k=1}^2\partial_{k}C(\h_0,v_0)\Big\{\frac{1}{T^2b}\sum_{1\leq t\neq t'\leq T}K_{0tt'}\Big(\frac{1}{n^2b^2}\sum_{1\leq i\neq i'\leq n}K_{0ii}
    ((\s_{in,k}-\s_{i'n,k})-h_{0k})\Big)\Big\}\\
    &\lesssim nTb^{1.5}(\lambda_n \sum_{k=1}^2\partial_{k}C(\h_0,v_0))b\lesssim nTb^{2.5}\to 0\ \ a.s.-[\mathbb{Q}].  \label{Theorem Feasible Inference eq 2}
\end{align}

Combining \eqref{Theorem Feasible Inference eq 1} and \eqref{Theorem Feasible Inference eq 2} finished the proof.

\subsection{\texorpdfstring{ Proof of Theorem \ref{Theorem MCLT Covariance Estimator}}{Proof of Theorem 3.3}} 
We will use Cramér–Wold device to prove this theorem. 
\par \textbf{Step 1} (Covariance) We first focus on investigating 
\begin{align}
    (nT)^2b^3\text{Cov}_{|S_n}(\hat{C}(\h_{1},v_1),\hat{C}(\h_{2},v_2))
\end{align}
for any given $(\h_1,v_1)\neq (\h_2,v_2)$. According to the definition of $\hat{C}(\h,v)$, some simple algebra yields
\begin{align}
    &(nT)^2b^3\text{Cov}_{|S_n}(\hat{C}(\h_{1},v_1),\hat{C}(\h_{2},v_2))\\
    = D^{-1}_{K_n}\frac{1}{(nT)^2b^3}&\sum_{s_{in}\neq s_{i'n}\atop t\neq t'}\sum_{s_{rn}\neq s_{r'n}\atop s\neq s'}K_{1b}(s_{in}-s_{i'n},t-t')K_{2b}(s_{rn}-s_{r'n},s-s')\text{Cov}_{|S_n}(X_{it}X_{i't'},X_{rs}X_{r's'}),
\end{align}
where $D_{K_n}=\Big(\frac{1}{(nT)^2b^3}\sum_{s_{in}\neq s_{i'n}\atop t\neq t'}K_{1b}(s_{in}-s_{i'n},t-t')\Big)\Big(\frac{1}{(nT)^2b^3}\sum_{s_{rn}\neq s_{r'n}\atop s\neq s'}K_{2b}(s_{rn}-s_{r'n},s-s')\Big)$. $K_{lb}(s_{in}-s_{i'n},t-t')$, for $l=1,2$, are defined similarly to $K_{0b}(s_{in}-s_{i'n},t-t')$ introduced above, except that $(\h_0,v_0)$ is replaced by $(\h_l,v_l)$, $l=1,2$. Together with Lemma \ref{lemma 1}, some simple algebra shows that $D_{K_n}$ convergence in $\mathbb P_{|S_n}$ to some strictly positive universal constant. i.e.
\begin{align}
    D_{K_n}\xrightarrow[n\to\infty]{\mathbb{P}_{|S_n}} (\mathbf{A}_{1}\mathbf{B}_{1})^2\ \ \ a.s.-[\mathbb{Q}]
\end{align}
Some algebra yields 
\begin{align}
    &\frac{1}{(nT)^2b^3}\sum_{s_{in}\neq s_{i'n}\atop t\neq t'}\sum_{s_{rn}\neq s_{r'n}\atop s\neq s'}K_{1b}(s_{in}-s_{i'n},t-t')K_{2b}(s_{rn}-s_{r'n},s-s')\text{Cov}_{|S_n}(X_{it}X_{i't'},X_{rs}X_{r's'}) \\
    =&\frac{1}{(nT)^2b^3}\sum_{\s_{in}\neq \s_{i'n}\atop t\neq t'}K_{1ii'tt'}K_{2ii'tt'} \text{Var}_{|S_n}(X_{it}X_{i't'}) \\
    &+\frac{1}{(nT)^2b^3}\sum_{s_{in}\neq s_{i'n}\atop t\neq t'}\sum_{s_{rn}\neq s_{r'n}, s\neq s'\atop ((r,s),(r',s'))\neq ((i,t),(i',t'))}K_{1ii'tt'}K_{2rr'ss'}\text{Cov}_{|S_n}(X_{it}X_{i't'},X_{rs}X_{r's'}) \\
    \label{Theorem MCLT Covariance Estimator eq 1}
    =& 0+ \frac{1}{(nT)^2b^3}\sum_{s_{in}\neq s_{i'n}\atop t\neq t'}K_{1ii'tt'}\sum_{s_{rn}\neq s_{r'n}, s\neq s'\atop \text{Card}(\{(r,s),(r',s')\}\cap \{(i,t),(i',t')\})=1}K_{2rr'ss'}\text{Cov}_{|S_n}(X_{it}X_{i't'},X_{rs}X_{r's'})\\
    &+\frac{1}{(nT)^2b^3}\sum_{s_{in}\neq s_{i'n}\atop t\neq t'}K_{1ii'tt'}\sum_{s_{rn}\neq s_{r'n}, s\neq s'\atop \text{Card}(\{(r,s),(r',s')\}\cap \{(i,t),(i',t')\})=0}K_{2rr'ss'}\text{Cov}_{|S_n}(X_{it}X_{i't'},X_{rs}X_{r's'})=: \textbf{Cov}_1+\textbf{Cov}_2, 
\end{align}
where the last equality holds almost surely with respect to $\mathbb{Q}$ for sufficiently large $n$. It can be proved by using Lemma \ref{lemma 1} and change of variable. Now we investigate $\textbf{Cov}_1$ and $\textbf{Cov}_2$.
\par ($\textbf{Cov}_1=o(1)\ \ a.s.-[\mathbb{Q}]$) 
\par When $\text{Card}(\{(r,s),(r',s')\}\cap \{(i,t),(i',t')\})=1$, we can focus on the case $(i,t)= (r,s)$ without loss of generality. For brevity, we sometimes would use $i$ ($r'$) to indicate location $\s_{in}$ $(\s_{r'n})$. Thus, when $(i,t)= (r,s)$,  
\begin{align}
   \textbf{Cov}_{1}=\frac{1}{(nT)^2b^3}\sum_{i\neq i', t\neq t'}K_{1ii'tt'}\sum_{i\neq r', t\neq s'}K_{2ir'ts'}\text{Cov}_{|S_n}(X_{it}X_{i't'},X_{it}X_{r's'}).
\end{align}
A noteworthy fact is that 
\begin{align}
    \text{Cov}_{|S_n}(X_{it}X_{i't'},X_{it}X_{r's'})&=E_{|S_n}[X^2_{it}X_{i't'}X_{r's'}]-E_{|S_n}[X_{it}X_{i't'}]E_{|S_n}[X_{it}X_{r's'}]\\
    &=\text{Cov}_{|S_n}(X^2_{it}X_{i't'},X_{r's'})-\text{Cov}_{|S_n}(X_{it},X_{i't'})\text{Cov}_{|S_n}(X_{it},X_{r's'}).
\end{align}
This indicates 
\begin{align}
    \textbf{Cov}_{1}=&\frac{1}{(nT)^2b^3}\sum_{i\neq i', t\neq t'}K_{1ii'tt'}\sum_{i\neq r', t\neq s'}K_{2ir'ts'}\text{Cov}_{|S_n}(X^2_{it}X_{i't'},X_{r's'})\\
    &-\frac{1}{(nT)^2b^3}\sum_{i\neq i', t\neq t'}K_{1ii'tt'}\sum_{i\neq r', t\neq s'}K_{2ir'ts'}\text{Cov}_{|S_n}(X_{it},X_{i't'})\text{Cov}_{|S_n}(X_{it},X_{r's'})\\
    =&:\textbf{Cov}_{11}+\textbf{Cov}_{12}.
\end{align}

\par ($\textbf{Cov}_{11}$) It is obvious that we can ignore the case $(r',s')=(i',t')$, since, as is pointed out by \eqref{Theorem MCLT Covariance Estimator eq 1}, this is almost surely a $o(1)$ term. We thus have
\begin{align}
    \textbf{Cov}_{11}&= \frac{1}{(nT)^2b^3}\sum_{(i,t)\neq (i',t')\neq (r',s')}K_{1ii'tt'}K_{2ir'ts'}\text{Cov}_{|S_n}(X^2_{it}X_{i't'},X_{r's'})\\
    =& \frac{1}{(nT)^2b^3}\sum_{(i,t)\neq (i',t')\neq (r',s')\atop d(\{(r',s')\},\{(i,t),(i',t')\})\leq q_n}K_{1ii'tt'}K_{2ir'ts'}\text{Cov}_{|S_n}(X^2_{it}X_{i't'},X_{r's'})\\
    &+\frac{1}{(nT)^2b^3}\sum_{(i,t)\neq (i',t')\neq (r',s')\atop d(\{(r',s')\},\{(i,t),(i',t')\})> q_n}K_{1ii'tt'}K_{2ir'ts'}\text{Cov}_{|S_n}(X^2_{it}X_{i't'},X_{r's'})\\
    =:& \triangle_{11}+\triangle_{12}.
\end{align}

Note that, $d(\{(r',s')\},\{(i,t),(i',t')\})\leq  q_n$ indicates 
\begin{align}
    \min\{||(r',s')-(i,t)||_{F},||(r',s')-(i',t')||_{F}\}\leq q_n.
\end{align}
If $||(r',s')-(i,t)||_{F}\leq q_n$ holds, we have $|t-s'|\leq q_n$. However, regarding $K_{2ir'ts'}=K_{2ir'}K_{2ts'}$ and $K_{2ts'}\neq 0$ holds only when $v_0-Tb\leq |t-s'|\leq v_0+Tb$, which indicates that the distance between $t$ and $s'$ is of order $O(T)$. Together with the fact that $q_n/T=o(1)$, there exists sufficiently large $T_0$ independent of sample size such that, for every $T\geq T_0$,  $K_{2ts'}= 0$. Thus, we only need to consider the case where $||(r',s')-(i,t)||_{F}\leq q_n$ holds. i.e.
\begin{align}
    \triangle_{11}= &\frac{1}{(nT)^2b^3}\sum_{(i,t)\neq (i',t')\neq (r',s')\atop ||(r',s')-(i',t')||_{F}\leq q_n}K_{1ii'tt'}K_{2ir'ts'}\text{Cov}_{|S_n}(X^2_{it}X_{i't'},X_{r's'})\\
    =&\frac{1}{(nT)^2b^3}\sum_{(i,t)\neq (i',t')\neq (r',s')\atop ||(r',s')-(i',t')||_{F}\leq q_n}K_{1ii'}K_{2ir'}K_{1tt'}K_{2ts'}\text{Cov}_{|S_n}(X^2_{it}X_{i't'},X_{r's'})
\end{align}

\par Similarly, $||(r',s')-(i',t')||_{F}\leq q_n$ implies $|t'-s'|\leq q_n$. However, as we frequently mentioned before, for each given $t=1,...,T$, $K_{1tt'}K_{2ts'}$ is not equal to $0$ if and only if
\begin{align}
  v_1-Tb\leq   |t-t'|\leq v_1+Tb\ \text{and}\ v_2-Tb\leq   |t-s'|\leq v_2+Tb
\end{align}
hold simultaneously. Without loss of generality, we set $v_1<v_2$. Since we have assumed $n=T$, there exists some universal $c>0$ and sufficiently large $T_0\in\mathbb{N}$ such that
\begin{align}
\label{Theorem MCLT Covariance Estimator proof eq 1}
  (t-v_1-Tb)- (t-v_2+Tb)>q_n,\ (t+v_2-Tb)-(t+v_1+Tb)>q_n
\end{align}
holds simultaneously for every $T\geq T_0$ and $T_0$ is independent of $t$. This indicates, for each given $t$, the $t'$ and $s'$ letting $K_{1tt'}K_{2ts'}\neq 0$ always satisfy $|t'-s'|>q_n$ for every $n=T\geq T_0$. However, we also require 
\begin{align}
\label{Theorem MCLT Covariance Estimator proof eq 2}
    |t'-s'|\leq q_n.
\end{align}
Thus, the contradiction between \eqref{Theorem MCLT Covariance Estimator eq 1} and \eqref{Theorem MCLT Covariance Estimator proof eq 2} implies
\begin{align}
    \label{Theorem MCLT Covariance Estimator proof eq 3}
    \triangle_{11}=0
\end{align}
holds for every $n=T\geq T_0$.

\par As for $\triangle_{12}$, some simple algebra shows that 
\begin{align}
    \triangle_{12}\leq C\beta^{\frac{\delta}{1+\delta}}(q_n)
     \frac{1}{(nT)^2b^3}\sum_{(i,t)\neq (i',t')}K_{1ii'}K_{1tt'}\sum_{(r',s')\atop d (\{(r',s')\},\{(i,t),(i',t')\})>q_n}K_{2ir'}K_{2ts'}
\end{align}
holds for every $\bS_n=S_n$, where $C$ is a universal constant depending on the moment conditions introduced in Assumption \ref{moment conditions}. By repeating the tricks used in proving Proposition \ref{prop 10} or \ref{prop 11}, using Proposition \ref{prop 5} yields
\begin{align}
   \lim_{n\to\infty} \frac{\sum_{(i,t)\neq (i',t')}K_{1ii'}K_{1tt'}\sum_{(r',s')\atop d (\{(r',s')\},\{(i,t),(i',t')\})>q_n}K_{2ir'}K_{2ts'}}{(nT)^3b^6}<\infty \ \ a.s.-[\mathbb{Q}].
\end{align}
Point (1) of (T4) of Assumption \ref{technical assumptions} implies
\begin{align}
    \label{Theorem MCLT Covariance Estimator proof eq 4}
    \lim_{n\to\infty}\triangle_{12}=0 \ \ a.s.-[\mathbb{Q}].
\end{align}
Combining \eqref{Theorem MCLT Covariance Estimator proof eq 3} and \eqref{Theorem MCLT Covariance Estimator proof eq 4} asserts
\begin{align}
\label{Theorem MCLT Covariance Estimator proof eq 5}
    \lim_{n\to\infty} \textbf{Cov}_{11}=0\ \ a.s.-[\mathbb{Q}].
\end{align}

\par ($\textbf{Cov}_{12}$) As for $\textbf{Cov}_{12}$, according to our previous argument, we only need to consider the case where $||(i,t)-(i',t')||_{F}>q_n$ and $||(i,t)-(r',s')||_{F}>q_n$ hold simultaneously, since $K_{1ii'tt'}$ and $K_{2ir'ts'}$ is non-zero only if $|t-s'|\land |t-t'|>q_n$. Then, we can show, for each $\bS_n=S_n$, 
\begin{align}
    \textbf{Cov}_{12}\leq C \beta^{\frac{2\delta}{1+\delta}}(q_n)\frac{1}{(nT)^2b^3}\sum_{(i,t)\neq (i',t')}K_{1ii'}K_{1tt'}\sum_{(i,t)\neq (r',s')}K_{2ir'}K_{2ts'},
\end{align}
which makes it a higher-order term of $\triangle_{12}$. Thus, using \eqref{Theorem MCLT Covariance Estimator proof eq 4} yields 
\begin{align}
\label{Theorem MCLT Covariance Estimator proof eq 6}
   \lim_{n\to\infty}\textbf{Cov}_1=0\ \ a.s.-[\mathbb{Q}].
\end{align}

\par ($\textbf{Cov}_2=o(1)\ \ a.s.-[\mathbb{Q}]$)
\par Similarly, we only need to consider the case where $||(i,t)-(i',t')||_{F}>q_n$ and $||(r,s)-(r',s')||_{F}>q_n$ holds simultaneously. Therefore, for each $\bS_n=S_n$,
\begin{align}
    \textbf{Cov}_2\leq &\frac{1}{(nT)^2b^3}\sum_{||(i,t)- (i',t')||_{F}>q_n}K_{1ii'tt'}\sum_{||(r,s)-(r',s')||_{F}>q_n}K_{2rr'ss'}|E_{|S_n}[X_{it}X_{i't'}X_{rs}X_{r's'}]|\\
    &+ \frac{1}{(nT)^2b^3}\sum_{||(i,t)-(i',t')||_{F}>q_n}K_{1ii'tt'}\sum_{||(r,s)- (r',s')||>q_n}K_{2rr'ss'}|\text{Cov}_{|S_n}(X_{it},X_{i't'})\text{Cov}_{|S_n}(X_{rs},X_{r's'})|\\
    =&: \textbf{Cov}_{21}+\textbf{Cov}_{22}
\end{align}
By repeating the procedures used in proving \eqref{Theorem MCLT Covariance Estimator proof eq 6}, we can obtain 
\begin{align}
    \label{Theorem MCLT Covariance Estimator proof eq 7}
    \lim_{n\to\infty}\textbf{Cov}_{22}=0\ \ a.s.-[\mathbb{Q}].
\end{align}

\par For $\textbf{Cov}_{21}$, we have decomposition 
\begin{align}
    \textbf{Cov}_{21}=(\sum_{(1)}+\sum_{(2)})K_{1ii'tt'}K_{2rr'ss'}|E_{|S_n}[X_{it}X_{i't'}X_{rs}X_{r's'}]|=:\bigcirc_1+\bigcirc_2,
\end{align}
where, by denoting $d_{rs}=d(\{(i,t),(i',t')\},\{(r,s)\})$ and  $d_{r's'}=d(\{(i,t),(i',t')\},\{(r',s')\})$
\begin{align}
    \sum_{(1)}=\sum_{||(i,t)-(i',t')||_{F}>q_n,||(r,s)-(r',s')||_{F}>q_n\atop d_{rs}\leq \frac{q_n}{3}, d_{r's'}\leq \frac{q_n}{3}}\ \ \text{and}\ \    \sum_{(2)}=\sum_{||(i,t)-(i',t')||_{F}>q_n,||(r,s)-(r',s')||_{F}>q_n\atop \max\{d_{rs},d_{r's'}\}> \frac{q_n}{3}}.
\end{align}

($\bigcirc_1$) Obviously, we only need to focus on the situations where $||(i,t)-(i',t')||>q_n$, $||(i,t)-(r,s)||_{F}\leq \frac{q_n}{3}$ and $||(i',t')-(r',s')||_{F}\leq \frac{q_n}{3}$ hold simultaneously. According to (M3) of Assumption \ref{moment conditions}, we have
\begin{align}
    \bigcirc_1&\lesssim \frac{1}{(nT)^2b^3}
   \mathbb{K}_S\mathbb{K}_T, 
\end{align}
where  
\begin{align}
    \mathbb{K}_S &:= \sum_{\s_{in}\neq \s_{i'n}}K_{1,ii'}\Big(\sum_{||\s_{rn}-\s_{in}||\leq q_n/3\atop ||\s_{r'n}-\s_{i'n}||\leq q_n/3}K_{2,rr'}\Big), \\
 \mathbb{K}_T &:=   
\sum_{t\neq t'}K_{1tt'}\Big(\sum_{|s-t|\leq q_n/3\atop |s'-t'|\leq q_n/3}K_{2ss'}\Big). 
\end{align}
Lemma \ref{lemma 4} implies $\mathbb{K}_T\lesssim T^2q_n^2b$. Meanwhile, using strong law of large number of 2nd-order U-statistic yields that 
\begin{align}
    \lim_{n\to\infty}\frac{\mathbb{K}_S}{n^2b^4q_n^2}<\infty\ \ a.s.-[\mathbb{Q}].
\end{align}
Above all, together with point (5) of T4 in Assumption \ref{technical assumptions}, we obtain 
\begin{align}
    \bigcirc_1\lesssim q_n^4b^2=o(1)\ \ a.s.-[\mathbb{Q}].
\end{align}

($\bigcirc_2$) Some simple algebra yields that, for each $\bS_n=S_n$, there exists some universal $C>0$ such that
\begin{align}
   \bigcirc_2 =&\sum_{(2)}K_{1ii'tt'}K_{2rr'ss'}|E_{|S_n}[X_{it}X_{i't'}X_{rs}X_{r's'}]|\\
   =& 2\sum_{||(i,t)-(i',t')||_{F}>q_n,||(r,s)-(r',s')||_{F}>q_n\atop d_{rs}>\frac{q_n}{3},d_{r's'}\leq\frac {q_n}{3}}K_{1ii'tt'}K_{2rr'ss'}|\text{Cov}_{|S_n}(X_{it}X_{i't'}X_{r's'},X_{rs})|\\
   &+\sum_{||(i,t)-(i',t')||_{F}>q_n,||(r,s)-(r',s')||_{F}>q_n\atop d_{rs}>\frac{q_n}{3},d_{r's'}> \frac{q_n}{3}}K_{1ii'tt'}K_{2rr'ss'}|\text{Cov}_{|S_n}(X_{it}X_{i't'},X_{rs}X_{r's'})|\\
   &+\sum_{||(i,t)-(i',t')||_{F}>q_n,||(r,s)-(r',s')||_{F}>q_n\atop d_{rs}>\frac{q_n}{3},d_{r's'}> \frac{q_n}{3}}K_{1ii'tt'}K_{2rr'ss'}|\text{Cov}_{|S_n}(X_{it},X_{i't'})||\text{Cov}_{|S_n}(X_{rs},X_{r's'})|\\
   =&:2\bigcirc_{21}+\bigcirc_{22}+\bigcirc_{23}.
\end{align}
Together with Assumption \ref{moment conditions}, we obtain
\begin{align}
   |\text{Cov}_{|S_n}(X_{it}X_{i't'}X_{r's'},X_{rs})|\lor |\text{Cov}_{|S_n}(X_{it}X_{i't'},X_{rs}X_{r's'})|&\lesssim \beta^{\frac{\delta}{1+\delta}}(q_n),\\
   |\text{Cov}_{|S_n}(X_{it},X_{i't'})||\text{Cov}_{|S_n}(X_{rs},X_{r's'})|&\lesssim \beta^{\frac{2\delta}{1+\delta}}(q_n).
\end{align}
Thus, by repeating the techniques used in proving the asymptotic variance of Theorem \ref{Theorem CLT Covariance Estimator}, we obtain 
\begin{align}
    \lim_{n\to\infty}\frac{\bigcirc_2}{(nT)^2b^3}=0 \ \ \ a.s.-[\mathbb{Q}].
\end{align}
Above all, we manage to prove that 
\begin{align}
    \label{eq of Step 1}
    \lim_{n,T} (nT)^2b^3\text{Cov}_{|S_n}(\hat{C}(\h_{1},v_1),\hat{C}(\h_{2},v_2))=0\ \ \ a.s.-[\mathbb{Q}].
\end{align}
 

\par \textbf{Step 2} (Cramér–Wold device) Please note that, in Step 1, we have proven that, for any $(\h_1,v_1)\neq (\h_2,v_2)$, $(nT)^2b^3\text{Cov}_{|S_n}(\hat{C}(\h_{1},v_1),\hat{C}(\h_{2},v_2))=o(1)$ holds almost surely. Meanwhile, regarding that $\text{vec}(\widehat{\mathbf C})-\text{vec}(\mathbf C)\in \mathbb{R}^{MN}$, we introduce notation
\begin{align}
 (  \text{vec}(\widehat{\mathbf C})-\text{vec}(\mathbf C))=:\mathbf T=(T_{11},...,T_{1N},...T_{M1},...,T_{MN})^\top.
\end{align}
Hence, for any $(w_{11},...,w_{MN})^{\top}=:\w \in \mathbb{R}^{MN}$, Cramér–Wold theorem implies that, to prove the asymptotic normality of $\mathbf T$, we only need to concentrate on $\w^{\top}\mathbf T$. 

\par Based on \textbf{Step 1}, by repeating the procedure used in proving the variance in Theorem \ref{Theorem CLT U-Statistic}, we obtain that 
\begin{align}
    \lim_{n\to \infty}nTb^{1.5}\text{Var}_{|S_n}(\w^\top\mathbf T)
    &=\lim_{n\to\infty }\sum_{k=1}^{MN}w_k\text{Var}_{|S_n}(T_{k})+o(1)=(\sigma/2\mf A_1)^2\sum_{i=1}^M\sum_{j=1}^Nw_{ij} \ \ a.s.-[\mathbb{Q}],
    \end{align}
where $\sigma^2=E[X_0^2]\mf A_2\mf B_2$. Equivalently, we have
\begin{align}
      \label{eq step 2 1}
    \text{Var}\left(\frac{nTb^{1.5}\w^\top \mathbf{T}}{((\sigma/2\mf A_1)^2\sum_{i=1}^M\sum_{j=1}^Nw_{ij})^{\frac{1}{2}}}\right)\xrightarrow[n\to\infty]{}1 \ \ a.s.-[\mathbb{Q}].
\end{align}

\par Note that 
\begin{align}
&nTb^{1.5}\w^{\top}\mathbf T=nTb^{1.5}\sum_{i=1}^{M}\sum_{j=1}^Nw_{ij}T_{ij}\\
=&nTb^{1.5}\sum_{i=1}^{M}\sum_{j=1}^Nw_{ij}\Big(\hat{C}(\h_i,v_j)-C(\h_i,v_j)\Big)\\
=&nTb^{1.5}\sum_{i=1}^{M}\sum_{j=1}^Nw_{ij}\widetilde{K}_n(\h_i,v_j)^{-1}\Big(U_n(\h_i,v_j)-E[U_n(\h_i,v_j)]+\textbf{Bias}(\h_i,v_j)\Big)\\
=& (\widetilde{K}+o_{a.s.}(1))^{-1}\Big\{nTb^{1.5}\Big[\sum_{i,j}(U_n(\h_i,v_j)-E[U_n(\h_i,v_j)])\Big]+  nTb^{1.5}\sum_{i,j}\textbf{Bias}(\h_i,v_j) \Big\}
\end{align}
where $U_n(\h_i,v_j)$ shares the same definition as $U_n$ in Theorem \ref{Theorem CLT U-Statistic}. Thus, by repeating the proof of the asymptotic normality in Theorem \ref{Theorem CLT U-Statistic}, together with $nTb^{2.5}\to 0$ and the finiteness of $M$ and $N$, we obtain 
\begin{align}
    nTb^{1.5}\w^\top\mathbf{T}\xrightarrow[n\to\infty]{d}N(0,(\sigma/2\mf A_1)^2\sum_{i=1}^N\sum_{j=1}^M w_{ij})\ \ \ \ a.s.-[\mathbb{Q}].
\end{align}
This finish the proof of Theorem \ref{Theorem MCLT Covariance Estimator}.


\subsection{\texorpdfstring{Proof of \cref{thm:partialexact}}{Proof of Theorem 4.1}}
  By Theorem \ref{Theorem MCLT Covariance Estimator}, we have $ nTb^{1.5} (\widehat{\mf C} - \mf C) \xrightarrow[]{d} \mathcal G~~~~a.s. ~[\mathbb{Q}]$, 
       where $ \mathcal G $ is $N(0,\tau^2I_{MN})$ distributed. Recalling the definition \eqref{d7}, we define 
       \begin{align}
           \phi(\mf C) :=   D_{\bs  \psi} (\mf C) \|\mf C \bs  \psi\|_F^2 
           = \|\mf C\|_F^2 \|\mf C \bs  \psi \|_F^2 - \|\bs  \psi^{\top}   \mf C^{\top} \mf C\|_F^2   
       \end{align}
       and use von Mises calculus to derive the limiting distribution of  $\phi(\widehat{\mf C})$. Then, since $\|\widehat {\mf C} \bs \psi \|_F \xrightarrow[]{p} \|\mf C \bs \psi \|_F $ the limiting distribution of $\widehat  D_{\bs \psi}$ will follow from Slusky's lemma. For a matrix $\mf G \in \R^{M\times N}$, we have 
       \begin{align}
       \label{d40}
           \phi(\mf C + t \mf G) = \|\mf C+ t \mf G\|_F^2 \|(\mf C+ t \mf G )\bs  \psi \|_F^2 - \|\bs  \psi^{\top}   (\mf C+ t \mf G)^{\top} (\mf C + t \mf G)\|_F^2 . 
       \end{align}
       Note that $\phi(\mf C + t \mf G)$ is a polynomial in $t$ of degree $4$, 
       \begin{align}
           \phi(\mf C + t \mf G) = (a_0 + a_1 t + a_2 t^2)(c_0 - c_1 t + c_2 t^2) - (b_0 + b_1 t + b_2 t^2 + b_3 t^3 + b_4 t^4), \label{d40b}
       \end{align}
       where 
       \begin{align}
           a_0 = \mathrm{tr}(\mf C \mf C^{\top}), \quad c_0 = \mathrm{tr}(\mf C \bs \psi \bs \psi^{\top} \mf C^{\top})\\
           a_1 = 2 \mathrm{tr}(\mf C \mf G^{\top}), \quad c_1 = 2 \mathrm{tr}(\mf G \bs \psi \bs\psi^{\top} \mf C^{\top})\\
           a_2 = \mathrm{tr}(\mf G \mf G^{\top}), \quad c_2 = \mathrm{tr}(\mf G \bs \psi \bs \psi^{\top} \mf G^{\top})\\
           b_0 = \|\bs \psi^{\top} \mf C^{\top} \mf C\|_F^2, \quad b_4 = \|\bs \psi^{\top} \mf G^{\top} \mf G\|_F^2\\ b_1 = 2 \{\mathrm{tr}[\bs \psi^{\top} \mf G^{\top} \mf C \mf C^{\top} \mf C \bs \psi] + \mathrm{tr}[\bs \psi^{\top} \mf C^{\top} \mf G \mf C^{\top} \mf C \bs \psi]\}\\
           b_2 = 2 \mathrm{tr}[\bs \psi^{\top} \mf G^{\top} \mf G \mf C^{\top} \mf C \bs \psi] + \mathrm{tr}[\bs \psi^{\top} \mf C^{\top} \mf G \mf G^{\top} \mf C  \bs \psi]\\ 
           + 2 \mathrm{tr}[\bs \psi^{\top} \mf G^{\top} \mf C \mf G^{\top} \mf C \bs \psi] + \mathrm{tr}[\bs \psi^{\top} \mf G^{\top} \mf C \mf C^{\top} \mf G \bs \psi]\\
           b_3 = 2 \mathrm{tr}[\bs \psi^{\top} \mf C^{\top} \mf G \mf G^{\top} \mf G \bs \psi] + 2 \mathrm{tr}[\bs \psi^{\top} \mf G^{\top} \mf C \mf G^{\top} \mf G \bs \psi].
       \end{align}
       A straightforward  calculation gives for the linear coefficient 
       \begin{align}
          \left. \frac{\partial \phi(\mf C + t \mf G)}{\partial t}\right|_{t = 0} &= 2\tr(\mf C \mf G^{\top}) \tr( \mf C \bs  \psi \bs  \psi^{\top} \mf C^{\top}) +  2\tr(\mf C \mf C^{\top}) \tr(\mf G \bs  \psi \bs \psi^{\top} \mf C^{\top}) \\ 
         & - 2 \tr (\bs  \psi^{\top} \mf C \mf G^{\top} \mf C \mf C^{\top} \bs  \psi )- 2 \tr (\bs  \psi^{\top} \mf G \mf C^{\top} \mf C \mf C^{\top} \bs \psi ). 
       \end{align}
       Under the null hypothesis, $\mf C = \mf c_1  \mf c_2^{\top} $, which implies
   \begin{align}   
      \tr (\bs  \psi^{\top} \mf C \mf G^{\top} \mf C \mf C^{\top} \bs \psi ) &= \tr (\bs \psi^{\top}\mf c_1 \mf c_2^{\top} \mf G^{\top} \mf c_1 \mf c_2^{\top} \mf C^{\top}\bs \psi )\\
      & =   \tr (\mf c_2^{\top} \mf G^{\top} \mf c_1) \tr (\bs \psi^{\top}\mf c_1 \mf c_2^{\top} \mf C^{\top}\bs \psi )  = \tr(\mf C \mf G^{\top}) \tr(\mf C \bs \psi \bs \psi^{\top}\mf C^{\top})   
      \end{align}
      and 
        \begin{align} 
       \tr (\bs \psi^{\top}\mf G \mf C^{\top} \mf C\mf C^{\top}\bs \psi ) & =  \tr (\bs \psi^{\top}\mf G (\mf c_1 \mf c_2^{\top})^{\top} \mf C(\mf c_1 \mf c_2^{\top})^{\top}\bs \psi ) \\
      & = \tr ( (\mf c_1 \mf c_2^{\top})^{\top} \mf C\mf c_1 \mf c_1^{\top}\bs \psi \bs \psi^{\top}\mf G \mf c_2) \tr(\mf c_1^{\top} \mf C \mf c_2) = \tr(\mf C \mf C^{\top}) \tr(\mf G \bs \psi \bs \psi^{\top}\mf C^{\top}).
          \end{align}
          Hence, under the null hypothesis, we have  for the linear coefficient
       \begin{align}
            \left. \frac{\partial \phi(\mf C + t \mf G)}{\partial t}\right|_{t = 0} = 0.
       \end{align}
       By a tedious calculation, we obtain under the null hypothesis for the second directional derivative the representation 
       \begin{align}
            \left. \frac{\partial^2 \phi(\mf C + t \mf G)}{2\partial t^2 }\right|_{t = 0} & = \| \mf C\|_F^2 \|\mf G \bs \psi \|_F^2 + 4 \tr(\mf C \mf G^{\top})\tr(\mf G \bs \psi \bs \psi^{\top} \mf C^{\top}) + \|\mf G\|_F^2 \| \mf C \bs \psi\|_F^2 \\ 
            & - 2 \tr(\bs \psi^{\top} \mf G^{\top} \mf G \mf C^{\top} \mf C \bs \psi)- \|\mf G^{\top} \mf C \bs \psi\|_F^2 -  2 \tr(\bs \psi^{\top} \mf G^{\top} \mf C \mf G^{\top} \mf C \bs \psi) - \|\mf C^{\top} \mf G \bs \psi\|_F^2\\
            &=\left \| \mf G \|\mf C \bs \psi\|_F - \frac{\mf G \bs \psi \bs \psi^{\top} \mf C^{\top} \mf C}{\|\mf C \bs \psi\|_F} \right\|_F^2 - \| \mf G^{\top} \mf C \bs \psi - \mf C^{\top} \mf G \bs \psi \|^2,
       \end{align}
which gives the coefficient of $t^2$ in \eqref{d40}. 
Now, we take \( \mf G := nTb^{1.5}\big (\widehat  {\mf C} -  \mf C\big) \), \( t = 1/\big (nTb^{1.5} \big) \) and expand the polynomial of degree $4$,
\( \phi(\mf C + t\mf G) \), in powers of \( t \). Note that \( \phi(\mf C) = 0 \) under
the null hypothesis and that the terms corresponding to \( t^3 \) and \( t^4 \) in \eqref{d40}
are at least of order \( \big (nTb^{1.5} \big ) ^{-3} \). Therefore, we obtain 
\begin{align}
(nT)^2b^{3}   \phi (\widehat  {\mf C} )  &=  (nT)^2b^{3} \big (  \phi (\widehat  {\mf C} ) - \phi ( \mf C  ) \big )   = \frac{1}{2}\,\frac{\partial^2}{\partial t^2}
\phi\!\big (\mf C + t  nTb^{1.5}\,(\widehat  {\mf C} - \mf C) \big)
\Big|_{t=0}
+ o_\mathbb{P}(1) \\ 
&
=\left \| \mf G \|\mf C \bs \psi\|_F - \frac{\mf G  \bs \psi \bs \psi^{\top} \mf C^{\top} \mf C}{\|\mf C \bs \psi\|_F} \right\|_F^2 - \| \mf G ^{\top} \mf C \bs \psi - \mf C^{\top}  \mf G  \bs \psi \|^2 + o_\mathbb{P}(1)
~~~a.s. ~[\mathbb{Q}],
\end{align}
and the assertion follows by the continuous mapping theorem, dividing by $\|\widehat  {\mf C} \bs \psi \|^2$ and Slutsky's lemma.

\subsection{\texorpdfstring{Proof of \cref{thm:alterexact}}{Proof of Theorem 4.2}}
The proof of Theorem \ref{thm:alterexact} follows from the arguments given in the  proof of \cref{thm:partialexact}. More precisely, note that under the alternative hypothesis,  \begin{align}
            \left. \frac{\partial \phi(\mf C + t \mf G)}{\partial t}\right|_{t = 0} \neq 0,
       \end{align}
       and that the coefficients of $t^2$, $t^3$ and $t^4$ of the polynomial 
       $$
       \phi (\widehat {\mf C} ) - \phi ( \mf C  )  = \phi ( \mf C + t \mf G ) - \phi (\mf C) 
       $$
       (with  \( \mf G := nTb^{1.5}\big (\widehat {\mf C} -  \mf C\big) \), \( t = 1/\big (nTb^{1.5} \big) \)
       ) 
       are at least of order $(nT)^{-2}b^{-3}$.
       Consequently, by von Mises calculus and the continuous mapping theorem, $nTb^{1.5}\{\phi(\widehat {\mf C}) - \phi(\mf C)\}$  converges weakly  to 
       \begin{align}
           \left. \frac{\partial \phi(\mf C + t \mathcal G)}{\partial t}\right|_{t = 0} & =  2\tr(\mf C \mathcal G^{\top}) \tr( \mf C \bs  \psi \bs  \psi^{\top} \mf C^{\top}) +  2\tr(\mf C \mf C^{\top}) \tr(\mathcal G \bs  \psi \bs \psi^{\top} \mf C^{\top}) \\ 
         & - 2 \tr (\bs  \psi^{\top} \mf C \mathcal G^{\top} \mf C \mf C^{\top} \bs  \psi )- 2 \tr (\bs  \psi^{\top} \mathcal G \mf C^{\top} \mf C \mf C^{\top} \bs \psi ), \label{eq:firstorder}
       \end{align}
       where the matrix $\mathcal G$ stands for the $N(0,\tau^2I_{MN})$ distributed matrix. The result follows from an elementary calculation of the covariance of \eqref{eq:firstorder}, $\|\widehat{\mf C} \bs \psi \|_F \xrightarrow[]{p} \|\mf C \bs \psi \|_F $ and Slusky's lemma.

\subsection{\texorpdfstring{Proof of Theorem \ref{th exact test SVD}}{Proof of Theorem 4.3}} 

With the notation $\mf E:=\widehat {\mf C}-\mf C$ we have by  \cref{Theorem MCLT Covariance Estimator}
\begin{align}
\label{eq proof SVD 1}
    nTb^{1.5}\text{vec}( \mf E)\xrightarrow[]{d}  N(0,\tau^2I_{MN}) \ \ \ \ \ a.s.-[\mathbb{Q}]
\end{align}
Then, under the null hypothesis, we obtain for the singular value decomposition of the matrix $\mf C$
\begin{align}
    \label{eq proof SVD 2}
    \mf C=\sigma_{sv,1}(\mf C)\u\mf v,
\end{align}
where $\mf u\in \mathbb{R}^M$ and $\mf v\in \mathbb{R}^N$ are the left and right singular vectors with $||\mf u||=||\mf v||=1$ corresponding to the (largest) singular value $\sigma_{sv,1}(\mf C)$. Moreover, we have
\begin{align}
    \widehat{D}=||\mf C+ \mf E||^2_{F}-\sigma^2_{sv,1}(\mf C+ \mf E).
\end{align}

\noindent \textbf{Step 1}(Orthogonal decomposition of the perturbation) With the notation 
\begin{align}
    P_{\mf u}=\mf u\mf u^{\top}, P_{\mf u}^{\perp}=I-P_{\mf u},\\
    P_{\mf v}=\mf v\mf v^{\top}, P_{\mf v}^{\perp}=I-P_{\mf v}.
\end{align}
we obtain the decomposition
\begin{align}
    &\mf E= A_E+ B_E+ C_E+D_E, \label{d41}
    \end{align}
    where 
    \begin{align}
  A_E &= P_{\mf u}EP_{\mf v}=(\mf u^{\top}E\mf v)\mf u\mf v^{\top}=\alpha \mf u\mf v^{\top},\\
  \alpha& =\mf u^{\top}E\mf v\in\mathbb{R}, \\
 C_E& =P_{\mf u}EP^{\perp}_{\mf v}=\mf u(\mf u^{\top}EP^{\perp}_{\mf v})=\mf u c^{\top},\\ c& =P^{\perp}_{\mf v}E^{\top}\mf u\  (\Rightarrow c\perp \mf v) \label{eq:defce}\\
 B_E& =P_{\mf u}^{\perp}EP_{\mf v}=(P_{\mf u}^{\perp}E\mf v\mf) \mf v^{\top}=b\mf v^{\top},\\ b&=P_{\mf u}^{\perp}E \mf v\  (\Rightarrow b\perp \mf u)\label{eq:defbe}\\
 D_E& = P^{\perp}_{\mf u}EP^{\perp}_{\mf v}\ (\Rightarrow\ \mf u^{\top}D_E=0,\  D_E\mf v=0).
 \label{nhd9}
\end{align}
A simple calculation shows that  $\text{span}\{\mf u\mf v^{\top}\}:=\{a\mf u\mf v^{\top}:a\in\mathbb{R}\} \subset \mathbb{R}^{M\times N} $  is orthogonal (with respect to the inner product $\langle \mf A, \mf B \rangle_F = \text{tr} (A B^\top ) $) to the  three subspaces 
\begin{align}
    \mathcal{C}& =\{\mf u \xi_{c}^{\top}: \xi_{c}\perp \mf v\},\\
    \mathcal{B}& =\{\xi_b\mf v^{\top}: \xi_b\perp \mf u\},\\ \mathcal{D}& =\{\Xi\in\mathbb{R}^{M\times N}:\mf u^{\top}\Xi= \Xi\mf v=0\}.
\end{align}
 Noting that $B_E\in \mathcal{B}$, $C_E\in \mathcal{C}$ and $D_E\in \mathcal{D}$, we obtain from \eqref{d41} that 
$$
\mf C+\mf  E=\mf C+A_E+B_E+C_E+D_E=(\sigma_{sv,1}(\mf C)+\alpha)\mf u\mf v^{\top}+b\mf v^{\top}+\mf uc^{\top}+P_{\mf u}^{\perp}EP_{\mf v}^{\perp},
$$
which implies  that 
\begin{align}
\label{nhd8}
||\mf C+\mf  E||^2_{F}
&=(\sigma_{sv,1}(\mf C)+\alpha)^2+||b||^2+||c||^2+||D_E||^2_{F}.
\end{align}

\noindent \textbf{Step 2} (Expansion of SVD mapping) 
This step is dedicated to the investigation of term $\sigma^2_{sv,1}(\mf C+\mf E)$. Let  $\mf u$ and $\mf v$ denote the vectors in the decomposition \eqref{eq proof SVD 2}. If  $\mf x\in \mathbb{R}^M$, $\mf y\in \mathbb{R}^N$ are two other vectors such that  $||\mf x||_F=||\mf y||_F=1$ and $\mf x^\top\mf u>0$, $\mf y^\top \mf v>0$, we can represent $\mf x$ and $\mf y$ as 
\begin{align} \label{d42}
  \mf x=\mf x(\h)=\frac{\mf u+\mf h}{\sqrt{1+||\h||^2}},\ \mf y= \mf y(\mf g)=\frac{\mf v+\mf g}{\sqrt{1+||\mf g||^2}} 
\end{align}
where the vectors $\mf h\in \mathbb{R}^{M}$ and $\mf g\in \mathbb{R}^{N}$  satisfy 
$ \h\perp \u$ and  $\mf g\perp \mf v$. 
With these notation we define 
\begin{align}
\label{d43}
    \Phi(\h,\mf g)=\mf x(\h)^\top\mf C \mf y(\mf g)&+\mf x(\h)^\top \mf E\mf y(\mf g) =:\Phi_{\mf C}(\h,\mf g) +\Phi_{E}(\h,\mf g).
\end{align}
To simplify the calculation later, we first linearize the  denominators in \eqref{d42}, that is 
\begin{align}
    &\mf x(\h)=\Big (1-\frac{1}{2}||\h||^2+O(||\h||^4)\Big  )(\mf u+\h)=\mf u+\h-\frac{1}{2}\mf u ||\h||^2+O(||\h||^3),\\
     &\mf y(\mf g)=\Big  (1-\frac{1}{2}||\mf g||^2+O(||\mf g||^4)\Big  )(\mf v+\mf g)=\mf v+\mf g-\frac{1}{2}\mf v ||\mf g||^2+O(||\mf g||^3),
\end{align}
which gives (using the orthogonality) 
\begin{align}
    \mf x(\h)^{\top}\mf u=1-\frac{1}{2}||\h||^2+O(||\h||^3)\ \ \text{and}\ \  \mf v^{\top}\mf y(\mf g)=1-\frac{1}{2}||\mf g||^2+O(||\mf g||^3). \label{eq proof of Prop 2.1 eq 1}
\end{align}
For the term $\Phi_{\mf C}$ in \eqref{d43} we obtain   
using \eqref{eq proof of Prop 2.1 eq 1}  (note  that  $\mf C=\sigma_{sv,1}(\mf C)\mf u\mf v^{\top}$)
\begin{align}
    \Phi_{\mf C} (\h,\mf g)
    &=\sigma_{sv,1}(\mf C)(\mf x(\h)^{\top}\mf u)(\mf v^{\top}\mf y(\mf g)) \\
    &=\sigma_{sv,1}(\mf C)\Big  (1-\frac{1}{2}||\h||^2\Big )\Big (1-\frac{1}{2}||\mf g||^2\Big )+O(||\h||^3+||\mf g||^3+||\mf h||^2||\mf g||^2)\\
    &=\sigma_{sv,1}(\mf C)-\frac{1}{2}
    \sigma_{sv,1}(\mf C )\Big (||\mf h||^2+||\mf g||^2\Big )+O(||\mf h||^3+||\mf g||^3+||\mf h||^2||\mf g||^2). \label{eq proof of Prop 2.1 Phi C}
\end{align}
Similarly, we have 
\begin{align}
    \Phi_{E} (\h,\mf g) =&(\mf u+\h-\frac{1}{2}\mf u ||\h||^2+O(||\h||^3))^{\top}\mf  E(\mf v+\mf g-\frac{1}{2}\mf v ||\mf g||^2+O(||\mf g||^3))\\
    =& \mf u^{\top} \mf E\mf v+\mf u^{\top} \mf E\mf g +\h^{\top} \mf E\mf v+\h^{\top}\mf E\mf g-\frac{1}{2}||\h||^2\mf u^{\top} \mf E\mf v-\frac{1}{2}||\mf g||^2\mf u^{\top} \mf E\mf v\\
    &+\underbrace{(-\frac{1}{2}||\h||^2\mf u^{\top}\mf E\mf g)}_{=O(||\h||^2||\mf g||)}+\underbrace{(-\frac{1}{2}||\mf g||^2\h^{\top}\mf E\mf v)}_{=O(||\mf g||^2||\h||)}+O(||\h||^2||\mf g||^2+||\h||^3+||\mf g||^3)\\
    =&\alpha+\mf u^{\top}\mf E\mf g +\h^{\top} \mf E\mf v+\h^{\top}\mf  E\mf g-\frac{\alpha}{2}(||\h||^2\mf+ ||\mf g||^2)+O(r_n),
\end{align}
where $r_n=O(||\h||^2||\mf g||^2+||\h||^3+||\mf g||^3+||\h||^2||\mf g||+||\mf g||^2 ||\h||)$. As  $P_{\mf u}^{\perp}$ and $P_{\mf v}^{\perp}$ are orthogonal projection matrices, we have 
\begin{align}
    &P_{\mf v}^{\perp}\mf g=\mf g~,~~  \h^{\top}P_{\mf u}^{\perp}\mf =\mf h^{\top},
\end{align}
which gives (observing the definition of  $b$ and $c$  in \eqref{eq:defbe} and \eqref{eq:defce})
\begin{align}
   \mf u^{\top}E\mf g=c^{\top}\mf g~,~~ \h^{\top}E\mf v=\h^{\top}b. 
\end{align}
Therefore, 
\begin{align}
    \Phi_{E}(\h,\mf g)=\alpha+c^{\top}\mf g +\h^{\top}b+\h^{\top}E\mf g-\frac{\alpha}{2}(||\h||^2\mf+\frac{1}{2}||\mf g||^2) { +O(r_n)
    }\label{eq proof of Prop 2.1 eq Phi E}.
\end{align}
Then, combining \eqref{eq proof of Prop 2.1 Phi C} and \eqref{eq proof of Prop 2.1 eq Phi E} yields
\begin{align}
    \Phi(\h,\mf g)=\sigma_{sv,1}(\mf C)+\alpha+c^{\top}\mf g +\h^{\top}b+\h^{\top}E\mf g-\frac{1}{2}(\sigma_{sv,1}(\mf C)+\alpha)(||\h||^2+||\mf g||^2)+R_3, \label{eq proof of Prop 2.1 eq 2}
\end{align}
where $ R_3 = O(r_n) =O(  ||\h||^3+||\mf g||^3+ ||\h||^2 ||\mf g||+||\mf g||^2||\h||)$  is a real-valued mapping of $(||\h||, ||\mf g||)$. Furthermore, it is easy to see that  $\Phi$  has partial derivatives of arbitrary order. Thus, a necessary condition for the maximum  is 
\begin{align}
   0&= \partial_{\mf h}\Phi(\h,\mf g)= b+\mf  E\mf g-(\sigma_{sv,1}(\mf C)+\alpha)\h+r_{\h},\label{eq proof of Prop 2.1 eq 2.5}\\
   0&=\partial_{\mf g}\Phi(\h,\mf g)=c+\mf  E^{\top}\mf h-(\sigma_{sv,1}(\mf C)+\alpha)\mf g+r_{\mf g}, \label{eq proof of Prop 2.1 eq 3}
\end{align}
where remainders $r_{\h}$ and $r_{\mf g}$ are  obtained from \eqref{eq proof of Prop 2.1 eq 2} and are of order $O(||\h||^2+||\mf g||^2)$. Moreover, by  definitions of $\alpha$, $b$ and $c$ we have 
\begin{align}
    |\alpha|\leq ||\mf E||_{op},\ ||b||=||P_{\mf u}^{\perp}\mf E\mf v||\leq ||\mf  E||_{op},\ ||c||=||P_{\mf v}^{\perp}\mf  E^{\top}\mf u||\leq ||\mf  E||_{op},
\end{align}
where $||\cdot||_{op}$ is the operator norm. Note that  $||\mf  E||_{op}\leq ||\mf  E||_{F} = o_{\mathbb{P}_{|S_n}}(1)$ since   $\widehat {\mf C}$ is a consistent estimator of $\mf C$.
Then it follows from  \eqref{eq proof of Prop 2.1 eq 2.5} and \eqref{eq proof of Prop 2.1 eq 3} that 
\begin{align}
    ||\h||\lesssim \frac{1}{|\sigma_{sv,1}(\mf C)+\alpha|}(||\mf E||_{op}+||\mf  E||_{op}||\mf g||+||\h||^2+||\mf g||^2)
    ,\\
    ||\mf g||\lesssim \frac{1}{|\sigma_{sv,1}(\mf C)+\alpha|}(||\mf E||_{op}+||\mf E||_{op}||\mf h||+||\h||^2+||\mf g||^2),
\end{align}
 which gives 
\begin{align}
    ||\h||\lor ||\mf g||=O_{\mathbb{P}_{|S_n}}( ||\mf E||_{op})=O_{\mathbb{P}_{|S_n}}(||\mf  E||_{F}).
\end{align}
We can now  solve the equations \eqref{eq proof of Prop 2.1 eq 2.5} and \eqref{eq proof of Prop 2.1 eq 3} to first order by dropping all higher order terms and obtain the follow leading terms of solutions
\begin{align}
    \h^*&=\frac{1}{\sigma_{sv,1}(\mf C)}b+O_{\mathbb{P}_{|S_n}}(||\mf E||_{F}^2)=\frac{1}{\sigma_{sv,1}(\mf C)}P_{\mf u}^{\perp}E\mf v+O_{\mathbb{P}_{|S_n}}(||\mf E||_{F}^2),\\
    \mf g^{*}&= \frac{1}{\sigma_{sv,1}(\mf C)}c+O_{\mathbb{P}_{|S_n}}(||\mf E||_{F}^2)=\frac{1}{\sigma_{sv,1}(\mf C)}P^{\perp}_{\mf v}E^{\top}\mf u+O_{\mathbb{P}_{|S_n}}(||\mf E||^2_{F}).
\end{align}

Then, substituting $\h^*$ and $\mf g^*$ into \eqref{eq proof of Prop 2.1 eq 2} yields
\begin{align}
    \sigma_{sv,1}(\mf C+\mf E)=\sigma_{sv,1}(\mf C)+\alpha+\frac{||b||^2+||c||^2}{2\sigma_{sv,1}(\mf C)}+O_{\mathbb{P}_{|S_n}}(||\mf E||^3_{F}),
\end{align}
which gives 
\begin{align}
    \sigma_{sv,1}^2(\mf C+\mf E)&=(\sigma_{sv,1}(\mf C)+\alpha)^2+\frac{(\sigma_{sv,1}(\mf C)+\alpha)(||b||^2+||c||^2)}{\sigma_{sv,1}(\mf C)}+O_{\mathbb{P}_{|S_n}}(||\mf E||^3_{F})\\
    &=(\sigma_{sv,1}(\mf C)+\alpha)^2+||b||^2+||c||^2+O_{\mathbb{P}_{|S_n}}(||\mf E||^3_{F}).
    \label{eq proof of Prop 2.1 eq 4}
\end{align}

\noindent \textbf{Step 3} Combining \eqref{eq proof of Prop 2.1 eq 4} with \eqref{nhd8} and \eqref{nhd9} gives
\begin{align}
    ||\mf C+\mf E||_{F}^2-\sigma^2_{sv,1}(\mf C+\mf E)=||P_{\mf u}^{\perp} \mf EP_{\mf v}^{\perp}||^2_{F}+O_{\mathbb{P}_{|S_n}}(||\mf E||^3_{F})
\end{align}
Observing that 
\begin{align}
&\text{vec}(P_{\mf u}^{\perp}\mf EP_{\mf v}^{\perp}) =(P_{\mf u}^{\perp}\otimes P_{\mf v}^{\perp})\text{vec}(\mf E),  
\end{align}
and (note that 
the matrix $P_{\mf u}^{\perp}\otimes P_{\mf v}^{\perp}$ is symmetric and idempotent)
\begin{align}
    ||P_{\mf u}^{\perp}\mf EP_{\mf v}^{\perp}||^2_{F}&=||\text{vec}(P_{\mf u}^{\perp}\mf EP_{\mf v}^{\perp})||_F^2=\text{vec}(\mf E)^{\top}(P_{\mf u}^{\perp}\otimes P_{\mf v}^{\perp})\text{vec}(\mf E)=\text{vec}(\mf E)^{\top}\mathbf{P} \text{vec}(\mf E),
\end{align}\
where $\mathbf{P}=P_{\mf u}^{\perp}\otimes P_{\mf v}^{\perp}$, 
 we obtain
\begin{align}
    \widehat{D}=||\mf C+\mf E||_{F}^2-\sigma^2_{sv,1}(\mf C+\mf E)=\text{vec}(\mf E)^{\top} \mathbf P  \text{ vec}(\mf E)+o_{\mathbb{P}_{|S_n}}(||\text{vec}(\mf E)||^2).  
\end{align}
Therefore, Slutsky lemma and  \cref{Theorem MCLT Covariance Estimator} yield
\begin{align}
    (nT)^2b^{3}\widehat{D}&=(nTb^{1.5}\text{vec}(\mf E))^{\top}\mf P(nTb^{1.5}\text{vec}(\mf E))+o_{\mathbb{P}_{|S_n}}(1)\\
    &=(nTb^{1.5}\text{vec}(\widehat {\mf C}-\mf C))^{\top}\mf P(nTb^{1.5}\text{vec}(\widehat{\mf C}-\mf C))+o_{\mathbb{P}_{|S_n}}(1) \\
    & \overset{d}{\to}  {\cal L}_{sv}:=\tau^2\mf G^{\top}  {\mf P}  \mf G  \ \ \ a.s.-[\mathbb{Q}],
\end{align}
which completes  the proof.

\subsection{\texorpdfstring{Proof of Theorem \ref{th relevant test SVD}}{Proof of Theorem 4.4}} 

 For any given $\mf X\in \mathbb{R}^{M\times N}$, we define  
    \begin{align}
        T(\mf X)=||\mf X||^2_{F}-\sigma^2_{sv,1}(\mf X)=:T_{1}(\mf X)-T_{2}(\mf X)\ 
    \end{align}
    and note that $ \widehat{D}=T(\widehat {\mf C})$.
Moreover, by letting $\mf  R=\mf C-\sigma_{sv,1}(\mf C)\u \mf v^{\top}$ (with $\mf u, \mf v$ denoting the singular vectors corresponding the largest singualr value), we have $D=|| \mf R||_{F}^2$ and $\langle \mf R, \mf u\mf v^{\top}\rangle_{F}=0$, where the latter one is based on the projection theorem in Hilbert spaces.  Simple algebra shows 
\begin{align}
\label{nhd7}
    \widehat{D}-D&=T_{1}(\widehat {\mf C})-T_{2}(\widehat {\mf C})-||\mf  R||_{F}^2 =T_{1}(\mf C+\mf E)-T_{2}(\mf C+\mf E)-||\mf R||^2_{F}
\end{align}
where $ \mf E =  \widehat{\mf C} - \mf C$. For the first term in this representation  (using  $\langle \mf R, \mf u \mf v^{\top}\rangle_{F}=0$) 
\begin{align}
    T_{1}(\mf C+\mf  E)&=||\mf C||^2_{F}+||\mf E||^2_{F}+2\langle \mf C, \mf E \rangle_{F}\\
    &=||\sigma_{sv,1}(\mf C)\u \mf v^{\top}||_{F}^2+||\mf R||^2_{F}+||\mf E||^2_{F}+2\langle \mf \sigma_{sv,1}(\mf C)\u \mf v^{\top}+\mf R,\mf E \rangle_{F}.
\end{align}
Since \cref{Theorem MCLT Covariance Estimator} implies the asymptotic normality of $nTb^{1.5}\text{vec}(\mf E)$, we obtain
\begin{align}
    nTb^{1.5}T_{1}(\mf C+\mf E)= nTb^{1.5}\sigma_{sv,1}^2(\mf C)+ nTb^{1.5}||\mf R||^2_{F}+o_{\mathbb{P}_{|S_n}}(1)+2 nTb^{1.5}\langle \mf \sigma_{sv,1}(\mf C)\u \mf v^{\top}+\mf R,\mf E \rangle_{F}. \label{eq proof of Prop 2.2 eq 1}
\end{align}
For the second term in \eqref{nhd7} we note that the assumption $\sigma_{sv,1}(\mf C)>\sigma_{sv,2}(\mf C)$ implies  that the mapping $
\mf C  \to \sigma^2_{sv,1} (\mf C) $ is Fréchet differentiable, which gives the expansion 
\begin{align}
    \sigma^2_{sv,1}(\mf C+ \mf E)=\sigma^2_{sv,1}(\mf C)+2\sigma_{sv,1}(\mf C)\langle \u\mf v^{\top}, \mf  E\rangle_F +o(||\mf  E||_{F}).
\end{align}
Therefore, we obtain 
\begin{align}
    nTb^{1.5}T_{2}(\mf C+\mf  E)=nTb^{1.5}\sigma^2_{sv,1}(\mf C)+2nTb^{1.5}\sigma_{sv,1}(\mf C)\langle \u\mf v^{\top}, \mf  E\rangle+o_{\mathbb{P}_{|S_n}}(1), \label{eq proof of Prop 2.2 eq 2}
\end{align}
and  combining this result with \eqref{eq proof of Prop 2.2 eq 1}  yields
\begin{align}
    nTb^{1.5}(\widehat{D}-D)=2nTb^{1.5}\langle \mf R,\mf E\rangle_{F}+o_{\mathbb{P}_{|S_n}}(1)=2 \text{vec}(\mf  R)^{\top}(nTb^{1.5}\text{vec}(\mf E))+o_{\mathbb{P}_{|S_n}}(1).
\end{align}
Finally,  Slutsky's lemma and the properties of the standard normal vector give
\begin{align}
    nTb^{1.5}(\widehat{D}-D)\xrightarrow[]{d} N(0,(\sigma/\mf A_1)^2||\mf Q||_F^2), \ \ a.s.-[\mathbb{Q}],
\end{align}
where $\mf Q=
\mf C-\sigma_{sv,1}(\mf C)\u \mf v^{\top}$. Together with the fact that $(\sigma/\mf A_1)^2||\mf Q||_F^2=4\tau^2||\mf Q||_F^2=\vartheta_{sv}^2$, we finish the proof.

\newpage 

\newpage
 \section{Almost Surely CLT For Triangular Array}
 \label{secb}

 \def\theequation{B.\arabic{equation}}
\setcounter{equation}{0}
\par Suppose $n\in\mathbb{N}$ and $m=m_n$ is a positive number sequence such that $n \to \infty$, $\lim_n m_n\nearrow\infty$. Let 
$\mathcal{P}_{n,m_n} := \{\mathbb{P}_{\theta_{n}}: \theta_{n}\in \Theta_{n,m_n}\}$ denote  a class of distributions on $\mathbb{R}^n$, such that  each element of $\mathcal{P}_{n,m_n}$, $\mathbb{P}_{\theta_{n}}$, is uniquely determined by a parameter $\theta_{n}\in \Theta_{n,m_n}\subset \mathbf{\Theta}$, where   $\mathbf{\Theta}$ is a metric space. Given a probability space $(\Omega,\mathcal{F},\mathbb{P})$, we introduce the following two groups of random variables defined on it.
\begin{itemize}
    \item [R1] Let $\mathbf{O}:= \{O_{n}:n\in\mathbb{N}\}=:\mathbf{O}$ is a sequence of random variables in $ \Theta_{n,m_n}$. We denote the distribution of process $\mathbf{O}$ as $Q$.
    \item [R2] Let $\{W_{in}:1\leq i\leq n;\ n\in\mathbb{N}\}$ be an array of real-valued random variables such that, for  $O_{n}=\theta_n\in\Theta_{n,m_n}$, the joint distribution of random vector $(W_{1n}, W_{2n},...,W_{nn})$ is $\mathbb{P}_{\theta_n}$. 
\end{itemize}
Define $\mu_{in}=E_{\mathbb{P}_{\theta_n}}[W_{in}|\mathcal{B}_{i-1,n}]$ and $\sigma^2_{in}=E_{\mathbb{P}_{\theta_n}}[W^2_{in}|\mathcal{B}_{i-1,n}]-\mu^2_{in}$, where $\mathcal{B}_{in}$ is the Borel field generated by $\{W_{jn}:j\leq  i\}$. For any real-valued random variable $Y$ and Borel field $\mathcal{A}$, $E_{\mathbb{P}_{\theta_n}}[Y]$ and $E_{\mathbb{P}_{\theta_n}}[Y|\mathcal{A}]$ denote  the expectation and conditional expectation of $Y$ with respect to distribution $\mathbb{P}_{\theta_n}(\cdot)$ and conditional distribution $\mathbb{P}_{\theta_n}(\cdot|\mathcal{A})$, respectively.

\par Without loss of generality, we assume the triangular array $\{W_{in}\}$ and the  process $\mathbf{O}$ are defined on the same probability space $(\Omega, \mathcal{F},\mathbb{P})$.

\begin{theorem}
    \label{theorem CLT 1}
    Assume the triangular array $\{W_{in}:1\leq i\leq n;\ n\in\mathbb{N}\}$ satisfies the following three conditions.
    \begin{itemize}
        \item [C1] For each $n$ and $i=1,2,...,n$, $\mu_{in}= 0$ holds almost surely with respect to $Q$.
        \item [C2] For each $n$, $\sum_{i=1}^{n}\sigma^2_{in}= 1$ holds almost surely with respect to $Q$.
        \item [C3] For every $\epsilon>0$, $\lim_{n\to\infty}\sum_{i=1}^{n}E_{\mathbb{P}_{\theta_n}}[W^2_{in}1[|W_{in}|>\epsilon]]=0$ holds almost surely with respect to $Q$.
    \end{itemize}
    Then, by denoting $S_{n}=\sum_{i=1}^nW_{in}$, for given any $x\in\mathbb{R}$, 
    \begin{align}
        \lim_{n\to \infty} \mathbb{P}_{\theta_n}(S_n\leq x)=\Phi(x)
    \end{align}
    holds almost surely with respect to $Q$, where $\Phi(x)$ is the cumulative distribution function of standard normal distribution.
\end{theorem}

\begin{theorem}
    \label{theorem CLT 2}
    Assume the triangular array $\{W_{in}:1\leq i\leq n;\ n\in\mathbb{N}\}$ satisfies the following three conditions.
    \begin{itemize}
        \item [(D1)] $\sum_{i=1}^{n}\mu_{in}\xrightarrow[n\to\infty]{\mathbb{P}} 0$ holds almost surely with respect to $Q$.
        \item [(D2)] $\sum_{i=1}^{n}\sigma^2_{in}\xrightarrow[n\to\infty]{\mathbb{P}} 1$ holds almost surely with respect to $Q$.
        \item [(D3)] For every $\epsilon>0$, $\lim_{n\to\infty}\sum_{i=1}^{n}E_{\mathbb{P}_{\theta_n}}[W^2_{in}1[|W_{in}|>\epsilon]|\mathcal{B}_{i-1,n}]=0$ holds almost surely with respect to distribution $Q$. 
    \end{itemize}
    Then, for given any $x\in\mathbb{R}$, 
    \begin{align}
        \lim_{n\to \infty} \mathbb{P}_{\theta_n}(S_n\leq x)=\Phi(x)
    \end{align}
    holds almost surely with respect to $Q$.
\end{theorem}

\noindent 
Our proofs of Theorem \ref{theorem CLT 1} and \ref{theorem CLT 2} are very similar to the proof of Theorems 2.1 and 2.2 in \cite{dvoretzky1972asymptotic}. However,  the background setting is different and we use more constructive (intuitive) method for  the proof. Thus, we still demonstrate the complete proof here for the sake of  completeness.

\subsection{\texorpdfstring{Proof of Theorem \ref{theorem CLT 1}}{Proof of Theorem B.1}}
 \label{secb1}
 
\par We decompose the proof into multiple steps. In each step, we first brief the goal we aim to achieve and then show the details. 
\medskip

\par \textbf{Step 1} In the first step, we prove the following proposition is true. 
\par 
\begin{proposition} Suppose $(\Omega,\mathcal{F},\mathbb{P})$ is a probability space and, for any given $n\geq 2$, $\{Z_{i}:1\leq i\leq n\}$ is a group of real-valued random variables defined on the probability $(\Omega,\mathcal{F},\mathbb{P})$ such that $\mathbb{P}_{\theta_i}$ is the distribution of $Z_{i}$, where $\mathbb{P}_{\theta_i}$ is introduced above. Then, there exists a probability space $(\widetilde{\Omega},\widetilde{\mathcal{F}},\widetilde{\mathbb{P}})$ and random variables $\widetilde{Z}_{1}$,...,$\widetilde{Z}_{n}$ defined on  $(\widetilde{\Omega},\widetilde{\mathcal{F}},\widetilde{\mathbb{P}})$ satisfying the following two conditions. 
\begin{itemize}
    \item[(1)] $\mathbb{P}(Z_{i} \leq  u, Z_{i-1} \leq v])=\widetilde{\mathbb{P}}(\widetilde{Z}_{i}\leq u,\widetilde{Z}_{i-1}\leq v)$ for all  $i=2,...,n$ and all  $u,v \in \mathbb R$.
    \item[ (2)] $\widetilde{\mathbb{P}}(\widetilde{Z}_{i}\leq u|\mathcal{B}(\widetilde{Z}_{1},...,\widetilde{Z}_{i-1}))=\widetilde{\mathbb{P}}(\widetilde{Z}_{i}\leq u|\mathcal{B}(\widetilde{Z}_{i-1}))$. 
    \end{itemize}
\end{proposition}

\begin{proof}
     When $n=2$, we only need to define $(\widetilde{Z}_{1},\widetilde{Z}_{2}): \mathbb{R}^2\to \mathbb{R}^2$ and set $\widetilde{\mathbb{P}}$ as the joint distribution of $(Z_{1},Z_{2})$. Then, we can regard $\{\widetilde{Z}_{1},\widetilde{Z}_{2}\}$ as an array of random variables defined on probability space $(\mathbb{R}^2,\mathcal{B}(\mathbb{R}^2),\widetilde{\mathbb{P}})=:(\widetilde{\Omega}_2,\widetilde{\mathcal{F}}_2,\widetilde{\mathbb{P}}_2)$. It is easy to check that  $\{\widetilde{Z}_{1},\widetilde{Z}_{2}\}$ and $(\widetilde{\Omega}_2,\widetilde{\mathcal{F}}_2,\widetilde{\mathbb{P}}_2)$ are the variables and space satisfying (1) and (2). When $n=3$, the previous argument implies that $Z_{i}$ has the same distribution as $\widetilde{Z}_{i}$, $i=1,2$, and $\mathbb{P}(Z_2\leq u|Z_{1}=v)=\widetilde{\mathbb{P}}(\widetilde{Z}_{2}\leq u|\widetilde{Z}_1=v)$. Define $F_{2}(u|v)=\mathbb{P}(Z_3\leq u|Z_{2}=v)$ and $F_2^{-1}(u|v)$ as the corresponding conditional quantile function. Define $U:[0,1]\to [0,1]$ as a uniform random variable. Let $\widetilde{\Omega}=\widetilde{\Omega}_2\times [0,1]$, $\widetilde{\mathcal{F}}=\widetilde{\mathcal{F}}_2\times \mathcal{B}[0,1]$ and $\widetilde{\mathbb{P}}=\widetilde{\mathbb{P}}_2 \bigotimes\mathbb{P}_U$, where $\mathbb{P}_U$ is the distribution function of random variable $U$ and $\mathcal{B}[0,1]$ is the Borel field on $[0,1]$. Then,  for any given $v$, $F_2^{-1}(U|v)$ is a real-valued random variable whose distribution is $F_{2}(u,v)$. Let $\widetilde{Z}_{3}=F^{-1}_2(U,\widetilde{Z}_2)$. Then, $\widetilde{Z}_{1}$, $\widetilde{Z}_{2}$ and $\widetilde{Z}_{3}$ are well defined random variables on probability space $(\widetilde{\Omega},\widetilde{\mathcal{F}},\widetilde{\mathbb{P}})$ and 
\begin{align}
    \widetilde{\mathbb{P}}(\widetilde{Z}_{3}\leq u|\widetilde{Z}_2=v)=&E[1[F^{-1}_{2}(U,\widetilde{Z}_{2})\leq u]|\widetilde{Z}_{2}=v]\\
=&E[1[Z_{3}\leq u]|Z_{2}=\widetilde{Z}_{2}=v]=\mathbb{P}(Z_{3}\leq u|Z_{2}=v).
\end{align}
Furthermore, since $U$ is independent of the Borel field $\mathcal{B}(\widetilde{Z}_{1},\widetilde{Z}_2)$, $\widetilde{\mathbb{P}}(\widetilde{Z}_3\leq u|\mathcal{B}(\widetilde{Z}_2, \widetilde{Z}_1))=\widetilde{\mathbb{P}}(\widetilde{Z}_3\leq u|\mathcal{B}(\widetilde{Z}_2))$ holds almost surely. The assertion now follows by an induction argument.
\end{proof}

\par \textbf{Step 2} In this step, we construct normal approximation based on proper introduction of independent normal random variables. Moreover, all of our arguments in this step holds for any given $n$ and $O_n=\theta_n$.

\par Based on the $\mathcal{B}_{in}$ introduced before, we have
\begin{align}
    \mathcal{B}_{1n}\subset \mathcal{B}_{2n}\subset \dots \subset\mathcal{B}_{kn}\subset \dots\subset \mathcal{B}_{nn}
\end{align}
and $S_{kn}:=\sum_{i=1}^kW_{in}$ is measurable with respect to $\mathcal{B}_{kn}$. The validity of this statement is independent of the realization of random variable $O_n$. Meanwhile, we introduce a group of mutually independent standard normal random variables, denoted as $\{G_{i}:1\leq i\leq n\}$, and let them be independent of sigma algebra $\mathcal{B}_{nn}$ and the sigma algebra generated by $\{O_n\}$. Let $V_{in}=\sigma_{in}G_{i}$ and $T_{kn}=\sum_{i=k+1}^{n}\sigma_{in}G_{i}$. Then, based on condition C2 and independence among $G_{i}$'s, there exists $\mathbf{\Theta}_0\subset \mathbf{\Theta} $ such that $Q(\mathbf{\Theta}_0)=0$ and, for every $\mathbf{O}=\theta\in \mathbf{\Theta}\backslash\mathbf{\Theta}_0$, $T_{0n}=:T_{n}$ is a standard normal random variable. Then, for any given $u\in\mathbb{R}$ and realization $\theta$, by setting $S_{0n}=T_{nn}=0$, we have
\begin{align}
  \triangle(u):=&| E_{\mathbb{P}_{\theta_n}}[\exp (iuS_n)]-\exp(-\frac{u^2}{2})|= |E_{\mathbb{P}_{\theta_n}}[\exp (iuS_n)]-E_{\mathbb{P}_{\theta_n}}[\exp(iuT_n)]|\\
  =&|E_{\mathbb{P}_{\theta_n}}[\exp (iu(S_{nn}+T_{nn}))]-E_{\mathbb{P}_{\theta_n}}[\exp(iu(T_{0n}+S_{0n}))]|\\
  \leq &\sum_{k=1}^n|E_{\mathbb{P}_{\theta_n}}[\exp (iu(S_{kn}+T_{kn}))]-E_{\mathbb{P}_{\theta_n}}[\exp(iu(S_{k-1,n}+T_{k-1,n}))]|\\
  =&\sum_{k=1}^n|E_{\mathbb{P}_{\theta_n}}[\exp (iu(S_{k-1,n}+T_{kn}))(\exp(iu W_{kn})-\exp(iu \sigma_{kn}G_{k}))]|.
\end{align}
According to Lemma 3.2 in \cite{dvoretzky1972asymptotic}, for any $n$ and $k=1,...,n$, there exists a standard normal random variable $\Phi_{k+1}$ independent of $\mathcal{B}(G_j:j\leq k)\times \mathcal{B}_{nn}$ such that
\begin{align}
    T_{kn}=\Big(\sum_{i=k+1}^{n}\sigma^2_{in}\Big)^{\frac{1}{2}}\Phi_{k+1}.
\end{align}
Then, due to the aforementioned independence of $\Phi_{k+1}$, we have
\begin{align}
    &|E_{\mathbb{P}_{\theta_n}}[\exp (iu(S_{k-1,n}+T_{kn}))(\exp(iu W_{kn})-\exp(iu \sigma_{kn}G_{k}))]|\\
    \leq &\int_{\mathbb{R}}|E_{\mathbb{P}_{\theta_n}}[\exp \{iu(S_{k-1,n}+(\sum_{i=k+1}^{n}\sigma^2_{in})^{\frac{1}{2}}x)\}(\exp(iu W_{kn})-\exp(iu \sigma_{kn}G_{k}))]|d\Phi(x).
\end{align}
By denoting $\mathcal{H}_{k+1}=\mathcal{B}(G_{j}:j\geq k+1)$, this asserts
\begin{align}
    \triangle(u) \leq &\sum_{k=1}^n\int_{\mathbb{R}}\Big|E_{\mathbb{P}_{\theta_n}}[\exp \{iu(S_{k-1,n}+(\sum_{i=k+1}^{n}\sigma^2_{in})^{\frac{1}{2}}x)\}\\& \times E_{\mathbb{P}_{\theta_n}}[(\exp(iu W_{kn})-\exp(iu \sigma_{kn}G_{k}))|\mathcal{B}_{k-1,n}\times \mathcal{H}_{k+1}]]\Big|d\Phi(x)\\
  \leq &\sum_{k=1}^nE_{\mathbb{P}_{\theta_n}}\Big|E_{\mathbb{P}_{\theta_n}}[(\exp(iu W_{kn})-\exp(iu \sigma_{kn}G_{k}))|\mathcal{B}_{k-1,n}\times \mathcal{H}_{k+1}]\Big|\\
    =&\sum_{k=1}^nE_{\mathbb{P}_{\theta_n}}\Big|E_{\mathbb{P}_{\theta_n}}[(\exp(iu W_{kn})-\exp(iu \sigma_{kn}G_{k}))|\mathcal{B}_{k-1,n}]\Big|=:\triangle^*(u),
\end{align}
where the last equality is according to the independence among $\mathcal{B}_{k}$, $\mathcal{H}_{k+1}$ and $\mathcal{B}(G_k)$.
\medskip

\par \textbf{Step 3} The main goal of this step is to show that, for every $u$, $\triangle'(u)$ converges to $0$ holds almost surely with respect to measure $Q$.
\smallskip

\par According to condition C1, there exists some $\mathbf{\Theta}_1\subset \mathbf{\Theta}$ such that $Q(\mathbf{\Theta}_1)=0$ for every $\mathbf{O}=\theta\in\mathbf{\Theta}\backslash(\mathbf{\Theta}_1 \bigcup\mathbf{\Theta}_0)$, $\mu_{in}=0$. Recall that, for any $x\in \mathbb{R}$, 
\begin{align}
    \Big|e^{ix}-\sum_{j=0}^{n}\frac{(ix)^n}{j!}\Big|\leq \min\Big\{\frac{|x|^{n+1}}{(n+1)!},\frac{2|x|^n}{n!}\Big\}.
\end{align}
The results above imply that, for every $n$, $k=1,2,...,n$ and $\epsilon>0$, 
\begin{align}
\label{inequality}
        &|E_{\mathbb{P}_{\theta_n}}[\exp(iu W_{kn})-1-\frac{1}{2}u^2\sigma^2_{kn}|\mathcal{B}_{k-1,n}]| \\
    \leq &\frac{1}{6}|u|^3E[|W_{kn}|^31[|W_{kn}|\leq \epsilon]|\mathcal{B}_{k-1,n}]+u^2E[W^2_{kn}1[|W_{kn}|> \epsilon]|\mathcal{B}_{k-1,n}] \\
    \leq & \frac{1}{6}|u|^3\epsilon\sigma_{kn}^2+u^2E[W^2_{kn}1[|W_{kn}|> \epsilon]|\mathcal{B}_{k-1,n}]
\end{align}
holds for every $\theta\in \mathbf{\Theta}\backslash(\mathbf{\Theta}_0\bigcup\mathbf{\Theta}_1)$, where $Q(\mathbf{\Theta}_1 \bigcup\mathbf{\Theta}_0)=0$. Then,
\begin{align}
    &\sum_{k=1}^{n}E_{\mathbb{P}_{\theta_n}}|E_{\mathbb{P}_{\theta_n}}[\exp(iu W_{kn})-1-\frac{1}{2}u^2\sigma^2_{kn}|\mathcal{B}_{k-1,n}]|\\
    \leq& \frac{1}{6}|u|^3\epsilon+u^2\sum_{k=1}^nE[W_{kn}^21[|W_{kn}|>\epsilon]],\ \forall\ \theta\in \mathbf{\Theta}\backslash(\mathbf{\Theta}_0\bigcup\mathbf{\Theta}_1).
\end{align}
Together with C3, for every $u\in\mathbb{R}$ and $\epsilon>0$,
\begin{align}
\label{star 1}
    \lim_{n\to\infty}\sum_{k=1}^{n}E_{\mathbb{P}_{\theta_n}}|E_{\mathbb{P}_{\theta_n}}[\exp(iu W_{kn})-1-\frac{1}{2}u^2\sigma^2_{kn}|\mathcal{B}_{k-1,n}]|\leq \frac{1}{6}|u|^3\epsilon
\end{align}
holds almost surely with respect to $Q$. 

\par On the other hand, based on the normality of $\sigma_{kn}G_{k}$ and the independence between $\mathcal{B}_{k-1,n}$ and $\mathcal{B}(G_k)$, repeating the argument above indicates
\begin{align}
    \label{star 2}
    &\sum_{k=1}^nE_{\mathbb{P}_{\theta_n}}|E_{\mathbb{P}_{\theta_n}}[\exp(iu \sigma_{kn}G_{k})-1-\frac{1}{2}u^2\sigma^2_{kn}|\mathcal{B}_{k-1,n}]| \\
    =&\sum_{k=1}^nE_{\mathbb{P}_{\theta_n}}\Big|\exp\Big(-\frac{\sigma^2_{kn}u^2}{2}\Big)-1-\frac{1}{2}u^2\sigma^2_{kn}\Big|\leq 3u^4\sum_{k=1}^nE[\sigma^4_{kn}]\\
    \leq& 3u^4E\Big[(\max_{1\leq k\leq n}\sigma^2_{kn})\Big(\sum_{k=1}^n\sigma_{kn}^2\Big)\Big]=3u^4E[\max_{1\leq k\leq n}\sigma^2_{kn}]
\end{align}
holds almost surely with respect to $Q$. Meanwhile, please note that C3 implies $\lim_{n\to\infty}E[\max_{1\leq k\leq n}W_{kn}^2]=0$. Using Jensen inequality asserts that, for every $u$,  
\begin{align}
\label{star 3}
    \lim_{n}3u^4E[\max_{1\leq k\leq n}\sigma^2_{kn}]=0
\end{align}
holds almost surely with respect to $Q$. 

\par Finally, by combining \eqref{star 1}-\eqref{star 3} and Jensen inequality, we obtain, for every $u\in\mathbb{R}$, $\lim_{n}\triangle'(u)=0$ holds almost surely with respect to $Q$, which finishes the proof of Theorem \ref{theorem CLT 1}.\\

\subsection{\texorpdfstring{Proof of Theorem \ref{theorem CLT 2}}{Proof of Theorem B.2}}
 \label{secb2}
 
\par Similar to the proof of Theorem 2.2 in \cite{dvoretzky1972asymptotic}, we transform conditions (D1)-(D3) to some settings similar to conditions C1-C3 in Theorem \ref{theorem CLT 1} and then apply Theorem \ref{theorem CLT 1} to finish the proof.

\par First, based on the definition of $\mu_{in}$, we have $W_{in}=W_{in}-\mu_{in}+\mu_{in}=:\bar{W}_{in}+\mu_{in}$. Condition (D1) indicates that there exists $\mathbf{\Theta}_2\subset \mathbf{\Theta}$ such that $Q(\mathbf{\Theta}_2)=Q(\mathbf{\Theta})$ and, for each $\theta_n\in\mathbf{\Theta}_2$, $|\sum_{i=1}^nW_{in}-\sum_{i=1}^n\bar{W}_{in}|=o_{\mathbb{P}_{\theta_n}}(1)$. Regarding $E[\bar{W}_{in}]=0$, we only need to prove Theorem \ref{theorem CLT 2} under conditions C1, (D2) and (D3). More specifically, from now on, we treat $W_{in}$ as random variable whose $\mu_{in}=E[W_{in}|\mathcal{B}_{in}]=0$ and focus on the case where $\theta_n\in\mathbf{\Theta}_{2}$.

\par Now we introduce random variable $t_{sn}=\max\{m\leq n: \sum_{i=1}^m\sigma^2_{in}\leq s \}$, $s\in\mathbb{N}$. Define 
\begin{align}
           &\widetilde{W}_{in}=W_{in}\ \ \ \text{if}\ i\leq t_{sn}\\
           &\widetilde{W}_{in}=0\ \ \ \ \ \ \ \text{otherwise},\\
\text{and}\ &\widetilde{W}_{n+1,n}=\Big(1-\sum_{i=1}^{t_{1n}}\sigma^2_{in}\Big)G,
\end{align}
where $G$ is a standard normal variable independent of $\mathcal{B}_{nn}$. Since $t_{sn}$ is measurable with respect to $\mathcal{B}_{nn}$ for each $n$ for each positive integer $s$, it can be shown that the following equations hold almost surely with respect to $\mathbb{P}_{\theta_n}$ for every $\theta_n\in \mathbf{\Theta}$,
\begin{align}
\label{star 4}
    &E[\widetilde{W}_{n+1,n}|\mathcal{B}_{nn}]=\Big(1-\sum_{i=1}^{t_{1n}}\sigma^2_{in}\Big)E[G]=0,\\
    \label{star 5}
    &E[\widetilde{W}^2_{n+1,n}|\mathcal{B}_{nn}]=\Big(1-\sum_{i=1}^{t_{1n}}\sigma^2_{in}\Big)E[G^2]=1-\sum_{i=1}^{t_{1n}}\sigma^2_{in}.
\end{align}
It can be shown that, for every $\mathbf{O}=\theta\in \mathbf{\Theta}_2$,
\begin{align}
\label{star 6}
    &E[\widetilde{W}_{in}|\mathcal{B}_{i-1,n}]=:\widetilde{\mu}_{in}=0,\ i=1,2,...,n;\\
\label{star 7}
    &\sum_{i=1}^{n+1}(E[\widetilde{W}^2_{in}|\mathcal{B}_{i-1,n}]-\widetilde{\mu}_{in})=1.
\end{align}

\par First, we show that, under condition (D2),  $\overline{\lim}_{n\to\infty}\mathbb{P}_{\theta_n}(|t_{1n}-n|\geq 1)=0$ holds almost surely with respect to $Q$. We prove this by showing its contradiction is wrong. Assume event $\overline{\lim}_{n\to\infty}\mathbb{P}_{\theta_n}(|t_{1n}-n|>1)=:\overline{\lim}_{n\to\infty}P_n>0$ has strictly positive probability mass with respect to $Q$. We denote this mass as $\xi>0$ and $\xi$ here is apparently independent of $n$. More specifically, there exists some $\Theta^*\subset \mathbf{\Theta}$ such that $Q(\Theta^*)>0$ and, for every $\mathbf{O}=\theta\in \Theta^*$, we can obtain a subsequence $\{P_{n_l}:l\in\mathbb{N}\}\subset \{P_{n}:n\in\mathbb{N}\}$ such that, for some $\eta>0$ independent of $n$ (and $n_l$), $P_{n_l}>\eta$ holds for every $l$. This means $\mathbb{P}_{\theta_{n_l}}(t_{1n_l}\leq n-1)>\eta$ holds for every $l$ and $\mathbf{O}_{n_l}=\theta_{n_l}\in \Theta^*$. By denoting $M_k=\sum_{j=1}^{k}\sigma^2_{jn}$, $k\in\mathbb{N}$, for each aforementioned $n_l$, the definition of $t_{1n_l}$ indicates
\begin{align}
   \{|t_{1n_l}-n_l|\geq 1\}=\{t_{1n_l}\leq n_l-1\}\subset \{M_{n_l}>1\}.
\end{align}
Thus, $Q(\{\theta:\mathbb{P}_{\theta}(M_{n_l}>1)>\eta\})>0$ holds for every $l$. On the other hand, condition (D2) means
$Q(\{\theta_{n}:\lim_{n}\mathbb{P}_{\theta_n}(|M_{n}-1|>\eta)=0\})=1$. Then, the paradox proves $\overline{\lim}_{n\to\infty}\mathbb{P}_{\theta_n}(|t_{1n}-n|\geq 1)=0$ holds almost surely with respect to $Q$.

\par By denoting $\widetilde{S}_{m}=\sum_{i=1}^m \widetilde{W}_{in}$, $m=1,2,...,n+1$, we have approximation error
\begin{align}
    &S_{n}- \widetilde{S}_{n+1}=\sum_{i=t_{1n}+1}^{t_{2n}} W_{in}+\sum_{i=t_{2n}+1}^{n} W_{in}-\widetilde{W}_{n+1,n}=:\mathbb{D}_{1}+\mathbb{D}_{2}-\mathbb{D}_3.
\end{align}
For $\mathbb{D}_1$, since we have treat $W_{in}$'s as random variables satisfying condition C1, we have
\begin{align}
    E[\mathbb{D}^2_1]&=E[\sum_{i=t_{1n}+1}^{t_{2n}}W^2_{in}]+2\sum_{t_n+1\leq i<i'\leq t_{2n}}E[W_{in}W_{i'n}]\\
    &=E[\sum_{i=t_{1n}+1}^{t_{2n}}\sigma^2_{in}]+2\sum_{t_n+1\leq i<i'\leq t_{2n}}E[W_{in}E[W_{i'n}|\mathcal{B}_{i'-1,n}]]=E[\sum_{i=t_{1n}+1}^{t_{2n}}\sigma^2_{in}].
\end{align}
Condition (D3) indicates that $\sum_{i=t_{1n}+1}^{t_{2n}}\sigma^2_{in}$ is integrable. Then, based on condition (D2) and dominated convergence theorem, $\mathbb{D}_1$ converges in probability to $0$ almost surely with respect to $Q$. 
For $\mathbb{D}_2$, regarding that condition (D2) also indicates $\mathbb{P}_{\theta_n}(t_{2n}\neq n)\to 0$ holds almost surely with respect to $Q$, condition (D3) guarantees $\mathbb{D}_2$ converges in probability to $0$ almost surely with respect to $Q$. This argument holds true for term  $\mathbb{D}_3$ as well due to \eqref{star 4}, \eqref{star 5} and condition (D2). Therefore, we only need to focus on $\widetilde{S}_{n+1}$ and equations \eqref{star 6} and \eqref{star 7} have ensured that array $\{\widetilde{W}_{in}:1\leq i\leq n+1\}$ satisfies condition C1 and C2 of Theorem \ref{theorem CLT 1}. 

\par Another crucial observation is that Theorem \ref{theorem CLT 1} still holds if we replace condition C3 with condition (D3). In the proof of Theorem \ref{theorem CLT 1}, inequality \eqref{inequality} is valid for the approximation of both $E[\exp(iuW_{kn})]$ and $E[\exp(iuG_{k})]$. Then, condition (D3) is sufficient to ensure this approximation of Fourier transformation converges to 0 and thus guarantees the pointwise convergence of $\triangle(u)$ there holds almost surely with respect to $Q$.

\par Therefore, to finish the proof, we need to prove, for any $\epsilon>0$, 
\begin{align}
    \lim_{n\to\infty}\sum_{i=1}^{n+1}E[\widetilde{W}^2_{in}1[|\widetilde{W}_{in}|>\epsilon]|\mathcal{B}_{i-1,n}]=0
\end{align}
holds almost surely with respect to $Q$. Since
\begin{align}
    &\sum_{i=1}^{n+1}E[\widetilde{W}^2_{in}1[|\widetilde{W}_{in}|>\epsilon]|\mathcal{B}_{i-1,n}]\\
    =&\sum_{i=1}^{t_{1n}}E[\widetilde{W}^2_{in}1[|\widetilde{W}_{in}|>\epsilon]|\mathcal{B}_{i-1,n}]+E[\widetilde{W}^2_{n+1,n}1[|\widetilde{W}_{n+1,n}|>\epsilon]|\mathcal{B}_{nn}]\\
    =&\sum_{i=1}^{t_{1n}}E[W^2_{in}1[|W_{in}|>\epsilon]|\mathcal{B}_{i-1,n}]+E[\widetilde{W}^2_{n+1,n}1[|\widetilde{W}_{n+1,n}|>\epsilon]|\mathcal{B}_{nn}]\\
    \leq &\sum_{i=1}^{n}E[W^2_{in}1[|W_{in}|>\epsilon]|\mathcal{B}_{i-1,n}]+E[\widetilde{W}^2_{n+1,n}1[|\widetilde{W}_{n+1,n}|>\epsilon]|\mathcal{B}_{nn}],
\end{align}
combining condition (D3) and \eqref{star 5} finish the proof.

\newpage
\section{Technical Results}\label{secc}

\def\theequation{C.\arabic{equation}}
\setcounter{equation}{0}

\subsection{Propositions}

\begin{proposition}
    \label{prop 5}
    Given sample size $n$, suppose $\{A_{kn}: k=1,2,...,K_n\}$ is a collection  of subsets of sampling region $\mathbf{R}_n$ and define $N_n(A)$ as the number of observed locations in the set $A\subset\mf R_n$. Suppose $\epsilon_n\searrow 0$ and define event
    \begin{align}
        \mathcal{E}_n&=\bigcap_{k=1}^{K_n}\{m_n\leq N_n(A_{kn})\leq M_n\},\\
        M_n&=(1+\log(1+\epsilon_n))n\textbf{Leb}(\lambda_n^{-1}A_{kn}\cap\mathbf{D}_n),\\
        m_n&=(1-\epsilon_n)n\textbf{Leb}(\lambda_n^{-1}A_{kn}\cap\mathbf{D}_n).
    \end{align}
   We further assume $n\textbf{Leb}(\lambda_n^{-1}A_{kn}\cap\mathbf{D}_n)\to \infty$ holds for every $k$ and there exists some $\eta>0$ such that $K_{n}= \mathcal{O}(n^{\eta})$. Then, by letting $\epsilon_n= (\frac{(\eta+2)\log n}{n\textbf{Leb}(\lambda_n^{-1}A_{kn}\cap\mathbf{D}_n)})^{0.5}$, together with Assumptions \ref{as spatial locations 1} and \ref{as spatial locations 1}, we obtain
   \begin{align}
       \mathbb{Q}(\underline{\lim}_{n\to\infty}\mathcal{E}_n)=1,
   \end{align}
    where $A$ is any Borel set $A\subset ([-\lambda_n,\lambda_n]^{2})^d$. $\mathbb{Q}$ is the distribution of vectorized triangular-array of spatial locations, denoted as $(\s_{in}:1\leq i\leq n)_{n\geq 1}$ such that $\mathbb{Q}(B)=\mathbb{P}(\{\omega:(\s_{in}:1\leq i\leq n)_{n\geq 1}(\omega)\in B\})$ holds for any Borel set $B\subset \mathbb{R}^{\infty}$. 
\end{proposition}
Proposition \ref{prop 5} is a direct corollary of Lemma 5.1 of \cite{lahiri2003central} and we thus omit its proof here. This proposition is very crucial for us to derive the asymptotic normality of our estimator. In particular, since big and small block technique (\cite{Bernstein1927}) is a main tool in our proof, Proposition \ref{prop 5} allows us to easily control the number of observed locations in each big or small block.

\begin{proposition}
    \label{prop 6}
    Under Assumption \ref{technical assumptions}, by defining $(\s_{in},t)=:(i,t)$, we obtain
    \begin{align}
       \sum_{(i,t)\in\tilde{B}^-_{k,j}}\sum_{(i',t'),(i'',t'')\in \tilde{B}^c_{k,j} \atop ||(i',t')-(i'',t'')||_{F}\leq q_n}&K_{0b}(\s_{in}-\s_{i'n},t-t')K_{0b}(\s_{in}-\s_{i''n},t-t'')\beta^{\frac{\delta}{1+\delta}}(||(i',t')-(i'',t'')||_{F})\\
       &=o(n\textbf{Card}(\widetilde{B}^{-}_{k,j}\cap\Gamma_{n})Tp_T b^4),\ \ \  a.s.-[\mathbb{Q}],
    \end{align}
    where “$a.s.-[\mathbb{Q}]$” indicates “almost surely with respect to measure $\mathbb{Q}$” and $\textbf{Card}(\widetilde{B}^{-}_{k,j}\cap\Gamma_{n})=\sum_{i=1}^{n}1[\s_{in}\in\widetilde{B}_{k,j}^-]$ is a random variable. 
\end{proposition}

\begin{proof}
    Note that,  for each $(k,j)$, we have the following representation.
\begin{align}
    \widetilde{B}^{-}_{k,j}=\widetilde{B}^{-}_{k,j,S}\times \widetilde{B}^{-}_{k,j,T},\\
    \widetilde{B}^{c}_{k,j}=\widetilde{B}^{c}_{k,j,S}\times \widetilde{B}^{c}_{k,j,T},
\end{align}
where the $\widetilde{B}^{-}_{k,j,S}$ and $\widetilde{B}^{-}_{k,j,T}$ denote the projections of the set $\widetilde{B}^{-}_{k,j}$ onto the spatial plane and time axis respectively. So do the $\widetilde{B}^{c}_{k,j,S}$ and $\widetilde{B}^{c}_{k,j,T}$. Let
\begin{align}
    &K_{0ir,S}=K \Big (\frac{\lambda_n^{-1}(\mf s_{in,1}-\mf s_{rn,1}-h_{0,1})}{b} \Big )K\Big (\frac{(\lambda_n\bb_n)^{-1}(\mf s_{in,2}-\mf s_{rn,2}-h_{0,2})}{b} \Big ) , \label{def:KS}\\    
& K_{0ts, T}=K \Big (\frac{|t-s|-v_0}{Tb} \Big )\label{def:KT}.
\end{align}
Additionally, basic geometrical properties of Euclidean space asserts that, for any $(\mathbf{u},v)\neq (\mathbf{u}',v')\in\mathbb{R}^3$ and $\delta>0$, 
\begin{align}
    \beta^{\frac{\delta}{1+\delta}}(||(\mathbf{u},v)-(\mathbf{u}',v')||_E)\leq \beta^{\frac{\delta}{1+\delta}}(C_E(||\mathbf{u}-\mathbf{u}'||_E+|v-v'|)),
\end{align}
where $C_{E}$ is a constant independent of $n$, $\mathbf{u}$ and $v$.

\par According to (T2), (i) of Assumption \ref{technical assumptions}, we obtain 
\begin{align}
    &\sum_{(i,t)\in\tilde{B}^-_{k,j}}\sum_{(i',t'),(i'',t'')\in \tilde{B}^c_{k,j} \atop ||(i',t')-(i'',t'')||_E\leq q_n}K_{0b}(\s_{in}-\s_{i'n},t-t')K_{0b}(\s_{in}-\s_{i''n},t-t'')\beta^{\frac{\delta}{1+\delta}}(||(i',t')-(i'',t'')||_E)\\
    \lesssim&  \sum_{\s_{in}\in \widetilde{B}^-_{k,j,S}}\sum_{ \s_{i'n},\s_{i''n}\in \widetilde{B}^c_{k,j,S} \atop ||\s_{i'n}-\s_{i''n}||_E\leq q_n}K_{0ii', S}K_{0ii'', S}\beta^{\frac{\delta}{2(1+\delta)}}(||\s_{i'n}-\s_{i''n}||_E)  \\&\times \sum_{t\in\widetilde{B}^-_{k,j,T}}\sum_{t',t''\in \widetilde{B}^c_{k,j,T}\atop |t'-t''|\leq q_n}K_{0tt', T}K_{0tt'', T}\beta^{\frac{\delta}{2(1+\delta)}}(|t'-t''|)\\
    =&:\mathbf{K}_1\cdot \mathbf{K}_2.
\end{align}

First, we bound $\mathbf{K}_2$. Simple algebra, observing the design of the blocks and  Lemma \ref{lemma 4} yield 
\begin{align}
    \mathbf{K}_2 & \lesssim\sum_{t\in\widetilde{B}^-_{k,j,T}}\sum_{t'\in\widetilde{B}^c_{k, j, T}}K_{0tt',T} \sum_{t''\in \widetilde{B}^c_{k,j,T}\atop 1\leq t''-t'\leq q_n}K_{0tt'', T}\beta^{\frac{\delta}{2(1+\delta)}}(t''-t')  \\
    & \lesssim\sum_{t\in\widetilde{B}^-_{k,j,T}}\sum_{t'\in\widetilde{B}^c_{k,j, T}}K_{0tt', T}\lesssim Tp_Tb.
\end{align}

Recall that $\{\s_{1n},...,\s_{nn}\}=:\Gamma_{n}$. As for $\mathbf{K}_1$, we introduce sets 
\begin{align}
    M_{i'l} &=\{\s_{i''n}\in \tilde{B}_{k,j}^c\cap \Gamma_n:||\s_{i''n}-\s_{i'n}||_E\leq (l,l+1]\},\\
    \widetilde{M}_{i'l} & =\{\s\in \tilde{B}_{k,j}^c\cap \R_n:||\s-\s_{i'n}||_E\leq (l,l+1]\},\ l\geq 1;\\
     M_{i'0} & =\{\s_{i''n}\in \tilde{B}_{k,j}^c\cap \Gamma_n:||\s_{i''n}-\s_{i'n}||_E<  1\},\\
     \widetilde{M}_{i'0} & =\{\s\in \tilde{B}_{k,j}^c\cap \R_n:||\s-\s_{i'n}||_E< 1\}.
\end{align}
We also define the following events used to control the number of locations in each block simultaneously. 
\begin{align}
    \mathcal{M}_n(k,j)=&\bigcap_{\s_{in}\in\widetilde{B}^-_{k,j}}\bigcap_{\s_{i'n}\in \widetilde{B}^c_{k,j}}\bigcap_{l=1}^{[q_n]}\Big\{n\textbf{Leb}(\Lambda_n^{-1}\widetilde{M}_{i'l}\cap D_n)\leq \textbf{Card}(M_{i'l})\leq \frac{2nl}{\lambda_n^2\bb_n}\Big\}\notag\\
    &\bigcap\Big\{n\textbf{Leb}(\Lambda_n^{-1}\widetilde{M}_{i'0}\cap D_n)\leq \textbf{Card}(M_{i'0})\leq \frac{2n}{\lambda_n^2\bb_n} \Big\},\notag\\
    \mathbf{B}^-_{n}(k,j)=&\Big\{n\textbf{Leb}(\Lambda_n^{-1}\widetilde{B}^-_{k,j}\cap D_n)\leq  \textbf{Card}(\widetilde{B}^-_{k,j}\cap\Gamma_n)\leq \frac{nr_{1n}r_{2n}}{\lambda_n^2\bb_n}  \Big\},\notag\\
    \mathbf{B}_{n}(k,j)=&\Big\{n\textbf{Leb}(\Lambda_n^{-1}\widetilde{B}_{k,j}\cap D_n)\leq  \textbf{Card}(\widetilde{B}_{k,j}\cap\Gamma_n)\leq \frac{nr_{1n}r_{2n}}{\lambda_n^2\bb_n} \Big\}, \label{def:Bkj}\\
    \mathbf{B}^c_{n}(k,j)=&\Big\{n\textbf{Leb}(\Lambda_n^{-1}\widetilde{B}^c_{k,j}\cap D_n)\leq  \textbf{Card}(\widetilde{B}^c_{k,j}\cap\Gamma_n)\leq n\Big\},\notag\\ \mathcal{M}_{n}=&\bigcap_{k=1}^{R_n}\bigcap_{j=1}^{J_n}\Big(\mathcal{M}_{n}(k,j)\bigcap \mathbf{B}_{n}(k,j)\bigcap \mathbf{B}^-_{n}(k,j)\bigcap\mathbf{B}^c_{n}(k,j)\Big) \notag,
\end{align}
where $R_n$ and $J_n$ are introduced above. Based on the definition of $R_n$ and $J_n$, we can regard $\mathcal{M}_n$ as an intersection of $K_n$-many $\mathcal{M}_{n}(k,j)$'s and $K_{n}\lesssim n^{\eta_1}$, for some $\eta_1>0$ independent of $n$. Then Proposition \ref{prop 5} asserts that 
\begin{align}
    \mathbb{Q}(\overline{\lim}_{n}\mathcal{M}^c_{n})=0.
\end{align}
For any set $A\subset \R_n$, by abbreviating $\textbf{Card}(A\cap\Gamma_n)$ as $|A|$, we have
\begin{align}
    \mathbf{K}_1&=\sum_{\s_{in}\in\widetilde{B}^-_{k,j}}\sum_{\s_{i'n}\in\widetilde{B}^c_{k,j}}K_{0ii'}\sum_{\s_{i''n}\in\widetilde{B}^c_{k,j}\atop ||\s_{i'n}-\s_{i''n}||_E\leq q_n}K_{0ii''}\beta^{\frac{\delta}{2(1+\delta)}}(||\s_{i'n}-\s_{i''n}||_E)\\
     & =\sum_{\s_{in}\in\widetilde{B}^-_{k,j}}\sum_{\s_{i'n}\in\widetilde{B}^c_{k,j}}K_{0ii'}\bigg (\sum_{l=1}^{[q_n]}\sum_{\s_{i''n}\in M_{i'l}}K_{0ii''}\beta^{\frac{\delta}{2(1+\delta)}}(||\s_{i'n}-\s_{i''n}||_E) \\
      &  ~~~~~~~~~~~~~~~~~~~~~~~~~~~~~~~~~~~~~~ +\sum_{\s_{i''n}\in M_{i'0}}K_{0ii''}\beta^{\frac{\delta}{2(1+\delta)}}(||\s_{i'n}-\s_{i''n}||_E) \bigg )\\
    & \leq \frac{|\widetilde{B}^-_{k,j}||\widetilde{B}^c_{k,j}|b^4}{|\widetilde{B}^-_{k,j}||\widetilde{B}^c_{k,j}|b^2}\sum_{\s_{in}\in\widetilde{B}^-_{k,j}}\sum_{\s_{i'n}\in\widetilde{B}^c_{k,j}}K_{0ii'} \\
     &  ~~~~\times \bigg (\sum_{l=1}^{[q_n]}|M_{i'l}|\beta^{\frac{\delta}{2(1+\delta)}}(l)\frac{1}{|M_{i'l}|b^2}\sum_{\s_{i''n}\in M_{i'l}}K_{0ii''}
    +\frac{|M_{i'0}|}{|M_{i'0}|b^2}\sum_{\s_{i''n}\in M_{i'0}}K_{0ii''}\bigg ).
\end{align}
Therefore, for arbitrary slow $\xi_{n}\nearrow\infty$, since we have assumed that $n=T$, we obtain
\begin{align}
   &\mathbb{Q} \left(\overline{\lim}_{n}\Big\{\mathbf{K}_1\geq \xi_{n}|\widetilde{B}^-_{k,j}|nb^4\Big(\frac{n}{\lambda_n^2\bb_n}\Big)\Big\}\right)\\
   \leq&  \mathbb{Q} \left(\overline{\lim}_{n}\Big\{\Big\{\mathbf{K}_1\geq \xi_{n}|\widetilde{B}^-_{k,j}|nb^4\Big(\frac{n}{\lambda_n^2\bb_n}\Big)\Big\}\bigcap\mathcal{M}_n\Big\}\right)+\mathbb{Q}(\overline{\lim}_n \mathcal{M}_n^c)=0.
\end{align}
Meanwhile, by letting $\xi_n=\log n$, T1 of Assumption \ref{technical assumptions} indicates that 
\begin{align}
    \xi_{n}n|\widetilde{B}^-_{k,j}|b^4\Big(\frac{n}{\lambda_n^2\bb_n}\Big)Tp_{T}b= o(n|\widetilde{B}^-_{k,j}|Tp_T b^4).
\end{align}
Thus, $\lim_{n\to\infty}\frac{\mathbf{K}_1\mathbf{K}_2}{n\textbf{Card}(\widetilde{B}^-_{k,j}\cap \Gamma_n)Tp_T b^4}=0$ holds almost surely with respect to $\mathbb{Q}$.
\end{proof}
\begin{proposition}
    \label{prop 7}
Under Assumption \ref{technical assumptions}, by defining 
\begin{align}
    A_{k,j}^c&=\{(i,t)\neq (i',t')\in \widetilde{B}^{-}_{k,j}; (r,s),(r',s')\in\widetilde{B}^{c}_{k,j}:\max\{||(i,t)-(i',t')||_{F},||(r,s)-(r',s')||_{F}\}\leq q_n \}\},\\
    A_{k,j}&=\{(i,t)\neq (i',t')\in \widetilde{B}^{-}_{k,j}; (r,s),(r',s')\in\widetilde{B}^{c}_{k,j}:\max\{||(i,t)-(i',t')||_{F},||(r,s)-(r',s')||_{F}\}> q_n \}\},
\end{align}
we obtain that
\begin{itemize}
    \item [(E1)] \begin{align}
         \sum_{A_{k,j}^c }K_{0b}(\s_{in}-\s_{rn},t-s)K_{0b}(\s_{i'n}-\s_{r'n},t'-s') &\beta^{\frac{\delta}{1+\delta}}(||(i,t)-(i',t')||_{F})\beta^{\frac{\delta}{1+\delta}}(||(r,s)-(r',s')||_{F})\\
        &=o(n\textbf{Card}(\widetilde{B}^{-}_{k,j}\cap\Gamma_{n})Tp_T b^4)\ \ \ a.s.-[\mathbb{Q}]; 
    \end{align}
    \item [(E2)]\begin{align}
        \sum_{A_{k,j} }K_{0b}(\s_{in}-\s_{rn},t-s)K_{0b}(\s_{i'n}-\s_{r'n},t'-s') \beta^{\frac{\delta}{1+\delta}}(q_n)=o(n\textbf{Card}(\widetilde{B}^{-}_{k,j}\cap\Gamma_{n})Tp_T b^3)\ \  a.s.-[\mathbb{Q}].
    \end{align}
\end{itemize}
\end{proposition}

\begin{proof}
The proof of Proposition \ref{prop 7} is nearly the same as the proof of Proposition \ref{prop 6}. Thus, we only sketch the proof and highlight the key steps.
\par First, note that
\begin{align}
    &\sum_{A^c_{k,j}}K_{0b}(\s_{in}-\s_{rn},t-s)K_{0b}(\s_{i'n}-\s_{r'n},t'-s') \\
    \lesssim &\sum_{(i,t)\neq (i',t')\in \tilde{B}^-_{k,j}\atop
    ||\s_{in}-\s_{i'n}||_{E}\leq q_n,|t-t'|\leq q_n}\sum_{(r,s),(r',s')\in \tilde{B}^c_{k,j}\atop
    ||\s_{rn}-\s_{r'n}||_{E}\leq q_n,|s-s'|\leq q_n}K_{0ir, S}K_{0ts, T}K_{0i'r', S}K_{0t's', T}\\
    \lesssim &  \sum_{\s_{in},\s_{i'n}\in \tilde{B}^-_{k,j,S}; \s_{rn},\s_{r'n}\in \tilde{B}^c_{k,j,S} \atop
    ||\s_{in}-\s_{i'n}||_{E}\leq q_n,||\s_{rn}-\s_{r'n}||_{E}\leq q_n}K_{0ir, S}K_{0i'r', S} \times \sum_{t,t'\in \tilde{B}^-_{k,j,T}; s,s'\in \tilde{B}^c_{k,j,T} \atop
   |t-t'|\leq q_n; |s-s'|\leq q_n}K_{0ts, T}K_{0t's', T} .
\end{align}
Then, similar to the argument used in the proof of Proposition \ref{prop 6}, (T2), (i) of Assumption \ref{technical assumptions} asserts
\begin{align}
     &\sum_{A_{k,j}^c }K_{0b}(\s_{in}-\s_{rn},t-s)K_{0b}(\s_{i'n}-\s_{r'n},t'-s') \beta^{\frac{\delta}{1+\delta}}(||(i,t)-(i',t')||_E)\beta^{\frac{\delta}{1+\delta}}(||(r,s)-(r',s')||_E)\\
     \lesssim&  \sum_{\s_{in}\neq \s_{i'n}\in \tilde{B}^-_{k,j,S}; \s_{rn}\neq \s_{r'n}\in \tilde{B}^c_{k,j,S} \atop
    ||\s_{in}-\s_{i'n}||_{E}\leq q_n;||\s_{rn}-\s_{r'n}||_{E}\leq q_n}K_{0ir, S}K_{0i'r',S}\beta^{\frac{\delta}{2(1+\delta)}}(||\s_{in}-\s_{i'n}||_E)\beta^{\frac{\delta}{2(1+\delta)}}(||\s_{rn}-\s_{r'n}||_E) \\
    &\times \sum_{t,t'\in \tilde{B}^-_{k,j,T}; s,s'\in \tilde{B}^c_{k,j,T} \atop
   |t-t'|\leq q_n; |s-s'|\leq q_n}K_{0ts, T}K_{0t's', T}\beta^{\frac{\delta}{2(1+\delta)}}(|t-t'|)\beta^{\frac{\delta}{2(1+\delta)}}(|s-s'|) \\
   \lesssim& 
    \mathbf{K}_3 Tp_Tb,
\end{align}
where 
$$
\mathbf{K}_3:= \sum_{\s_{in}\neq \s_{i'n}\in \tilde{B}^-_{k,j,S}; \s_{rn}\neq \s_{r'n}\in \tilde{B}^c_{k,j,S} \atop
    ||\s_{in}-\s_{i'n}||_{E}\leq q_n;||\s_{rn}-\s_{r'n}||_{E}\leq q_n}K_{0ir, S}K_{0i'r',S}\beta^{\frac{\delta}{2(1+\delta)}}(||\s_{in}-\s_{i'n}||_E)\beta^{\frac{\delta}{2(1+\delta)}}(||\s_{rn}-\s_{r'n}||_E) 
$$
and 
the last inequality follows from  Lemma \ref{lemma 4} and condition (T3) in Assumption \ref{technical assumptions}. To bound $\mathbf{K}_3$, similar to the proof of Proposition \ref{prop 6}, we introduce the following sets
\begin{align}
        N_{il}&=\{\s_{i'n}\in \widetilde{B}^-_{k,j,S}\cap\Gamma_n:||\s_{i'n}-\s_{in}||_E\in (l,l+1]\}, \\
        N_{rl}& =\{\s_{r'n}\in \widetilde{B}^c_{k,j,S}\cap\Gamma_n:||\s_{r'n}-\s_{rn}||_E\in (l,l+1]\},\ l\geq 1;\\
    N_{i0}&=\{\s_{i'n}\in \widetilde{B}^-_{k,j,S}\cap\Gamma_n:0<||\s_{i'n}-\s_{in}||_E< 1\}, \\
    N_{r0} & =\{\s_{r'n}\in \widetilde{B}^c_{k,j,S}\cap\Gamma_n:0<||\s_{r'n}-\s_{rn}||_E< 1\};\\
    \widetilde N_{il}&=\{\s\in \widetilde{B}^-_{k,j,S}\cap \R_n:||\s-\s_{in}||_E\in (l,l+1]\},  \\
    \widetilde N_{rl} & =\{\s\in \widetilde{B}^c_{k,j,S}\cap \R_n:||\s-\s_{rn}||_E\in (l,l+1]\},\ l\geq 1;\\
     \widetilde N_{i0}&=\{\s\in \widetilde{B}^-_{k,j,S}\cap \R_n:0<||\s-\s_{in}||_E\leq 1\}, \\\widetilde N_{r0} & =\{\s\in \widetilde{B}^c_{k,j,S}\cap \R_n:0<||\s-\s_{rn}||_E\leq 1\}.
\end{align}
We also introduce the following events by denoting $\textbf{Card}(A\cap\Gamma_n)$ as $|A|$, for any $A\subset \R_n$,
\begin{align}
     \mathcal{E}^*_{n}(k,j) &=\bigcap_{\s_{in}\in \widetilde{B}^-_{k,j}}\bigcap_{l=0}^{[q_n]}\Big\{n\textbf{Leb}(\Lambda_n^{-1}\widetilde N_{il}\cap D_n)\leq |N_{il}| \leq \frac{2ln}{\lambda_n^2\bb_n}\Big\},\\
    \mathcal{E}^{**}_{n}(k,j) &=\bigcap_{\s_{rn}\in \widetilde{B}^c_{k,j}}\bigcap_{l=0}^{[q_n]}\Big\{n\textbf{Leb}(\Lambda_n^{-1}\widetilde N_{rl}\cap D_n)\leq |N_{rl}| \leq \frac{2ln}{\lambda_n^2\bb_n}\Big\},\\
\mathcal{E}_n&=\bigcap_{k=1}^{R_n}\bigcap_{j=1}^{J_n}\Big(\mathcal{E}^*_{n}(k,j)\bigcap \mathcal{E}^{**}_{n}(k,j)\bigcap \mathbf{B}^-_{n}(k,j)\bigcap \mathbf{B}_{n}(k,j)\bigcap \mathbf{B}^c_{n}(k,j)\Big),
\end{align}
where $ \mathbf{B}^-_{n}(k,j)$, $ \mathbf{B}^c_{n}(k,j)$ and $ \mathbf{B}_{n}(k,j)$ are defined in the proof of Proposition \ref{prop 6}. Similarly, Proposition \ref{prop 5} asserts
\begin{align}
    \mathbb{Q}(\overline{\lim}_n \mathcal{E}^c_n)=0. 
\end{align}
Moreover, simple algebra can show that, for every realization of $\bS_n=(\s_{in}:1\leq i\leq n)$, 
\begin{align}
    \mathbf{K}_3\lesssim b^2\sum_{\s_{in}\in\widetilde{B}^-_{k,j},\s_{rn}\in\widetilde{B}^c_{k,j}}K_{0ir}\left(\sum_{l=1}^{[q_n]}|N_{il}||N_{rl}|\beta^{\frac{\delta}{1+\delta}}(l)\left(\sum_{\s_{i'n}\in N_{il},\s_{r'n}\in N_{il}}\frac{K_{0i'r'}}{|N_{il}||N_{rl}|b^2}\right)\right).
\end{align}
Then, repeating the argument used in the proof of Proposition \ref{prop 6} yields that
\begin{align}
    \lim_{n\to\infty}\frac{\mathbf{K}_3}{n\textbf{Card}(\widetilde{B}_{k,j}^-\cap\Gamma_n)Tp_Tb^4}=0
\end{align}
holds almost surely with respect to $\mathbb{Q}$. Considering the proof of E2 is nearly the same as the proof of Proposition \ref{prop 8}, we omit it here for brevity.
\end{proof}
\begin{proposition}
    \label{prop 8}
    Under Assumption \ref{technical assumptions}, 
    \begin{align}
        \sum_{(i,t)\in\tilde{B}^-_{k,j}}\sum_{(i',t'),(i'',t'')\in \tilde{B}^c_{k,j}}K_{0b}(\s_{in}-\s_{i'n},t-t')K_{0b}(\s_{in}-\s_{i''n},t-t'')\beta^{\frac{\delta}{1+\delta}}(q_n)\\
        =o(n\textbf{Card}(\widetilde{B}^{-}_{k,j}\cap\Gamma_{n})Tp_T b^3),\ \ a.s.-[\mathbb{Q}].
    \end{align}
\end{proposition}

\begin{proof}
 Since the term $\beta^{\frac{\delta}{1+\delta}}(q_n)$ can be regarded as a deterministic constant depending only on sample size $n$, we only need to focus on the grow conditions of the sums of product-kernels. Some simple algebra indicates that
\begin{align}
      &\sum_{(i,t)\in\tilde{B}^-_{k,j}}\sum_{(i',t'),(i'',t'')\in \tilde{B}^c_{k,j}}K_{0b}(\s_{in}-\s_{i'n},t-t')K_{0b}(\s_{in}-\s_{i''n},t-t'')\\
      =&  \sum_{\s_{in}\in\widetilde{B}^-_{k,j,S}}\sum_{(i',t'),(i'',t'')\in \widetilde{B}^c_{k,j,S}}K_{0ii', S}K_{0ii'', S}\times \sum_{t\in\widetilde{B}^-_{k,j,T}}\sum_{t',t''\in\widetilde{B}^c_{k,j, T}}K_{0tt', T}K_{0tt'', T} \\
      \lesssim & p_TT^2b^2 \sum_{\s_{in}\in\widetilde{B}^-_{k,j,S}}\sum_{(i',t'),(i'',t'')\in \widetilde{B}^c_{k,j,S}}K_{0ii', S}K_{0ii'', S} =:p_{T}(Tb)^2\mathbf{K}_4,
\end{align}
where $\widetilde{B}^-_{k,j,S}$ ($\widetilde{B}^c_{k,j,T}$) and $\widetilde{B}^-_{k,j,T}$ ($\widetilde{B}^c_{k,j,T}$) are defined in the proof of Proposition \ref{prop 6}. Then, based on the notation used in the proof of Proposition \ref{prop 6}, define event 
\begin{align}
    \mathbb{B}_n= \bigcap_{k=1}^{R_n}\bigcap_{j=1}^{J_n} (\mathbf{B}_n^{-}(k,j)\bigcap\mathbf{B}_n(k,j)\bigcap\mathbf{B}_n^{c}(k,j)),
\end{align}
we obtain that 
\begin{align}
    \mathbb{Q}(\overline{\lim}_{n\to\infty}\mathbb{B}^c_n)=0.
\end{align}
Then, similar to the argument used in the proof of Propositions \ref{prop 6} and \ref{prop 7},
\begin{align}
    \lim_{n\to\infty}\frac{|\mathbf{K}_4|}{\textbf{Card}(\widetilde{B}_{k,j}^-\cap\Gamma_n)(nb^2)^2}<\infty
\end{align}
holds almost surely with respect to $\mathbb{Q}$. Above all, we obtain 
\begin{align}
    \sum_{(i,t)\in\tilde{B}^-_{k,j}}\sum_{(i',t'),(i'',t'')\in \tilde{B}^c_{k,j}}K_{0b}(\s_{in}-\s_{i'n},t-t')K_{0b}(\s_{in}-\s_{i''n},t-t'')\beta^{\frac{\delta}{1+\delta}}(q_n)\\=O(p_T\textbf{Card}(\widetilde{B}_{k,j}^-\cap\Gamma_n)(Tnb^3)^2)\beta^{\frac{\delta}{1+\delta}}(q_n)=o(n\textbf{Card}(\widetilde{B}_{k,j}^-\cap\Gamma_n)Tp_Tb^3)
\end{align}
holds almost surely with respect to measure $\mathbb{Q}$, where the last equality follows from  condition (1) in (T4) in  Assumption \ref{technical assumptions}. This completes the proof.
\end{proof}
\begin{proposition}
    \label{prop 9}
    Under Assumption \ref{technical assumptions}, 
\begin{align}
\lim_{n}\frac{\sum_{k=1}^{R_n}\sum_{j=1}^{J_n}\sum_{(i,t)\in\tilde{B}^-_{k,j}}\sum_{(i',t')\in \tilde{B}^c_{k,j}}K^2_{0b}(\s_{in}-\s_{i'n},t-t')(E[X^2_{0}])^2}{\sigma_n^2}=1.
\end{align}
\end{proposition}

\begin{proof}
\noindent By  simple algebra we obtain
\begin{align}
    &\sum_{k=1}^{R_n}\sum_{j=1}^{J_n}\sum_{(i,t)\in\tilde{B}^-_{k,j}}\sum_{(i',t')\in \tilde{B}^c_{k,j}}K^2_{0b}(\s_{in}-\s_{i'n},t-t')(E[X^2_{0}])^2=:\sum_{k=1}^{R_n}\sum_{j=1}^{J_n}H_{2}(k,j)\\
    =&\sum_{k=1}^{R_n}\sum_{j=1}^{J_n+1}H_{2}(k,j)-\sum_{k=1}^{R_n}H_{2}(k,J_n+1).
\end{align}
For each given $k=1,2,\ldots ,R_n$ and $j=1,2,.\ldots  ,J_n$, 
\begin{align}
    H_{2}(k,j) & =\sum_{(i,t)\in\tilde{B}^-_{k,j}}\sum_{(i',t')\in \tilde{B}^c_{k,j}}K^2_{0b}(\s_{in}-\s_{i'n},t-t')(E[X^2_{0}])^2\\
    =&\sum_{(i,t)\in\tilde{B}^-_{k,j}}\bigg (\sum_{(i',t')\in \mathbf{R}_n\times\{1,..,T\}}K^2_{0b}(\s_{in}-\s_{i'n},t-t')(E[X^2_{0}])^2 \\
    & ~~~~~~~~~~~~~~~~~~~~~~~~~~~~~~
    -\sum_{(i',t')\in \widetilde{B}_{k,j}}K^2_{0b}(\s_{in}-\s_{i'n},t-t')(E[X^2_{0}])^2\bigg)\\
    & =\sum_{(i,t)\in\tilde{B}_{k,j}}\sum_{(i',t')\in \mathbf{R}_n\times\{1,..,T\}}K^2_{0b}(\s_{in}
    \\
    & ~~~~~~~~~~~~~~~~~~~~~
    -\s_{i'n},t-t')(E[X^2_{0}])^2-\sum_{(i,t)\in\tilde{B}_{k,j}}\sum_{(i',t')\in \widetilde{B}_{k,j}}K^2_{0b}(\s_{in}-\s_{i'n},t-t')(E[X^2_{0}])^2\\
    & \\
    & ~~~~~~~~~~~~~~~~~~~~~
    - \sum_{(i,t)\in\tilde{B}_{k,j}\backslash\widetilde{B}^-_{k,j}}\sum_{(i',t')\in \mathbf{R}_n\times\{1,..,T\}}K^2_{0b}(\s_{in}-\s_{i'n},t-t')(E[X^2_{0}])^2
    \\
    & ~~~~~~~~~~~~~~~~~~~~~
    +\sum_{(i,t)\in\tilde{B}_{k,j}\backslash\widetilde{B}^-_{k,j}}\sum_{(i',t')\in \widetilde{B}_{k,j}}K^2_{0b}(\s_{in}-\s_{i'n},t-t')(E[X^2_{0}])^2\\
    & =: H_{21}(k,j)-H_{22}(k,j)-H_{23}(k,j)+H_{24}(k,j).
\end{align}
Similar to the proof of Propositions \ref{prop 6} and \ref{prop 7}, based on the design of blocks, $q_n/ (r_{1n} \wedge r_{2n}) = o(1)$ by introducing some sets and events, Proposition \ref{prop 5} indicates that, for each $k$ and $j$,
\begin{align}
    \lim_{n\to\infty}\Big|\frac{H_{23}(k,j)}{H_{21}(k,j)}\Big|=0
\end{align}
holds almost surely with respect to $\mathbb{Q}$. On the other hand, note that
\begin{align}
   |H_{24}(k,j)-H_{22}(k,j)|=\sum_{(i,t)\in \widetilde{B}^-_{k,j}}\sum_{(i',t')\in\widetilde{B}_{k,j}}K^2_{0b}(\s_{in}-\s_{i'n},t-t')(E[X^2_{0}])^2
\end{align}
and 
\begin{align}
    &\sum_{k=1}^{R_n}\sum_{j=1}^{J_n}\sum_{(i,t)\in \widetilde{B}^-_{k,j}}\sum_{(i',t')\in\widetilde{B}_{k,j}}K^2_{0b}(\s_{in}-\s_{i'n},t-t')\\
    =& \sum_{k=1}^{R_n}\sum_{\s_{in}\in\widetilde{B}^-_{k,j,S}}\sum_{\s_{i'n}\in\widetilde{B}_{k,j,S}}K^2_{0ii', S} \times \sum_{j=1}^{J_n}\sum_{t\in\widetilde{B}^-_{k,j,T}}\sum_{t\in\widetilde{B}_{k,j,T}}K^2_{0tt', T} =:\mathbf{K}_{S}\mathbf{K}_{T}.
\end{align}
Recall that, at each given location $\s_{in}$, we can observe a time series $\{X_{(\s_{in},t)}:1\leq t\leq T\}$ and these time locations are invariant with respect to spatial location $\s_{in}$. We immediately have the following two points hold for every $n$ and $T$.
\begin{itemize}
    \item [1] $\widetilde{B}_{k,j,S}$, $\widetilde{B}^{-}_{k,j,S}$ and $\widetilde{B}^c_{k,j,S}$ are independent of index $j$. 
    \item [2] $\widetilde{B}_{k,j,T}$, $\widetilde{B}^{-}_{k,j,T}$ and $\widetilde{B}^c_{k,j,T}$ are independent of index $k$. More specifically, for each $j$, $\widetilde{B}^{-}_{k,j,T}=I_{bj}$ and $\widetilde{B}_{k,j,T}=I_{bj}\cup I_{sj}$, where $I_{bj}$ and $I_{sj}$ are “big and small” block on time axis introduced above.
\end{itemize}
For $\mathbf{K}_{T}$, together with point 2 above, \eqref{lemma 5 eq-2} of Lemma \ref{lemma 5} and \eqref{lemma 6 eq-2} of Lemma \ref{lemma 6} indicate that 
\begin{align}
    \mathbf{K}_{T}\lesssim (Tbp_T\log T )\land(p_Tq_T\log T).
\end{align}
For $\mathbf{K}_S$, similar to the argument used in the proof of Propositions \ref{prop 6} and \ref{prop 7}, by introducing proper sets and events, we can show 
\begin{align}
    \lim_{n\to\infty}\frac{\mathbf{K}_{S}}{n^2b^2}<\infty,
\end{align}
holds almost surely with respect to $\mathbb{Q}$, which asserts that
\begin{align}
    \lim_{n\to\infty}\frac{|H_{24}(k,j)-H_{22}(k,j)|}{n^2Tb^3p_T\log T}<\infty
\end{align}
holds almost surely with respect to $\mathbb{Q}$.

Above all, the following identity holds almost surely with respect to $\mathbb{Q}$,
\begin{align}
    \lim_{n\to\infty}\frac{\sum_{k,j}H_{2}(k,j)}{\sigma_n^2}=\lim_{n\to\infty}\frac{\sum_{k,j}H_{21}(k,j)}{\sigma^2_n}=1,
\end{align}
which completes the proof.
\end{proof}
\begin{proposition}
    \label{prop 10}
    Under Assumption \ref{technical assumptions}, by denoting $K_{irtu}=K_{0b}(\s_{in}-\s_{rn},t-u)$, we obtain that
    \begin{align}
     \label{prop 10 eq-2}
    \sum_{(k,j)\neq (k',j')}\sum_{(i,t)\in\widetilde{B}^{-}_{k,j}}\sum_{(i',t')\in\widetilde{B}^-_{k',j'}}\sum_{(r,u),(r',u')\in C^c_{(i,t)(i',t')}}&K_{irtu}K_{i'r't'u'}\text{Cov}_{|S_n}(X_{it}X_{ru},X_{i't'}X_{r'u'}) \\
      =:\sum_{(k,j)\neq (k',j')}V_{21}(k,j,k',j') &=o((nT)^2b^3),\ \ \ a.s.-[\mathbb{Q}]\\
    \label{prop 10 eq-1}
        \sum_{(k,j)\neq (k',j')}\sum_{(i,t)\in\widetilde{B}^{-}_{k,j}}\sum_{(i',t')\in\widetilde{B}^-_{k',j'}}\sum_{(r,u),(r',u')\in C_{(i,t)(i',t')}}&K_{irtu}K_{i'r't'u'}\text{Cov}_{|S_n}(X_{it}X_{ru},X_{i't'}X_{r'u'}) \\
       =:\sum_{(k,j)\neq (k',j')}V_{22}(k,j,k',j') &=o((nT)^2b^3),\ \ \ a.s.-[\mathbb{Q}],
    \end{align}
    
    where
\begin{align}
     & C_{(i,t)(i',t')}=\{(r,u)\in \widetilde{B}_{k,j}^c, (r',u')\in\widetilde{B}_{k',j'}^c:||(r,u)-(i',t')||_{F}< q_n, ||(r',u')-(i,t)||_{F}< q_n\},\\
  &C^c_{(i,t)(i',t')}=\{(r,u)\in \widetilde{B}_{k,j}^c, (r',u')\in\widetilde{B}_{k',j'}^c: \max\{||(r,u)-(i',t')||_{F}, ||(r',u')-(i,t)||_{F}\}\geq q_n\}.
\end{align}
\end{proposition}

\begin{proof} ~~

 \textbf{Proof of \eqref{prop 10 eq-1}}.  Note that the definition of set $C_{(i,t) (i',t')}$ and design of block (see also Figure \ref{fig:cit}) yield that, for every realization of $\bS_n$ and $(k,j)\neq (k',j')$, 
\begin{align}
   V_{22}(k,j,k',j')\leq  \sum_{(i,t)\in\widetilde{B}^-_{k,j}, (r',u')\in \widetilde{B}_{k,j}\atop ||(i,t)-(r',u')||_E\leq q_n}\sum_{(i',t')\in\widetilde{B}^-_{k',j'}, (r,u)\in \widetilde{B}_{k',j'}\atop ||(i',t')-(r,u)||_E\leq q_n}K_{irtu}K_{i'r't'u'}|\text{Cov}(X_{it}X_{r'u'},X_{i't'}X_{ru})|.
\end{align}
Denote $\mathcal{N}_{k,j}=\{\widetilde{B}_{k',j'}: (k',j')\neq (k,j),\ d(\widetilde{B}_{k,j},\widetilde{B}_{k',j'})=0\}$, where $d(A,B)=\inf_{x\in A, y\in B} ||x-y||_E$. Basic properties of Euclidean space shows that $\max_{k,j}\textbf{Card}(\mathcal{N}_{k,j})\leq 26$.
Thus, for every realization of $\bS_n$, by Assumption \ref{moment conditions},
\begin{align}
&\sum_{(k,j)\neq (k',j')}V_{22}(k,j,k',j')\\
\leq &2(\widetilde{M}_1\lor M_1)\sum_{(k,j)}\sum_{(i,t)\in\widetilde{B}^-_{k,j},(r',u')\in\widetilde{B}_{k,j}\atop ||(i,t)-(r',u')||_E<q_n}\sum_{\widetilde{B}_{k',j'}\in \mathcal{N}_{k,j}}\sum_{(i',t')\in\widetilde{B}^-_{k',j'},(r,u)\in\widetilde{B}_{k',j'}\atop ||(i',t')-(r,u)||_E<q_n}K_{irtu}K_{i'r't'u'}\\
&+4\sum_{(k,j)}\sum_{(i,t)\in\widetilde{B}^-_{k,j},(r',u')\in\widetilde{B}_{k,j}\atop ||(i,t)-(r',u')||_E<q_n}\sum_{\widetilde{B}_{k',j'}\notin \mathcal{N}_{k,j}}\sum_{(i',t')\in\widetilde{B}^-_{k',j'},(r,u)\in\widetilde{B}_{k',j'}\atop ||(i',t')-(r,u)||_E<q_n}K_{irtu}K_{i'r't'u'}|\text{Cov}(X_{it}X_{r'u'},X_{i't'}X_{ru})|\\
=&: \mathbf{L}_1+\mathbf{L}_2.
\end{align}

For each realization of $\bS_n$, based on the notation  $K_{0ir, S}$ and $K_{0tu, T}$ introduced {in \eqref{def:KS} and \eqref{def:KT}}, we have $K_{irtu}=K_{0ir, S}K_{0tu, T}$ and
\begin{align}
    \mathbf{L}_1\leq& 2(\widetilde{M}_1\lor M_1) \sum_{k=1}^{R_n}\sum_{\s_{in}\in\widetilde{B}^-_{k,j,S},\s_{r'n}\in \widetilde{B}_{k,j,S}\atop ||\s_{in}-\s_{r'n}||_E<q_n}\sum_{\widetilde{B}_{k',j',S}\in \mathcal{N}_{k,j,S}}\sum_{\s_{i'n}\in\widetilde{B}^-_{k',j',S},\s_{rn}\in\widetilde{B}_{k',j',S} \atop||\s_{i'n}-\s_{rn}||_E<q_n}K_{0ir}K_{0i'r'} \\
    &\times  \sum_{j=1}^{J_n}\sum_{t\in\widetilde{B}^-_{k,j,T},u'\in \widetilde{B}_{k,j,T}\atop |t-u'|_E<q_n}\sum_{\widetilde{B}_{k',j',T}\in \mathcal{N}_{k,j,T}}\sum_{t'\in\widetilde{B}^-_{k',j',T},u\in\widetilde{B}_{k',j',T}\atop |t'-u|<q_n }K_{0tu}K_{0t'u'} \\
    &=: C_1\mathbf{L}_{1S}\mathbf{L}_{1T},
\end{align}
where $\mathcal{N}_{k,j,S}$ and $\mathcal{N}_{k,j,T}$ are the collections of the projections of $\widetilde{B}_{k',j'}$'s in set $\mathcal{N}_{k,j}$ onto spatial plane and time axis respectively. More specifically, $\max_{k,j}\textbf{Card}(\mathcal{N}_{k,j,S})\leq 4$ and $\max_{k,j}\textbf{Card}(\mathcal{N}_{k,j,T})\leq 2$.

\par To bound $\mathbf{L}_{1T}$, we first note that the compactness of the support of kernel function indicates that, for each given $t_0\in\widetilde{B}^-_{k,j,T}$, $v_0-Tb\leq |t_0-u|\leq v_0+Tb$. However, the definition of $\mathcal{N}_{k,j,T}$ implies $|t_0-u|\leq 2(p_T+q_T)\leq 4p_T$. Since $\frac{p_T}{T}=o(1)$ and $n=T$, for any given $v_0\in (0,1)$, there exists some $T_0$ such that $4p_T<v_0-Tb$ holds for every $n>T_0$ and $T_0$ is independent of $(k,j)$ and sample size $n$, which asserts $K_{0tu}=0$, $\mathbf{L}_{1T}=0$ and $\mathbf{L}_1=0$.

Recall that Euclidean spatial-temporal distance is always longer than Euclidean spatial distance and T2, (i) in Assumption \ref{technical assumptions} indicates the mixing coefficient $\beta(r)$ is non-increasing on $\mathbb{R}^+$. Therefore, for any two different spatial-temporal locations $(\s,t)\neq (\s',t')$, 
\begin{align}
    \beta^{a}(||(\s,t)-(\s',t')||_E)= \beta^{a}\Big(\sqrt{||\s-\s'||^2_E+|t-t'|}\Big)\leq \beta^{a}(||\s-\s'||_E)
\end{align}
holds for any given $a\geq 0$. Then, similar to $\mathbf{L}_1$, we can show that there exists some universal constant $C_2>0$ such that 
\begin{align}
    \mathbf{L}_2\leq& C_2 \sum_{k=1}^{R_n}\sum_{\s_{in}\in\widetilde{B}^-_{k,j,S},\s_{r'n}\in \widetilde{B}_{k,j,S}\atop ||\s_{in}-\s_{r'n}||_E<q_n}\sum_{\widetilde{B}_{k',j',S}\notin \mathcal{N}_{k,j,S}}\sum_{\s_{i'n}\in\widetilde{B}^-_{k',j',S},\s_{rn}\in\widetilde{B}_{k',j',S} \atop||\s_{i'n}-\s_{rn}||_E<q_n}K_{0ir, S}K_{0i'r', S}\beta^{\frac{\delta}{(1+\delta)}}(d(\widetilde{B}_{k,j,S},\widetilde{B}_{k',j',S})) \\
    &\times \sum_{j=1}^{J_n}\sum_{t\in\widetilde{B}^-_{k,j,T},u'\in \widetilde{B}_{k,j,T}\atop |t-u'|_E<q_n}\sum_{\widetilde{B}_{k',j',T}\notin \mathcal{N}_{k,j,T}}\sum_{t'\in\widetilde{B}^-_{k',j',T},u\in\widetilde{B}_{k',j',T}\atop |t'-u|<q_n }K_{0tu, T}K_{0t'u', T} =:C_2\mathbf{L}_{2S}\mathbf{L}_{2T}
\end{align}
holds for every realization of $\bS_n$, where $d(A,B)=\inf_{a\in A,b\in B}||a-b||_E$, $A, B\subset \mathbb{R}^{2}$ (or $\mathbb{R}$).

\par By defining $\mathcal{M}_{k,j,l}=\{\widetilde{B}_{k',j',S}\notin\mathcal{N}_{k,j,S}:lq_n\leq d(\widetilde{B}_{k,j,S},\widetilde{B}_{k',j',S})\leq (l+1)q_n\}$, 
\begin{align}
\mathbf{L}_{2S}\leq \sum_{k=1}^{R_n}\sum_{\s_{in}\in\widetilde{B}^-_{k,j,S},\s_{r'n}\in \widetilde{B}_{k,j,S}\atop ||\s_{in}-\s_{r'n}||_E<q_n}\sum_{l=0}^{\infty}\sum_{\widetilde{B}_{k',j',S}\in \mathcal{M}_{k,j,l}}\sum_{\s_{i'n}\in\widetilde{B}^-_{k',j',S},\s_{rn}\in\widetilde{B}_{k',j',S} \atop||\s_{i'n}-\s_{rn}||_E<q_n}K_{0ir}K_{0i'r'}\beta^{\frac{\delta}{(1+\delta)}}(lq_n)   
\end{align}
holds for every $\bS_n=S_n$. T2, (ii) of Assumption \ref{technical assumptions} implies $\beta^{\frac{\delta}{1+\delta}}(lq_n)\leq C_{\beta}\beta^{\frac{\delta}{1+\delta}}(q_n)\beta^{\frac{\delta}{1+\delta}}(l)$, where $C_{\beta}$ is independent of $l$ and $q_n$. Then, for every $\bS_n=S_n$
\begin{align}
    C_2\mathbf{L}_{2S}\mathbf{L}_{2T}\leq C_2C_{\beta}&  \sum_{k=1}^{R_n}\sum_{s_{in}\in\widetilde{B}^-_{k,j,S},s_{r'n}\in \widetilde{B}_{k,j,S}\atop |s_{in}-s_{r'n}||_E<q_n}\sum_{l=1}^{\infty}\sum_{\widetilde{B}_{k',j',S}\in \mathcal{M}_{k,j,l}}\sum_{s_{i'n}\in\widetilde{B}^-_{k',j',S},s_{rn}\in\widetilde{B}_{k',j',S} \atop||s_{i'n}-s_{rn}||_E<q_n}K_{0ir}K_{0i'r'}\beta^{\frac{\delta}{(1+\delta)}}(l) \\
    &\times \sum_{j=1}^{J_n}\sum_{t\in\widetilde{B}^-_{k,j,T},u'\in \widetilde{B}_{k,j,T}\atop |t-u'|<q_n}\sum_{\widetilde{B}_{k',j',T}\notin \mathcal{N}_{k,j,T}}\sum_{t'\in\widetilde{B}^-_{k',j',T},u\in\widetilde{B}_{k',j',T}\atop |t'-u|<q_n }K_{0tu}K_{0t'u'}\beta^{\frac{\delta}{1+\delta}}(q_n) \\
    =:&C_2C_{\beta}\mathbf{L}_{2S}^*\mathbf{L}^*_{2T}.
\end{align}

As we discussed before, given $v_0\in (0,1)$, $K_{0tu}$ and $K_{0t'u'}$ are not equal to $0$ if and only if $v_0-Tb\leq |t-u|\leq v_0+Tb$ and $v_0-Tb\leq |t'-u'|\leq v_0+Tb$ hold. This indicates that
\begin{align}
    \mathbf{L}^*_{2T}\lesssim (T/p_T) p_Tq_n Tb q_n \beta^{\frac{\delta}{1+\delta}}(q_n)=T^2bq_n^2\beta^{\frac{\delta}{1+\delta}}(q_n)=o(bq_n^{2}),\label{eq:L2T}
\end{align}
where the last equation is due to T4 of Assumption \ref{technical assumptions} and the assumption $n=T$.

\par Similar to the proof of Proposition \ref{prop 7}, denote $\textbf{Card}(A\cap \Gamma_n)=:|A|$ for any $A\subset \R_n$ and let $E_{\s_{in}}[K_{0ir}]$ indicate the expectation with respect to random variable $\s_{in}$. Since $\{\s_{in}\}$ is an array which is row-wise iid, $E_{\s_{in}}[K_{0ir}]=E_{\s_{1n}}[K_{0ir}]$ holds for each $1\leq i\leq n$ and $E_{\s_{1n}}[K_{0ir}]$ can be regarded as a real-valued mapping measurable with respect to the sigma algebra generated by random variable $\s_{rn}$. Some simple algebra yields
\begin{align}
    K_{0ir}K_{0i'r'}&=\triangle_{0ir}\triangle_{0i'r'}-E_{\s_{in}}[K_{0ir}]\triangle_{0i'r'}-E_{\s_{i'n}}[K_{0i'r'}]\triangle_{0ir}+E_{\s_{in}}[K_{0ir}]E_{\s_{i'n}}[K_{0i'r'}]\\
    &= \triangle_{0ir}\triangle_{0i'r'}-E_{\s_{1n}}[K_{0ir}]\triangle_{0i'r'}-E_{\s_{1n}}[K_{0i'r'}]\triangle_{0ir}+E_{\s_{1n}}[K_{0ir}]E_{\s_{1n}}[K_{0i'r'}]
\end{align}
 where $\triangle_{0ir}=K_{0ir}-E_{\s_{in}}[K_{0ir}]$. Additionally, some simple algebra can reveal that 
 $$\max_{n}\max_{i,r}||b^{-2}E_{s_{in}}[K_{0ir}]||_{\infty}<\infty.$$

 \par Recall that, in the proof of Proposition \ref{prop 6}, we have shown that 
 \begin{align}
     \mathbb{Q}(\{\underline{\lim}_{n\to\infty}\bigcap_{k=1}^{R_n}\mathbf{B}_{n}(k,j)\}^{C})=0,
 \end{align}
 where $A^C$ indicates the complete of event $A$. $\mathbf{B}_n(k,j)$ is defined in \eqref{def:Bkj} in the proof of Proposition \ref{prop 6} and it is invariant with respect to $j$. Hence, according to strong law of large number, up to some positive constants independent of $n$, the following statement holds almost surely with respect to $\mathbb{Q}$.
\begin{align}
    \mathbf{L}^*_{2S}&\leq b^4\sum_{k=1}^{R_n}\frac{|\widetilde{B}_{k,j,S}|nq^2_n}{\lambda_n^2\bb_n} \sum_{l=0}^{\infty}\sum_{\widetilde{B}_{k',j',S}\in \mathcal{M}_{k,j,l}}\sum_{\s_{i'n}\in\widetilde{B}^-_{k',j',S},\s_{rn}\in\widetilde{B}_{k',j',S} \atop||\s_{i'n}-\s_{rn}||_E<q_n}(b^{-4}E_{\s_{in}}[K_{0ir}]E_{\s_{r'n}}[K_{0i'r'}])\beta^{\frac{\delta}{(1+\delta)}}(l)\\
    &\leq b^4\sum_{k=1}^{R_n}\frac{|\widetilde{B}_{k,j,S}|nq^2_n}{\lambda_n^2\bb_n}\Big(\sum_{l=0}^{\infty}l^2\beta^{\frac{\delta}{(1+\delta)}}(l)\Big) \Big(\frac{nr_{1n}r_{2n}}{\lambda^2_n\bb_n}\Big)\Big(\frac{nq^2_{n}}{\lambda^2_n\bb_n}\Big).  \notag
\end{align}
Therefore, according to Condition (T1) in Assumption \ref{technical assumptions} and the design of the spatial blocks,
\begin{align}
    \lim_{n\to\infty}\frac{\mathbf{L}^*_{2S}}{nb^{2.5}r_{1n}r_{2n}q_n^4}=0 \ \ \ a.s.-[\mathbb{Q}].\label{eq:L2S}
\end{align}
Regarding our design of spatial block indicates that $r_{1n}r_{2n}=O( (q_n\log n)^2)$, combining \eqref{eq:L2S} and \eqref{eq:L2T}, we obtain 
\begin{align}
    \lim_{n\to\infty}\frac{\mathbf{L}^*_{2S}\mathbf{L}^*_{2T}}{nb^{3.5}q_n^8(\log n)^2}=0 \ \ \ a.s.-[\mathbb{Q}].
\end{align}
Therefore, (3) of T4 in Assumption \ref{technical assumptions}, $\lim_{n\to\infty}\frac{\sum_{(k,j)\neq (k',j')}V_{22}(k,j,k',j')}{(nT)^2b^3}=0$ holds almost surely with respect to $\mathbb{Q}$, which completes the proof of \eqref{prop 10 eq-1}. 

As for the proof of \eqref{prop 10 eq-2}, since it is nearly the same as \eqref{prop 10 eq-1}, we omit it here. 
\end{proof}
\begin{figure}[!ht]
    \centering
    \includegraphics[width=0.5\linewidth]{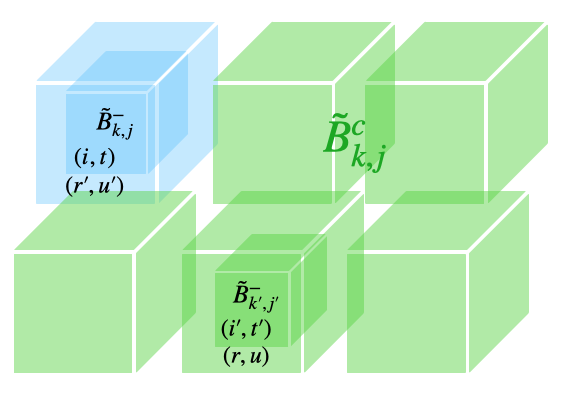}
    \caption{Illustration of $C_{(i,t) (i^{\prime},t^{\prime})}$, where we draw fewer cubes for better illustration. The green area is $\widetilde{B}_{k,j}^c$. The light blue and green areas in the outer part of each cube are for $(r,u)$ and $(r^{\prime}, u^{\prime})$}.
    \label{fig:cit}
\end{figure}
\begin{proposition}
    \label{prop 11}
    Under Assumptions \ref{moment conditions} and \ref{technical assumptions}, we have
    \begin{align}
    \label{prop 11 eq-1}
    &\sum_{k,j}\sum_{(i,t)\in\widetilde{B}^-_{k,j}}\sum_{(i',t')\in \widetilde{B}_{k,j}}K_{0b}^2(\s_{in}-\s_{i'n},t-t')=o((nT)^2b^3)\ \ \ a.s.-[\mathbb{Q}];\\
    \label{prop 11 eq-2}
     \sum_{k,j}\sum_{(i,t)\in\widetilde{B}^-_{k,j}}&\sum_{(i',t')\neq (i'',t'')\in \widetilde{B}_{k,j}}K_{0b}(\s_{in}-\s_{i'n},t-t')K_{0b}(\s_{in}-\s_{i''n},t-t'')\\&\times \text{Cov}_{|S_n}(X_{it}X_{i't'},X_{it}X_{i''t''})  =o((nT)^2b^3)\ \ \ \ \ \ a.s.-[\mathbb{Q}].  
    \end{align}
\end{proposition}

\begin{proof}
 To prove \eqref{prop 11 eq-1}, similar to the previous argument, we only need to notice that, we have  
\begin{align}
&\sum_{k,j}\sum_{(i,t)\in\widetilde{B}^-_{k,j}}\sum_{(i',t')\in \widetilde{B}_{k,j}}K_{0b}^2(\s_{in}-\s_{i'n},t-t') = \mathbf{K}_{S}\mathbf{K}_T ,  
\end{align}
where 
\begin{align}
    \mathbf{K}_{S} &= \sum_{k=1}^{R_n}\sum_{\s_{in}\in\widetilde{B}^-_{k,j,S},\s_{i'n}\in\widetilde{B}_{k,j,S}}K_{0ii'} \\
    \mathbf{K}_T &= \sum_{j=1}^{J_n}\sum_{t\in\widetilde{B}^-_{k,j,T},t'\in\widetilde{B}_{k,j,T}}K_{0tt'}.
\end{align}
The design of the time blocks, \eqref{lemma 5 eq-2} of Lemma \ref{lemma 5} and \eqref{lemma 6 eq-2} of Lemma \ref{lemma 6} yield
\begin{align}
    \mathbf{K}_T\lesssim \min\{p_T^2\log T, Tbp_T\log T\}.
\end{align}
Meanwhile, by repeating the tricks used in the proof of Proposition \ref{prop 6}, we can show 
\begin{align}
    \lim_{n\to\infty}\frac{\mathbf{K}_S}{nb^2(nr_{1n}r_{2n}/ \lambda_n^2\bb_n)}=0\ \ \ a.s.-[\mathbb{Q}].
\end{align}
The design of spatial blocks implies that $r_{1n}r_{2n}=O((q_n\log n)^2)$, which, together with T1 of Assumption \ref{technical assumptions}, asserts that $\mathbf{K}_S=o(b^{-0.5}q_n^2nb^2)$ holds almost surely with respect to $\mathbb{Q}$. Then, based on (4) of T4 in Assumption \ref{technical assumptions}, we finish the proof of \eqref{prop 11 eq-1}.

\par To prove \eqref{prop 11 eq-2}, a crucial observation is that 
\begin{align}
    |\text{Cov}(X_{it}X_{i't'},X_{it}X_{i''t''})|\leq |\text{Cov}(X^2_{it}X_{i't'},X_{i''t''})|+|\text{Cov}(X_{it},X_{i't'})||\text{Cov}(X_{it},X_{i''t''})|.
\end{align}
Hence, it suffice to prove 
\begin{align}
    \lim_{n}&\frac{\mathbf{C}_1\lor \mathbf{C}_2}{(nT)^2b^3}=0\ \ \ a.s.-[\mathbb{Q}],\ \ \text{where}\\
   \mathbf{C}_1= \sum_{k,j}\sum_{(i,t)\in\widetilde{B}^-_{k,j}}&\sum_{(i',t')\neq (i'',t'')\in \widetilde{B}_{k,j}}K_{0b}(\s_{in}-\s_{i'n},t-t')K_{0b}(\s_{in}-\s_{i''n},t-t'')|\text{Cov}(X^2_{it}X_{i't'},X_{i''t''})|,\\
  \mathbf{C}_2=  \sum_{k,j}\sum_{(i,t)\in\widetilde{B}^-_{k,j}}&\sum_{(i',t')\neq (i'',t'')\in \widetilde{B}_{k,j}}K_{0b}(\s_{in}-\s_{i'n},t-t')K_{0b}(\s_{in}-\s_{i''n},t-t'')| \\
  & ~~~~~~~~~~~
  ~~~~~~~~~~~~~
 ~~~~~~~~~~~~~
  ~~~~~~~~~~~~~
 \times \text{Cov}(X_{it},X_{i't'})||\text{Cov}(X_{it},X_{i''t''})|.
\end{align}

First, we introduce sets 
\begin{align}
    \mathcal{W}_{i,i',l}&=\{\s_{i''n}\in\widetilde{B}_{k,j,S}\cap\Gamma_n: d(\{\s_{i''n}\}, \{\s_{in},\s_{i'n}\})\in (l-1,l]\},\\
    \widetilde{\mathcal{W}}_{i,i',l}&=\{\s_{i''n}\in\widetilde{B}_{k,j,S}\cap\R_n: d(\{\s_{i''n}\}, \{\s_{in},\s_{i'n}\})\in (l-1,l]\},
\end{align}
where $l=1,2,...,\max\{[r_{1n}],[r_{2n}]\}+1$. Then, we have 
\begin{align}
    \mathbf{C}_1
    &\leq \mathbf{C}_{1S}\mathbf{C}_{1T} ,
\end{align}
where
\begin{align}
    \mathbf{C}_{1S} &= \sum_{k=1}^{R_n}\sum_{\s_{in}\in\widetilde{B}^-_{k,j,S}}\sum_{\s_{i'n}\in\widetilde{B}_{k,j,S}}K_{0ii'}\sum_{l=1}^{\infty}\sum_{\s_{i''n}\in \mathcal{W}_{i,i',l}}K_{0ii''}\beta^{\frac{\delta}{1+\delta}}(l)\\
    \mathbf{C}_{1T}  &= \sum_{j=1}^{J_n}\sum_{t\in\widetilde{B}^-_{k,j,T}}\sum_{t',t''\in\widetilde{B}_{k,j,T}}K_{0tt'}K_{0tt''}.
\end{align}
The estimates \eqref{lemma 5 eq-1}, \eqref{lemma 5 eq-3} in  Lemma \ref{lemma 5} and \eqref{lemma 6 eq-1}, \eqref{lemma 6 eq-3} in  Lemma \ref{lemma 6} imply that 
\begin{align}
\label{proof 11}
    \mathbf{C}_{1T}\lesssim (p_T\lor Tb)p_T^2\log T=(q_n\log T \lor Tb)q_n^2(\log T)^3,\ \ \text{where}\ \ n=T. 
\end{align}
 Recall that $\mathbf s_{i,n}$ is an independent sequence. As for $\mathbf{C}_{1S}$, similar to the previous argument, by introducing proper sets and events, strong law of large number can show that
\begin{align}
    \mathbf{C}_{1S}\leq &Cb^4\Big(\frac{n}{\lambda_n^2\bb_n}\Big)\sum_{k=1}^{R_n}|\widetilde{B}^-_{k,j,S}|\Big(\sum_{\s_{i'n}\in\widetilde{B}_{k,j,S}}\max_{\s_{in}}(b^{-2}E_{\s_{in}}[K_{0ii'}])\Big)\Big(\sum_{l=1}^{\infty}l^2\beta^{\frac{\delta}{1+\delta}}(l)\Big)\Big(\max_{\s_{i''n}}(b^{-2}E_{\s_{i''n}}[K_{0ii''}])\Big)\\
    \leq & C'\Big(\frac{n}{\lambda_n^2\bb_n}\Big)\Big(\sum_{l=1}^{\infty}l^2\beta^{\frac{\delta}{1+\delta}}(l)\Big)b^4\sum_{k=1}^{R_n}|\widetilde{B}^-_{k,j,S}||\widetilde{B}_{k,j,S}|
\end{align}
holds almost surely with respect to $\mathbb{Q}$, where $C, C'>0$ are constants independent of $(i,t)$ and $n$. 

\par Then, T3 and T1 of Assumption \ref{technical assumptions} imply that 
\begin{align}
    \lim_{n\to\infty}\frac{\mathbf{C}_{1S}}{n(q_n\log n)^{2}b^3}=0\ \ \ a.s.-[\mathbb{Q}].
\end{align}
Point (5) of T4 in Assumption \ref{technical assumptions} and \eqref{proof 11} indicates that $\lim_{n\to\infty}\frac{\mathbf{C}_1}{(nT)^2b^3}=0$ holds almost surely with respect to $\mathbb{Q}$. By repeating the same procedure, we can show $\lim_{n\to\infty}\frac{\mathbf{C}_2}{(nT)^2b^3}=0$ holds almost surely with respect to $\mathbb{Q}$. Thus, we finish the proof. 
\end{proof}
\begin{proposition}
    \label{prop 12}
     Provided that $n=T$, by denoting $K_{irts}=K_{0b}(s_{in}-s_{rn},t-s)$, there exists some $N_0\in\mathbb{N}$ independent of sample size $n$ such that, for every $n>N_0$ and realization $\bS_n=S_n$,
    \begin{align}
     \label{prop 12 eq-1}
     \triangle:=\sum_{k,j}\sum_{(i,t)\neq (i',t')\in\widetilde{B}^-_{k,j}}\sum_{(r,s)\neq (r',s')\in\widetilde{B}_{k,j}}K_{irts}K_{i'r't's'}\text{Cov}_{|S_n}(X_{it}X_{rs},X_{i't'}X_{r's'})=0.
    \end{align}
\end{proposition}

\begin{proof}
The proof  is very similar to the argument for the  term $\mathbf{L}_{1T}$ in the proof of Proposition \ref{prop 10}.
Recall that $K_{irts}=K_{0ir}K_{0ts}$. For each $t\in \widetilde{B}^-_{k,j,T}$, $K_{0ts}$ is non-zero if and only if $s\in\widetilde{B}_{k,j,T}$ satisfies
\begin{align}
    v_0-Tb\leq |t-s|\leq v_0+ Tb.
\end{align}
Meanwhile, for each $j$, because of $t\in \widetilde{B}^-_{k,j,T}$ and $s\in \widetilde{B}_{k,j,T}$, $|t-s|\leq p_T+q_n<2p_T$ holds for every given $T$. Since  $Tb=o(T)$ and $p_T=o(T)$ holds simultaneously,  for any given $v_0$, $n=T$ indicates that there exists some $N_0\in\mathbb{N}$ independent of sample size $n$ and $(k,j)$ such that
\begin{align}
    \{s\in\widetilde{B}_{k,j,T}:|t-s|<2p_T\}\bigcap \{s\in\widetilde{B}_{k,j,T}: v_0-Tb\leq |t-s|\leq v_0+ Tb\}=\emptyset
\end{align}
holds for every $n\geq N_0$. This implies $K_{0ts}=0$ holds for every $k,j$, when $n\geq N_0$. Thus, $\triangle=0$ when $n\geq N_0$, provided that $n=T$.
\end{proof}
\begin{proposition}
\label{prop 13}
    Under Assumptions \ref{technical assumptions} and \ref{moment conditions}, we obtain
    \begin{align}
        \sum_{(k,j)\neq (k',j')}\sum_{(i,t)\in\widetilde{B}_{k,j}^{-}}\sum_{(i',t')\in\widetilde{B}_{k',j'}^{-}}\sum_{(r,u)\in \widetilde{B}_{k,j}}\sum_{(r',u')\in \widetilde{B}_{k',j'}}K_{irtu}K_{i't'r'u'}\text{Cov}_{|S_n}(X_{it}X_{ru},X_{i't'}X_{r'u'})=o((nT)^2b^3) 
    \end{align}
    $a.s.-[\mathbb{Q}].$
\end{proposition}

\begin{proof}
  The proof is a direct combination of Propositions \ref{prop 11} and \ref{prop 12}.  
\end{proof}

\subsection{Further auxiliary results}
\begin{lemma} (Theorem A, P201, \cite{serfling2009approximation}) 
 \label{lemma 1}
Suppose $\{X_{i}:i=1,...,n\}$ is a set of independent and identically distributed random variables taking value in measurable space $(\mathcal{X},\mathcal{T})$. Define 
\begin{align}
    U_n=\frac{1}{\left(\begin{array}{c}
          m \\
          n
    \end{array}\right)}\sum_{i_1<...<i_m}h(X_{i_1},...,X_{i_m}),
\end{align}
where $h:\mathcal{X}^m\rightarrow\mathbb{R}$ is a measurable function such that $||h||_{\infty}<b$, $E[h(X_{i_1},...,X_{i_m})]=:\theta$ and $\text{Var}(h(X_{i_1},...,X_{i_m}))=\sigma^2$. Then, for any $t>0$ and $n>m$, we have 
   \begin{align}
       \mathbb{P}(|U_n-\theta|>t)\leq 2\exp\Big(-\frac{[n/m]t^2}{2(\sigma^2+\frac{b-\theta}{3}t)}\Big).
   \end{align}    
\end{lemma}
\begin{lemma} (Lemma 2, \cite{YOSHIHARA1987143})
    \label{lemma 2}
    Suppose $(\xi_n)_{n\in\mathbb{N}}$ be a group of $\mathbb{R}^p$-valued random vectors. Let $Y$ be a $\mathcal{M}_{1}^k$-measurable $\mathbb{R}^{d_1}$-valued random vector and $Z$ be a $\mathcal{M}_{k+m}^{\infty}$-measurable $\mathbb{R}^{d_2}$-valued random vector, where $\mathcal{M}_{k_1}^{k_2}$ is the sigma algebra generated by $\xi_{k}$, $k_1\leq k\leq k_2$. Given a Borel measurable function $h(y,z)$ on $\mathbb{R}^{d_1+d_2}$, we assume there exists some $\delta>0$ such that 
    $$M=:\max\left\{E|h(Y,Z)|^{1+\delta}, \int_{\mathbb{R}^{d_1}}\int_{\mathbb{R}^{d_2}}|h(y,z)|^{1+\delta}d\mathbb{P}_{Y}d\mathbb{P}_{Z}\right\}$$ 
    is finite, where $\mathbb{P}_X$ denotes the distribution of random vector $X$. Then, we get 
    \begin{align}
        &E|E[h(Y,Z)|\mathcal{M}_{1}^k]-g(Y)|\leq 3M^{\frac{1}{1+\delta}}\beta^{\frac{\delta}{1+\delta}} (m),
    \end{align}
    where $g(y)=\int_{\mathbb{R}^{d_2}}h(y,z)d\mathbb{P}_{Z}=E[h(y,Z)]$ and function $\beta$ is introduced in Assumption \ref{dependence}.
\end{lemma}
\begin{lemma}
    \label{lemma 3} (Lemma 1, \cite{yoshihara1976limiting}) Suppose $(\xi_n)_{n\in\mathbb{N}}$ be a group of $\mathbb{R}^p$-valued random vectors. Let $i_1<i_2<\dots<i_k$ be arbitrary positive integers. For any $1\leq j\leq k-1$, we denote $\mathbb{P}_{j}^k=\mathbb{P}_{\xi_{i_1},..,\xi_{i_j}}\bigotimes\mathbb{P}_{\xi_{i_j+1},..,\xi_{i_k}}$ and $\mathbb{P}_0^k=\mathbb{P}_{\xi_{i_1},..,\xi_{i_j}}$. Let $h(x_1,...,x_k)$ be any Borel measurable function such that $\int_{\mathbb{R}^{kp}}|h(x_1,...,x_k)|^{1+\delta}d\mathbb{P}^{k}_{j}=:M'$ is finite. Then, 
    \begin{align}
        \left|\int_{\mathbb{R}^{pk}}h(x_1,..,x_k)d\mathbb{P}_{0}^k- \int_{\mathbb{R}^{pk}}h(x_1,..,x_k)d\mathbb{P}_{j}^k\right|\leq 4(M')^{\frac{1}{1+\delta}}\beta^{\frac{\delta}{1+\delta}}(i_{j+1}-i_j).
    \end{align}
\end{lemma}

\begin{lemma}
    \label{lemma 4}
     Suppose $A_{Ts}=\{a_{s1},...,a_{s1}+m_{A}-1\}$ and $B_{Ts}=\{b_{s1},...,b_{s1}+m_B-1\}$ are all subsets of $\{1,2,..,T\}$, where $s\lesssim [T/m_A]$, $\lim_{T}(m_A\land m_B)=\infty$ and $m_A\leq m_B$. Suppose $b$ is the bandwidth of kernel function satisfying $\lim_{T\to\infty}b\searrow0$ and $\lim_{T\to\infty}Tb\nearrow\infty$. Assume $\lim_{T\to\infty}\rho\nearrow\infty$ and $\rho\leq  m_A$. For any given $u_{0}\in (0,1)$ and each $s$, when $\lim_{T}\frac{m_B}{T}=1$, we have 
        \begin{align}
        \label{lemma 4 eq-1}
            &\sum_{t\in A_{Ts}}\sum_{t',t''\in B_{Ts}\atop
            0<|t'-t''|\leq \rho}K\Big(\frac{|t-t'|/T-u_{0}}{b}\Big)K\Big(\frac{(t-t'')/T-u_{0}}{b}\Big)\lesssim m_{A}m_{B}\rho b,\\
            \label{lemma 4 eq-2}
             &\sum_{t\in A_{Ts}}\sum_{t'\in B_{Ts}}K^2\Big(\frac{(t-t')/T-u_{0}}{b}\Big)\lesssim m_{A}m_{B} b.
        \end{align}
\end{lemma}

\begin{proof}
 In the proof, we denote $A_{Ts}$ and $B_{Ts}$ as $A_T$ and $B_T$ since our following arguments hold uniformly over $s$. First, the following basic equality holds for every given $T$ and $b$.
\begin{align}
\mathcal{K}_T:= &\sum_{t\in A_T}\sum_{(t',t'')\in B_T^2
      \atop
      0<|t'-t''|<\rho}K\Big(\frac{(t-t')/T-u_{0}}{b}\Big)K\Big(\frac{(t-t'')/T-u_{0}}{b}\Big)\\
     = &\sum_{t\in A_T}\sum_{(t',t'')\in B_T^2
      \atop
      0<|t'-t''|<\rho}K\Big(\frac{\frac{m_{A}}{T}\frac{t}{m_{A}}-\frac{m_{B}}{T}\frac{t'}{m_{B}}-u_{0}}{b}\Big)K\Big(\frac{\frac{m_{A}}{T}\frac{t}{m_{A}}-\frac{m_{B}}{T}\frac{t''}{m_{B}}-u_{0}}{b}\Big).
\end{align}
By denoting the $l$-th and $j$-th smallest elements of sets $A_{T}$ and $B_{T}$ as $a_{l}$ and $b_{j}$ respectively,
\begin{align}
      &\mathcal{K}_T\lesssim\sum_{t\in A_T}\sum_{t_1\in B_T}\sum_{
     k=1}^{[\rho]}K\Big(\frac{\frac{m_{A}}{T}\frac{t-a_{l}+1}{m_{A}}-\frac{m_{B}}{T}\frac{t_1-b_{j}+1}{m_{B}}-v_{ljT}}{b}\Big)K\Big(\frac{\frac{m_{A}}{T}\frac{t-a_{l}+1}{m_{A}}-\frac{m_{B}}{T}\frac{t_1-b_{j}+1}{m_{B}}+\frac{[\rho]}{T}\frac{k}{[\rho]}-v_{ljT}}{b}\Big)\\
     & =\sum_{t=a_{1}-a_{l}+1}^{a_{1}+m_A-a_{l}}\sum_{t_1=b_{1}-b_{j}+1}^{b_{1}+m_B-b_{j}-[\rho]}K\Big(\frac{\frac{m_{A}}{T}\frac{t}{m_{A}}-\frac{m_{B}}{T}\frac{t_1}{m_{B}}-v_{ljT}}{b}\Big) \sum_{
     k=1}^{[\rho]}K\Big(\frac{\frac{m_{A}}{T}\frac{t}{m_{A}}-\frac{m_{B}}{T}\frac{t_1}{m_{B}}+\frac{[\rho]}{T}\frac{k}{[\rho]}-v_{ljT}}{b}\Big)  \\ & =: \mathcal{K}'_T,
\end{align}
where $v_{ljT}=u_{0}-\frac{a_{l}+b_{j}-2}{T}$ and $[\rho]$ is the integer part of $\rho>0$. When $\lim_{T}\frac{m_{B}}{T}=1$, by naming $j=m_{B}$, we obtain $\frac{b_{m_{B}}}{T}\to 1$ and $\lim_{T}v_{lm_{B}T}\leq u_{0}-1<0$ holds for any $u_{0}\in (0,1)$ and $l=1,2,..,m_{A}$. Then, by letting $l=1$, some simple algebra generates
\begin{align}
    \mathcal{K}'_{T}&=\sum_{t=1}^{m_A}\sum_{t_1=-m_B}^{-[\rho]}K\Big(\frac{\frac{m_{A}}{T}\frac{t}{m_{A}}-\frac{m_{B}}{T}\frac{t_1}{m_{B}}-v_{1m_BT}}{b}\Big) \sum_{
     k=1}^{[\rho]}K\Big(\frac{\frac{m_{A}}{T}\frac{t}{m_{A}}-\frac{m_{B}}{T}\frac{t_1}{m_{B}}+\frac{[\rho]}{T}\frac{k}{[\rho]}-v_{1m_BT}}{b}\Big)\\
     &=\sum_{t=1}^{m_A}\sum_{t_1=1}^{m_B-[\rho]+1}K\Big(\frac{\frac{m_{A}}{T}\frac{t}{m_{A}}+\frac{m_{B}}{T}\frac{t_1}{m_{B}}+\frac{[\rho]-1}{T}-v_{1m_BT}}{b}\Big) \sum_{
     k=1}^{[\rho]}K\Big(\frac{\frac{m_{A}}{T}\frac{t}{m_{A}}+\frac{m_{B}}{T}\frac{t_1}{m_{B}}+\frac{[\rho]}{T}\frac{k}{[\rho]}-v_{1m_BT}}{b}\Big) .
\end{align}
Define $z_{1T}=\frac{m_A}{T}\frac{t}{m_A}+\frac{[\rho]-1}{T}-v_{1m_B T}$ and $z_{2T}=\frac{m_A}{T}\frac{t}{m_A}+\frac{[\rho]}{T}\frac{k}{[\rho]}-v_{1m_B T}$. The assumptions that $\rho=o(m_A)$ and $m_A=o(T)$ assert $\lim_{T}z_{1T}>0$ and $\lim_{T}z_{2T}>0$, which indicates that there exists sufficiently large $T_0$ independent of $T$ such that $\min\{z_{1T},z_{2T}\}>0$ holds for any $T>T_0$. Then, according to the monotonicity of kernel function mentioned in Section \ref{sec23}, we have
\begin{align}
    \mathcal{K}'_{T}\leq m_A \rho \sum_{t_1=1}^{m_B}K^2\Big(\frac{\frac{m_B}{T}\frac{t_1}{m_B}}{b}\Big)\lesssim m_Am_B\rho b \int_{-1}^{1}K^2(u)du.
\end{align}
Thus, we finish the proof of \eqref{lemma 4 eq-1}. As for the proof of \eqref{lemma 4 eq-2}, regarding the it is actually a special case of \eqref{lemma 4 eq-1} and the proof is nearly the same, we omit it here for brevity.
\end{proof}
\begin{lemma}
    \label{lemma 5}
    Based on the $m_A$, $m_B$, $A_{Ts}$, $B_{Ts}$, $b$ and $\rho$ introduced in Lemma \ref{lemma 4}, we assume $\lim_{T}\frac{m_B}{T}=0$ and $B_{Ts}\cap B_{Ts'}=\emptyset$ hold for any $1\leq s\neq s'\leq [T/m_B]+1$. Then, for any given $u_{0}\in (0,1)$, we have the following conclusions hold.
    \begin{itemize}
        \item [1] Suppose $A_{Ts}$ and $B_{Ts}$ have the same smallest element. Provided that $\lim_{T}\frac{m_B}{Tb}=\infty$ and $\lim_{T}\frac{m_A}{m_B}=1$ hold, we have 
         \begin{align}
            \label{lemma 5 eq-1}
            &\sum_{s=1}^{[T/m_B]}\sum_{t\in A_{Ts}}\sum_{t',t''\in B_{Ts}\atop
            0<|t'-t''|\leq \rho}K\Big(\frac{(t-t')/T-u_{0}}{b}\Big)K\Big(\frac{(t-t'')/T-u_{0}}{b}\Big)\lesssim m_Am_B\rho(\log T),\\
            \label{lemma 5 eq-2}
            &\sum_{s=1}^{[T/m_B]}\sum_{t\in A_{Ts}}\sum_{t'\in B_{Ts}}K\Big(\frac{(t-t')/T-u_{0}}{b}\Big)\lesssim m_Am_B(\log T).
        \end{align}
        \item [2] Suppose $A_{Ts}$ and $B_{Ts}$ have the same largest element. Provided that  $\lim_{T}\frac{m_B}{Tb}=\infty$ and $\lim_{T}\frac{m_A}{m_B}=0$ hold, 
        \begin{align}
            \label{lemma 5 eq-3}
            &\sum_{s=1}^{[T/m_B]}\sum_{t\in A_{Ts}}\sum_{t',t''\in B_{Ts}\atop
            0<|t'-t''|\leq \rho}K\Big(\frac{(t-t')/T-u_{0}}{b}\Big)K\Big(\frac{(t-t'')/T-u_{0}}{b}\Big)\lesssim m_Am_B\rho(\log T),\\
            \label{lemma 5 eq-4}
            &\sum_{s=1}^{[T/m_B]}\sum_{t\in A_{Ts}}\sum_{t'\in B_{Ts}}K\Big(\frac{(t-t')/T-u_{0}}{b}\Big)\lesssim m_Am_B(\log T).
        \end{align} 
    \end{itemize}       
\end{lemma}

\begin{proof}
We first focus on proving \eqref{lemma 5 eq-1}. Based on the conditions introduced in Lemma \ref{lemma 5}, we know $A_{Ts}\subset B_{Ts}$ and $a_{s1}=b_{s1}$. Repeating the arguments used in the proof of Lemma \ref{lemma 4} obtains 
\begin{align}
    &\sum_{s=1}^{[T/m_B]}\sum_{t\in A_{Ts}}\sum_{t',t''\in B_{Ts}\atop
            0<|t'-t''|\leq \rho}K\Big(\frac{(t-t')/T-u_{0}}{b}\Big)K\Big(\frac{(t-t'')/T-u_{0}}{b}\Big)\\
            =&\sum_{s=1}^{[T/m_B]}\sum_{t=a_{s1}-a_{sl}+1}^{a_{s1}+m_A-a_{sl}}\sum_{t_1=b_{s1}-b_{sj}+1}^{b_{s1}+m_B-b_{sj}-[\rho]}K\Big(\frac{\frac{m_{A}}{T}\frac{t}{m_{A}}-\frac{m_{B}}{T}\frac{t_1}{m_{B}}-v_{ljT}}{b}\Big) \sum_{
     k=1}^{[\rho]}K\Big(\frac{\frac{m_{A}}{T}\frac{t}{m_{A}}-\frac{m_{B}}{T}\frac{t_1}{m_{B}}+\frac{[\rho]}{T}\frac{k}{[\rho]}-v_{ljT}}{b}\Big) \\
     =&:\sum_{s=1}^{[T/m_B]}\mathcal{K}_{Ts},
\end{align}
where $a_{sl}$ and $b_{sj}$ denotes the $l$-th and $j$-th smallest elements in set $A_{Ts}$ and $B_{Ts}$ respectively. Setting $l=j=1$ implies $v_{11T}=u_{0}-\frac{a_{s1}+b_{s1}-2}{T}=u_{0}-\frac{2b_{s1}}{T}+\frac{2}{T}$. 

\par First, we show that it suffices to only consider the $B_{Ts}$'s satisfying $b_{s1}\in [\frac{u_{0}}{2}-\epsilon T,\frac{u_{0}}{2}+\epsilon T]$, where $\epsilon$ is allowed to be arbitrary small $\epsilon>0$ independent of $T$ and $s$. Since $|b_{s1}-\frac{u_{0}}{2}|>\epsilon T$ implies $\lim_{T}|v_{11T}|>\epsilon>0$, the $\frac{v_{11T}}{b}$ is the leading term of $\frac{\frac{m_{A}}{T}\frac{t}{m_{A}}-\frac{m_{B}}{T}\frac{t_1}{m_{B}}+\frac{[\rho]}{T}\frac{k}{[\rho]}-v_{11T}}{b}$ and $\frac{\frac{m_{A}}{T}\frac{t}{m_{A}}-\frac{m_{B}}{T}\frac{t_1}{m_{B}}-v_{11T}}{b}$. Then, there exists some $T_1>0$ independent of $T$ and $s$ such that
\begin{align}
\label{proof lemma 5 eq-1}
    \min\left\{\left|\frac{\frac{m_{A}}{T}\frac{t}{m_{A}}-\frac{m_{B}}{T}\frac{t_1}{m_{B}}+\frac{[\rho]}{T}\frac{k}{[\rho]}-v_{11T}}{b}\right|,\left|\frac{\frac{m_{A}}{T}\frac{t}{m_{A}}-\frac{m_{B}}{T}\frac{t_1}{m_{B}}-v_{11T}}{b}\right|\right\}>1
\end{align}
holds for any $s$ and $T>T_1$. Then, when $T>T_1$, $\mathcal{K}_{Ts}=0$ which indicates $\sum_{s=1}^{[T/m_B]}\mathcal{K}_{Ts}=0$.

\par Second, another crucial observation is that assumptions $\lim_{T}\frac{m_B}{m_A}=1$ implies
\begin{align}
 b^{-1}\Big(\frac{m_{A}}{T}\frac{t}{m_{A}}-\frac{m_{B}}{T}\frac{t_1}{m_{B}}\Big)=\left(\frac{m_A}{Tb}\Big(\frac{t}{m_A}-\frac{t_1}{m_B}\Big)-\frac{m_B-m_A}{Tb}\frac{t_1}{m_B}\right)=O(\frac{m_A}{Tb})=O(\frac{m_B}{Tb}).
\end{align}

Together with the assumption $\lim_{T}\frac{m_B}{Tb}=\infty$, we now show that we only need to pay attention to the $B_{Ts}$'s satisfying $b_{1s}\in [(\frac{u_{0}}{2}-m_{B})(1-c_T),(\frac{u_{0}}{2}-m_{B})(1+c_T)]$ or $b_{1s}\in [(\frac{u_{0}}{2}+m_{B})(1-c_T),(\frac{u_{0}}{2}+m_{B})(1+c_T)]$, where $c_T$ is any given diverging positive sequence diverging to $\infty$. For convenience, we let $c_T=\log T$.  

\par If $b_{1s}\in (-\infty,(\frac{u_{0}}{2}-m_{B})(1+\log T))\cup((\frac{u_{0}}{2}-m_{B})(1-\log T),(\frac{u_{0}}{2}+m_{B})(1-\log T))\cup((\frac{u_{0}}{2}+m_{B})(1+\log T),\infty)$, it can be easily shown that the leading terms of $\frac{\frac{m_{A}}{T}\frac{t}{m_{A}}-\frac{m_{B}}{T}\frac{t_1}{m_{B}}+\frac{[\rho]}{T}\frac{k}{[\rho]}-v_{11T}}{b}$ and $\frac{\frac{m_{A}}{T}\frac{t}{m_{A}}-\frac{m_{B}}{T}\frac{t_1}{m_{B}}-v_{11T}}{b}$ are either $\frac{m_B}{Tb}$ or $\frac{v_{11}}{b}$ and both of them diverge as $T\nearrow \infty$. Thus, there exists some $T_2>0$ independent of $T$ and $s$ such that \eqref{proof lemma 5 eq-1} holds for any $T>T_2$. Thus, $\mathcal{K}_{Ts}=0$ holds for any $T>T_2$ and $s=1,2,...,[T/m_B]$, which proves the necessity of focusing solely on the $B_{Ts}$'s such that $b_{1s}\in [(\frac{u_{0}}{2}+m_{B})(1-\log T),(\frac{u_{0}}{2}+m_{B})(1+\log T)]$ or $b_{1s}\in [(\frac{u_{0}}{2}-m_{B})(1-\log T),(\frac{u_{0}}{2}-m_{B})(1+\log T)]$

\par Considering the width of these two intervals is of order $m_B\log T$ and $B_{Ts}$'s are disjoint, the quantity of $B_{Ts}$ ($A_{Ts}$) included is of order $\log T$. Hence, we obtain 
\begin{align}
    \sum_{s=1}^{[T/m_B]}\mathcal{K}_{Ts}\lesssim (\log T) m_Am_B\rho.
\end{align}
Then, we finish the proof of \eqref{lemma 5 eq-1}. To prove \eqref{lemma 5 eq-3}, we can simply regard the $B_{Ts}\backslash A_{Ts}$ defined in the proof of \eqref{lemma 5 eq-1} as the $A_{Ts}$ mentioned in precondition of \eqref{lemma 5 eq-3}. Then, by simply repeating the arguments above, we can prove \eqref{lemma 5 eq-3}. As for \eqref{lemma 5 eq-2} and \eqref{lemma 5 eq-4}, they are actually special cases of \eqref{lemma 5 eq-1} and \eqref{lemma 5 eq-3} and the proof is basically the same. We hence omit them here for brevity.
\end{proof}

\begin{lemma}
    \label{lemma 6}
    Based on the $m_A$, $m_B$, $A_{Ts}$, $B_{Ts}$, $b$ and $\rho$ introduced in Lemma \ref{lemma 4}, we assume $\lim_{T}\frac{m_B}{T}=0$ and $B_{Ts}\cap B_{Ts'}=\emptyset$ hold for any $1\leq s\neq s'\leq [T/m_B]+1$. Then, for any given $u_{0}\in (0,1)$, we have the following conclusions hold.
    \begin{itemize}
        \item [1] Suppose $A_{Ts}$ and $B_{Ts}$ have the same smallest element. Provided that $\lim_{T}\frac{m_B}{Tb} < C$, for a sufficiently large constant $C$, and $\lim_{T}\frac{m_A}{m_B}=1$ hold, we have 
         \begin{align}
            \label{lemma 6 eq-1}
            &\sum_{s=1}^{[T/m_B]}\sum_{t\in A_{Ts}}\sum_{t',t''\in B_{Ts}\atop
            0<|t'-t''|\leq \rho}K\Big(\frac{(t-t')/T-u_{0}}{b}\Big)K\Big(\frac{(t-t'')/T-u_{0}}{b}\Big)\lesssim Tbm_A\rho\log T,\\
            \label{lemma 6 eq-2}
            &\sum_{s=1}^{[T/m_B]}\sum_{t\in A_{Ts}}\sum_{t'\in B_{Ts}}K\Big(\frac{(t-t')/T-u_{0}}{b}\Big)\lesssim Tbm_A\log T.
        \end{align}
        \item [2] Suppose $A_{Ts}$ and $B_{Ts}$ have the same largest element. Provided that  $\lim_{T}\frac{m_B}{Tb}=0$ and $\lim_{T}\frac{m_A}{m_B}=0$ hold, 
        \begin{align}
            \label{lemma 6 eq-3}
            &\sum_{s=1}^{[T/m_B]}\sum_{t\in A_{Ts}}\sum_{t',t''\in B_{Ts}\atop
            0<|t'-t''|\leq \rho}K\Big(\frac{(t-t')/T-u_{0}}{b}\Big)K\Big(\frac{(t-t'')/T-u_{0}}{b}\Big)\lesssim Tbm_A\rho \log T,\\
            \label{lemma 6 eq-4}
            &\sum_{s=1}^{[T/m_B]}\sum_{t\in A_{Ts}}\sum_{t'\in B_{Ts}}K\Big(\frac{(t-t')/T-u_{0}}{b}\Big)\lesssim Tbm_A \log T.
        \end{align} 
    \end{itemize}       
\end{lemma}

%

%
%

%
%
\begin{proof}
    Similar to the proof of Lemma \ref{lemma 5}, we here only focus on \eqref{lemma 6 eq-1}. More specifically, since the proof of \eqref{lemma 6 eq-1} is very similar to that of \eqref{lemma 5 eq-1}, we here only sketch the difference. Based on $\mathcal{K}_{Ts}$ introduced in the proof of Lemma \ref{lemma 5}, by letting $l=j=1$, we obtain
\begin{align}
    \mathcal{K}_{Ts}=\sum_{t=1}^{m_A}\sum_{t_1=1}^{m_B}K\Big(\frac{\frac{m_{A}}{T}\frac{t}{m_{A}}-\frac{m_{B}}{T}\frac{t_1}{m_{B}}-v_{11T}}{b}\Big) \sum_{
     k=1}^{[\rho]}K\Big(\frac{\frac{m_{A}}{T}\frac{t}{m_{A}}-\frac{m_{B}}{T}\frac{t_1}{m_{B}}+\frac{[\rho]}{T}\frac{k}{[\rho]}-v_{11T}}{b}\Big) ,
\end{align}
where $v_{11T}=u_{0}-\frac{a_{s1}+b_{s1}-2}{T}=u_{0}-\frac{2b_{s1}-2}{T}$.
Since $\lim_{T}\frac{m_{A}}{m_{B}}=1$ and $\frac{m_{B}}{Tb}\leq C$, for a large constant $C$, we have $\frac{m_B-m_A}{Tb}=o(1)$, which further implies
\begin{align}
    \frac{m_{A}}{T}\frac{t}{m_{A}}-\frac{m_{B}}{T}\frac{t_1}{m_{B}}+\frac{[\rho]}{T}\frac{k}{[\rho]}=O(b),\ \frac{m_{A}}{T}\frac{t}{m_{A}}-\frac{m_{B}}{T}\frac{t_1}{m_{B}}=O(b).
\end{align}
This asserts that
\begin{align}
     &b^{-1}\Big(\frac{m_{A}}{T}\frac{t}{m_{A}}-\frac{m_{B}}{T}\frac{t_1}{m_{B}}+\frac{[\rho]}{T}\frac{k}{[\rho]}-v_{11T}\Big)=(C_1+b^{-1}v_{11T})(1+o(1)),\\ &b^{-1}\Big(\frac{m_{A}}{T}\frac{t}{m_{A}}-\frac{m_{B}}{T}\frac{t_1}{m_{B}}-v_{11T}\Big)=(C_2+b^{-1}v_{11T})(1+o(1)),
\end{align}
where $C_1$ and $C_2$ are two universal constants independent of $T$ and $s$. Thus, we only need to focus on the cases where $B_{Ts}$'s satisfying $b_{s1}\in [\frac{u_{0}}{2}- T b(1+c_T),\frac{u_{0}}{2}+ T b(1+c_T)]$, where $c_T$ is arbitrary positive number sequence such that $\lim_T c_T\nearrow\infty$. Without loss of generality, by setting $c_T=\log T-1$, the number of such kind of $B_{Ts}$'s is of order $\frac{Tb\log T}{m_B}$. Then, we obtain 
\begin{align}
    \sum_{s=1}^{[T/m_B]}\mathcal{K}_{Ts}\lesssim Tbm_A\rho\log T.
\end{align}
\end{proof}

\end{appendix}

\end{document}